\documentclass{ws-rmta}

\usepackage[utf8]{inputenc}
\usepackage{amsmath}
\usepackage{amsfonts}
\usepackage{amssymb}
\usepackage{graphicx}
\usepackage{enumerate}
\usepackage{xcolor}
\usepackage{enumitem}
\usepackage{tcolorbox}
\tcbuselibrary{skins,breakable}
\tcbset{before= \vspace{8pt}}

\numberwithin{equation}{section}
\newcommand{\I}{\mathbf{I}}
\newcommand{\dd}{\mathbf{d}}
\newcommand{\y}{\mathbf{y}}
\newcommand{\W}{\mathcal{W}}
\renewcommand{\S}{\mathbf{S}}
\newcommand{\bs}{\boldsymbol}
\newcommand{\Imm}{\mathrm{Im}}
\renewcommand{\labelitemi}{\textendash}
\newtheorem{assum}{Assumption}
\newtheorem{defn}{Definition}
\numberwithin{equation}{section}

\begin{document}

\title{On the asymptotic behavior of the eigenvalue distribution of block correlation matrices of high-dimensional time series}

\author{Philippe Loubaton}
\address{Laboratoire d'Informatique Gaspard Monge, UMR 8049, Universit\'e Paris-Est Marne la Vall\'ee \\
5 Bd. Descartes, Cite Descartes, 77454 Marne la Vallée Cedex 2, France}

\author{Xavier Mestre}
\address{Centre Tecnol\`ogic de Telecomunicacions de Catalunya\\
Av. Carl Friedrich Gauss, 7, Parc Mediterrani de la Tecnologia\\
08860 Castelldefels, Spain}

	\setcounter{page}{1}

\maketitle

\begin{abstract}
%We consider a set of $M$ mutually independent Gaussian scalar time series that are observed between time $1$ to time $N$.
We consider linear spectral statistics built from the block-normalized correlation matrix of a set of $M$ mutually independent scalar time series. This matrix is composed of $M \times M$ blocks that contain the sample cross correlation between pairs of time series. In particular, each block has size $L \times L$ and contains the sample cross-correlation measured at $L$ consecutive time lags between each pair of time series. Let $N$ denote the total number of consecutively observed windows that are used to estimate these correlation matrices. We analyze the asymptotic regime where $M,L,N \rightarrow +\infty$ while $ML/N \rightarrow c_\star$, $0<c_\star<\infty$. We study the behavior of linear statistics of the eigenvalues of this block correlation matrix under these asymptotic conditions and show that the empirical eigenvalue distribution converges to a Marcenko-Pastur distribution.
Our results are potentially useful in order to address the problem of testing whether a large number of time series are uncorrelated or not. 
\end{abstract}


\section{Introduction}

\subsection{Problem addressed and motivation}

We consider a set of $M$ jointly stationary zero mean complex-valued scalar time series,
denoted as $y_{1,n} ,\ldots,y_{M,n}$, where $n\in
\mathbb{Z}$. We assume that the joint distribution of $\left((y_{m,n})_{n \in \mathbb{Z}}\right)_{m=1, \ldots, M}$ is the circularly symmetric complex Gaussian law\footnote{
Any finite linear combination $z = \sum_{m=1}^{M} \sum_{j=1}^{J} \alpha_j y_{m,n_j}$ of the random 
variables $((y_{m,n})_{n \in \mathbb{Z}})_{m=1, \ldots, M}$ is distributed according to the distribution 
$\mathcal{N}_\mathbb{C}(0, \delta^{2})$, i.e. $\mathrm{Re} z$ and $\mathrm{Im} z$ are independent and 
$\mathcal{N}(0, \delta^{2}/2)$ distributed, where $\delta^{2} > 0$ is the corresponding variance}. In this paper, 
we study the behaviour of linear statistics of the eigenvalues of a certain large random matrix built from the available data when the $M$ time series $(y_m)_{m=1, \ldots, M}$ are uncorrelated (i.e. independent), assuming that both the number of available samples and the number of series are large. Our results are potentially useful in order to address the problem of testing whether a large number of time series are uncorrelated or not. 

In order to introduce the large random matrix models that we will address in the following, we consider a column vector gathering $L$
consecutive observations of the $m$th time series starting at time $n$, namely
\begin{equation*}
\mathbf{y}^{L}_{m,n}=\left[ y_{m,n},\ldots,y_{m,n+L-1}\right] ^{T}
\end{equation*}
and from this build an $ML$-dimensional column vector
\begin{equation*}
\mathbf{y}^{L}_n=\left[ \left(\mathbf{y}_{1,n}^{L}\right)^{T},\ldots,\left(\mathbf{y}_{M,n}^{L}\right)^{T}\right] ^{T}.
\end{equation*}
We will denote by $\mathcal{R}_{L}$ the $ML\times ML$
covariance matrix of this random vector, i.e. $\mathcal{R}_{L}=\mathbb{E}
\left[ \mathbf{y}^{L}_{n} \left(\mathbf{y}^{L}_n\right)^{H}\right]$ where $($\textperiodcentered$)^{H}$ stands for transpose conjugate. This matrix is sometimes referred to as the spatio-temporal covariance matrix.  Clearly, the $M$ series 
$(y_m)_{m=1, \ldots, M}$ are uncorrelated, to be referred to as the hypothesis $\mathrm{H}_0$ in the following, if and only if, for each integer $L$, matrix $\mathcal{R}_{L}$ is block-diagonal, namely
\begin{equation*}
\mathcal{R}_{L}=\mathrm{Bdiag}\left( \mathcal{R}_{L}\right)
\end{equation*}
where, for an $ML \times ML$ matrix $\mathbf{A}$, $\mathrm{Bdiag}\left( \mathbf{A}\right) $ is the block-diagonal matrix
of the same dimension whose $L\times L$ blocks are those of $\mathbf{A}$. We notice that 
the $L \times L$ diagonal blocks of $\mathrm{Bdiag}\left( \mathcal{R}_{L}\right)$ are the $L \times L$ Toeplitz matrices $\mathcal{R}_{m,L}$, $m=1,\ldots,M$, defined by
\begin{equation}
\label{eq:defRmL}
\left\{ \mathcal{R}_{m,L}\right\} _{k,k^{\prime}}=r_{m}\left(k-k^{\prime}\right).
\end{equation}
Here, $r_{m}\left( k\right) $, $k\in\mathbb{Z}$, is the covariance sequence
of the $m$th time series, defined as
\begin{equation}
r_{m}\left( k\right) =\int_{0}^{1}\mathcal{S}_{m}\left( \nu\right) \mathrm{e}
^{2\pi\mathrm{i}\nu k}d\nu   \label{eq:defcovarianceseq}
\end{equation}
where, for each $m$, $\mathcal{S}_{m}$ represents the spectral density of 
$(y_{m,n})_{n \in \mathbb{Z}}$. 
We will denote by $\mathcal{R}_{\mathrm{corr},L}$ the block correlation matrix defined by 
\begin{equation}
\label{eq:def-generalized-correlation-matrix}
%\mathcal{R}_{\mathrm{corr},L} = \left( \mathrm{Bdiag}\left( \mathcal{R}_{L}\right) \right)^{-1/2} \mathcal{R}_{L} \left( \mathrm{Bdiag}\left( \mathcal{R}_{L}\right) \right)^{-1/2}
\mathcal{R}_{\mathrm{corr},L} = \mathcal{B}^{-1/2}_{L} \mathcal{R}_{L} \mathcal{B}^{-1/2}_{L}
\end{equation}
where 
\begin{equation*}
    \mathcal{B}_L = \mathrm{Bdiag}(\mathcal{R}_L).
\end{equation*}
Consequently, $\mathcal{R}_{L}$ is block diagonal for each $L$ if and only if $\mathcal{R}_{\mathrm{corr},L} = \I_{ML}$ for each $L$. 

A possible way to test whether the time series $(y_m)_{m=1, \ldots, M}$ are uncorrelated thus consists in estimating 
$\mathcal{R}_{\mathrm{corr},L}$ for a suitable value of $L$, and subsequently comparing the corresponding estimate 
with $\I_{ML}$.  
%{\color{red} While we have denoted above by $N$ the number of available samples $(\mathbf{y}_n)_{n=1, \ldots, N}$, we will assume from now on that} for each $m=1, \ldots, M$, the observations $y_{m,1}, \ldots, y_{m,N+L-1}$ are available  where $N$ represents {\color{red} this time}  the number of observations that are averaged to build the test statistic for each time lag.  
We will assume from now on that, for each $m=1, \ldots, M$, the observations $y_{m,1}, \ldots, y_{m,N+L-1}$ are available  where $N$ represents the number of observations that are averaged to build the test statistic for each time lag.  
In the following, we consider the standard sample estimate $\widehat{\mathcal{R}}_{\mathrm{corr},L}$ defined by
\begin{equation}
\label{eq:def-hat-RcorrL}
\widehat{\mathcal{R}}_{\mathrm{corr},L}=\widehat{\mathcal{B}}^{-1/2}_{L} \widehat{\mathcal{R}}_{L} \widehat{\mathcal{B}}^{-1/2}_{L}
\end{equation}
where $\widehat{\mathcal{R}}_{L}$ is the empirical spatio-temporal covariance matrix given by %\footnote{{\color{red} The matrix $\widehat{\mathcal{R}}_{L}$ depends on samples $(\mathbf{y}_n)_{n=1, \ldots, N+L-1}$ while we have assumed that only samples $(\mathbf{y}_n)_{n=1, \ldots, N}$ are available. In order to be in accordance with our assumptions, the estimate $\widehat{\mathcal{R}}_{L}$ should have been defined by $\widehat{\mathcal{R}}_{L} = \frac{1}{N-L} \sum_{n=1}^{N-L} \mathbf{y}^{L}_{n} \left(\mathbf{y}^{L}_n\right)^{H}$. This definition would have complicated the notations, so that we prefer to use the estimate (\ref{eq:def_st_corr_mtx}). This, implicitely, means that the samples $(\mathbf{y}_n)_{n=1, \ldots, N+L-1}$ are actually available. As in the asymptotic regime considered in the paper, the ratio $\frac{L}{N}$ converges towards  $0$, the actual sample size $N+L-1$ can be written as $N(1+o(1))$. Changing $N$ with $N+L-1$ does therefore not modify the significance of the results of this paper}}
\begin{equation}
\label{eq:def_st_corr_mtx}
\widehat{\mathcal{R}}_{L}=\frac{1}{N}\sum_{n=1}^{N}\mathbf{y}^{L}_{n} \left(\mathbf{y
}^{L}_n\right)^{H}
\end{equation}
and where $\widehat{\mathcal{B}}_{L}$ is the corresponding block diagonal
\begin{equation}
    \label{eq:defRhatmL}
    \widehat{\mathcal{B}}_{L}=\mathrm{Bdiag}(\widehat{\mathcal{R}}_{L}) = \left(\begin{array}{ccc}
         \widehat{\mathcal{R}}_{1,L} & &  \\
         & \ddots & \\
         & &\widehat{\mathcal{R}}_{M,L}
    \end{array}\right).
\end{equation}
with $\widehat{\mathcal{R}}_{m,L}$, $m=1,\ldots,M$, denoting the corresponding $L \times L$ diagonal blocks. The expression (\ref{eq:def_st_corr_mtx}) of $\widehat{\mathcal{R}}_{L}$ explains why we assume that $N+L-1$ samples are available, because if the sample size had been defined as $N$, $\widehat{\mathcal{R}}_{L}$ should have been defined by $\widehat{\mathcal{R}}_{L} = \frac{1}{N-L} \sum_{n=1}^{N-L} \mathbf{y}^{L}_{n} \left(\mathbf{y}^{L}_n\right)^{H}$, which would have complicated the notations. In any case, in the asymptotic regime considered in the paper, the ratio $\frac{L}{N}$ converges towards $0$. Therefore, the actual sample size $N+L-1$ can be written as $N(1+o(1))$. Changing $N$ with $N+L-1$ does therefore not modify the significance of the results of this paper. 
\begin{remark}
\label{rem:dimensions}
A relevant question here is how to choose the lag parameter $L$. 
On the one hand, $L$ should be sufficiently large, because this allows to identify correlations among samples in
different time series that are well spaced in time. For instance, two time
series chosen as copies of the same temporally white noise with a relative
delay higher than $L$ lags will be perceived as uncorrelated by examination
of $\widehat{\mathcal{R}}_{corr,L}$, which is of course far from true. 
On the other hand, $L$ should be chosen sufficiently low so that $ML/N \ll 1$ in order to make the
estimation error $\Vert \widehat{\mathcal{R}}_{\mathrm{corr},L} - \I_{ML}\Vert$
reasonably low under the hypothesis $\mathrm{H}_0$.
%
%, to be referred to as the hypothesis $\mathrm{H}_0$ in the following. 
If the number $M$ of time series is large and that the number of observations $N$ is not 
unlimited, the condition $ML/N \ll 1$  requires the selection of a small value for $L$. Such a choice may thus reduce 
drastically the efficiency of the uncorrelation
tests based on  $\| \widehat{\mathcal{R}}_{\mathrm{corr},L} - \I_{ML} \|$. 
Finding statistics having a well defined behaviour under $\mathrm{H}_0$ when $ML$ and $N$ 
are of the same order of magnitude would allow to consider larger values of $L$, thus 
improving the performance of the corresponding tests. 

\end{remark}

In this paper, we propose 
to study the behavior of spectral statistics built from the eigenvalues of $\widehat{\mathcal{R}}_{\mathrm{corr},L}$, which
will be denoted by $(\hat{\lambda}_{k,N})_{k=1,\ldots ,ML}$. More
specifically, we will consider statistics of the form\footnote{ The application of a function $\phi$ to a Hermitian matrix should be understood as directly applied to its eigenvalues.}
\begin{equation}
\label{eq:defLSS}
\widehat{\phi }_{N}=\frac{1}{ML}\mathrm{Tr}\left[ \phi \left( \widehat{
\mathcal{R}}_{\mathrm{corr},L}\right) \right] =\frac{1}{ML}
\sum_{k=1}^{ML}\phi \left( \hat{\lambda}_{k,N}\right)
\end{equation}
where $\phi $ is assumed to be a suitable function, and will study
the behaviour of $\widehat{\phi }_{N}$ under $\mathrm{H}_0$ in asymptotic regimes where $M,N,L$ converge towards $+\infty$ in such a way that $c_N = \frac{ML}{N}$ converges towards a non zero constant $c_* \in (0,+\infty)$. 

The main result of this paper establishes the asymptotic conditions under which $\widehat{\phi }_{N}$ converges almost surely towards the integral of $\phi$ with respect to 
the Marcenko-Pastur distribution. 
In order to analyze the asymptotic behavior of the above class of
statistics, we use large random matrix methods that relate the quantity $\widehat{\phi}_{N}$ with the empirical
eigenvalue distribution of $\widehat{\mathcal{R}}_{\mathrm{corr},L}$,
denoted as
\begin{equation}
    \label{eq:def-edf-Rcorr}
    d\hat {\mu}_{N}(\lambda)=\frac{1}{ML}\sum_{k=1}^{ML}\delta_{\lambda-
\hat{\lambda}_{k,N}},
\end{equation} that is
\begin{equation*}
\widehat{\phi}_{N}=\int\phi\left( \lambda\right) d\hat{\mu}_{N}(\lambda).
\end{equation*}
We will establish the behavior of $\widehat{\phi}_{N}$ by studying the
empirical eigenvalue distribution $d\hat{\mu}_{N}(\lambda)$. 
%More specifically, we will first establish that there exists a deterministic probability measure $d \mu _{N}(\lambda)$ such that, for each bounded continuous function $\phi$ defined on $\mathbb{R}$, it holds that
% \begin{equation}
% \widehat{\phi}_{N}-\int_{\mathbb{R}^{+}}\phi(\lambda)\,d\mu_{N}(\lambda)\rightarrow0   \label{eq:convergence-linear-statistics}
% \end{equation}
% almost surely. We will also prove that,
\begin{defn}
\label{def:MPlaw} 
Let $\mu_{mp,d}$ denote the Marcenko-Pastur distribution of parameter $d$. We recall that for each $d > 0$, $\mu_{mp,d}$ is the limit of the empirical eigenvalue distribution of a large random matrix $\frac{1}{K} {\bf X} {\bf X}^{H}$ where ${\bf X}$ is a $J \times K$ random matrix
with zero mean unit variance i.i.d. entries and where both $J$ and $K$ converge towards $+\infty$ in such a way that $\frac{J}{K} \rightarrow d$. 
\end{defn}
We will prove that, under certain asymptotic assumptions, the statistic $\widehat{\phi}_N$ can be described (up to some error terms) as the integral of $\phi(\lambda)$ with respect to Marchenko-Pastur distribution of parameter $c_N$, in the sense that
\begin{equation}
    \label{eq:conv-MP-intro}
\widehat{\phi}_{N} -\int_{\mathbb{R}^{+}} \phi(\lambda) \, d \mu_{mp,c_N}(\lambda) \rightarrow 0
\end{equation}
almost surely. We will also characterize the rate of convergence to zero of the corresponding error term in (\ref{eq:conv-MP-intro}). This result will establish the conditions under which we can test whether the $M$ time series 
$y_1, \ldots, y_M$ are uncorrelated by comparing linear spectral statistics $\widehat{\phi}_{N}$ with the corresponding limits under $\mathrm{H}_0$ as established above. 
%The detailed study of such class of tests will be conducted in a future work.  

\subsection{On the literature}
\label{sec:literature}
Testing whether $M$ time series are uncorrelated is an important problem 
that has been extensively addressed in the past. Apart from a few works devoted to 
the case where the number of time series $M$ converges towards $+\infty$ (see below), the vast majority of published papers assumed that $M$ is a fixed integer. In this 
context, we can first mention spectral domain approaches based on the observation 
that the $M$ time series $(y_{1,n})_{n \in \mathbb{Z}} , \ldots, (y_{M,n})_{n \in \mathbb{Z}}$ are uncorrelated if and only 
the spectral coherence matrix of the $M$--variate time series $(\y_n)_{n \in \mathbb{Z}}$, where $\y_n = (y_{1,n}, \ldots, y_{M,n})^{T}$, is reduced to 
$\mathbf{I}_M$ at each frequency. Some examples following this approach are 
\cite{wahba1971some}, \cite{taniguchi1996nonparametric}, \cite{eichler2007frequency},
\cite{eichler2008testing}. A number of papers also proposed to develop 
lag domain approaches, e.g. \cite{haugh1976checking}, \cite{hong1996testing}, 
\cite{duchesne2003robust}, \cite{kim2005test} which considered test statistics based on empirical estimates of the autocorrelation coefficients between the residuals of the various time series. See also \cite{elhimdiduchesneroy2003} for a more direct 
approach. 

We next review the very few existing works devoted to the case where the number $M$ of time series converges towards $+\infty$. 
We are just aware of papers addressing the case where the observations $\y_1, \ldots, \y_N$ are independent identically distributed (i.i.d.) and where the ratio $\frac{M}{N}$
converges towards a constant $d \in (0,1)$. In particular, in contrast with the asymptotic regime considered in the present work, these papers assume that $M$ and $N$ are of the same order of magnitude. This is because, in this context, the time series are mutually uncorrelated if and only the covariance matrix $\mathbb{E}(\y_n \y_n^{H})$ is diagonal. Therefore, it is reasonable to consider test statistics that are functionals of the sample covariance matrix $\frac{1}{N} \sum_{n=1}^{N} \y_n \y_n^{H}$. In particular, when the observations are i.i.d. Gaussian random vectors, the generalized likelihood ratio test (GLRT) consists in comparing the test statistics $\log \mathrm{det}(\widehat{\mathcal{R}}_{\mathrm{corr}})$ to a threshold, where $\widehat{\mathcal{R}}_{\mathrm{corr}} = \widehat{\mathcal{R}}_{\mathrm{corr},1}$ represents the sample correlation matrix. 
\cite{jiang2004} proved that under $\mathrm{H}_0$, the empirical eigenvalue distribution of 
$\widehat{\mathcal{R}}_{\mathrm{corr}}$ converges almost surely towards the Marcenko-Pastur distribution $\mu_{mp,d}$ and therefore, that 
$\frac{1}{M} \mathrm{Tr}(\phi(\widehat{\mathcal{R}}_{\mathrm{corr}}))$ converges towards $\int \phi(\lambda) d\mu_{mp,d}(\lambda)$ for each bounded continuous function $\phi$. In the Gaussian case, \cite{jiangyang2013} also established a central limit theorem (CLT) for $\log \mathrm{det}(\widehat{\mathcal{R}}_{\mathrm{corr}})$ under $\mathrm{H}_{0}$ using the moment method. 
\cite{dette-dornemann-2020} remarked that, in the Gaussian real case,  $(\mathrm{det}(\widehat{\mathcal{R}}_{\mathrm{corr}}))^{M/2}$ is the product of independent beta distributed random variables. Therefore, $\log \mathrm{det}(\widehat{\mathcal{R}}_{\mathrm{corr}})$ 
appears as the sum of independent random variables, thus deducing the CLT. We finally mention \cite{mestre2017correlation} in which a CLT on linear statistics of the eigenvalues of 
$\widehat{\mathcal{R}}_{\mathrm{corr}}$ is established in the Gaussian case using large random matrix techniques when the covariance matrix $\mathbb{E}(\y_n \y_n^H)$ is not necessarily diagonal. This allows to study the asymptotic performance of the GLRT under certain class of alternatives. 
%Finally,  \cite{nowak17} studied the asymptotic eigenvalue distribution of the single time-lag correlation matrix for increasingly large $M,N$, under the assumption that the time series are i.i.d. standard Gaussian random variables.

%{\color{red} Not clear whether this paragraph should be removed altogether} Regarding the time-lagged sample correlation matrix estimator, efforts have mainly been devoted to the study of the single-lag estimation matrix. In our notation, this would correspond (up to tail effects) to the matrix that is obtained by selecting every $((m-1)L+\ell,(m-1)L+\ell+\tau)$ entry of $\widehat{\mathcal{R}}_L$, $m=1,\ldots,M$, where $\ell$ is a fixed integer ($1\leq \ell \leq L$) and $\tau$ is the time lag ($0\leq \tau \leq L-\ell$). Of particular interest are the results in \cite{nowak17}, which also contains a good account of the different approaches to this problem that have been developed within the field of mechanical statistics. More specifically, \cite{nowak17} studied the asymptotic eigenvalue distribution of the single time-lag correlation matrix for increasingly large $M,N$, under the assumption that the time series are i.i.d. standard Gaussian random variables and where $\tau$ is either fixed or increases to infinity at the same rate as $N,M$. {\color{red} Perhaps we should also introduce here the work by Philippe and Daria on cross-covariances between past and future}
Regarding the asymptotic behaviour of the empirical eigenvalue distribution of the complete matrix $\widehat{\mathcal{R}}_L$, it seems relevant to highlight the work in \cite{loubaton2016} and \cite{loubaton-mestre-2017},
which also addressed in the asymptotic regime considered in the present paper. More specifically, \cite{loubaton2016} assumed that the $M$ mutually independent time series $y_1, \ldots, y_M$ 
are i.i.d. Gaussian and 
established that the empirical eigenvalue distribution of 
$\widehat{\mathcal{R}}_L$ converges towards the Marcenko Pastur distribution 
$\mu_{mp,c_*}$. Moreover, if $L = \mathcal{O}(N^{\beta})$ with $\beta < 2/3$, 
it is established that almost surely, for $N$ large enough, all the eigenvalues 
of $\widehat{\mathcal{R}}_L$ are located in a neighbourhood of the support 
of $\mu_{mp,c_*}$. In \cite{loubaton-mestre-2017}, the mutually independent time series 
$y_1, \ldots, y_M$ are no longer assumed i.i.d. and it is established that
the empirical eigenvalue distribution has a deterministic behaviour. 
The corresponding deterministic equivalent is characterized, and 
some results on the corresponding speed of convergence are given. 
As it will appear below, the present paper uses extensively in Sections \ref{sec:evalbarphi} and \ref{sec:approximation-MP} the 
tools developed in \cite{loubaton-mestre-2017}.  

We also mention \cite{loubaton-rosuel}, which developed large random matrix 
methods in order to test the hypothesis $\mathrm{H}_0$. However, the approach used in \cite{loubaton-rosuel} is based on the study of the asymptotic 
behaviour of the empirical eigenvalue distribution of a frequency smoothed estimator of the spectral coherence matrix. While the techniques developed in \cite{loubaton-rosuel} appear in general completely different from the technical content of the present paper, we mention that our Section \ref{sec:influenceblockmat} was inspired by Section 4.1 in \cite{loubaton-rosuel}, even though the technical problem solved in section \ref{sec:influenceblockmat} appears harder to solve than that in \cite[Section 4.1]{loubaton-rosuel}. 

We finally point out that a number of previous works addressed the behaviour of the estimated auto-covariance matrix $\hat{{\bf R}}_x(\tau) = \frac{1}{N} \sum_{n=1}^{N-\tau} {\bf x}_{n+\tau} {\bf x}_{n}^{H}$ of a $M$ dimensional time series ${\bf x} = ({\bf x}_n)_{n \in \mathbb{Z}}$ at a given lag $\tau$ in the asymptotic 
regime where $\frac{M}{N} \rightarrow d$ with $d > 0$. We can mention 
\cite{jin-bai-el-al-2014}, \cite{li-pan-yao-jmva-2015}, \cite{liu-aue-paul-2015}, \cite{batta-bose-2016}, 
\cite{nowak17}, which, under various assumptions on ${\bf x}$, study the behaviour of the empirical eigenvalue distribution of $\hat{{\bf R}}_x(\tau) + \hat{{\bf R}}^{H}_x(\tau)$, $\hat{{\bf R}}_x(\tau) \hat{{\bf R}}^{H}_x(\tau)$, symmetric polynomials of $(\hat{{\bf R}}_x(\tau), \hat{{\bf R}}^{H}_x(\tau))$, or of $\hat{{\bf R}}_x(\tau)$. We also mention the work in \cite{loubaton-tieplova-2020}, where the asymptotic behaviour of the singular values distribution of the estimated auto-covariance matrix 
between finite dimensional past and future of ${\bf x}$ (which, up to the end effects, depend 
on matrices $(\hat{{\bf R}}_x(\tau))_{\tau=1, \ldots, K}$ for a fixed integer $K$) is studied
when $\frac{M}{N} \rightarrow d$ with $d > 0$. These 
contributions are not directly related to the present paper in that they study the properties of 
$\hat{{\bf R}}_x(\tau)$ for a single value of $\tau$ (or for a finite number of values of $\tau$ in 
\cite{loubaton-tieplova-2020}) when $M$ and $N$ are of the same order of magnitude, while our random matrix model
depends, up to a block Toeplitzification of matrix $\widehat{\mathcal{R}}_L$, on $(\hat{{\bf R}}_y(\tau))_{\tau=0 \ldots, L}$, where, this time,
$M,N,L$ converge towards $+\infty$ in such a way that $\frac{ML}{N} \rightarrow c_\star$.

%\newline

\subsection{Assumptions}

\begin{assum}
\label{assum:statistics}
The complex scalar time series $y_m$, $m \geq 1$, are mutually independent, stationary, zero mean and circularly symmetric Gaussian distributed with autocovariance sequence $r_m = (r_{m}(k))_{k\in\mathbb{N}}$ defined as $r_m(k) = \mathbb{E}[y_{m,n+k}y^\ast_{m,n}]$ and associated spectral densities  $(\mathcal{S}_m(\nu))_{m \geq 1}$.
\end{assum}
\begin{assum}
\label{as:asymptotic-regime}
All along the paper, we assume that  $M\rightarrow
+\infty,N\rightarrow+\infty$ in such a way that $c_{N}=\frac{ML}{N}
\rightarrow c_\star$, where $0<c_\star<+\infty$, and 
that $L = L(N) = \mathcal{O}(N^{\beta})$ for some constant $\beta \in (0,1)$. In order to shorten the notations, $
N \rightarrow +\infty$ should be understood as the above asymptotic regime.
\end{assum}

%As $M = M(N) \rightarrow +\infty$, we assume that an infinite sequence $(y_m)_{m \geq 1}$ of mutually independent zero mean circularly symmetric complex  Gaussian time series with spectral densities $(\mathcal{S}_m(\nu))_{m \geq 1}$ is given. 
We will need that the spectral densities are bounded above and below
uniformly in $M$, namely
\begin{assum}
\label{ass:bounds-spectral-densities}
The spectral densities are such that
\begin{align}
\sup_{m \geq 1} \max_{\nu\in\lbrack0,1]}\mathcal{S}_{m}(\nu) = s_{max} & <+\infty
\label{eq:upperbound-S} \\
\inf_{m \geq 1} \min_{\nu\in\lbrack0,1]}\mathcal{S}_{m}(\nu) = s_{min} & >0.
\label{eq:lowerbound-S}
\end{align}
\end{assum}

Note that, for each $m=1, \ldots, M$, the matrix $\mathcal{R}_{m,L}$ can be seen as an $L \times L$ diagonal block of an infinite Toeplitz matrix with symbol $\mathcal{S}_m(\nu)$. Therefore, Assumption \ref{ass:bounds-spectral-densities} directly implies that, for each $N$, these matrices verify $s_{min} {\bf I}_L \leq \mathcal{R}_{m,L} \leq s_{max} {\bf I}_L$. This property will be used a number of times throughout the text. 

%{\color{red} [ Is this used anywhere now? ] 
Let us denote by $\mathbf{r}_M$ the $M$-dimensional sequence of covariances, namely
\begin{equation}
    \mathbf{r}_M (k)= \left[r_1(k),\ldots,r_M(k)\right]^T 
    \label{eq:multivariate_cov}
\end{equation}
where $r_m(k)$, $m=1,\ldots,M$ are defined in (\ref{eq:defcovarianceseq}). We can consider the sequence of Euclidean norms $\left\{\Vert\mathbf{r}_M(k)\Vert\right\}_{k\in\mathbb{Z}}$. At some points, we will need the corresponding series to be of order $\mathcal{O}(\sqrt {M})$.
\begin{assum}
\label{as:norm-vector-r} 
The multivariate covariance sequence $\mathbf{r}_M$ defined in (\ref{eq:multivariate_cov}) is such that 
\begin{equation*}
    \sup_{M \geq 1} \frac{1}{\sqrt{M}} \sum_{k\in\mathbb{Z}} \Vert\mathbf{r}_M(k)\Vert < +\infty.
\end{equation*}
\end{assum}
% }

% {\color{blue} Yes, it is useful, we used this assumption in our previous paper 
% to establish that $\frac{1}{ML} \mathrm{Tr}\left[ {\bf A}({\bf R}(z) - \mathbb{E}({\bf Q}(z))) \right] \leq C(z) \, \frac{L}{MN}$. As the proof of this property in the context of the present paper is similar, we still need the assumption.} 
%If $(r_m)_{m \geq 1}$ represents the corresponding infinite autocovariance sequences, 

On the other hand, we will also need to impose a certain rate of decayment of $\sup_{m \geq 1} \sum_{|k|  \geq n+1} |r_m(k)|$ when 
$n \rightarrow +\infty$. To that effect, we introduce the weighting sequence $\left(
\omega(n)\right) _{n\in\mathbb{Z}}$ defined as
\begin{equation*}
\omega(n)=\left( 1+\left\vert n\right\vert \right) ^{\gamma}
\end{equation*}
where $\gamma \geq 0$ is given. This sequence belongs to the class of strong
Beurling weights (see \cite{barry-simon-book}, Chapter 5), which are functions $\omega$ on $\mathbb{Z}$\ with
the properties: (i) $\omega(n)\geq1$, (ii) $\omega(n)=\omega(-n)$, (iii) 
$\omega(m+n)\leq\omega(m)\omega(n)$ for all $m,n\in\mathbb{Z}$ and (iv) 
$n^{-1}\log\omega(n) \rightarrow 0$ as $n \rightarrow \infty$. We define 
$\ell_{\omega}$ as the Banach space of two sided sequences 
$a=\left(a(n)\right) _{n\in\mathbb{Z}}$ such that
\begin{equation*}
\left\Vert a\right\Vert_{\omega}
=\sum_{n=-\infty}^{\infty}\omega(n)\left\vert a(n)\right\vert
=\sum_{n=-\infty}^{\infty}\left( 1+\left\vert n\right\vert \right)^{\gamma}\left\vert a(n) \right\vert < +\infty.
\end{equation*}
When $\gamma=0$, $\omega(n) = 1$ for each $n$, and  $\ell_{\omega}$ coincides with the Wiener algebra $\ell_{1} = \{ a=\left(a(n)\right) _{n\in\mathbb{Z}}, \left\Vert a\right\Vert _{1} < +\infty \}$. For each $\gamma \geq 0$, it holds that  $\left\Vert a\right\Vert _{1}\leq\left\Vert a\right\Vert _{\omega}$, and 
that $\ell_\omega$ is included in $\ell_1$. The function 
$\sum_{n \in \mathbb{Z}} a(n) e^{2 i \pi n \nu}$ is thus well defined and continuous on $[0,1]$, and we will identify the sequence $a$ to the above function. In particular, with a certain abuse of notation, $\sum_{n \in \mathbb{Z}} a(n) e^{2 i \pi n \nu}$ will be denoted by $a(e^{2 i \pi \nu})$ in the following.  We can of course define the convolution product of sequences in $\ell_{\omega}$, namely
\begin{equation*}
\left( a_{1}\ast a_{2}\right) \left( n\right) =\sum_{m\in\mathbb{Z}
}a_{1}(m)a_{2}(n-m)
\end{equation*}
which has the property that $\left\Vert a_{1}\ast a_{2}\right\Vert _{\omega
}\leq\left\Vert a_{1}\right\Vert _{\omega}\left\Vert a_{2}\right\Vert
_{\omega}$, and therefore $a_{1}\ast a_{2}\in\ell_{\omega}\,$. Under the
convolution product, we can see $\ell_{\omega}$ as an algebra (the Beurling
algebra) associated with the weight $\omega$. 

\begin{assum}
\label{as:norm-r-omega}
For some $\gamma_0 > 0$, the covariance sequence $r_{m}$ defined in (\ref{eq:defcovarianceseq})
belongs to $\ell_{\omega_0}$ for each $m$, where $\omega_0(n) = (1 + |n|)^{\gamma_0}$. Moreover, it is assumed that 
\begin{equation}
\label{eq:uniform-norm-omega0-rm}
\sup_{m \geq 1}\left\Vert r_{m}\right\Vert _{\omega_0}<\infty.
\end{equation}
\end{assum}
Note that the fact that $r_{m}\in\ell_{\omega_0}$ implies that, for each $0 \leq \gamma < \gamma_0$, we have $r_{m}\in\ell_{\omega}$, where $\omega(n) = (1 + |n|)^{\gamma}$. Moreover, (\ref{eq:uniform-norm-omega0-rm}) allows to control uniformly w.r.t. $m$ of the remainder $\sum_{|k| \geq n+1} |r_m(k)|$. Indeed, observe that we can write
$$
\|r_m\|_{\omega_0} \geq \sum_{|k| \geq n+1} (1 + |k|)^{\gamma_0} |r_m(k)| \geq n^{\gamma_0} \sum_{|k| \geq n+1} |r_m(k)|.
$$
Therefore, (\ref{eq:uniform-norm-omega0-rm}) implies that 
\begin{equation}
    \label{eq:uniform-control-reminder-rm}
 \sup_{m \geq 1} \sum_{|k| \geq n+1} |r_m(k)| \leq \frac{\kappa}{n^{\gamma_0}}   
\end{equation}
for some constant $\kappa$. \\

In order to provide some insights on the significance of Assumptions \ref{as:norm-vector-r} and \ref{as:norm-r-omega}, we provide examples and counterexamples. If there exists $\gamma > \gamma_0$ for which $\sup_{m} |r_m(n)| \leq \frac{\kappa}{n^{1+\gamma}}$ for each $n$ large enough and $\sup_{m} |r_m(0)| < +\infty$, then, Assumptions \ref{as:norm-vector-r} and \ref{as:norm-r-omega} hold. If one of the time series is such that $\sum_{n} |r_m(n)| = +\infty$, then neither Assumption \ref{as:norm-vector-r} nor Assumption \ref{as:norm-r-omega} hold true (we recall (see \cite{loubaton-mestre-2017}) that Assumption \ref{as:norm-vector-r} implies that for each $m$, $\sum_{n \in \mathbb{Z}} |r_m(n)| < +\infty$)) . Finally, if one of the time series (say $y_1$) verifies 
$|r_1(n)| \sim_{n \rightarrow +\infty}  \frac{\kappa}{n (\log n)^{1+\delta}}$ for $\delta > 0$ while 
$\sum_{m \geq 2} |r_m(n)| \leq \frac{\kappa}{n^{1+\gamma}}$ for $\gamma > \gamma_0$, then 
Assumption \ref{as:norm-vector-r} holds, but Assumption \ref{as:norm-r-omega} does not hold.

% Our final assumption is regarding the function defining the linear spectral statistics $\phi$. This function should be regular enough on a sufficiently large compact support including zero. However, in order to simplify the exposition, we will focus on functions that are infinitely differentiable on the whole positive real axis including zero. 
% \begin{assum}
% \label{ass:phismooth}
%     The function $\phi$ is defined on $(-\delta,+\infty)$ for some $\delta >0$ and smooth (infinitely differentiable) on that domain. 
% \end{assum}

\subsection{Main Result}

The main objective of this paper is to establish the asymptotic conditions that guarantee that we can approximate the original statistic in $\widehat{\phi}_N$ by the corresponding integral with respect to the Marchenko-Pastur distribution as in (\ref{eq:conv-MP-intro}). To that effect, we will introduce two intermediate quantities that will provide some refined approximations of the original statistic $\widehat{\phi}_N$. 

In order to introduce the first intermediate quantity, we need to consider the matrix
\begin{equation}  \label{eq:def-Rcorrbar}
   \overline{\mathcal{R}}_{\mathrm{corr}, L} =   
   \mathcal{B}^{-1/2}_{L} \, \widehat{\mathcal{R}}_{L} \,  \mathcal{B}^{-1/2}_{L}. 
\end{equation}
Note that $\overline{\mathcal{R}}_{\mathrm{corr}, L}$ is matrix defined in the same way as  $\widehat{\mathcal{R}}_{\mathrm{corr}, L}$ by replacing the estimated block-diagonal autocorrelation matrix 
$\widehat{\mathcal{B}}_L = \mathrm{Bdiag}( \widehat{\mathcal{R}}_{L})$ by its true value 
$\mathcal{B}_L = \mathrm{Bdiag}( \mathcal{R}_{L})$, which in fact coincides with $ \mathcal{R}_{L}$ (we are assuming independent sequences). We define $\overline{\phi}_N$ as the modified linear statistic
\begin{equation}
\label{eq:defbarphi}
\overline{\phi}_N =  \frac{1}{ML} \sum_{k=1}^{ML} \phi(\overline{\lambda}_{k,N}) = \int_{\mathbb{R}^{+}} \phi(\lambda) d\overline{\mu}_N(\lambda)
\end{equation}
where $(\overline{\lambda}_{k,N})_{k=1, \ldots, ML}$ are the eigenvalues of matrix $\overline{\mathcal{R}}_{\mathrm{corr},L}$ and where $\overline{\mu}_N(\lambda)$ is the associated empirical eigenvalue distribution. 

In order to introduce the second intermediate quantity, we recall that, given an integer $K$, a $K \times K$ matrix-valued positive measure ${\boldsymbol \mu}$ is a $\sigma$--additive function from the Borel sets of $\mathbb{R}$ onto the set of all positive definite $K \times K$ matrices (see e.g. \cite[Chapter 1]{rozanov} for more details). 
\begin{defn}
\label{def:MStieltjes}
We denote by $\mathcal{S}_{ML}(\mathbb{R}^{+})$ the set of all $ML\times ML$
matrix valued functions defined on $\mathbb{C}\setminus\mathbb{R}^{+}$ by
\begin{equation*}
\mathcal{S}_{ML}(\mathbb{R}^{+})=\left\{ \int_{\mathbb{R}^{+}}\frac {1}{\lambda-z}\,d{\boldsymbol\mu}(\lambda)\right\}
\end{equation*}
where ${\boldsymbol\mu}$ is a positive $ML \times ML$ matrix-valued measure
carried by $\mathbb{R}^{+}$ satisfying ${\boldsymbol\mu}(\mathbb{R}^{+})=
\mathbf{I}_{ML}$.
\end{defn}
We will next introduce a deterministic scalar measure $\mu_N(\lambda)$ that will allow us to 
describe the asymptotic behavior of the modified statistic $\overline{\phi}_N$. To that effect,
we need to introduce some operators that were originally used in \cite{loubaton-mestre-2017}, which inherently depend on the covariance sequences $(r_{m})_{m \geq 1} $. In order to introduce these operators, for $\nu \in [0,1]$ and $R \in \mathbb{N}$, we define the column vector
\begin{equation}
\label{eq:def-d_R}
    {\bf d}_R(\nu) =  \left( 1, \mathrm{e}^{2 i \pi \nu}, 
\ldots, \mathrm{e}^{2 i \pi (R-1) \nu}\right)^{T}
\end{equation}
and let ${\bf a}_R(\nu)$ denote the corresponding normalized vector
\begin{equation}
    \label{eq:def_a_nu}
    {\bf a}_R(\nu) = \frac{1}{\sqrt{R} }{\bf d}_R(\nu).
\end{equation}
With these two definitions, we are now able to introduce the Toeplitzation operators used to define the above deterministic measure $\mu_N(\lambda)$.
\begin{defn}
\label{def:Toeplizationm}
For a given squared matrix $\mathbf{M}$ with dimensions $R\times R$, we
define $\Psi_{K}^{(m)}( \mathbf{M}) $, $m = 1,\ldots,M$, as the $K\times K$ Toeplitz matrix given by 
%with $(i,j)$th entry equal to
\begin{equation*}
\Psi_{K}^{(m)}\left( \mathbf{M}\right)  =\int_{0}^{1}\mathcal{S}_{m}\left( \nu\right) \mathbf{a}_{R}^{H}\left(
\nu\right) \mathbf{Ma}_{R}\left( \nu\right) \mathbf{d}_{K}\left( \nu\right)
\mathbf{d}_{K}^{H}\left( \nu\right) d\nu.
\end{equation*}
\end{defn}
The above operator is the key building block that defines $\Psi$ and $\overline{\Psi}$, which are the ones that determine the master equations that define $\mu_N(\lambda)$. 
\begin{defn}
\label{def:ToeplizationDerived}
Consider an $N\times N$ matrix $\mathbf{M}$. We define $\Psi\left(
\mathbf{M}\right) $ as an $ML\times ML$ block diagonal matrix with $m$th
diagonal block given by $\Psi_{L}^{(m)}\left( \mathbf{M}\right) $. Finally, consider an $ML\times ML$ matrix $\mathbf{M}$, and let $\mathbf{M}
_{m,m}$ denote its $m$th $L\times L$ diagonal block. We define $\overline {
\Psi}\left( \mathbf{M}\right) $ as the $N\times N$ matrix given by
\begin{equation}
\overline{\Psi}\left( \mathbf{M}\right) =\frac{1}{M}\sum_{m=1}^{M}\Psi
_{N}^{(m)}\left( \mathbf{M}_{m,m}\right).
\label{eq:def_Phi_average}
\end{equation}
\end{defn}
Having now introduced the above operators, we are now ready to present the master equations that define the deterministic measure $\mu_N(\lambda)$. Consider a $z \in \mathbb{C}^+$ and the following pair of equations in $\mathbf{T}_N(z), \widetilde{\mathbf{T}}_N(z)$:
\begin{align}
\mathbf{T}_N(z)  &  =-\frac{1}{z}\left(  \mathbf{I}_{ML}+\mathcal{B}_{L}
^{-1/2}\Psi\left(  \widetilde{\mathbf{T}}_N^{T}(z)\right)  \mathcal{B}
_{L}^{-1/2}\right)  ^{-1}\label{eq:canonical-T}\\
\widetilde{\mathbf{T}}_N(z)  &  =-\frac{1}{z}\left(  \mathbf{I}_{N}
+c_{N}\overline{\Psi}^{T}\left(  \mathcal{B}_{L}^{-1/2}\mathbf{T}_N
(z)\mathcal{B}_{L}^{-1/2}\right)  \right)  ^{-1}. \label{eq:canonical-tildeT}
\end{align}
We will see that there exists a unique pair of solutions $(\mathbf{T}_N(z),\widetilde{\mathbf{T}}_N(z))$ to the above equations in the set $\mathcal{S}_{ML}(\mathbb{R}^+) \times \mathcal{S}_{N}(\mathbb{R}^+)$. 
We will denote as ${\boldsymbol \mu}_N(\lambda)$ the matrix valued measure with Stieltjes transform $\mathbf{T}(z)$ and $\mu_N$ the probability measure
\begin{equation} 
\label{eq:defmumeasure}
\mu_N(\lambda) = \frac{1}{ML}\mathrm{Tr}({\boldsymbol \mu}_N(\lambda)).
\end{equation} With this, we have now all the ingredients to present the main result of this paper. 
%%%%%%%%%%%%%%%%%%%
\begin{theorem}
\label{thm:main_result}
Let Assumptions \ref{assum:statistics}-\ref{as:norm-r-omega} hold true. Then, $\hat{\mu}_N(\lambda)$ converges weakly almost surely to ${\mu}_{mp,c_\star}(\lambda)$. Furthermore:

(i) Consider $\widehat{\phi}_N$ and $\overline{\phi}_N$ defined in (\ref{eq:defLSS}) and (\ref{eq:defbarphi}) respectively and assume that $\phi$ is well defined and smooth on a open subset containing $[0,+\infty)$ to be defined in section \ref{sec:influenceblockmat}. For every small enough $\epsilon >0$, there exists a $\gamma>0$ independent of $N$ such that
\begin{equation}
    \label{eq:convergenceth(i)}
    \mathbb{P}\left( \left| \widehat{\phi}_N - \overline{\phi}_N \right| > N^\epsilon \max \left(\frac{1}{M},\frac{1}{L^{\gamma_0}}\right) \right) < \exp(-N^\gamma)
\end{equation}
for all $N$ sufficiently large. 

(ii) Let $\beta < 4/5$ and assume that $\phi$ is a smooth function with compact support. Then, for every small $\epsilon >0$ there exists a $\gamma > 0$ independent of $N$ such that
\begin{equation}
\label{eq:convergenceth(ii)}
    \mathbb{P} \left( \left| \overline{\phi}_N - \int_{\mathbb{R}^+} \phi(\lambda) d\mu_N(\lambda) \right| > N^\epsilon \max\left( \frac{1}{M\sqrt{L}}, \frac{1}{M^2} \right) \right) < \exp(-N^\gamma)
\end{equation}
for all $N$ sufficiently large. 

(iii) Consider the Marchenko-Pastur distribution with parameter $c_N = \frac{ML}{N}$ as given in Definition \ref{def:MPlaw}. Then, for every $\gamma < \gamma_0$, $\gamma \neq 1$ and every compactly supported smooth function $\phi$, we have
\begin{equation}
    \label{eq:convergencemu-MP}
    \left| \int_{\mathbb{R}^+} \phi(\lambda) d\mu_N(\lambda) - \int_{\mathbb{R}^+} \phi(\lambda) d\mu_{mp,c_N}(\lambda) \right| < \kappa \frac{1}{L^{2\min(\gamma,1)}}
\end{equation}
for some universal constant $\kappa>0$.
\end{theorem}

The above theorem basically establishes three levels of approximation of the original linear spectral statistic $\widehat{\phi}_N$ and provides the speed of convergence to zero of the corresponding error terms. In particular, it is interesting to observe that the error term in (\ref{eq:convergencemu-MP}) becomes the dominant one as soon as $\beta < 1/3$ if $\gamma_0 > 1$. Note that the situation where $\beta$ is small (or, equivalently, $L \ll M$) is the most relevant asymptotic scenario. Otherwise, the ratio $M/N$ converges quickly towards $0$, 
which, in practice represents situations in which $M \ll N$. Therefore, it  
may be possible to choose a reasonably large value of $L$ such that $\frac{ML}{N} \ll 1$. In this context, the simpler asymptotic regime where $M,N,L$ converge towards $+\infty$ in such a way that $\frac{ML}{N} \rightarrow 0$ may be relevant.
%and it might be possible to choose $L$ in such a way that $ML/N \rightarrow 0$, a regime which is of course easier to manage than the one considered in this paper. 

As a consequence of all the above, we observe that when $\beta < 1/3$ and $\gamma_0 > 1$, the dominant error incurred by approximating the linear spectra statistic $\widehat{\phi}_N$ as an integral with respect to the Marchenko-Pastur law is in fact an unknown deterministic term as established in (\ref{eq:convergencemu-MP}). %This fact clearly discourages the study of the fluctuations of $\widehat{\phi}_N$ around the corresponding Marchenko-Pastur integral.  

\subsection{Outline of the proof of Theorem \ref{thm:main_result}}

In this section, we provide some detail on the strategy that is followed in the proof of Theorem \ref{thm:main_result}. In order to present the main steps, we first review the concept of stochastic domination introduced in \cite{erdos13}  and slightly adapted in \cite{loubaton-rosuel}. We summarize next the formulation in \cite{loubaton-rosuel}. This definition will allow to denote the convergence of (\ref{eq:convergenceth(i)}) and (\ref{eq:convergenceth(ii)}) in a more compact and convenient way. More details can be found in \cite{loubaton-rosuel}.

\begin{defn}[Stochastic Domination]
Consider two families of non-negative random variables, namely $X= \{X^{(N)}(u)$, $N \in \mathbb{N}$, $u \in U^{(N)}\}$ and $Y = \{Y^{(N)}(u)$, $N \in \mathbb{N}$, $u \in U^{(N)}\}$, where $U^{(N)}$ is a set that may depend on $N$. We say that $X$ is stochastically dominated by $Y$ and write $X \prec Y$ if, for all small $\epsilon > 0$, there exists some $\gamma > 0$ depending on $\epsilon$ such that 
\begin{equation*}
    \sup_{u \in U^{(N)}} \mathbb{P}\left[  X^{(N)}(u) > N^{\epsilon} Y^{(N)}(u) \right] \leq \exp{-N^{\gamma}}
\end{equation*}
for each large enough $N > N_0(\epsilon)$.% {\color{red} I propose to omit the end of the sentence because the above formulation is independent from $u$}, where $N_0(\epsilon)$ is independent of $u$. 

On the other hand, we will say that a family of events $\Omega = \Omega^{(N)}(u)$ holds with exponentially high (resp. small) probability if there exist $N_0$ and $\gamma >0$ such that, for any $N \geq N_0$, $\mathbb{P}(\Omega^{(N)}(u)) > 1- \exp(-N^\gamma)$ (resp. $\mathbb{P}(\Omega^{(N)}(u)) < \exp(-N^\gamma)$) for each $u \in U^{(N)}$. %Furthermore, if for some general complex family $X$ we have $|X| \prec Y$, we also write $X = \mathcal{O}_{\prec}(Y)$ {\color{red} Need to check if this is used anywhere}.
\end{defn}

It can be seen that $\prec$ satisfies the usual arithmetic properties of order relations. In particular, given four families of non-negative random variables $X_1, X_2, Y_1, Y_2$ such that $X_1 \prec Y_1$ and $X_2 \prec Y_2$, then $X_1 + X_2 \prec Y_1 + Y_2$ and $X_1 X_2 \prec Y_1 Y_2$ (see Lemma 2.1 in \cite{loubaton-rosuel}). 

The proof of Theorem \ref{thm:main_result} is developed in Sections \ref{sec:simplification} to \ref{sec:approximation-MP}. The main steps are outlined in what follows. 
\begin{enumerate}[label=(\roman*),leftmargin=3pt]
\item Before we begin with the proper technical content of the paper, we close the present section with some useful properties and technical results that will become useful in the rest of the paper. %{\color{red}Include here the constraint on compact support?}
\item Section \ref{sec:simplification} provides some preliminary results on the asymptotic behavior of the sample estimate of the spatio-temporal covariance matrix $\widehat{\mathcal{R}}_L$ and its $L \times L$ diagonal blocks. % $\widehat{\mathcal{R}}_{m,L}$, $m=1,\ldots,M$.  %deals with the term $\widehat{\phi}_N - \overline{\phi}_N$. 
The objective is to show that the eigenvalue behavior of $\widehat{\mathcal{R}}_{\mathrm{corr},L}$ can be studied by examining the eigenvalue behavior of the matrix  $\overline{\mathcal{R}}_{\mathrm{corr}, L}$. More specifically, we will first prove that $\|\widehat{\mathcal{R}}_L \| $ is bounded with an exponentially large probability (recall that $\| \cdot \|$ denotes spectral norm) and that $\|\widehat{\mathcal{R}}_{m,L} - \mathcal{R}_{m,L}\| \prec \max (M^{-1/2},L^{-\gamma_0})$ for each $m=1,\ldots,M$, where we recall that $\widehat{\mathcal{R}}_{m,L}$ and $\mathcal{R}_{m,L}$ denote the $m$th diagonal block of $\widehat{\mathcal{R}}_{L}$ and $\mathcal{R}_{L}$ respectively (see (\ref{eq:defRmL}) and (\ref{eq:defRhatmL})). 
This will immediately imply that 
 \begin{equation}
     \label{eq:dominationRcorr}
     \|\widehat{\mathcal{R}}_{\mathrm{corr},L} - \overline{\mathcal{R}}_{\mathrm{corr},L}\| \prec \max \left(\frac{1}{\sqrt{M}},\frac{1}{L^{\gamma_0}}\right).
 \end{equation}
%Therefore, the individual eigenvalues of both matrices have the same asymptotic behaviour with exponentially high probability. %In particular, the asymptotic spectral eigenvalue distribution of both $\widehat{\mathcal{R}}_{\mathrm{corr}, L}$ and $\overline{\mathcal{R}}_{\mathrm{corr}, L}$ asymptotically coincide with probability one. This, together with the fact that $\|\widehat{\mathcal{R}}_{\mathrm{corr},L} \| \prec 1$ and $\|\overline{\mathcal{R}}_{\mathrm{corr},L} \| \prec 1$  will directly imply that  $\widehat{\phi}_N - \overline{\phi}_N \rightarrow 0$ almost surely.

\item Section \ref{sec:influenceblockmat} is devoted to the proof of (\ref{eq:convergenceth(i)}), which basically quantifies the influence of replacing $\widehat{\mathcal{R}}_{\mathrm{corr}, L}$ with $\overline{\mathcal{R}}_{\mathrm{corr}, L}$ in the corresponding linear spectral statistics. More specifically, by exploiting the Helffer-Sj\"ostrand formula in combination with the preliminary results in Section \ref{sec:simplification}, Theorem \ref{th:hatphi-overlinephi} establishes that
\begin{equation}
\label{eq:improved-speed}
\left| \widehat{\phi}_N - \overline{\phi}_N \right| \prec \max{\left( \frac{1}{M}, \frac{1}{L^{\gamma_0}} \right)}
\end{equation}
Notice that (\ref{eq:convergenceth(i)}) implies that 
\[
\left| \widehat{\phi}_N - \overline{\phi}_N \right| \prec \max{\left( \frac{1}{\sqrt{M}}, \frac{1}{L^{\gamma_0}} \right)}
\]
Therefore, (\ref{eq:improved-speed}) appears as a stronger result. As shown in Section 
 \ref{sec:influenceblockmat}, its proof is demanding.
%where we recall that $\gamma_0$ is defined in Assumption \ref{as:norm-r-omega}. 
%which is another way of expressing (\ref{eq:convergenceth(i)}). 

%From all the above, we can conclude that in order to study the asymptotic behavior of $\widehat{\phi}_N$, it is sufficient to study the behaviour of the empirical eigenvalue distribution of $\overline{\mathcal{R}}_{\mathrm{corr}, L}$. In particular, the Helffer-Sj\"ostrand formula will allow us to study the statistic $\overline{\phi}_N$ by characterizing (TBC).
\item Section \ref{sec:evalbarphi} studies the error term $\overline{\phi}_N - \int \phi(\lambda) d \mu_N(\lambda)$. First, this section shows that, for any smooth function $\phi$ with domain containing $[0,+\infty)$ the study of the modified statistic $\overline{\phi}_N$ can be reduced to the study of the corresponding expectation, $\mathbb{E}\overline{\phi}_N$. More specifically, we establish that
\begin{equation*}
   \left|\overline{\phi}_N-\mathbb{E}\overline{\phi}_N \right| \prec \frac{1}{M\sqrt{L}}.
\end{equation*}
%In particular, $\overline{\phi}_N-\mathbb{E}\overline{\phi}_N \rightarrow 0$ almost surely as $N \rightarrow +\infty$. 
The remaining error term $\mathbb{E}\overline{\phi}_N - \int \phi(\lambda) d \mu_N(\lambda)$ will be characterized by adapting the tools developed in \cite{loubaton-mestre-2017}, which was devoted to the study of the empirical eigenvalue distribution of matrix $\widehat{\mathcal{R}}_L$. The main difference between the matrix model considered here and the one in \cite{loubaton-mestre-2017} is the fact that here the matrix $\widehat{\mathcal{R}}_L$ is multiplied on both sides by the block diagonal deterministic matrix $\mathcal{B}^{-1/2}_{L}$, see further (\ref{eq:def-Rcorrbar}). This multiplication on both sides introduces some modifications in the master equation that defines $\mathbf{T}_{N}(z)$, which is obviously different from the one in \cite{loubaton-mestre-2017}. Other than that, the strategy of the proof will follow \cite{loubaton-mestre-2017} almost verbatim, and will mostly be omitted. 
% The main idea here will be to prove that
% $$
%     \mathbb{E}q_N(z) - \frac{1}{ML}\mathrm{Tr}\left(\mathbf{T}_N(z)\right) \rightarrow 0
% $$
% almost surely when $N \rightarrow +\infty$. The proof of Theorem \ref{thm:main_result} will then follow from a tightness argument. 
First, we will establish the almost sure weak convergence of $\bar{\mu}_N - \mu_N$  towards zero (cf. Proposition \ref{prop:convergence-Tr(Q-T)}). Then, by additionally imposing $\beta < 4/5$ in Assumption \ref{as:asymptotic-regime} and assuming that $\phi$ is compactly supported, we will be able to conclude that 
\begin{equation}
    \left|\mathbb{E}\overline{\phi}_N - \int \phi(\lambda) d \mu_N(\lambda)\right| %= \mathcal{O}\left(\frac{1}{M^{2}}\right).
    \leq \kappa \frac{1}{M^2}
\end{equation}
for some universal constant $\kappa > 0$. This will directly imply (\ref{eq:convergenceth(ii)}). 

\item Section \ref{sec:approximation-MP} finally shows that the deterministic sequence of probability measures $(\mu_N)_{N \geq 1}$ can be approximated by $\mu_{mp,c_N}$ in the sense of (\ref{eq:convergencemu-MP}) for compactly supported smooth functions $\phi$. A central step in the proof will be to establish that 
\begin{eqnarray*}
\sup_{m} \sup_{\nu \in [0,1]} \left| \mathcal{S}_m(\nu) {\bf a}_L(\nu)^{H} \mathcal{R}_{m,L}^{-1} {\bf a}_L(\nu) - 1 \right| & =  & 
\mathcal{O}\left( \frac{1}{L^{\min(1,\gamma)}}\right), \, \gamma < \gamma_0, \gamma \neq 1 \\
     & =  &  \frac{\log L}{L}, \; \gamma = 1 < \gamma_0
\end{eqnarray*}
(see further Lemma \ref{lem:orthogonal-polynomials}). This result, proved in Appendix \ref{sec:orthogonal-polynomials}, is obtained by noting that 
${\bf a}_L(\nu)^{H} \mathcal{R}_{m,L}^{-1} {\bf a}_L(\nu)$ can be expressed in terms of the orthogonal Szegö 
polynomials associated to the measure $\mathcal{S}_m(\nu) d \nu$, and by adapting to our context certain asymptotic related results presented in \cite[Chapter 5]{barry-simon-book}.
\item Section \ref{sec:sims} concludes the paper with a numerical validation that confirms the converge rates as established in Theorem \ref{thm:main_result}.

% \begin{remark}
% {\color{red} [To be reformulated] }The above results allow to conclude that the variance of $ \frac{1}{ML}\mathrm{Tr}\left( \phi(\overline{\mathcal{R}}_{\mathrm{corr}, L})\right) - \int  \phi(\lambda) d \mu_{mp,c_N}(\lambda)$ is a $\mathcal{O}(\frac{1}{MN})$ term while its bias is $\mathcal{O}(\frac {L}{MN}) + \mathcal{O}\left( \frac{1}{L^{\min(1,\gamma)}}\right)$ if $\beta < \frac{4}{5}$ for each $\gamma < \gamma_0$, $\gamma \neq 1$. This analysis should of course 
% be extended in order to have a clear understanding of the asymptotic behaviour of the probability distribution of $\frac{1}{ML}\mathrm{Tr}\left( \phi(\widehat{\mathcal{R}}_{\mathrm{corr}, L})\right) -  \int  \phi(\lambda) d \mu_{mp,c_N}(\lambda)$. 
% \end{remark}
\end{enumerate}
%In the last part of the paper, we will see that . 
%For this, it will be crucial to rely on Assumption \ref{as:norm-r-omega}, which basically states that 
%$$
%\sup_{m} \sum_{n \in \mathbb{Z}} (1 + |n|)^{\gamma_0} |r_m(n)| < +\infty 
%$$
%for some $\gamma_0 > 0$. 

%%%%%%%%%%%%%%%%%%%%%%%%%%%%%%%%%%%%%%%%%%%%%%%%%%%%%%%%%%%%%%%%%%%%%%%%%%%%%%%%%%%%%%%%

The main tool in order to study the linear spectral statistics of the estimated block correlation matrix $\widehat{\mathcal{R}}_{\mathrm{corr},L}$ will be the Stieltjes transform of its empirical eigenvalue distribution defined by (\ref{eq:def-edf-Rcorr}). More specifically, we will denote by $\hat{q}_N(z)$ the Stieltjes transform of $d\hat{\mu}_{N}(\lambda)$, that is 
\begin{equation*}
\hat{q}_{N}(z)=\int_{\mathbb{R}^{+}}\frac{1}{\lambda-z}\,d\hat{\mu}_{N}(\lambda)=
\frac{1}{ML}\sum_{k=1}^{ML}\frac{1}{\hat{\lambda}_{k,N}-z}
\end{equation*}
which is well defined for $z\in\mathbb{C}^{+}$.  This function can also be
written as $ \hat{q}_{N}(z)=\frac{1}{ML}\mathrm{Tr} \widehat{\mathbf{Q}}_{N}(z)$ where 
$\widehat{\mathbf{Q}}_{N}(z)$ is the resolvent of matrix $\widehat{\mathcal{R}}_{\mathrm{corr},L}$, namely
\begin{equation}
\label{eq:def-resolvent-tildeRcorr}
\widehat{\mathbf{Q}}_{N}(z)=\left(\widehat{\mathcal{R}}_{\mathrm{corr},L}-z\mathbf{I}_{ML}\right)^{-1}.
\end{equation}
%Each realization of the resolvent can be identified with the Stieltjes transform of a positive matrix valued measure carried by $\mathbb{R}^{+}$ with total measure $\mathbf{I}_{ML}$ (see Proposition \ref{prop:class-S} below for details). 
Likewise, for $z \in \mathbb{C}^+$, we will respectively denote by $\mathbf{Q}_N(z)$ the resolvent of $\overline{\mathcal{R}}_{\mathrm{corr}, L}$ and by $q_N(z)$ the Stieltjes transform associated to its empirical eigenvalue distribution $ d\overline{\mu}_N(\lambda)$, namely
\begin{equation}
\label{eq:def_resolvent_barRcorr}
    \mathbf{Q}_N(z) = \left( \overline{\mathcal{R}}_{\mathrm{corr}, L} - z\I_{ML}\right)^{-1}
\end{equation}
and
\[
q_N(z) = \int_{\mathbb{R}^+} \frac{1}{\lambda-z} d\overline{\mu}_N(\lambda) = \frac{1}{ML}\mathrm{Tr}\mathbf{Q}_N(z).
\]

\subsection{Notations}

% In general terms, if $\mathbf{A}$ is a $ML\times ML$ matrix, we denote by $
% \mathbf{A}_{i_{1},i_{2}}^{m_{1},m_{2}}$ the entry $
% (i_{1}+(m_{1}-1)L,i_{2}+(m_{2}-1)L)$ of matrix $\mathbf{A}$, while $\mathbf{A
% }^{m_{1},m_{2}}$ represents the $L\times L$ matrix $(\mathbf{A}
% ^{m_{1},m_{2}}_{i_1,i_2})_{1\leq(i_{1},i_{2})\leq L}$. For each $j=1, \ldots, N$, 
% ${\bf e}_j$ (resp. ${\bf f}_m$) represents the $j^{th}$ (resp. $m^{th}$) vector of the canonical basis of $\mathbb{C}^{N}$ (resp. $\mathbb{C}^{M}$)
% and for $i=1, \ldots, L$, $m=1, \ldots, M$, ${\bf f}_{i}^{m}$ is the  $(i + (m-1)L)^{th}$ vector of the canonical basis of $\mathbb{C}^{ML}$. 

The set $\mathbb{C}^{+}$ is composed of the complex numbers with strictly
positive imaginary parts. The conjugate of a complex number $z$ is denoted $
z^{\ast}$. 
% If $z\in\mathbb{C}\setminus\mathbb{R}^{+}$, we denote by $
% \delta_{z}$ the term
% \begin{equation}
% \delta_{z}=\mathrm{dist}(z,\mathbb{R}^{+}).   \label{eq:def-delta_z}
% \end{equation}
The conjugate transpose of a matrix $\mathbf{A}$ is denoted $\mathbf{A}^{H}$
while the conjugate of $\mathbf{A}$ (i.e. the matrix whose entries are the
conjugates of the entries of $\mathbf{A}$) is denoted $\mathbf{A}^{\ast}$. $
\Vert\mathbf{A}\Vert$ and $
\Vert\mathbf{A}\Vert_F$ represent the spectral norm and the Frobenius norm of matrix $\mathbf{A}$,
respectively. For a square matrix $\mathbf{A}$, we write $\mathbf{A}>0$ (resp. $\mathbf{A} \geq 0$) to state that $\mathbf{A}$ is positive definite (resp. positive semi-definite). If $\mathbf{A}$ and $\mathbf{B}$ are two square matrices, $\mathbf{A}<\mathbf{B}$ (resp. $\mathbf{A} \leq \mathbf{B}$) should be read as $\mathbf{B}-\mathbf{A}>0$ (resp. $\mathbf{B}-\mathbf{A} \geq 0$). Also, for two general matrices  $\mathbf{A}$ and $\mathbf{B}$,  $\mathbf{A}\otimes\mathbf{B}$ represents the Kronecker product of $\mathbf{A}$ and $\mathbf{B}$, i.e.
the block matrix whose block $(i,j)$ is $\mathbf{A}_{i,j}\,\mathbf{B}$. If $
\mathbf{A}$ is a square matrix, $\mathrm{Im}(\mathbf{A})$ and $\mathrm{Re}(
\mathbf{A})$ represent the Hermitian matrices
\begin{equation*}
\mathrm{Im}(\mathbf{A})=\frac{\mathbf{A}-\mathbf{A}^{H}}{2i},\;\mathrm{Re}(
\mathbf{A})=\frac{\mathbf{A}+\mathbf{A}^{H}}{2}.
\end{equation*}
If $(\mathbf{A}_{N})_{N\geq1}$ (resp. $(\mathbf{b}_{N})_{N\geq1}$) is a
sequence of matrices (resp. vectors) whose dimensions increase with $N$, $(
\mathbf{A}_{N})_{N\geq1}$ (resp. $(\mathbf{b}_{N})_{N\geq1}$) is said to be
uniformly bounded if $\sup_{N\geq1}\Vert\mathbf{A}_{N}\Vert<+\infty$ (resp. $
\sup_{N\geq1}\Vert\mathbf{b}_{N}\Vert<+\infty$).

 We will let $\mathbf{J}_{K}$ denote the $K\times K$ shift matrix with ones in the first upper
diagonal and zeros elsewhere, namely $\{ \mathbf{J}_{K}\}
_{i,j}=\delta_{j-i=1}$. We will denote by $\mathbf{J}_{K}^{-1}$ its transpose in order to simplify the notation. Likewise, $\mathbf{J}_K^{0} = \I_K$ will denote the $K \times K$ identity matrix. 

% If $\nu\in\lbrack0,1]$ and if $R$ is an integer, we denote by $\mathbf{d}
% _{R}(\nu)$ the $R$--dimensional vector
% \begin{equation}
% \mathbf{d}_{R}(\nu)=\left( 1,\mathrm{e}^{2\mathrm{i}\pi\nu},\ldots ,\mathrm{e
% }^{2\mathrm{i}\pi(R-1)\nu}\right) ^{T}   \label{eq:def-d_R}
% \end{equation}
% and by $\mathbf{a}_{L}(\nu)$ the normalized vector $\mathbf{a}_{L}(\nu )=
% \frac{1}{\sqrt{R}}\,\mathbf{d}_{R}(\nu)$, cf. (\ref{eq:def_a_nu}).

If $x$ is a complex-valued random variable, its expectation is denoted by $
\mathbb{E}\left( x\right) $ and its variance as
\begin{equation*}
\mathrm{Var}(x)=\mathbb{E}(|x|^{2})-\left\vert \mathbb{E}(x)\right\vert ^{2}.
\end{equation*}
The zero-mean random variable $x-\mathbb{E}(x)$ is denoted $x^{\circ}$.
%\newline

%\newline
In some parts of the paper, we will need to bound quantities by constants that do not depend on the system dimensions nor on the complex variable $z$. These will be referred to as ``nice constants".
\begin{defn}[\textbf{Nice constants and nice polynomials}]
\label{def:nice}
A nice constant is a positive
constant independent of the dimensions $L,M,N$ and the complex variable $z$. A
nice polynomial is a polynomial whose degree is independent from $L,M,N$,
and whose coefficients are nice constants. Throughout the paper, $\kappa$ and $P_{1},
P_{2}$ will represent a generic nice constant and two generic nice polynomials respectively, whose values may change from one line to another. 
Finally, $C(z)$ will denote a general term of the form $C(z)=P_{1}(|z|)P_{2}(1/{\Imm z})$. 
%, where $\delta_{z} = \mathrm{dist}(z, \mathbb{R}^{+})$. 
\end{defn}

\subsection{\textbf{Background on Stieltjes transforms of positive matrix
valued measures}}
%We recall that if $K$ is a positive integer, then a $K \times K$ matrix-valued positive measure $\boldsymbol{\mu}$ is a $\sigma$--additive function from the Borel sets of $\mathbb{R}$ onto the set of all positive $K \times K$ matrices (see e.g. \cite{rozanov}, Chapter 1 for more details). 
We recall that $\mathcal{S}_{K}(\mathbb{R}^{+})$ denotes the set of
all Stieltjes transforms of $K\times K$ positive matrix-valued measures ${
\boldsymbol\mu}$ carried by $\mathbb{R}^{+}$ verifying ${\boldsymbol\mu}(
\mathbb{R}^{+})=\mathbf{I}_{K}$. The elements of the class $\mathcal{S}_{K}(
\mathbb{R}^{+})$ satisfy the following properties:

\begin{proposition}
\label{prop:class-S} Consider an element $\mathbf{S}(z) = \int_{\mathbb{R}
^{+}} \frac{d \, {\boldsymbol \mu}(\lambda)}{\lambda- z}$ of $\mathcal{S}
_{K}(\mathbb{R}^{+})$. Then, the following properties hold true:

\begin{enumerate}[label=(\roman*)]
\item $\mathbf{S}$ is analytic on $\mathbb{C}^{+}$.

\item $\mathrm{Im}(\mathbf{S}(z)) \geq0$ and $\mathrm{Im}(z \, \mathbf{S
}(z)) \geq0$ if $z \in\mathbb{C}^{+}$.

\item $\lim_{y\rightarrow+\infty}-\mathrm{i}y\mathbf{S}(\mathrm{i} y)=
\mathbf{I}_{K}$.

\item $\mathbf{S}(z)\mathbf{S}^{H}(z)\leq\frac{\mathbf{I}_{K}}{(\Imm
{z})^{2}}$ for each $z\in\mathbb{C}^{+}$.

%\item  ${\mathrm{Im} \mathbf{S}(z)}\leq \frac{1}{\Imm{z}} {\mathbf{I}_{K}}$ for each $z\in\mathbb{C}^{+}$. %, where $\frac{\mathrm{Im} \mathbf{S}(z)}{\mathrm{Im} z}$ should be interpreted as the derivative $\mathbf{S}^{'}(z)$ of $\mathbf{S}(z)$ w.r.t. $z$ when $z < 0$.

\item $\int_{\mathbf{R}^{+}}\lambda\,d{\boldsymbol\mu}(\lambda
)=\lim_{y\rightarrow+\infty}\mathrm{Re}\left( -\mathrm{i}y(\mathbf{I} _{K}+
\mathrm{i}y\mathbf{S}(iy)\right) $.
\end{enumerate}
Conversely, if a function $\mathbf{S}(z)$ satisfy properties (i), (ii),
(iii), then $\mathbf{S}(z)\in\mathcal{S}_{K}(\mathbb{R}^{+})$.

\end{proposition}

While we have not been able to find a paper in which this result is proved,
it has been well known for a long time (see however \cite
{hachem-loubaton-najim-aap-2007} for more details on (i), (ii), (iii), (v)),
as well as Theorem 3 of \cite{alpay-tsekanovskii} from which (iv) follows immediately). 
We however provide an elementary proof of (iv) because it is based on a
version of the matrix Schwarz inequality that will be used later. Given a certain $K \times K$ positive matrix measure ${\boldsymbol\mu}$ carried by $\mathbb{R}^{+}$, we denote by $\mathbb{L}^{2}({
\boldsymbol\mu})$ the Hilbert space of all $K$-dimensional row vector-valued functions $\mathbf{u}(\lambda)$ defined on $\mathbb{R}^{+}$ satisfying $\int_{\mathbb{R}
^{+}}\mathbf{u}(\lambda)\,d\,{\boldsymbol\mu}(\lambda)\,\mathbf{u}^{H}
(\lambda)\,<+\infty$ endowed with the scalar product
\begin{equation*}
\langle \mathbf{u},\mathbf{v} \rangle =\int_{\mathbb{R}^{+}}\mathbf{u}(\lambda )\,d\,{
\boldsymbol\mu}(\lambda)\,\mathbf{v}^{H}(\lambda).
\end{equation*}
Then, if $\mathbf{U}(\lambda)=(\mathbf{u}_{1}(\lambda)^{T},\ldots ,\mathbf{u}
_{K_{u}}(\lambda))^{T})^{T}$ and $\mathbf{V}(\lambda)=(\mathbf{v}
_{1}(\lambda)^{T},\ldots,\mathbf{v}_{K_{v}}(\lambda))^{T})^{T}$ are matrices with $K_u$ and $K_v$ rows respectively, all of which are elements of $\mathbb{L}^{2}({\boldsymbol\mu})$, it holds that
\begin{equation}
\langle\mathbf{U},\mathbf{V}\rangle\, \langle\mathbf{V},\mathbf{V}\rangle^{-1}\,\langle
\mathbf{U},\mathbf{V}\rangle^{H}\leq\langle\mathbf{U},\mathbf{U}\rangle
\label{eq:schwartz-matriciel-v2}
\end{equation}
where, with some abuse of notation, $\langle\mathbf{U},\mathbf{V}\rangle$ denotes the matrix defined by $\left( \langle\mathbf{U},
\mathbf{V}\rangle\right) _{i,j}=\langle \mathbf{u}_{i},\mathbf{v}_{j} \rangle$. This inequality can be directly proven by considering the $(K_u + K_v) \times K$ matrix $\mathbf{W} = [\mathbf{U}^T,\mathbf{V}^T]^T$ and noting that $\langle \mathbf{W}, \mathbf{W} \rangle$ is positive semi-definite. This implies that its Schur complement is also positive semi-definite, which directly implies (\ref{eq:schwartz-matriciel-v2}). 
Now, using (\ref
{eq:schwartz-matriciel-v2}) for $\mathbf{U}(\lambda)=\frac{\mathbf{I}}{
\lambda-z}$ and $\mathbf{V}=\mathbf{I}$, and remarking that $|\lambda
-z|^{2}\geq (\Imm z)^{2}$ for each $\lambda\in\mathbb{R}^{+}$, we immediately obtain (iv).

\subsection{Further properties of stochastic domination and concentration inequalities}

The following result is a direct consequence of the union bound.

\begin{lemma}
\label{lem:stochastic_dom_maximum}
Let $X_1,\ldots,X_P$ denote a collection of $P \in \mathbb{N}$ families of non-negative random variables, each one defined as $X_p=\{X_p^{(N)}(u)$, $N \in \mathbb{N}$, $u \in U^{(N)}\}$. Let $Y$ denote an equivalently defined family of non-negative random variables such that $X_p \prec Y$ for each $p = 1,\ldots,P$. Assuming that $P \leq N^C$ for some universal constant $C$, we have
\begin{equation*}
    \max_{p=1,\ldots,P}X_p \prec Y.
\end{equation*}
\end{lemma}

On the other hand, we will be using a number of inequalities based on concentration of functions of random Gaussian vectors. More specifically, consider the real-valued function $f(\mathbf{x},\mathbf{x}^\ast)$ where $\mathbf{x}$ is a complex $N$-dimensional variable and where $(\cdot)^\ast$ denotes complex conjugate. If $\mathbf{x} \sim \mathcal{N}_\mathbb{C}(0,\mathbf{I}_N)$, $f(\mathbf{x},\mathbf{x}^\ast)$ can be interpreted as a function of the $2N$--dimensional $\mathcal{N}(0,\mathbf{I}_{2N})$ vector $\left(\sqrt{2} \mathrm{Re}({\bf x}^{T}), \sqrt{2} \mathrm{Im}({\bf x}^{T})\right)^{T}$. If $f$ is  1-Lipschitz, there exists a universal constant $C>0$ such that 
\begin{equation}
\label{eq:concentrationGaussian}
\mathbb{P}\left[|f(\mathbf{x},\mathbf{x}^\ast) - \mathbb{E}f(\mathbf{x},\mathbf{x}^\ast) | > t\right] \leq C \exp{-C t^2}.
\end{equation}
This concentration inequality is well known if $f$ is a function of a $\mathcal{N}(0,\mathbf{I}_N)$ real-valued vector ${\bf x}$  (\cite[Theorem 2.1.12]{tao-book}). We notice that 
(\ref{eq:concentrationGaussian}) implies that $|f(\mathbf{x},\mathbf{x}^\ast) - \mathbb{E}f(\mathbf{x},\mathbf{x}^\ast) | \prec 1$ (see \cite{loubaton-rosuel} for more details).

Finally, we will also make extensive use of the Hanson-Wright inequality, proven in \cite{rudelson2013hanson} for the subgaussian real-valued case, but easily extended to the complex Gaussian context. If $\mathbf{A}$ denotes an $N \times N$ matrix of complex entries and if $\mathbf{x} \sim \mathcal{N}_\mathbb{C}(0,\mathbf{I}_N)$, then 
\begin{equation}
\label{eq:hanson-wright}
\mathbb{P}\left[|\mathbf{x}^H \mathbf{A} \mathbf{x} - \mathbb{E}\mathbf{x}^H \mathbf{A} \mathbf{x} | > t\right] \leq 2 \exp{\left[-C \min\left(\frac{t^2}{\|\mathbf{A}\|_F^2},\frac{t}{\|\mathbf{A}\|}\right)\right]}
\end{equation}
where here again $C >0$ is a universal constant and where $\|\cdot\|_F$ and $\|\cdot\|$ respectively denote Frobenius and spectral norms.

\subsection{Additional properties of the Toeplitzification operators}

We introduce here some additional properties of the Toeplitzification operators introduced in Definitions \ref{def:Toeplizationm}-\ref{def:ToeplizationDerived}, which will prove useful in the course of the derivations. 
For a given squared matrix $\mathbf{M}$ with dimensions $R\times R$, the operator $\Psi_{K}^{(m)}( \mathbf{M}) $ in Definition \ref{def:Toeplizationm} can alternatively be represented as an $K\times K$ Toeplitz
matrix with $(i,j)$th entry equal to
\begin{equation}
\left\{ \Psi_{K}^{(m)}\left( \mathbf{M}\right) \right\}
_{i,j}=\sum_{l=-R+1}^{R-1}r_{m}\left( i-j-l\right) \tau\left( \mathbf{M}
\right) \left( l\right)   \label{eq:definition_operator_Psi}
\end{equation}
or, alternatively, as the matrix
\begin{equation}
\Psi_{K}^{(m)}\left( \mathbf{M}\right) =\sum_{n=-K+1}^{K-1}\left(
\sum_{l=-R+1}^{R-1}r_{m}\left( n-l\right) \tau\left( \mathbf{M}\right)
\left( l\right) \right) \mathbf{J}_{K}^{-n}
\label{eq:definition_operator_Psi2}
\end{equation}
where the sequence $\tau\left( \mathbf{M}\right) \left( l\right) $, $-R<l<R$, is defined as
\begin{equation}
\tau\left( \mathbf{M}\right) \left( l\right) =\frac{1}{R}\mathrm{Tr}\left[
\mathbf{MJ}_{R}^{l}\right] .   \label{eq:definition_tau}
\end{equation}
We observe that, with this definition,
\begin{equation}
    \label{eq:parsevalintro}
    \sum_{r=-(R-1)}^{R-1} \left| \tau({\bf M})(r) \right|^{2} \leq \frac{1}{R} \mathrm{Tr}({\bf M} {\bf M}^{H}).
\end{equation}
This inequality can be proven by noting that $ \tau({\bf M})(r)$, $r=-R+1,\ldots,R-1$ are the Fourier coefficients of the function $\nu \mapsto \mathbf{a}^H_R(\nu) \mathbf{M} \mathbf{a}_R(\nu)$ so that, by Parseval's identity,
\begin{align*}
     \sum_{r=-(R-1)}^{R-1} \left| \tau({\bf M})(r) \right|^{2} & =  \int_0^1 \left| \mathbf{a}^H_R(\nu) \mathbf{M} \mathbf{a}_R(\nu) \right|^{2} d\nu \leq \\ 
     & \leq  \int_0^1 \mathbf{a}^H_R(\nu) \mathbf{MM}^H \mathbf{a}_R(\nu) d\nu  = \frac{1}{R} \mathrm{Tr}({\bf M} {\bf M}^{H})
\end{align*}
where we have used the Cauchy-Schwarz inequality. 

We also mention the following property: If ${\bf A}$ is a $R \times R$ Toeplitz 
matrix with entries ${\bf A}_{i,j} = a(i-j)$ for some sequence $(a(l))_{l=-(R-1), \ldots, R-1}$, and if ${\bf B}$ is another $R \times R$ matrix, we have
\begin{equation}
    \label{eq:trace-toeplitz}
    \frac{1}{R} \mathrm{Tr} ( {\bf A} {\bf B} ) = \sum_{l=-(R-1)}^{R-1} a(l) \tau({\bf B})(-l)
\end{equation}

% We can also express this operator more compactly using frequency notation, namely
% \begin{equation*}
% \Psi_{K}^{(m)}\left( \mathbf{M}\right) =\sum_{n=-K+1}^{K-1}\left(
% \int_{0}^{1}\mathcal{S}_{m}\left( \nu\right) \mathbf{a}_{R}^{H}\left(
% \nu\right) \mathbf{Ma}_{R}\left( \nu\right) \mathrm{e}^{2\pi\mathrm{i}\nu
% n}d\nu\right) \mathbf{J}_{K}^{-n} %\\
% %& =\int_{0}^{1}\mathcal{S}_{m}\left( \nu\right) \mathbf{a}_{R}^{H}\left(\nu\right)\mathbf{Ma}_{R}\left( \nu\right) \mathbf{d}_{K}\left( \nu\right)\mathbf{d}_{K}^{H}\left( \nu\right) d\nu    
% \end{equation*}
% where we recall that $\mathbf{a}_{R}\left( \nu\right) =\mathbf{d}_{R}\left( \nu\right) /
% \sqrt{R}$ and $\mathbf{d}_{R}\left( \nu\right) $ as defined in (\ref
% {eq:def-d_R}). 
The following properties are easily checked (see \cite{loubaton-mestre-2017}). 
\begin{itemize}[leftmargin=3pt]
\item Given a square matrix $\mathbf{A}$ of dimension $K\times K$ and a
square matrix$\ \mathbf{B}$ of dimension $R\times R$, we can write
\begin{equation}
\frac{1}{K}\mathrm{Tr}\left[ \mathbf{A}\Psi_{K}^{(m)}\left( \mathbf{B}
\right) \right] =\frac{1}{R}\mathrm{Tr}\left[ \Psi_{R}^{(m)}\left( \mathbf{A}
\right) \mathbf{B}\right]   \label{eq:property_commutative}
\end{equation}

\item Given a square matrix $\mathbf{M}$ and a positive integer $K$, we have
\begin{equation*}
\left\Vert \Psi_{K}^{(m)}\left( \mathbf{M}\right) \right\Vert \leq\sup
_{\nu\in\lbrack0,1]}\left\vert \mathcal{S}_{m}\left( \nu\right) \right\vert
\left\Vert \mathbf{M}\right\Vert .
\end{equation*}

\item Given a square positive definite matrix $\mathbf{M}$ and a positive
integer $K$, the hypothesis $\inf_{\nu}\mathcal{S}_{m}\left( \nu\right) >0$
implies that
\begin{equation}
\Psi_{K}^{(m)}\left( \mathbf{M}\right) >0.   \label{eq:property_positive}
\end{equation}
\end{itemize}

Consider now the two other linear operators in Definition \ref{def:ToeplizationDerived},
which respectively operate on $N\times N$ and $ML\times ML$ matrices. 
%In order to keep the notation as simple as possible, we will drop the dimensions in the notation of these operators. 
% \begin{itemize}[leftmargin=3pt]
% \item Consider an $N\times N$ matrix $\mathbf{M}$. We define $\Psi\left(
% \mathbf{M}\right) $ as an $ML\times ML$ block diagonal matrix with $m$th
% diagonal block given by $\Psi_{L}^{(m)}\left( \mathbf{M}\right) $.
% \item Consider an $ML\times ML$ matrix $\mathbf{M}$, and let $\mathbf{M}
% _{m,m}$ denote its $m$th $L\times L$ diagonal block. We define $\overline {
% \Psi}\left( \mathbf{M}\right) $ as the $N\times N$ matrix given by
% \begin{equation}
% \overline{\Psi}\left( \mathbf{M}\right) =\frac{1}{M}\sum_{m=1}^{M}\Psi
% _{N}^{(m)}\left( \mathbf{M}_{m,m}\right) = \int_{0}^{1} \frac{1}{M} \sum_{m=1}^{M} \mathcal{S}_m(\nu) {\bf a}_L^{H}(\nu)  \mathbf{M}_{m,m} {\bf a}_L(\nu) \, 
% \dd_N(\nu) \dd_N^{H}(\nu) \, d\nu.
% \label{eq:def_Phi_average}
% \end{equation}
% Observe, in particular, that $\overline{\Psi}\left( \mathbf{M}\right)$ 
% coincides with  $\overline{\Psi}\left( \mathrm{Bdiag}(\mathbf{M})\right)$.
% \end{itemize}
If $\mathbf{A}$ and $\mathbf{B}$ are $
ML\times ML$ and $N\times N$ matrices, we see directly from (\ref
{eq:property_commutative}) that
\begin{equation}
\frac{1}{N}\mathrm{Tr}\left[ \overline{\Psi}\left( \mathbf{A}\right) \mathbf{
B}\right] =\frac{1}{ML}\mathrm{Tr}\left[ \mathbf{A}\Psi\left( \mathbf{B}
\right) \right] .   \label{eq:transpose_ops_gen}
\end{equation}

We finally conclude this section by two useful propositions that follow directly 
from  \cite{loubaton-mestre-2017}.
\begin{proposition}
\label{prop:Upsilon-tildeUpsilon}Let $\mathbf{\Gamma}^{m}(z),$ $m=1,\ldots,M$, be a
collection of $L\times L$ matrix-valued complex functions belonging to
$\mathcal{S}_{L}\left(  \mathbb{R}^{\mathbb{+}}\right)  $ and define 
$\mathbf{\Gamma}(z)$ as the $ML \times ML$ block diagonal matrix given by 
$\mathbf{\Gamma}(z)=\mathrm{diag}\left(  \mathbf{\Gamma}^{1}(z),\ldots
,\mathbf{\Gamma}^{M}(z)\right)  $. Then, for each $z \in \mathbb{C}^{+}$, 
the matrix $ \mathbf{I}_{N}
+c_{N}\overline{\Psi}^{T}\left(  \mathcal{B}_{L}^{-1/2}\mathbf{\Gamma
}(z)\mathcal{B}_{L}^{-1/2}\right)$ is invertible, so that we can define 
\begin{equation}
    \label{eq:def-tildeUpsilon}
    \widetilde{\mathbf{\Upsilon}}(z)   =-\frac{1}{z}\left(  \mathbf{I}_{N}
+c_{N}\overline{\Psi}^{T}\left(  \mathcal{B}_{L}^{-1/2}\mathbf{\Gamma
}(z)\mathcal{B}_{L}^{-1/2}\right)  \right)  ^{-1}.
\end{equation}
On the other hand, the matrix $ \mathbf{I}_{ML}+\mathcal{B}
_{L}^{-1/2}\Psi\left(  \widetilde{\mathbf{\Upsilon}}^{T}(z)\right)
\mathcal{B}_{L}^{-1/2}$ is also invertible, and we define 
\begin{equation}
\label{eq:def-Upsilon}
\mathbf{\Upsilon}(z)  =-\frac{1}{z}\left(  \mathbf{I}_{ML}+\mathcal{B}
_{L}^{-1/2}\Psi\left(  \widetilde{\mathbf{\Upsilon}}^{T}(z)\right)
\mathcal{B}_{L}^{-1/2}\right)  ^{-1}.
\end{equation}
Furthermore, $\widetilde{\mathbf{\Upsilon}}(z)$ and $\mathbf{\Upsilon}(z)$ are elements of
$\mathcal{S}_{N}(\mathbb{R}^{+})$ and $\mathcal{S}_{ML}(\mathbb{R}^{+})$
respectively. In particular, they are holomorphic on $\mathbb{C}^{+}$ and satisfy
\begin{equation}
\mathbf{\Upsilon}(z)\mathbf{\Upsilon}^{H}(z)\leq\frac{\mathbf{I}_{ML}}{(\Imm {z})^{2}
},\;\widetilde{\mathbf{\Upsilon}}(z)\widetilde{\mathbf{\Upsilon}}^{H}(z)\leq\frac
{\mathbf{I}_{N}}{(\Imm {z})^{2}}\label{eq:upperbound-R-tildeR}.
\end{equation}
Moreover, there exist two nice constants $\eta$ and $\widetilde{\eta}$ such that
\begin{align}
\mathbf{\Upsilon}(z)\mathbf{\Upsilon}^{H}(z) &  \geq\frac{(\Imm {z})^{2}}{16(\eta
^{2}+|z|^{2})^{2}}\mathbf{I}_{ML}\label{eq:lower-bound-UU*}\\
\widetilde{\mathbf{\Upsilon}}(z)\widetilde{\mathbf{\Upsilon}}^{H}(z) &  \geq\frac{(\Imm
{z})^{2}}{16(\widetilde{\eta}^{2}+|z|^{2})^{2}}\mathbf{I}_{N}
.\label{eq:lower-bound-tildeUtildeU*}
\end{align}
\end{proposition}
\begin{proof}
The proof is an easy adaptation of the proof of Lemma 4.1 in \cite{loubaton-mestre-2017}. More precisely, if we replace in this Lemma matrix 
$\mathrm{Bdiag}(\mathbb{E}\mathbf{Q}(z))$ by $\mathbf{\Gamma}(z)$ and matrices $(\mathbf{R}(z),\widetilde{\mathbf{R}}(z))$ by $(\mathbf{\Upsilon}(z),\widetilde{\mathbf{\Upsilon}}(z))$, it is easy to check that 
the arguments of the proof of Lemma 4.1 in \cite{loubaton-mestre-2017} can be extended
to the particular context considered in the present paper. 
%$\blacksquare$
\end{proof}
In order to state the next result, we consider two $ML \times ML$ block diagonal 
matrices  ${\bf S}, {\bf T}$ and two $N \times N$ matrices 
$\widetilde{{\bf S}}, \widetilde{{\bf T}}$. We also assume
that ${\bf S}, {\bf T},\widetilde{{\bf S}}, \widetilde{{\bf T}}$ are full rank matrices. 
For each fixed $z$, we define the linear operator $\Phi$ on the set of all $ML \times ML$
matrices by
\begin{equation}
    \label{eq:def-Phi-general}
\Phi \left(  \mathbf{X}\right)  =z^{2}c_{N}\mathbf{S}
\Psi\left(  \widetilde{\mathbf{S}}^{T} \overline{\Psi}\left(
\mathbf{X}\right)  \widetilde{\mathbf{T}}^{T}\right)  \mathbf{T}.
\end{equation}
Note that the operator $\Phi$ of course depends on ${\bf S}, {\bf T},\widetilde{{\bf S}}, \widetilde{{\bf T}}$, $M,L,N$ and $z$. We also define the following linear operators 
on the set of all  $ML \times ML$  Hermitian matrices:
\begin{align}
\Phi_{\mathbf{T}^{H}}\left(  \mathbf{X}\right)    & =\left\vert
z\right\vert ^{2}c_{N}\mathbf{T}^{H}\Psi\left(
\widetilde{\mathbf{T}}^{\ast}\overline{\Psi}\left(  \mathbf{X}\right)
\widetilde{\mathbf{T}}^{T}\right)  \mathbf{T}
\label{eq:defPhi-general-TH}\\
\Phi_{\mathbf{S}} \left(  \mathbf{X}\right)    & =\left\vert
z\right\vert ^{2}c_{N}\mathbf{S}\Psi\left(  \widetilde
{\mathbf{S}}^{T}\overline{\Psi}\left(  \mathbf{X}\right)  \widetilde
{\mathbf{S}}^{\ast}\right)  \mathbf{S}^{H}.
\label{eq:defPhi-general-S}
\end{align}
We remark that both operators are positive in the sense that if ${\bf X} \geq 0$, then 
$\Phi_{\mathbf{S}} \left(  \mathbf{X}\right) \geq 0$ and $\Phi_{\mathbf{T}^{H}} \left(  \mathbf{X}\right) \geq 0$. Let $\Phi^{(1)}\left(  \mathbf{X}\right)  =\Phi \left(  \mathbf{X}
\right)  $ and recursively define $\Phi^{(n+1)}\left(  \mathbf{X}\right)
=\Phi \left(  \Phi^{(n)}\left(  \mathbf{X}\right)  \right)  $ for
$n\geq1$. Then, the following result holds. 
\begin{proposition}
\label{prop:Phi-Phi_S-Phi_TH}
For any two $L$-dimensional column vectors $\mathbf{a}$,
$\mathbf{b}$ and for each $m=1, \ldots, M$, the inequality
\begin{equation}
\label{eq:inequality-general-phi-phiS-phiTH}
\left\vert 
\mathbf{a}^{H} \left(\Phi^{(n)}\left(\mathbf{X}\right)\right)_{m,m}
\mathbf{b} 
\right\vert
 \leq
 \left[  \mathbf{a}^{H} \left( \Phi_{\mathbf{S}}
^{(n)} \left(  \mathbf{X} \mathbf{X}^{H} \right) \right)_{m,m} \mathbf{a} \right]^{1/2}
  \left[ \mathbf{b}^{H} \left(\Phi_{\mathbf{T}^{H}}^{(n)}\left(
\I_{ML} \right) \right)_{m,m} \mathbf{b} \right]^{1/2}
\end{equation}
holds, where $(\mathbf{A})_{m,m}$ denotes the $m$th $L \times L$ diagonal block of $\mathbf{A}$. Moreover, if there exist two $ML \times ML$ positive definite matrices $\mathbf{Y}_1$ 
and $\mathbf{Y}_2$ such that 
\begin{eqnarray}
\label{eq:hyp-phiSn-0}
\lim_{n \rightarrow +\infty}\Phi_{\mathbf{S}}^{(n)}\left( \mathbf{Y}_1 \right) & \rightarrow & 0 \\
\label{eq:hyp-phiTHn-0}
\lim_{n \rightarrow +\infty} \Phi_{\mathbf{T}^{H}}^{(n)}\left( \mathbf{Y}_2 \right) & \rightarrow & 0
\end{eqnarray}
then, for each $ML \times ML$ matrix $\mathbf{X}$, 
\begin{equation}
\label{eq:convergence-Phin-general}
\lim_{n \rightarrow +\infty} \Phi^{(n)}\left(  \mathbf{X}\right) \rightarrow 0   
\end{equation}
If, moreover, $\sum_{n=0}^{+\infty} \Phi_{\mathbf{S}}
^{(n)}\left( \mathbf{Y}_1 \right) < +\infty$ and $\sum_{n=0}^{+\infty} \Phi_{\mathbf{T}^{H}}
^{(n)}\left( \mathbf{Y}_2 \right) < +\infty$, then, for each $ML \times ML$ hermitian matrix $\mathbf{Y}$, the two series $\sum_{n=0}^{+\infty} \Phi_{\mathbf{S}}
^{(n)}\left( \mathbf{Y} \right)$ and $\sum_{n=0}^{+\infty} \Phi_{\mathbf{T}^{H}}
^{(n)}\left( \mathbf{Y} \right)$ are convergent. Finally, for each $ML \times ML$ matrix $\mathbf{X}$, 
 $\sum_{n=0}^{+\infty} \Phi^{(n)}\left( \mathbf{X} \right)$ is also convergent, and we have
 \begin{equation}
     \label{eq:inequality-general-norm-sum-Phin}
     \left\| \sum_{n=0}^{+\infty} \Phi^{(n)}\left( \mathbf{X} \right) 
     \right\|  \leq  
  \left\| \sum_{n=0}^{+\infty} \Phi_{\mathbf{S}}^{(n)}\left( \mathbf{X} \mathbf{X}^{H} \right) \right\|^{1/2}  \, \left\| \sum_{n=0}^{+\infty} \Phi_{\mathbf{T}^{H}}^{(n)}\left(\I_{ML} \right) \right\|^{1/2}
\end{equation}
as well as 
\begin{equation}
\label{eqinequality-general-norm-sum-Phin-bis}
   \left\| \sum_{n=0}^{+\infty} \Phi^{(n)}\left( \mathbf{X} \right) \right\|  \leq  \| {\bf X} \| \; \left\| \sum_{n=0}^{+\infty} \Phi_{\mathbf{S}}^{(n)}\left( \I_{ML} \right) \right\|^{1/2}  \, \left\| \sum_{n=0}^{+\infty} \Phi_{\mathbf{T}^{H}}^{(n)}\left(\I_{ML} \right) \right\|^{1/2}.
\end{equation}
\end{proposition}
\begin{proof}
Inequality (\ref{eq:inequality-general-phi-phiS-phiTH}) is established in Section 5 of  \cite{loubaton-mestre-2017}. We now prove 
(\ref{eq:convergence-Phin-general}). For this, we first remark that since matrices $(\mathbf{Y}_i)_{i=1,2}$ are positive definite, 
there exist $\alpha_1 > 0$ and $\alpha_2 > 0$ such that $\mathbf{Y}_i \geq \alpha_i \I_{ML}$ for $i=1,2$. As the operators $\Phi_{\mathbf{S}}$ 
and $ \Phi_{\mathbf{T}^{H}}$ are positive, it holds that $\Phi^{(n)}_{\mathbf{S}}(\mathbf{Y}_1) \geq \alpha_1 \Phi^{(n)}_{\mathbf{S}}(\I_{ML})$
and $ \Phi_{\mathbf{T}^{H}}^{(n)}(\mathbf{Y}_2) > \alpha_2  \Phi_{\mathbf{T}^{H}}^{(n)}(\I_{ML})$ for each $n$. Therefore, conditions
(\ref{eq:hyp-phiSn-0}) and (\ref{eq:hyp-phiTHn-0}) imply that $\Phi^{(n)}_{\mathbf{S}}(\I_{ML}) \rightarrow 0$ and 
$\Phi^{(n)}_{\mathbf{T}^{H}}(\I_{ML}) \rightarrow 0$. If $\mathbf{X}$ is a generic $ML \times ML$ matrix, the inequality $\mathbf{X} \mathbf{X}^{H} \leq \| {\bf X} \|^{2} \I_{ML}$
implies that $\Phi^{(n)}_{\mathbf{S}}(\mathbf{X} \mathbf{X}^{H}) \leq \| {\bf X} \|^{2} \Phi^{(n)}_{\mathbf{S}}(\I_{ML})$. Therefore, we deduce that 
for each matrix ${\bf X}$, $\Phi^{(n)}_{\mathbf{S}}(\mathbf{X} \mathbf{X}^{H}) \rightarrow 0$ when $n \rightarrow +\infty$. The inequality in
(\ref{eq:inequality-general-phi-phiS-phiTH}) thus leads to (\ref{eq:convergence-Phin-general}). Using similar arguments, 
we check that the convergence $\sum_{n=0}^{+\infty} \Phi_{\mathbf{S}}
^{(n)}\left( \mathbf{Y}_1 \right)$ and $\sum_{n=0}^{+\infty} \Phi_{\mathbf{T}^{H}}
^{(n)}\left( \mathbf{Y}_2 \right)$ implies the convergence of $\sum_{n=0}^{+\infty} \Phi_{\mathbf{S}}
^{(n)}\left( \mathbf{Y} \right)$ and $\sum_{n=0}^{+\infty} \Phi_{\mathbf{T}^{H}}
^{(n)}\left( \mathbf{Y} \right)$ for each positive matrix $\mathbf{Y}$. If $\mathbf{Y}$ is not positive, it is sufficient to 
remark that $\mathbf{Y}$ can be written as the difference of 2 positive matrices to conclude to the convergence of the 
above two series. We finally consider a general matrix ${\bf X}$, and establish that $\sum_{n=0}^{+\infty} \Phi^{(n)}\left( \mathbf{X} \right)$ is convergent. For this, we remark that (\ref{eq:inequality-general-phi-phiS-phiTH}) implies that for each $m$ and each $k$, the inequality
\begin{multline}
\label{eq:inequality-partial-sum-quadratic-phi}
 \sum_{n=0}^{k} \left\vert \mathbf{a}^{H} \left( \Phi^{(n)}\left(  \mathbf{X}\right) \right)
_{m,m} \mathbf{b} \right\vert \leq \\ \leq \left[   \mathbf{a}^{H} \left(  \sum_{n=0}^{k}  \left( \Phi_{\mathbf{S}}
^{(n)}\left(  \mathbf{X} \mathbf{X}^{H} \right) \right)_{m,m} \right) \mathbf{a}\right]
^{1/2}  \left[ \mathbf{b}^{H}\left( \sum_{n=0}^{k} \left( \Phi_{\mathbf{T}^{H}}^{(n)}\left(
\I_{ML} \right) \right)_{m,m} \right)\mathbf{b}\right]  ^{1/2}
\end{multline}
holds. This implies that 
$$
 \sum_{n=0}^{+\infty} \left\vert \mathbf{a}^{H} \left( \Phi^{(n)}\left(  \mathbf{X}\right) \right)
_{m,m} \mathbf{b} \right\vert < +\infty
$$
and that the series  $\sum_{n=0}^{+\infty} \Phi^{(n)}\left( \mathbf{X} \right)$ is convergent. The result in (\ref{eq:inequality-general-norm-sum-Phin})
is obtained by taking the limit in the inequality (\ref{eq:inequality-partial-sum-quadratic-phi}),
while (\ref{eqinequality-general-norm-sum-Phin-bis}) is an immediate consequence of 
(\ref{eq:inequality-general-norm-sum-Phin}). 
%$\blacksquare$
\end{proof}

\section{Preliminary results on the empirical estimates $\widehat{\mathcal{R}}_L$ and $\widehat{\mathcal{R}}_{m,L}$}

\label{sec:simplification}
Consider again the  sample block correlation matrix, namely $\widehat{\mathcal{R}}_{\mathrm{corr},L} = \widehat{\mathcal{B}}^{-1/2}_L \widehat{\mathcal{R}}_L \widehat{\mathcal{B}}^{-1/2}_L$, where we recall that $\widehat{\mathcal{B}}_L = \mathrm{Bdiag}(\widehat{\mathcal{R}}_L)$. In this section, we will show that we can replace the block diagonal sample covariance matrix $\widehat{\mathcal{B}}_L$ by the true matrix $\mathcal{B}_L = \mathcal{R}_L$ without altering the asymptotic behavior of the empirical eigenvalue distribution of $\overline{\mathcal{R}}_{\mathrm{corr},L} = \mathcal{B}^{-1/2}_L \widehat{\mathcal{R}}_L \mathcal{B}^{-1/2}_L$. %More specifically, we establish that the two matrices $\widehat{\mathcal{R}}_{\mathrm{corr},L}$ and $\overline{\mathcal{R}}_{\mathrm{corr},L}$ have almost surely the same behavior in terms of spectral norm, where $\overline{\mathcal{R}}_{\mathrm{corr},L} = \mathcal{B}^{-1/2}_L \widehat{\mathcal{R}}_L \mathcal{B}^{-1/2}_L$. 
%Recalling the definition of these two matrices, namely
%\[
%\widehat{\mathcal{R}}_{\mathrm{corr},L} =  \mathrm{Bdiag}(\widehat{\mathcal{R}}_L)^{-1/2} \, \widehat{\mathcal{R}}_L \, 
%\mathrm{Bdiag}(\widehat{\mathcal{R}}_L)^{-1/2} 
%= \widehat{\mathcal{B}}^{-1/2}_L \widehat{\mathcal{R}}_L \widehat{\mathcal{B}}^{-1/2}_L 
%\]
%and  
%\[\overline{\mathcal{R}}_{\mathrm{corr},L} = \mathrm{Bdiag}(\mathcal{R}_L)^{-1/2} \, \widehat{\mathcal{R}}_L \, \mathrm{Bdiag}(\mathcal{R}_L)^{-1/2}
%= \mathcal{B}^{-1/2}_L \widehat{\mathcal{R}}_L \mathcal{B}^{-1/2}_L
%\]
 
For this, we proceed in three steps. First, in Section \ref{subsec:boundedness-hatR}, we prove that the spectral norm of $\widehat{\mathcal{R}}_L$ is bounded with exponentially high probability. 
%Then, we show that $\| \mathrm{Bdiag}(\widehat{\mathcal{R}}_L) - \mathrm{Bdiag}(\mathcal{R}_L) \| \rightarrow 0$,
Then, using similar arguments, we show in Section \ref{subsec:convergence-estimators-RmL} that $\| \widehat{\mathcal{B}}_L - \mathcal{B}_L \| \prec \max(M^{-1/2},L^{-\gamma_0})$. Finally, in Section \ref{sec:evalTheta} we establish that $\| \widehat{\mathcal{B}}^{-1/2}_L - \mathcal{B}^{-1/2}_L \| \prec \max(M^{-1/2},L^{-\gamma_0})$ using Hermitian matrix perturbation results. 
%$\| \mathrm{Bdiag}(\widehat{\mathcal{R}}_L)^{-1/2} - \mathrm{Bdiag}(\mathcal{R}_L)^{-1/2} \| \rightarrow 0$. 
The fact that $\|\widehat{\mathcal{R}}_L\|$ is bounded with exponentially high probability will immediately imply that 
%\begin{equation}
%    \label{eq:simplification}
% \left \| \mathrm{Bdiag}(\widehat{\mathcal{R}}_L)^{-1/2}  \widehat{\mathcal{R}}_L
%\mathrm{Bdiag}(\widehat{\mathcal{R}}_L)^{-1/2} - %\mathrm{Bdiag}(\mathcal{R}_L)^{-1/2} \widehat{\mathcal{R}}_L
%\mathrm{Bdiag}\mathcal{R}_L)^{-1/2} \right \| \rightarrow 0   
%\end{equation}
\begin{equation}
    \label{eq:simplification}
% \left \| \widehat{\mathcal{B}}^{-1/2}_L \widehat{\mathcal{R}}_L \widehat{\mathcal{B}}^{-1/2}_L - \mathcal{B}^{-1/2}_L \widehat{\mathcal{R}}_L \mathcal{B}^{-1/2}_L \right \| \prec \frac{1}{\sqrt{M}}.
\left \| \widehat{\mathcal{R}}_{\mathrm{corr},L} - \overline{\mathcal{R}}_{\mathrm{corr},L} \right \| \prec \max\left(\frac{1}{\sqrt{M}},\frac{1}{L^{\gamma_0}}\right).
\end{equation}
%This will allow us to focus our analysis on $\overline{\mathcal{R}}_{\mathrm{corr},L}$ for the rest of the paper. 
% In Section \ref{sec:influenceblockmat} we will use this result to show that 
% %, under some additional regularity conditions on the function $\phi(\lambda)$, we have 
% $\|\widehat{\phi}_N - \overline{\phi}_N \| \prec \max{\left(M^{-1},L^{-\gamma_0}\right)}$.

We will write the normalized observations as $\mathbf{w}_{n,N}=\frac{1}{\sqrt{N
}}\mathbf{y}_n^L $, where $n=1,\ldots ,N$ and
\begin{equation}
\mathbf{W}_{N}=\left[ \mathbf{w}_{1,N},\ldots ,\mathbf{w}_{N,N}\right] .
\label{eq:def-W_N}
\end{equation}
Therefore $\widehat{\mathcal{R}}_{L}$ coincides with  $\widehat{\mathcal{R}}_{L}=\mathbf{W}_{N}\mathbf{W}_{N}^{H}$. 
%We recall here that we denote by $\mathcal{B}_L$ the block diagonal matrix $\mathrm{Bdiag}\mathcal{R}_{L} = \mathcal{R}_L$. 
% Now, the matrix $\overline{\mathcal{R}}_{\mathrm{corr},L}$ can also be written as
% $$
% \overline{\mathcal{R}}_{\mathrm{corr},L} = \mathcal{B}_L^{-1/2} \, \mathbf{W}_{N}\mathbf{W}_{N}^{H} 
% \mathcal{B}_L^{-1/2}.
% $$
In the following, we will often drop the index $N$, and will denote $\mathbf{W}
_{N}, {\bf w}_{j,N}, \mathbf{Q}_{N},\ldots $ by $\mathbf{W}, {\bf w}_j, \mathbf{Q},\ldots $ in order to
simplify the notations. 
%For $1\leq l\leq L$, $1\leq m\leq M$, and $1\leq
%j\leq N$. %, $\mathbf{W}_{i,j}^{m}$ represents the entry $\left(i+(m-1)L,j\right) $ of matrix $\mathbf{W}$.

  \subsection{Control of the largest eigenvalue of $\widehat{\mathcal{R}}_L$}
  \label{subsec:boundedness-hatR}

  The approach we follow is based on the observation that it is possible to
  add a bounded matrix to $\mathbf{W}_N \mathbf{W}_N^{H}$ to produce a block Toeplitz matrix.
Controlling the largest eigenvalue of $\mathbf{W}_N \mathbf{W}_N^{H}$
  becomes therefore equivalent to controlling the largest eigenvalue of the block Toeplitz matrix, a problem that can be solved by studying
  the supremum over the frequency interval of the spectral norm of the corresponding symbol. 
  
  \subsubsection{Modifying $\mathbf{W}_N \mathbf{W}_N^{H}$ into a block Toeplitz matrix}
  
  In order to present this result, it is more convenient
  to reorganize the rows of matrix $\mathbf{W}_N$. For this, we define for each $n$ the $M$ dimensional random vector
  $\y_n$ defined by
  \begin{equation}
    \label{eq:def-y}
    \y_n = \left( \begin{array}{c} y_{1,n} \\ \vdots \\ y_{M,n} \end{array} \right).
    \end{equation}
  $(\y_n)_{n \in \mathbb{Z}}$ is thus an $M$--dimensional stationary random sequence whose spectral density matrix $\S(\nu)$
coincides with the diagonal matrix $\S(\nu) = \mathrm{Diag}(\mathcal{S}_1(\nu), \ldots, \mathcal{S}_M(\nu))$. 
  We next consider the $ML \times N$ matrix $\W_N$, which is defined as
  \begin{equation}
    \label{eq:new-definition-W}
   \W_N = \frac{1}{\sqrt{N}} \, \left( \begin{array}{ccccc} \y_{1} & \y_{2} & \ldots & \y_{N-1} & \y_{N} \\
                                            \y_{2} & \y_{3} & \ldots & \y_{N} & \y_{N+1} \\
                                             \vdots & \vdots & \vdots & \vdots & \vdots \\
                                             \vdots & \vdots & \vdots & \vdots & \vdots \\
                                             \y_{L} & \y_{L+1} & \ldots & \y_{N+L-2} & \y_{N+L-1} \\
\end{array} \right).
\end{equation} 
Observe that $\W_{N}$ can be obtained by simple permutation of the rows of $\mathbf{W}_N$ and consequently $\W_{N}\W^H_{N}$ and $\mathbf{W}_N\mathbf{W}^H_N$ have the same eigenvalues. In particular, they have the same spectral norm. For this reason, we may focus on the behavior of $\W_{N}$ from now on. 

  We define matrices $\W_{N,1}$ and $\W_{N,2}$ as the $ML \times (N-L+1)$ and $ML \times (L-1)$ matrices such that
  $\W_N = (\W_{N,1}, \W_{N,2})$. In particular, matrix $\W_{N,2}$ is given by
  \begin{equation}
    \label{eq:def-W2}
    \W_{N,2} = \frac{1}{\sqrt{N}} \, \left( \begin{array}{ccccc} \y_{N-L+2} & \y_{N-L+3} & \ldots & \y_{N-1} & \y_{N} \\
                                            \y_{N-L+3} & \y_{N-L+4} & \ldots & \y_{N} & \y_{N+1} \\
                                             \vdots & \vdots & \vdots & \vdots & \vdots \\
                                             \vdots & \vdots & \vdots & \vdots & \vdots \\
                                             \y_{N+1} & \y_{N+2} & \ldots & \y_{N+L-2} & \y_{N+L-1} \\
\end{array} \right).
  \end{equation}
  We now express $\W_{N,2}$ as $\W_{N,2} = \W_{N,2,1} + \W_{N,2,2}$ where $\W_{N,2,1}$ is the upper block triangular matrix given by
 \begin{equation}
    \label{eq:def-W21}
    \W_{N,2,1} = \frac{1}{\sqrt{N}} \, \left( \begin{array}{ccccc} \y_{N-L+2} & \y_{N-L+3} & \ldots & \y_{N-1} & \y_{N} \\
                                                                 \y_{N-L+3} & \y_{N-L+4} & \ldots & \y_{N} & 0 \\
                                                                 \y_{N-L+4} & \ldots  & \y_N & 0 & 0 \\
                                             \vdots & \vdots & \vdots & \vdots & \vdots \\
                                            \y_N & 0 & \vdots & \vdots & 0 \\
                                            0  & 0 & \ldots & 0 & 0 \\
\end{array} \right)
  \end{equation}
 and where $\W_{N,2,2}$ is the lower block triangular matrix defined by
 \begin{equation}
    \label{eq:def-W22}
    \W_{N,2,2} = \frac{1}{\sqrt{N}} \, \left( \begin{array}{ccccc} 0 & 0 & \ldots & 0 & 0 \\
                                            0 & 0 & \ldots & 0 & \y_{N+1} \\
                                            0 & \vdots & 0 & \y_{N+1} & \y_{N+2} \\
                                            \vdots & \vdots & \vdots & \vdots & \vdots \\
                                            0 & \y_{N+1} & \ldots & \ldots & \y_{N+L-2} \\
                                             \y_{N+1} & \y_{N+2} & \ldots & \y_{N+L-2} & \y_{N+L-1} \\
\end{array} \right).
  \end{equation} 
 In other words, matrix $\W_{N,2,1}$ is obtained by replacing in $\W_{N,2}$ vectors $\y_{N+1}, \ldots, \y_{N+L-1}$ by
 $\bf{0}, \ldots, \bf{0}$ while  $\W_{N,2,2}$ is obtained by replacing in $\W_{N,2}$ vectors $\y_{N-L+2}, \ldots, \y_{N}$ by
 $\bf{0}, \ldots, \bf{0}$. We also define $\W_{N,0}$ as the $ML \times (L-1)$ lower block triangular matrix given by
 \begin{equation}
   \label{eq:def-W0}
     \W_{N,0} = \frac{1}{\sqrt{N}} \, \left( \begin{array}{ccccc} 0 & 0 & \ldots & 0 & 0 \\
                                            0 & 0 & \ldots & 0 & \y_{1} \\
                                            0 & \vdots & 0 & \y_{1} & \y_{2} \\
                                            \vdots & \vdots & \vdots & \vdots & \vdots \\
                                            0 & \y_{1} & \ldots & \ldots & \y_{L-2} \\
                                             \y_{1} & \y_{2} & \ldots & \y_{L-2} & \y_{L-1} \\
\end{array} \right).
  \end{equation} 
 We finally introduce the $ML \times (N+L-1)$ block Hankel matrix $\widetilde{\W}_N$ defined by
 \begin{equation}
   \label{eq:def-tildeW}
   \widetilde{\W}_N = (\W_{N,0}, \W_{N,1}, \W_{N,2,1}).
   \end{equation}
 It is easy to check that $\widetilde{\W}_N \widetilde{\W}_N^{H}$ is the block Toeplitz matrix whose $M \times M$  blocks
 $\left((\widetilde{\W}_N \widetilde{\W}_N^{H})_{k,l}\right)_{k,l=1, \ldots, L}$ are given by
 $$
 (\widetilde{\W}_N \widetilde{\W}_N^{H})_{k,l} = \widehat{{\bf R}}_{k-l}
 $$
 where the $M \times M$ matrices $(\widehat{{\bf R}}_l)_{l=-(L-1), \ldots, L-1}$ are defined by
 $$
 \widehat{{\bf R}}_l = \frac{1}{N} \sum_{n=1}^{N-l} \y_{n+l} \y_n^H
 $$
 for $l \geq 0$ and $\widehat{{\bf R}}_l = \widehat{{\bf R}}^{H}_{-l}$ for $l \leq 0$. In other words, for each
 $l$, $\widehat{{\bf R}}_l$ is the standard empirical biased estimate of the autocovariance matrix at lag $l$ of the multivariate time
 series $(\y_n)_{n \in \mathbb{Z}}$. 
 
 Matrix $\widetilde{\W}_N \widetilde{\W}_N^{H}$ also coincides with the
 block Toeplitz matrix associated to the symbol $\widehat{\S}(\nu)$ defined by
 \begin{equation}
   \label{eq:def-hatS}
   \widehat{\S}(\nu) = \sum_{l=-(L-1)}^{L-1} \widehat{{\bf R}}_l \mathrm{e}^{-2 i \pi l \nu}
 \end{equation}
 so that we can write
 \begin{equation}
   \label{eq:expre-integrale-modif-gram}
   \widetilde{\W}_N \widetilde{\W}_N^H = \int_{0}^{1} \dd_L(\nu) \dd_{L}^{H}(\nu) \otimes \widehat{\S}(\nu)  \, d\nu.
 \end{equation}
The $M \times M$ matrix $\widehat{\S}(\nu)$ coincides with a lag window estimator of the spectral density of $(\y_n)_{n \in \mathbb{Z}}$. Evaluating the spectral norm
 of $\widetilde{\W}_N \widetilde{\W}_N^{H}$ is easier than that of $\W_N \W_N^{H}$, because the spectral norm of $\widetilde{\W}_N \widetilde{\W}_N^{H}$ is upper bounded by $\sup_{\nu \in [0,1]} \|  \widehat{\S}(\nu) \|$,
 a term that can be controlled using a discretization in the frequency domain and the epsilon net argument in $\mathbb{C}^{M}$ (see e.g. \cite{tao-book} for an introduction to the concept of epsilon net). In the reminder of this section, we first prove that $\| \W_N \W_N^{H} - \widetilde{\W}_N \widetilde{\W}_N^{H} \|$
 is bounded with exponentially high probability and then establish that  $\sup_{\nu \in [0,1]} \|  \widehat{\S}(\nu) \|$, and thus $\|\widetilde{\W}_N \widetilde{\W}_N^{H}\|$ is
 also bounded with exponentially high probability. 

 We first state the following lemma, which will allow to reduce various suprema on the interval $[0,1]$ to the corresponding suprema on a finite grid of the same interval.
 This result is adapted from Zygmund \cite{zygmund}, and was used in \cite{xiao-wu-2012}. 
 \begin{lemma}
   \label{le:zygmund}
   Let $h(\nu) = \sum_{l=-(L-1)}^{L-1} h_l\mathrm{e}^{-2 i \pi l \nu}$ an order  $L-1$ real valued trigonometric polynomial. Then, for each $\nu_0 \in [0,1]$,
   $\delta > 0$, $K \geq 2(1+\delta)(L-1)$, we define $\nu_k = \nu_0 + k/K$ for $k=0, \ldots, K$. Then, it holds that
   \begin{equation}
     \label{eq:sygmund}
     \max_{\nu \in [0,1]} | h(\nu)| \leq \left( 1 + \frac{1}{\delta} \right) \max_{k=0, \ldots, K} |h(\nu_k)|.
   \end{equation}
 \end{lemma}
 We now compare the spectral norms of  $\W_N \W_N^{H}$ and $\widetilde{\W}_N \widetilde{\W}_N^{H}$.
 \begin{proposition}
   \label{prop:spectral-norm-w-tildew}
   Let $\alpha$ denote a large enough constant. Under Assumptions \ref{assum:statistics}-\ref{ass:bounds-spectral-densities} and \ref{as:norm-r-omega}, it holds that
   \begin{equation}
       \label{eq:spectral-norm-w-tildew}
       \mathbb{P} \left(\|  \W_N \W_N^{H} - \widetilde{\W}_N \widetilde{\W}_N^{H} \| > \alpha \right) \leq \kappa_1 L \exp{( -\kappa_2 M \alpha)}.
   \end{equation}
   for two nice constants $\kappa_1$ and $\kappa_2$. 
%   If $\alpha$ is a large enough constant, then, it holds that
%   \begin{equation}
%     \label{eq:spectral-norm-w-tildew}
%     \mathbb{P}\left( \|  \W_N \W_N^{H} - \widetilde{\W}_N \widetilde{\W}_N^{H} \| > \alpha \right) \leq \kappa_1 L \exp{( -\kappa_2 M \alpha)}
%   \end{equation}
%   for some nice constants $\kappa_1$ and $\kappa_2$.
 \end{proposition}
   \begin{proof}
   We drop all the subindexes $N$ from all the matrices for clarity of exposition. Matrix $\W \W^{H}$ is equal to $\W \W^{H} = \W_1 \W_1^{H} + (\W_{2,1} + \W_{2,2}) (\W_{2,1} + \W_{2,2})^{H}$ while
     $\widetilde{\W} \widetilde{\W}^{H} = \W_0 \W_0^{H} + \W_1 \W_1^{H} + \W_{2,1} \W_{2,1}^{H}$. Therefore,
     $$
     \W \W^{H} - \widetilde{\W} \widetilde{\W}^{H} =  \W_{2,2} \W_{2,2}^{H} +  \W_{2,2} \W_{2,1}^{H} +  \W_{2,1} \W_{2,2}^{H} - \W_0 \W_0^{H}.
     $$
     In order to establish (\ref{eq:spectral-norm-w-tildew}), we have to show that $\mathbb{P}(\| \W_{2,i} \W_{2,j}^{H} \| > \alpha)$, $i,j=1,2$,  
     and $\mathbb{P}(\| \W_{0} \W_{0}^{H} \| > \alpha)$ decrease at the same rate as the right hand side of (\ref{eq:spectral-norm-w-tildew}).
     We just establish this property for matrix $\W_0 \W_0^{H}$, or equivalently for matrix $\widetilde{\W}_0 \widetilde{\W}^H_0$, where $\widetilde{\W}_0$ is defined as
     $$
     \widetilde{\W}_0 = \frac{1}{\sqrt{N}} \, \left( \begin{array}{ccccc} \y_1 & 0 & \ldots & 0 & 0 \\
                                            \y_2 & \y_1 & \ldots & 0 & 0 \\
                                            \vdots & \ddots & \ddots & 0 & 0 \\
                                            \y_{L-2} & \y_{L-3} & \ddots & \y_1 & 0 \\
                                             \y_{L-1} & \y_{L-2} & \ldots & \y_{2} & \y_{1} \\
     \end{array} \right).
     $$
     It is easily seen that $\widetilde{\W}_0$ can be expressed as
     \begin{equation}
       \label{eq:expre-integrale-W0}
       \widetilde{\W}_0 = \sqrt{\frac{L}{N}} \int_{0}^{1}  \dd_{L-1}(\nu) \dd_{L-1}^{H}(\nu)  \otimes \bs{\xi}_{L,y}(\nu)  \, d\nu
     \end{equation}
     where $\bs{\xi}_{L,y}(\nu)$ is an $M$-dimensional column vector defined as $\bs{\xi}_{L,y}(\nu) = \frac{1}{\sqrt{L}} \sum_{l=0}^{L-2} \y_{l+1}\mathrm{e}^{-2 i \pi l \nu} $. The matrix version of the Cauchy-Schwarz inequality in (\ref{eq:schwartz-matriciel-v2}) with $\mathbf{U}(\nu)=\sqrt{\frac{L}{N}}\dd_{L-1}(\nu)  \otimes \bs{\xi}_{L,y}(\nu)$ and $\mathbf{V}(\nu) =  \dd_{L-1}(\nu)$ leads immediately to
     $$
     \widetilde{\W}_0 \widetilde{\W}_0^{H} \leq \frac{L}{N} \, \int_{0}^{1}  \dd_{L-1}(\nu) \dd_{L-1}^{H} (\nu)\otimes  \bs{\xi}_{L,y}(\nu)  \bs{\xi}_{L,y}^{H}(\nu) \, d\nu.
     $$
     From this, we obtain immediately that
     $$
     \| \widetilde{\W}_0 \widetilde{\W}_0^{H} \| \leq \sup_{\nu \in [0,1]} \frac{L}{N} \, \|  \bs{\xi}_{L,y}(\nu) \|^{2}.
     $$
     Next, observe that $\nu \rightarrow  \frac{L}{N} \|  \bs{\xi}_{L,y}(\nu) \|^{2}$ is a real valued trigonometric polynomial of order $L-2$. Therefore, if $K, \delta$ and the points
     $(\nu_k)_{k=0, \ldots, K}$ are given as in Lemma \ref{le:zygmund}, it holds that
    %  $$
    %      \left\| \widetilde{\W}_0 \widetilde{\W}_0^{H} \right\| \leq \left( 1+ \frac{1}{\delta} \right) \sup_{k=0,\ldots,K} \frac{L}{N} \| \boldsymbol{\xi}_{K,y} (\nu_k)\|^2.
    %  $$
    $$
         \left\| \widetilde{\W}_0 \widetilde{\W}_0^{H} \right\| \leq \left( 1+ \frac{1}{\delta} \right) \sup_{k=0,\ldots,K} \frac{L}{N} \| \boldsymbol{\xi}_{L,y} (\nu_k)\|^2.
     $$
     Noting that $K = \mathcal{O}(L)$, it is sufficient to evaluate $\mathbb{P}\left(  \frac{L}{N} \| \boldsymbol{\xi}_{L,y} (\nu)\|^2 > \eta \right)$ for some fixed $\nu$ and some well chosen constant $\eta$, and then use the union bound.
    %  $$
    %  \mathbb{P}\left( \| \widetilde{\W}_0 \widetilde{\W}_0^{H} \| > \alpha\right) \leq\mathbb{P}\left(\sup_{\nu \in [0,1]}  \frac{L}{N} \|  \bs{\xi}_{L,y}(\nu) \|^{2} >  \alpha\right) \leq
    %  \mathbb{P}\left(\sup_{k=0, \ldots, K}  \frac{L}{N} \|  \bs{\xi}_{L,y}(\nu_k) \|^{2} > \kappa \alpha\right)
    %  $$ 
    %  for some nice constant $\kappa$. Using the union bound, we get that
    %  $$
    %  \mathbb{P}\left( \| \widetilde{\W}_0 \widetilde{\W}_0^{H} \| > \alpha\right) \leq \sum_{k=0}^{K} \mathbb{P} \left(  \frac{L}{N} \|  \bs{\xi}_{L,y}(\nu_k) \|^{2} > \kappa \alpha \right).
    %  $$
    % Therefore, we just have to evaluate an upper bound of $\mathbb{P} (  \frac{L}{N} \|  \bs{\xi}_{L,y}(\nu) \|^{2} > \kappa \alpha)$ uniform w.r.t. the frequency $\nu$, where
     Observe first that we can express %$ \frac{L}{N} \|  \bs{\xi}_{L,y}(\nu) \|^{2}$ as 
     $$
      \frac{L}{N} \|  \bs{\xi}_{L,y}(\nu) \|^{2} = \frac{ML}{N} \frac{1}{M} \sum_{m=1}^{M} \left| \xi_{L,y_m}(\nu) \right|^{2}
     $$
     where  $\xi_{L,y_m}(\nu)$, $m=1,\ldots,M$ are components of $\bs{\xi}_{L,y}(\nu)$. These are mutually independent complex Gaussian random variables, so that we can use the Hanson-Wright inequality 
    in order to establish an exponential concentration inequality on   $\frac{L}{N} \|  \xi_{L,y}(\nu) \|^{2}$. % control $\mathbb{P} (  \frac{L}{N} \|  \xi_{L,y}(\nu) \|^{2} > \kappa \alpha)$.

     In order to use (\ref{eq:hanson-wright}), we remark that for each $m$, $\xi_{L,y_m}(\nu)$
     can be written as $\xi_{L,y_m}(\nu)= \left(\mathbb{E} | \xi_{L,y_m}(\nu) |^{2}\right)^{1/2} x_m$ where $x_1, \ldots, x_M$ are $\mathcal{N}_\mathbb{C}(0,1)$ 
     i.i.d. random variables. If ${\bf x}=(x_1, \ldots, x_M)$, $\frac{1}{M} \sum_{m=1}^{M} \left| \xi_{L,y_m}(\nu) \right|^{2}$ can be written as
     $$
     \frac{1}{M} \sum_{m=1}^{M} \left| \xi_{L,y_m}(\nu) \right|^{2} = 
     {\bf x}^{H} {\boldsymbol \Xi}(\nu)  {\bf x}
     $$
     where ${\boldsymbol \Xi}(\nu)$ is the $M \times M$ diagonal matrix with $m$th diagonal entry equal to 
     $$
     \left[\boldsymbol{\Xi}(\nu)\right]_{m,m} = \frac{1}{M} \mathbb{E} | \xi_{L,y_m}(\nu) |^{2}.
     $$
     In order to evaluate 
     $\|{\boldsymbol \Xi}(\nu)\|$ and $\|{\boldsymbol \Xi}(\nu)\|_F^{2}$, we have to study the behaviour of  $\mathbb{E} | \xi_{L,y_m}(\nu) |^{2}$, i.e. the expectation of the periodogram of the sequence $y_{m,1}, \ldots, y_{m,L-1}$. The following result establishes that the diagonal entries of this matrix are equal to scaled versions of the spectral densities $\frac{1}{M} \mathcal{S}_m(\nu), m=1,\ldots,M,$ up to an error that decays as $\mathcal{O}\left(\frac{1}{ML^{\min(1,\gamma_0)}}\right)$. 
     \begin{lemma}
     \label{le:expectation-periodogram} Under Assumptions \ref{assum:statistics} and \ref{as:norm-r-omega},
     $ \mathbb{E}| \xi_{L,y_m}(\nu) |^{2}$ can be written as
     $\mathbb{E}| \xi_{L,y_m}(\nu) |^{2} =  \mathcal{S}_m(\nu) + \epsilon_{m,L}(\nu)$ where $\epsilon_{m,L}(\nu)$ verifies
     \begin{equation}
         \label{eq:reminder-expectation-periodogram}
     | \epsilon_{m,L}(\nu) | \leq \frac{\kappa}{(L-1)^{\min(1,\gamma_0)}}    
     \end{equation}
     for each $m$ and for some nice constant $\kappa$.
     \end{lemma}
     Lemma \ref{le:expectation-periodogram} is proved in Appendix \ref{sec:app_proof_lemma_period}.
     
     This lemma implies that there exists a nice constant $\kappa$ for 
     which $ \mathbb{E}| \xi_{L,y_m}(\nu) |^{2} \leq \kappa$ for each 
     $\nu$ and each $m$ and $L > 1$. Therefore, if ${\boldsymbol \Xi}(\nu)$ is the above mentioned diagonal matrix, ${\boldsymbol \Xi}(\nu)$ verifies $\|{\boldsymbol \Xi}(\nu)\| \leq \frac{\kappa}{M}$
     and $\| {\boldsymbol \Xi}(\nu) \|_{F}^{2} \leq  \frac{\kappa^{2}}{M}$. 
    %  Consequently, from (\ref{eq:hanson-wright-domination}) we obtain 
    %  \begin{equation*}
    %      \left| \mathbf{x}^H {\boldsymbol \Xi} (\nu) \mathbf{x} - \mathrm{Tr}\left({\boldsymbol \Xi}(\nu)\right) \right| \prec \frac{1}{M}
    %  \end{equation*}
    %  and, since $\mathrm{Tr}\left({\boldsymbol \Xi}(\nu)\right) < \kappa$, we directly have
    %  \begin{equation*}
    %      \frac{L}{N} \|\boldsymbol{\xi}_{L,y}(\nu_k) \|^2 = \frac{ML}{N}\left| \mathbf{x}^H {\boldsymbol \Xi} (\nu_k) \mathbf{x} \right| \prec 1.
    %  \end{equation*}
    %  for $k=0,\ldots,K$. 
     Consider a nice constant $\eta > 2 \kappa$. Then, 
     \begin{eqnarray*}
     \mathbb{P} \left(\frac{1}{M} \sum_{m=1}^{M} \left| \xi_{L,y_m}(\nu) \right|^{2} > \eta\right) & \leq & 
     \mathbb{P} \left(\frac{1}{M} \sum_{m=1}^{M} \left| \xi_{L,y_m}(\nu) \right|^{2} - \mathbb{E}| \xi_{L,y_m}(\nu)|^{2}  > \eta - \kappa \right) \\
         & \leq &  \mathbb{P} \left(\frac{1}{M} \sum_{m=1}^{M} \left| \xi_{L,y_m}(\nu) \right|^{2} - \mathbb{E}| \xi_{L,y_m}(\nu)|^{2}  > \eta/2 \right).
    \end{eqnarray*}
    As $\min\left(\frac{M(\eta/2)}{\kappa},  \frac{M(\eta/2)^{2}}{\kappa^{2}} \right) = 
     \frac{M(\eta/2)}{\kappa}$  and $ML/N \rightarrow c_\star$, the Hanson-Wright inequality leads to 
    $$
     \mathbb{P} \left(\frac{ML}{N} \, \frac{1}{M} \sum_{m=1}^{M} \left| \xi_{L,y_m}(\nu) \right|^{2} > \eta\right) \leq \kappa_1 \exp (- M \kappa_2 \, \eta)
     $$
     for some nice constants $\kappa_1$ and $\kappa_2$.
%Recalling that $K = \mathcal{O}(L)$, we have shown by Lemma \ref{lem:stochastic_dom_maximum} that $\| \widetilde{\W}_0 \widetilde{\W}_0^{H} \| \prec 1 $.
     Recalling that $K = \mathcal{O}(L)$, and using the union bound to evaluate 
     $\mathbb{P}\left( \sup_{k=0,\ldots,K} \frac{L}{N} \| \boldsymbol{\xi}_{K,y} (\nu_k)\|^2 > \eta \right)$, we have shown that, if $\alpha$ is large enough, there exist two nice constants $\kappa_1$ and 
     $\kappa_2$ such that 
     $$
      \mathbb{P}( \| \widetilde{\W}_0 \widetilde{\W}_0^{H} \| > \alpha) \leq L \, \kappa_1 \exp (- M \kappa_2 \alpha).
     $$
     Following the same approach to evaluate the other terms $\|\W_{N,i} \W_{N,j}^{H}\|$, we can conclude that
     (\ref{eq:spectral-norm-w-tildew}) is established.
     %$\blacksquare$
     \end{proof}

     As a consequence of Proposition \ref{prop:spectral-norm-w-tildew}, the evaluation of $\mathbb{P}(\|\W_N \W_N^{H}\| > \alpha)$ can be alternatively formulated in terms of the evaluation of $\mathbb{P}(\|\widetilde{\W}_N \widetilde{\W}_N^{H}\| > \alpha)$.
 %As a consequence of Proposition \ref{prop:spectral-norm-w-tildew}, we will be able to establish that $\|\W_N \W_N^{H}\| \prec 1$ if we are able to see that $\|\widetilde{\W}_N \widetilde{\W}_N^{H}\| \prec 1$. 
     
 \subsubsection{Controlling the spectral norm of $\widetilde{\W}_N \widetilde{\W}_N^{H}$}
 
In order to establish the fact that $\|\widetilde{\W}_N \widetilde{\W}_N^{H}\|$ is bounded with exponentially large probability, we use the expression in (\ref{eq:expre-integrale-modif-gram}) and remark that 
$$
\left\|\widetilde{\W}_N \widetilde{\W}_N^{H}\right\| \leq \sup_{\nu \in [0,1]} \|\widehat{\S}(\nu)\|.
$$
In the following, we thus control the spectral norm of $\widehat{\S}(\nu)$. In particular, we have the following result. 
%{\color{blue} This property is not sufficient to deduce Proposition \ref{prop:concentration-hatcorr-overlinecorr}. We should replace it by what was stated in an earlier version, i.e. that $\mathbb{P}(\sup_{\nu} \| \widehat{\S}(\nu) \| > \alpha) \leq \kappa_1 L \exp -(\kappa_2 M \alpha)$ for each $\alpha$ large enough.} {\color{red} Sorry about that, I agree this was completely stupid. I have now restored the original version.}
% \begin{proposition}
%     \label{prop:sup-hatS}
%     Under Assumptions \ref{assum:statistics}, \ref{as:asymptotic-regime}, \ref{ass:bounds-spectral-densities}, \ref{as:norm-r-omega} we have 
%     \begin{equation}
%         \label{eq:concentration-sup-hatS}
%         \sup_{\nu \in [0,1]} \left\| \widehat{\S}(\nu) \right\| \prec 1.
%     \end{equation}
% \end{proposition}
\begin{proposition}
\label{prop:sup-hatS}
If $\alpha$ is a large enough constant, under Assumptions \ref{assum:statistics}-\ref{ass:bounds-spectral-densities} and \ref{as:norm-r-omega}, it holds that
\begin{equation}
\label{eq:concentration-sup-hatS}
\mathbb{P}\left( \sup_{\nu \in [0,1]} \left\| \widehat{\S}(\nu) \right\| > \alpha \right) < \kappa_1 L \, \exp{ (-\kappa_2 M \alpha)}
\end{equation}
for some nice constants $\kappa_1$ and $\kappa_2$. 
\end{proposition}
\begin{proof}
We denote by $\widehat{\S}^{\circ}(\nu)$ the centered matrix  $\widehat{\S}^{\circ}(\nu) =  \widehat{\S}(\nu) - \mathbb{E}\widehat{\S}(\nu)$. We first notice that 
$$
\sup_{\nu} \| \widehat{\S}(\nu) \| \leq \sup_{\nu} \| \mathbb{E}\widehat{\S}(\nu) \| + \sup_{\nu} \| \widehat{\S}^\circ(\nu) \|
$$ 
and work on the two terms separately. First, we prove that $\sup_{\nu} \| \mathbb{E}\widehat{\S}(\nu) \|$ is bounded. Indeed, it is clear that
$\mathbb{E}(\widehat{\S}(\nu)) = \sum_{l=-(L-1)}^{L-1} (1 - \frac{|l|}{L}) {\bf R}(l) \mathrm{e}^{-2 i \pi l \nu}$
where ${\bf R}(l) = \mathbb{E}(\y_{n+l}\y_n^H)$ is the autocovariance matrix of $\y_n$ at lag $l$. 
Since the components of $\y_n$ are independent time series, matrix ${\bf R}(l)$ coincides with 
${\bf R}(l) = \mathrm{Diag}\left( (r_m(l))_{m=1, \ldots, M} \right)$. Therefore, 
$$
\| \mathbb{E}\widehat{\S}(\nu) \| \leq \sup_{m=1, \ldots, M} \sum_{l=-(L-1)}^{L-1} |r_m(l)| \leq \sup_{m \geq 1} \sum_{l \in \mathbb{Z}}  |r_m(l)|.
$$
Condition (\ref{eq:uniform-norm-omega0-rm}) thus implies that  $\sup_{\nu} \| \mathbb{E}\widehat{\S}(\nu) \| < +\infty$. 
Therefore, in order to establish (\ref{eq:concentration-sup-hatS}), we need to study %$\sup_{\nu} \| \widehat{\S}^\circ(\nu) \| \prec 1$. To that effect, we will study
$\mathbb{P} ( \sup_{\nu \in [0,1]} \| \widehat{\S}^\circ(\nu)  \| > \alpha )$ for $\alpha$ sufficiently large. 
% We first observe that, for each $\nu$, we have 
% $ \| \widehat{\S}(\nu) - \mathbb{E}\widehat{\S}(\nu) \| \geq  \| \widehat{\S}(\nu) \| - \| \mathbb{E}\widehat{\S}(\nu) \| \geq \| \widehat{\S}(\nu) \| - \sup_{\nu} \| \mathbb{E}\widehat{\S}(\nu) \|$. Therefore, it holds that 
% $$
% \sup_{\nu} \| \widehat{\S}(\nu) - \mathbb{E}\widehat{\S}(\nu) \| \geq \sup_{\nu} \| \widehat{\S}(\nu) \| -  \sup_{\nu} \| \mathbb{E}\widehat{\S}(\nu) \|.
% $$
% If we choose $\alpha > 2 \sup_{\nu} \|  \mathbb{E}\widehat{\S}(\nu) \|$ we have $\alpha -  \sup_{\nu} \|  \mathbb{E}\widehat{\S}(\nu) \| \geq \alpha/2$. 
% Consequently, the set $\{ \sup_{\nu} \| \widehat{\S}(\nu) \| > \alpha \}$ is included in the 
% set  $\{ \sup_{\nu} \| \widehat{\S}(\nu) - \mathbb{E} \widehat{\S}(\nu) \| > \alpha/2 \}$ and the left hand side of (\ref{eq:concentration-sup-hatS}) is upper bounded by
% $$
% P\left( \sup_{\nu \in [0,1]} \left\| \widehat{\S}(\nu) - \mathbb{E}\widehat{\S}(\nu) \right\| > \alpha/2 \right).
% $$
% It therefore remains to establish that the above probability can be upper bounded by a term of the form $\kappa_1 L \, \exp{ (-\kappa_2 M \alpha)}$. 

We first show that the study of the supremum of $\| \widehat{\S}^{\circ}(\nu) \|$ over 
$[0,1]$ can be reduced to the supremum over a discrete grid with $\mathcal{O}(L)$ elements. The idea is to make use Lemma \ref{le:zygmund} by conveniently expressing $\| \widehat{\S}^{\circ}(\nu) \|$ in terms of trigonometric polynomials.
\begin{lemma}
\label{le:discrete-grid}
We consider $\delta$, $K$, and $(\nu_k)_{k=0, \ldots, K}$ as in Lemma \ref{le:zygmund}. Then, the following 
result holds:
\begin{equation}
\label{eq:discrete-grid-hatScirc}
\sup_{\nu \in [0,1]} \| \widehat{\S}^{\circ}(\nu) \| \leq \left( 1 + \frac{1}{\delta} \right) \, \sup_{k=0, \ldots, K} \| \widehat{\S}^{\circ}(\nu_k) \|.
\end{equation}
\end{lemma}
\begin{proof}
We will first verify that 
\begin{equation}
\label{eq:max-max}
\sup_{\nu \in [0,1]} \| \widehat{\S}^{\circ}(\nu) \| = \sup_{\nu \in [0,1], \mathbf{h} \in \mathbb{S}^{M-1}} \left| 
\mathbf{h}^H \widehat{\S}^{\circ}(\nu) \mathbf{h} \right|
\end{equation}
where $\mathbb{S}^{M-1}$ is the unit sphere in $\mathbb{C}^{M}$. We remark that, because of the continuity of the spectral norm as well as the continuity of both true and estimated spectral densities, there exists a certain $\widehat{\nu}$ that achieves the supremum on the left hand side of (\ref{eq:max-max}), that is $\sup_{\nu \in [0,1]} \| \widehat{\S}^{\circ}(\nu) \| = 
\| \widehat{\S}^{\circ}(\widehat{\nu}) \|$. Moreover, for such given $\widehat{\nu}$, there exists a $\mathbf{h}_{\widehat{\nu}} \in \mathbb{S}^{M-1}$ for which 
$\| \widehat{\S}^{\circ}(\widehat{\nu})  \| = | \mathbf{h}^H_{\widehat{\nu}} \widehat{\S}^{\circ}(\widehat{\nu}) \mathbf{h}_{\widehat{\nu}} |$. In other words, 
$\sup_{\nu \in [0,1]} \| \widehat{\S}^{\circ}(\nu) \|$ coincides with $ | 
\mathbf{h}_{\widehat{\nu}}^H \widehat{\S}^{\circ}(\widehat{\nu}) \mathbf{h}_{\widehat{\nu}} |$. Hence, we obtain that the left hand side of (\ref{eq:max-max})
is less than the right hand side of (\ref{eq:max-max}). The converse inequality is obvious. 

Using a similar continuity argument, we can readily see that 
$$
\sup_{\nu \in [0,1], \mathbf{h} \in \mathbb{S}^{M-1}} | 
\mathbf{h}^H \widehat{\S}^{\circ}(\nu) \mathbf{h} | = | \mathbf{h}^H_{\widehat{\nu}} \widehat{\S}^{\circ}(\widehat{\nu}) \mathbf{h}_{\widehat{\nu}} |
$$ 
also coincides with $\sup_{\nu \in [0,1]} | \mathbf{h}^H_{\widehat{\nu}} \widehat{\S}^{\circ}(\nu) \mathbf{h}_{\widehat{\nu}} |$. 
The function $\nu \rightarrow \mathbf{h}^H_{\widehat{\nu}} \widehat{\S}^{\circ}(\nu) \mathbf{h}_{\widehat{\nu}}$ is a real valued trigonometric 
polynomial of order $L-1$. Therefore, Lemma \ref{le:zygmund} implies that 
$$
\sup_{\nu \in [0,1]} \left| \mathbf{h}^H_{\widehat{\nu}} \widehat{\S}^{\circ}(\nu) \mathbf{h}_{\widehat{\nu}} \right| \leq \left( 1 + \frac{1}{\delta} \right) \, \sup_{k=0, \ldots, K} \left| \mathbf{h}^H_{\widehat{\nu}} \widehat{\S}^{\circ}(\nu_k) \mathbf{h}_{\widehat{\nu}} \right|.
$$
Since $| \mathbf{h}^H_{\widehat{\nu}} \widehat{\S}^{\circ}(\nu_k) \mathbf{h}_{\widehat{\nu}} | \leq \|  \widehat{\S}^{\circ}(\nu_k) \|$, 
we have shown that 
$$
\sup_{\nu \in [0,1]} \| \widehat{\S}^{\circ}(\nu) \| = \sup_{\nu \in [0,1]} \left| \mathbf{h}^H_{\widehat{\nu}} \widehat{\S}^{\circ}(\nu) \mathbf{h}_{\widehat{\nu}} \right| \leq \left( 1 + \frac{1}{\delta} \right) \sup_{k=0, \ldots, K}  \|  \widehat{\S}^{\circ}(\nu_k) \|.
$$
This establishes (\ref{eq:discrete-grid-hatScirc}).
\end{proof}

We now complete the proof of (\ref{eq:concentration-sup-hatS}) in Proposition \ref{prop:sup-hatS}. The union bound 
leads to 
\begin{equation}
\label{eq:union-bound-frequency-domain}
\mathbb{P}\left( \sup_{\nu \in [0,1]} \| \widehat{\S}^{\circ}(\nu) \| > \alpha_N \right) 
\leq \sum_{k=0}^{K} \mathbb{P}\left( \| \widehat{\S}^{\circ}(\nu_k) \| > \frac{\delta}{1+\delta} \; \alpha_N \right).
\end{equation}
Thus, we only need to evaluate $ \mathbb{P}( \| \widehat{\S}^{\circ}(\nu) \| > \eta_N )$, where $\nu$ is a fixed 
frequency and where $\eta_N =  \frac{\delta}{1+\delta} \; \alpha_N$. For this, we use the epsilon net argument 
in $\mathbb{C}^{M}$. We recall that an epsilon net $\mathcal{N}_{\epsilon}$ of $\mathbb{C}^{M}$ is a finite set of unit norm vectors of $\mathbb{C}^{M}$ having the property 
that for each $\mathbf{g} \in \mathbb{S}^{M-1}$, there exists an $\mathbf{h} \in \mathcal{N}_{\epsilon}$ such that 
$\| \mathbf{g} - \mathbf{h} \| \leq \epsilon$. It is well known that the cardinal $|\mathcal{N}_{\epsilon}|$ is upper bounded by $\left(\frac{\kappa}{\epsilon}\right)^{2M}$ for some nice constant $\kappa$. We consider such an epsilon net $\mathcal{N}_{\epsilon}$ and 
denote by $\widehat{\mathbf{h}}$ a vector of $\mathbb{S}^{M-1}$ for which 
$\| \widehat{\S}^{\circ}(\nu) \| = | \widehat{\mathbf{h}}^H  \widehat{\S}^{\circ}(\nu) \widehat{\mathbf{h}} | $, and consider 
a vector $\widetilde{\mathbf{h}} \in \mathcal{N}_{\epsilon}$ such that $\| \widehat{\mathbf{h}} - \widetilde{\mathbf{h}} \| \leq \epsilon$. 
We express $\widetilde{\mathbf{h}}^H \widehat{\S}^{\circ}(\nu)  \widetilde{\mathbf{h}}$ as 
$$
\widetilde{\mathbf{h}}^H \widehat{\S}^{\circ}(\nu)  \widetilde{\mathbf{h}} = \left(\widehat{\mathbf{h}} + \widetilde{\mathbf{h}} - \widehat{\mathbf{h}}\right)^H  \widehat{\S}^{\circ}(\nu) 
\left(\widehat{\mathbf{h}} + \widetilde{\mathbf{h}} - \widehat{\mathbf{h}}\right).
$$
Using the triangular inequality, we obtain that 
$$
\left| \widetilde{\mathbf{h}}^H \widehat{\S}^{\circ}(\nu)  \widetilde{\mathbf{h}} \right| \geq \left| \widehat{\mathbf{h}}^H \widehat{\S}^{\circ}(\nu)  \widehat{\mathbf{h}} \right|
- 2 \left|  (\widetilde{\mathbf{h}} - \widehat{\mathbf{h}})^H  \widehat{\S}^{\circ}(\nu) \widehat{\mathbf{h}} \right| - 
\left|  (\widetilde{\mathbf{h}} - \widehat{\mathbf{h}})^H  \widehat{\S}^{\circ}(\nu) (\widetilde{\mathbf{h}} - \widehat{\mathbf{h}})^H \right|.
$$
Since $\widetilde{\mathbf{h}} \in \mathbb{S}^{M-1}$ and $\| \widehat{\mathbf{h}} - \widetilde{\mathbf{h}} \| \leq \epsilon$, we can write
$$
\left|  (\widetilde{\mathbf{h}} - \widehat{\mathbf{h}})^H  \widehat{\S}^{\circ}(\nu) \widehat{\mathbf{h}} \right| \leq \|  \widehat{\S}^{\circ}(\nu) (\widetilde{\mathbf{h}} - \widehat{\mathbf{h}}) \| \leq \epsilon \, \| \widehat{\S}^{\circ}(\nu) \|
$$ 
together with $\left|  (\widetilde{\mathbf{h}} - \widehat{\mathbf{h}})^H  \widehat{\S}^{\circ}(\nu) (\widetilde{\mathbf{h}} - \widehat{\mathbf{h}})^H \right| \leq \epsilon^{2} 
\| \widehat{\S}^{\circ}(\nu) \|$. This implies that
$$
\left| \widetilde{\mathbf{h}}^H \widehat{\S}^{\circ}(\nu)  \widetilde{\mathbf{h}} \right| \geq (1 - 2 \epsilon - \epsilon^{2}) \| \widehat{\S}^{\circ}(\nu) \|.
$$
In the following, we assume that $\epsilon$ satisfies $1 - 2 \epsilon - \epsilon^{2} > 0$. Therefore, using again the union bound, we obtain that 
\begin{equation}
\label{eq:union-bound-in-CM}
 \mathbb{P}\left( \| \widehat{\S}^{\circ}(\nu) \| > \eta_N \right) \leq \sum_{\mathbf{h} \in \mathcal{N}_{\epsilon}} 
\mathbb{P}\left( \left| \mathbf{h}^H \widehat{\S}^{\circ}(\nu)  \mathbf{h} \right| \geq (1 - 2 \epsilon - \epsilon^{2}) \eta_N \right).
\end{equation}
In order to evaluate $\mathbb{P}( | \mathbf{h}^H \widehat{\S}^{\circ}(\nu)  \mathbf{h} | \geq (1 - 2 \epsilon - \epsilon^{2}) \eta_N )$ for each unit norm vector $\mathbf{h}$, we denote by $z_n$ the scalar time series defined by 
$z_n = \mathbf{h}^H \y_n$. Then, the quadratic form $\mathbf{h}^H \widehat{\S}^{\circ}(\nu)  \mathbf{h}$ coincides 
with $\widehat{s}_z(\nu) - \mathbb{E}\widehat{s}_z(\nu)$ where $\widehat{s}_z(\nu)$ represents the 
lag-window estimator of the spectral density of $z$ defined by 
$\widehat{s}_z(\nu) = \sum_{l=-(L-1)}^{L-1} \widehat{r}_z(l) \mathrm{e}^{-2 i \pi l \nu}$. Here, 
$\widehat{r}_z(l)$ is the standard empirical estimate of the autocovariance coefficient of $z$ at lag $l$. 
We denote by ${\bf z}$ the $N$--dimensional vector ${\bf z} = (z_1, \ldots, z_N)^{T}$. As is well known, 
$\widehat{s}_z(\nu)$ can be expressed as 
\begin{equation}
\label{eq:expre-quadratic-hatsz}
\widehat{s}_z(\nu) = \int_{0}^{1} w(\nu - \mu) \frac{1}{N} \left| \sum_{n=0}^{N-1} z_{n+1} \mathrm{e}^{-2 i \pi n \mu} \right|^{2} \, d\mu
\end{equation}
where $w(\mu)$ is the Fourier transform of the rectangular window $\mathbb{I}_{l \in \{ -(L-1), \ldots, L-1 \}}$. The expression in
(\ref{eq:expre-quadratic-hatsz}) can also be written as a quadratic form of vector ${\bf z}$: 
\begin{equation}
\label{eq:expre-quadratic-hatsz-bis}
\widehat{s}_z(\nu) = {\bf z}^{H} \, \left( \frac{1}{N} \int_{0}^{1} w(\nu - \mu) \, \dd_N(\mu) \dd_N^{H}(\mu) \, d\mu \right) \, {\bf z}.
\end{equation}
where we recall that $\dd_N(\mu)$ is defined by (\ref{eq:def-d_R}).
If ${\bf R}_z$ represents the covariance matrix of vector ${\bf z}$, ${\bf z}$ can be written as ${\bf z} = {\bf R}_z^{1/2} {\bf x}$ 
for some $\mathcal{N}_\mathbb{C}(0,\I_N)$ distributed random vector ${\bf x}$. Therefore, if we denote by $\boldsymbol{\Omega}$ the 
$N \times N$ matrix defined by 
$$
\boldsymbol{\Omega} = {\bf R}_z^{1/2}  \, \frac{1}{N} \int_{0}^{1} w(\nu - \mu) \, \dd_N(\mu) \dd_N^{H}(\mu) \, d\mu \, {\bf R}_z^{1/2},
$$
the quantity $\widehat{s}_z(\nu) - \mathbb{E}\widehat{s}_z(\nu)$ can be written as $\widehat{s}_z(\nu) - \mathbb{E}\widehat{s}_z(\nu) = 
{\bf x}^H \boldsymbol{\Omega} {\bf x} - \mathbb{E} {\bf x}^H \boldsymbol{\Omega} {\bf x}$. Therefore,  
$$
P\left( |\widehat{s}_z(\nu) - \mathbb{E}\widehat{s}_z(\nu)| > (1 - 2 \epsilon - \epsilon^{2}) \eta \right)
$$ 
can be evaluated using the Hanson-Wright inequality (\ref{eq:hanson-wright}). This requires the evaluation of the spectral and the Frobenius norm of $\boldsymbol{\Omega}$. Observe that we can express $\boldsymbol{\Omega}= {\bf R}_z^{1/2} \boldsymbol{\Omega}_w {\bf R}_z^{1/2}$ where $\boldsymbol{\Omega}_w$ is a Toeplitz matrix defined as
\[
\boldsymbol{\Omega}_w = \frac{1}{N} \int_{0}^{1} w(\nu - \mu) \, \dd_N(\mu) \dd_N^{H}(\mu) \, d\mu.
\]
It is easy to check that the spectral norm of ${\bf R}_z$ is uniformly bounded. Moreover, the spectral norm of  $\boldsymbol{\Omega}_w$ is bounded by $\frac{1}{N} \sup_{\nu} |w(\nu)| = L/N$. Therefore, $\| \boldsymbol{\Omega} \| \leq \kappa \frac{L}{N}$ for some nice constant $\kappa$. 
In order to evaluate the Frobenius norm of $\boldsymbol{\Omega}$, observe that $\boldsymbol{\Omega}_w$ is band Toeplitz matrix with entries given by 
$(\boldsymbol{\Omega}_w)_{k,l} = \frac{1}{N} \mathrm{e}^{2 i \pi (k-l) \nu} \mathbb{I}_{|k-l| \leq L-1}$. Therefore, $\| \Omega_w \|_F^{2} \leq 
\kappa \frac{L}{N}$, which implies that  $\| \boldsymbol{\Omega} \|_F^{2} \leq \kappa \frac{L}{N}$. Consequently, the Hanson-Wright inequality in (\ref{eq:hanson-wright}) implies that, if $\eta$ is large enough,
$$
\mathbb{P}\left( \left| \mathbf{h}^H \widehat{\S}^{\circ}(\nu)  \mathbf{h} \right| \geq (1 - 2 \epsilon - \epsilon^{2}) \eta \right) \leq 
\kappa_1 \exp{ (-\kappa_2 M \eta)}
$$
where we have introduced two nice constants $\kappa_1$ and $\kappa_2$.
Recalling that $| \mathcal{N}_{\epsilon} | \leq \left(\frac{\kappa}{\epsilon}\right)^{2M}$, the union bound 
(\ref{eq:union-bound-in-CM})  implies that  
$$
\mathbb{P}\left( \| \widehat{\S}^{\circ}(\nu) \| > \eta \right) \leq  \left(\frac{\kappa}{\epsilon}\right)^{2M}  \kappa_1  \exp{(-\kappa_2 M \alpha})
$$
The right hand side of the above inequality can clearly be bounded by 
$\kappa_3  \exp{(-\kappa_4 M \alpha})$ for $\alpha$ large enough, where $\kappa_3$ and $\kappa_4$ are two new nice constants.
%\mathbb{P}\left( \| \widehat{\S}^{\circ}(\nu) \| > \eta \right) \leq  \kappa_1  \exp M\left({2\log\left(\frac{\kappa}{\epsilon}\right) - \kappa_2(1-2\epsilon-\epsilon^2) \eta_N}\right)
Finally, (\ref {eq:union-bound-frequency-domain}) leads to 
% Let a sufficiently small $\epsilon>0$ be given. Choosing 
% $\alpha_N = N^\epsilon$  in (\ref{eq:union-bound-frequency-domain}) we readily obtain 
% $$
% \mathbb{P}\left( \sup_{\nu \in [0,1]} \| \widehat{\S}^{\circ}(\nu) \| > N^\epsilon \right) \leq
% \kappa_1 L \exp{ M\left({2\log\left(\frac{\kappa}{\epsilon}\right) - \kappa_2(1-2\epsilon-\epsilon^2) N^\epsilon}\right)}
$$
\mathbb{P}\left( \sup_{\nu \in [0,1]} \| \widehat{\S}^{\circ}(\nu) \| > \alpha/2 \right) \leq \kappa_1 L \exp{ (-\kappa_2 M \alpha)}
$$
for $N$ sufficiently large and two nice constants $\kappa_1$, $\kappa_2$. This completes the proof of Proposition \ref{prop:sup-hatS}. 
\hfill $\square$
\end{proof}
\vspace{1em}

As a direct sequence of Propositions \ref{prop:spectral-norm-w-tildew} and \ref{prop:sup-hatS}, we have the following corollary. 

\begin{corollary}
\label{cor:control-norm-calR}
%   Under Assumptions \ref{assum:statistics}-\ref{as:norm-r-omega} we have
%   \begin{equation}
%       \label{eq:control-norm-calR}
%       \| \W_N \W_N^{H} \| \prec  1.       
%   \end{equation}
For each $\alpha$  larger than a certain positive constant, then, it holds that
\begin{eqnarray}
    \label{eq:control-norm-calR}
    \mathbb{P}( \| \W_N \W_N^{H} \| & > & \alpha) \leq \kappa_1 \, L  \exp( - \kappa_2 M \alpha) \\
     \mathbb{P}( \| \hat{\mathcal{R}}_L  \| & > & \alpha) \leq \kappa_1 \, L  \exp( - \kappa_2 M \alpha)
     \label{eq:control-norm-calRbis}
\end{eqnarray}
for some nice constants $\kappa_1$ and $\kappa_2$.  Moreover, $ \| \widehat{\mathcal{R}}_L  \|$
satisfies 
\begin{equation}
    \label{eq:hatcalR-domination}
     \| \widehat{\mathcal{R}}_L  \| \prec 1.
\end{equation}
\end{corollary}

%Note that Assumption \ref{as:asymptotic-regime} implies that $M = \mathcal{O}(N^{1 - \beta})$. Thus, (\ref{eq:control-norm-calR}) and the Borel-Cantelli Lemma imply that $\| \W_N \W_N^{H} \| = \| \mathbf{W}_N \mathbf{W}_N^{H} \|$ is almost surely bounded by a constant for each $N$ large enough. 

\subsection{Evaluation of the behaviour of $\| \mathrm{Bdiag}(\widehat{\mathcal{R}}_L) - \mathrm{Bdiag}(\mathcal{R}_L) \|$}
\label{subsec:convergence-estimators-RmL}
Recall that $\mathcal{R}_{m,L}$, $m=1,\ldots,M$, denote the $L \times L$ diagonal blocks of the matrix $\mathrm{Bdiag}(\mathcal{R}_L)$. We will denote by $\widehat{\mathcal{R}}_{m,L}$ the $m$th $L \times L$ diagonal block of $\widehat{\mathcal{R}}_L$. In this section, we establish that
\begin{equation}
     \label{eq:convergence-blocks}
     \| \widehat{\mathcal{R}}_{m,L} - \mathcal{R}_{m,L} \| \prec \max\left(\frac{1}{\sqrt{M}},\frac{1}{L^{\gamma_0}}\right).
\end{equation}
% \begin{equation}
%     \label{eq:convergence-blocks}
%  \lim_{N \rightarrow +\infty} \sup_{m=1, \ldots, M} \| \widehat{\mathcal{R}}_{m,L} - \mathcal{R}_{m,L} \| = 0 
% \end{equation}
% almost surely. 
Note first that we can express $\widehat{\mathcal{R}}_{m,L}$ as the empirical estimate 
of $\mathcal{R}_{m,L}$, that is 
$$
 \widehat{\mathcal{R}}_{m,L} = \frac{1}{N} \sum_{n=1}^{N} \y_{m,n}^{L} \left(\y_{m,n}^{L}\right)^{H}
 $$
 or equivalently by $\widehat{\mathcal{R}}_{m,L} = \mathbf{W}^{m}_N \left(\mathbf{W}^{m}_N\right)^H$ where $ \mathbf{W}^{m}_N$ is the 
 $L \times N$ matrix defined by 
 $$
 \mathbf{W}^{m}_N = \frac{1}{\sqrt{N}} \left(\y_{m,1}^{L}, \ldots, \y_{m,N}^{L} \right).
 $$
The arguments used in this section are based on the techniques used in Section \ref{subsec:boundedness-hatR}. Therefore, we just provide a sketch of proof of (\ref{eq:convergence-blocks}) based on the same two steps as above: first, we approximate $\mathbf{W}^{m}_N  \left(\mathbf{W}^{m}_N\right)^H$ with a Toeplitz matrix and then study the equivalent Toeplitz version of (\ref{eq:convergence-blocks}). 

\subsubsection{Modifying $\mathbf{W}^{m}_N  \left(\mathbf{W}^{m}_N\right)^H$ into a Toeplitz matrix}

We prove here that $\mathbf{W}^{m}_N \left(\mathbf{W}^{m}_N\right)^H$ 
can be approximated as the Toeplitz matrix $\widetilde{\mathbf{W}}^{m}_N (\widetilde{\mathbf{W}}^{m}_N)^H$
where $\widetilde{\mathbf{W}}^{m}_N$ is obtained by replacing vectors $(\y_n)_{n=1, \ldots, N}$ 
by the scalars $(y_{m,n})_{n=1, \ldots, N}$ in the definition of matrix $\W_N$ in (\ref{eq:def-tildeW}) above. In particular,  it holds that
\begin{equation}
\label{eq:def_RLmToep}
 \widetilde{\mathbf{W}}_N^{m} \left(\widetilde{\mathbf{W}}_N^{m}\right)^{H} = \int_{0}^{1} \widehat{\mathcal{S}}_{m}(\nu)  \dd_L(\nu) \dd_{L}^{H}(\nu) \, d\nu
 \end{equation}
 where  $ \widehat{\mathcal{S}}_{m}(\nu)$ represents the $m$th diagonal entry of the 
 lag window estimator (\ref{eq:def-hatS}). More specifically, following the proof 
 of Proposition \ref{prop:spectral-norm-w-tildew}, we justify that
\begin{equation}
    \label{eq:toeplitz-approximation-hatRm}
    \left\|  \mathbf{W}_N^{m} \left(\mathbf{W}_N^{m}\right)^{H} -  \widetilde{\mathbf{W}}_N^{m} \left(\widetilde{\mathbf{W}}_N^{m}\right)^{H}\right\| \prec \frac{1}{M}.
\end{equation}
% for each $0 < \epsilon < 1$,
% \begin{equation}
%     \label{eq:toeplitz-approximation-hatRm}
%      P\left( \sup_{m=1, \ldots, M} \left\|  \mathbf{W}_N^{m} \left(\mathbf{W}_N^{m}\right)^{H} -  \widetilde{\mathbf{W}}_N^{m} \left(\widetilde{\mathbf{W}}_N^{m}\right)^{H}\right\| > \frac{1}{M^{(1-\epsilon)}} \right) \leq \kappa_1 N \exp{ (-\kappa_2 M^{\epsilon})}
% \end{equation}
%   holds for some nice constants $\kappa_1$ and $\kappa_2$. 
   To verify (\ref{eq:toeplitz-approximation-hatRm}), we drop the dependence on $N$ of all matrices to simplify the notation and remark that 
   $$
     \mathbf{W}^{m} (\mathbf{W}^{m})^{H} - \widetilde{\mathbf{W}}^{m} (\widetilde{\mathbf{W}}^{m})^{H} =  \mathbf{W}_{2,2}^{m} (\mathbf{W}_{2,2}^{m})^{H} +  \mathbf{W}_{2,2}^{m} (\mathbf{W}_{2,1}^{m})^{H} +  \mathbf{W}_{2,1}^{m} (\mathbf{W}_{2,2}^{m})^{H} - \mathbf{W}_0^{m} (\mathbf{W}_0^{m})^{H}
  $$
  where the various matrices of the right hand side are obtained by replacing vectors 
  $(\y_n)_{n=1, \ldots, N+L-1}$ in the definition of matrices $\W_{2,2}, \W_{2,1}, \W_{0}$ used in Section \ref{subsec:boundedness-hatR} by the scalars  $(y_{m,n})_{n=1, \ldots, N+L-1}$. In order to verify (\ref{eq:toeplitz-approximation-hatRm}), we just briefly check that
%   $$
%     P\left( \sup_{m=1, \ldots, M} \left\|  \mathbf{W}_{0}^{m} \left(\mathbf{W}_{0}^{m}\right)^{H} \right\| > \frac{1}{M^{(1-\epsilon)}} \right) \leq \kappa_1 N \exp{ (-\kappa_2 M^{\epsilon})}
%     $$
$$
\left\|  \mathbf{W}_{0}^{m} \left(\mathbf{W}_{0}^{m}\right)^{H} \right\| \prec \frac{1}{M}
$$
   or equivalently (after proper column permutation of $(\mathbf{W}_{0}^{m}$) that 
$$
\left\|  \widetilde{\mathbf{W}}_{0}^{m} \left(\widetilde{\mathbf{W}}_{0}^{m}\right)^{H} \right\| \prec \frac{1}{M}
$$
    % $$
    % P\left( \sup_{m=1, \ldots, M} \left\|  \widetilde{\mathbf{W}}_{0}^{m} \left(\widetilde{\mathbf{W}}_{0}^{m}\right)^{H} \right\| > \frac{1}{M^{(1-\epsilon)}} \right) \leq \kappa_1 N \exp{ (-\kappa_2 M^{\epsilon})}
    % $$
   where $ \widetilde{\mathbf{W}}_{0}^{m}$ is defined by 
   $$
    \widetilde{\mathbf{W}}^{m}_0 = \sqrt{\frac{L}{N}} \int_{0}^{1}  \dd_{L-1}(\nu) \dd_{L-1}^{H}(\nu)   \xi_{L,y_m}(\nu) \, d\nu.
   $$
   As in Section \ref{subsec:boundedness-hatR}, we notice that the matrix-valued Cauchy-Schwarz inequality in (\ref{eq:schwartz-matriciel-v2}) with $\mathbf{U}(\nu)=\sqrt{\frac{L}{N}}\dd_{L-1}(\nu)  \xi_{L,y_m}(\nu)$ and $\mathbf{V}(\nu) =  \dd_{L-1}(\nu)$ implies that 
    $$
     \widetilde{\mathbf{W}}_0^{m} (\widetilde{\mathbf{W}}_0^{m})^{H} \leq \frac{L}{N} \, \int_{0}^{1}  \dd_{L-1}(\nu) \dd_{L-1}^{H} (\nu)  |\xi_{L,y_m}(\nu)|^{2}  \, d\nu.
     $$
     This allow us to establish that 
     $$
     \left\|  \widetilde{\mathbf{W}}_0^{m} (\widetilde{\mathbf{W}}_0^{m})^{H} \right\| \leq  \sup_{\nu \in [0,1]} \frac{L}{N} |\xi_{L,y_m}(\nu)|^{2}.
     $$
     By Lemma \ref{le:zygmund} we know that the supremum can be replaced by a maximum over $\mathcal{O}(L)$ points, so that by Lemma \ref{lem:stochastic_dom_maximum} it is sufficient to establish that
     $$
     \frac{L}{N} |\xi_{L,y_m}(\nu)|^{2} \prec \frac{1}{M}
     $$
     for some fixed $\nu$. Following the same reasoning as in Section \ref{subsec:boundedness-hatR}, a direct application of the Hanson-Wright inequality shows that $ |\xi_{L,y_m}(\nu)|^{2} \prec 1$ for any fixed $\nu$, from where the result follows.

     \subsubsection{Studying the Toeplitz equivalent of (\ref{eq:convergence-blocks})}
     
     In order to prove (\ref{eq:convergence-blocks}), it thus remains to establish that 
     $$
      \left\|  \widetilde{\mathbf{W}}_N^{m} \left(\widetilde{\mathbf{W}}_N^{m}\right)^{H} - \mathcal{R}_{m,L} \right\| \prec \max \left(\frac{1}{\sqrt{M}},\frac{1}{L^{\gamma_0}}\right)
     $$
%   $$
%     \sup_{m=1, \ldots, M} \left\|  \widetilde{\mathbf{W}}_N^{m} \left(\widetilde{\mathbf{W}}_N^{m}\right)^{H} - \mathcal{R}_{m,L} \right\| \rightarrow 0
%     $$
%     almost surely. 
    Noting that $\widetilde{\mathbf{W}}_N^{m} (\widetilde{\mathbf{W}}_N^{m})^{H} - \mathcal{R}_{m,L}$ is the $L \times L$ Toeplitz matrix associated to the symbol $\widehat{\mathcal{S}}_m(\nu) - \mathcal{S}_m(\nu)$, and using Lemma \ref{le:zygmund}, it is sufficient to prove that 
    $$
    \left| \widehat{\mathcal{S}}_m(\nu) - \mathcal{S}_m(\nu) \right| \prec \max \left(\frac{1}{\sqrt{M}},\frac{1}{L^{\gamma_0}}\right)
     $$
    % $$
    % \sup_{m=1, \ldots, M} \sup_{\nu} \left| \widehat{\mathcal{S}}_m(\nu) - \mathcal{S}_m(\nu) \right| \rightarrow 0
    % $$
    % almost surely. 
    for each $\nu$. In order to see this, we write $\widehat{\mathcal{S}}_m(\nu) - \mathcal{S}_m(\nu)$
    as 
    $$
    \widehat{\mathcal{S}}_m(\nu) - \mathcal{S}_m(\nu) =  \widehat{\mathcal{S}}_m(\nu) - 
    \mathbb{E}\left(\widehat{\mathcal{S}}_m(\nu)\right) + \mathbb{E}\left(\widehat{\mathcal{S}}_m(\nu)\right) - \mathcal{S}_m(\nu).
    $$
    The bias $\mathbb{E}(\widehat{\mathcal{S}}_m(\nu)) - \mathcal{S}_m(\nu)$ is equal to 
    $$
    \mathbb{E}\left(\widehat{\mathcal{S}}_m(\nu)\right) - \mathcal{S}_m(\nu) = 
    - \sum_{|l| \geq L} r_m(l)\mathrm{e}^{-2 i \pi l \nu} - \frac{1}{N} \sum_{l=-(L-1)}^{L-1} |l|  r_m(l)\mathrm{e}^{-2 i \pi l \nu}.
    $$
    An easy adaptation of the proof of Lemma \ref{le:expectation-periodogram} in Appendix \ref{sec:app_proof_lemma_period} establishes that 
    \begin{equation}
    \label{eq:biais-hatsm}
    \left| \mathbb{E}\left(\widehat{\mathcal{S}}_m(\nu)\right) - \mathcal{S}_m(\nu) \right| \leq \kappa \left( \frac{1}{L^{\gamma_0}}+ \frac{L^{(1-\gamma_0)_{+}}}{N} \right)
    \end{equation}
    for some nice constant $\kappa$, where $(\cdot)_+=\max{(\cdot,0)}$. This implies that 
      \begin{equation}
        \label{eq:sup-m-nu-convergence-biais}
        \sup_{m=1, \ldots, M} \sup_{\nu} \left|  \mathbb{E}\left(\widehat{\mathcal{S}}_m(\nu)\right) - \mathcal{S}_m(\nu) \right| \leq \kappa \max\left(\frac{1}{L^{\gamma_0}},\frac{1}{M}\right)
    \end{equation}
    for some nice constant $\kappa$. 
    % \begin{equation}
    %     \label{eq:sup-m-nu-convergence-biais}
    %     \sup_{m=1, \ldots, M} \sup_{\nu} \left|  \mathbb{E}\left(\widehat{\mathcal{S}}_m(\nu)\right) - \mathcal{S}_m(\nu) \right| \rightarrow 0.
    % \end{equation}
    In order to study the term $\widehat{\mathcal{S}}_m(\nu) -\mathbb{E}(\widehat{\mathcal{S}}_m(\nu))$, 
    we remark that it can be written as 
    $$
   \widehat{\mathcal{S}}_m(\nu) - \mathbb{E}\left(\widehat{\mathcal{S}}_m(\nu)\right)  = 
    \mathbf{e}_m^{T} \widehat{\S}^{\circ}(\nu) \mathbf{e}_m
    $$
    where $\mathbf{e}_m$ is the $m$th vector of the canonical basis of $\mathbb{C}^{M}$. 
     Using the Hanson-Wright inequality as in Section 
    \ref{subsec:boundedness-hatR}, we obtain immediately that for each $\nu$ and for each $m$, 
    there exist two nice constants $\kappa_1$ and $\kappa_2$ such that
    $$
    \mathbb{P}\left( |\mathbf{e}_m^{T} \widehat{\S}^{\circ}(\nu) \mathbf{e}_m| > \alpha_N \right) \leq \kappa_1 
    \exp ( - \kappa_2 M \alpha_N^{2})
    $$
    where $(\alpha_N)_{N \geq 1}$ satisfies $\alpha_N \rightarrow 0$ and 
    $M \alpha_N^{2} \rightarrow +\infty$. 
    In particular, the choice $\alpha_N = {N^\epsilon}/{\sqrt{M}}$ satisfies this property for all small enough $\epsilon>0$, which allows to conclude that $ |\mathbf{e}_m^{T} \widehat{\S}^{\circ}(\nu) \mathbf{e}_m| \prec M^{-1/2}$ for any fixed $m$ and $\nu$. However, noting again that $\mathbf{e}_m^{T} \widehat{\S}^{\circ}(\nu) \mathbf{e}_m$ is a real valued trigonometric polynomial, we see by Lemma \ref{le:discrete-grid} and Lemma \ref{lem:stochastic_dom_maximum} that $\sup_{m,\nu} |\mathbf{e}_m^{T} \widehat{\S}^{\circ}(\nu) \mathbf{e}_m| \prec M^{-1/2}$.
    %(uniformly in $m$ and $\nu$). 
    %Any sequence verifying $\alpha_N = \mathcal{O}(\frac{1}{M^{(1 -\epsilon)/2}})$, where $0 < \epsilon < 1$ satisfies this condition. Using Lemma \ref{le:zygmund} as well as the union bound, we obtain that 
    % $$
    % P\left( \sup_{m=1, \ldots, M} \sup_{\nu} |\mathbf{e}_m^{T} \widehat{\S}^{\circ}(\nu) \mathbf{e}_m| > \alpha_N \right) \leq N \kappa_1 
    % \exp (- \kappa_2 M \alpha_N^{2}).
    % $$
    % The Borel Cantelli Lemma thus implies that 
    % almost surely, for each $N$ large enough, 
    % $$ 
    % \sup_{m=1, \ldots, M} \sup_{\nu} \left| \widehat{\mathcal{S}}_m(\nu) -\mathbb{E}\left(\widehat{\mathcal{S}}_m(\nu)\right) \right| \leq \alpha_N. 
    % $$
    % This, in conjunction with (\ref{eq:sup-m-nu-convergence-biais}) implies that 
    %  $$
    % \sup_{m=1, \ldots, M} \sup_{\nu} \left| \widehat{\mathcal{S}}_m(\nu) - \mathcal{S}_m(\nu) \right| \rightarrow 0
    % $$
    % almost surely, and that 
    % $$
    % \sup_{m=1, \ldots, M} \left\|  \widetilde{\mathbf{W}}_N^{m} \left(\widetilde{\mathbf{W}}_N^{m}\right)^{H} - \mathcal{R}_{m,L} \right\| \rightarrow 0.
    % $$
    
    As a consequence of all the above, we have established that $\|\widehat{\mathcal{R}}_{m,L}-\mathcal{R}_{m,L}\| \prec \max(M^{-1/2},L^{-\gamma_0})$, which directly implies that
    \begin{equation}
        \label{eq:dominationblockR}
        \| \mathrm{Bdiag}(\widehat{\mathcal{R}}_L) - \mathrm{Bdiag}(\mathcal{R}_L) \| \prec \max\left( \frac{1}{\sqrt{M}}, \frac{1}{L^{\gamma_0}} \right).  
    \end{equation}
    %$\| \mathrm{Bdiag}(\widehat{\mathcal{R}}_L) - \mathrm{Bdiag}(\mathcal{R}_L) \| \rightarrow 0$ almost surely. 
    All these results are all the ingredients that we need in order to evaluate the spectral norm of the matrix
    %\|\widehat{\mathcal{R}}_{\mathrm{corr},L} - \overline{\mathcal{R}}_{\mathrm{corr},L}\|$ in the following subsection.
    \begin{equation}
    \label{eq:def-Theta}
    \bs{\Theta}_N =  \widehat{\mathcal{R}}_{\mathrm{corr},L} - \overline{\mathcal{R}}_{\mathrm{corr},L}
    \end{equation}
    which is carried out in the following section. 
    
\subsection{Evaluation of $\| {\boldsymbol \Theta}_N\|= \|\widehat{\mathcal{R}}_{\mathrm{corr},L} - \overline{\mathcal{R}}_{\mathrm{corr},L}\|$ }
\label{sec:evalTheta}

We first precise that almost surely, all the matrices $(\hat{\mathcal{R}}_{m,L})_{m=1, \ldots, M(N), N \geq 1}$ are invertible. To verify this, we remark that the random variable $\mathrm{det}(\hat{\mathcal{R}}_{m,L})$ is a differentiable function of the $2(N+L-1)$ 
entries of the Gaussian vector $(\mathrm{Re}(y_{m,1}, \ldots, y_{m,N+L-1}), \mathrm{Im}(y_{m,1}, \ldots, y_{m,N+L-1}))$. Therefore, the probability distribution of $\mathrm{det}(\hat{\mathcal{R}}_{m,L})$ 
is absolutely continuous, and the event $\{ \mathrm{det}(\hat{\mathcal{R}}_{m,L}) = 0 \}$ 
has probability $0$. Therefore, the union of the above events is also negligible, thus showing the 
almost sure invertibility of the matrices $(\hat{\mathcal{R}}_{m,L})_{m=1, \ldots, M(N), N \geq 1}$.

Using the above definition of $\bs{\Theta}_N$, we are able to write 
\begin{eqnarray}
{\bs \Theta}_N & = & \widehat{\mathcal{B}}_L^{-1/2} \widehat{\mathcal{R}}_L \widehat{\mathcal{B}}_L^{-1/2} - \mathcal{B}_L^{-1/2} \widehat{\mathcal{R}}_L \mathcal{B}_L^{-1/2} \nonumber \\
& = & (\widehat{\mathcal{B}}_L^{-1/2} -  \mathcal{B}_L^{-1/2}) \widehat{\mathcal{R}}_L \widehat{\mathcal{B}}_L^{-1/2} + \mathcal{B}_L^{-1/2} \widehat{\mathcal{R}}_L (\widehat{\mathcal{B}}_L^{-1/2} -  \mathcal{B}_L^{-1/2}) \label{eq:expre-Theta}
\end{eqnarray}
We have shown above that $\| \widehat{\mathcal{R}}_{m,L} - \mathcal{R}_{m,L} \| \prec \max{(M^{-1/2},L^{-\gamma_0})}$. Our first objective here is to show that $\| \widehat{\mathcal{R}}_{m,L}^{-1/2} - \mathcal{R}_{m,L}^{-1/2} \| \prec \max{(M^{-1/2},L^{-\gamma_0})}$. For this, we use perturbation theory 
of Hermitian matrices arguments (see e.g. \cite[Sec. 2, Ch. 1 and Sec. 1, Ch. 2]{kato}) that will also be needed in Section \ref{sec:influenceblockmat}.

We first recall that Assumption \ref{ass:bounds-spectral-densities} implies that for each $N$, matrices 
$(\mathcal{R}_{m,L})_{m=1, \ldots, M}$ verify $s_{min} {\bf I}_L \leq \mathcal{R}_{m,L} \leq s_{max} {\bf I}_L$. %{\color{blue} could be mentionned just after Assumption \ref{ass:bounds-spectral-densities}}. 
Therefore, if we denote by $\mathcal{C}$ a simple closed contour included in 
the half plane $\{ \mathrm{Re}(\lambda) > 0 \}$ and enclosing the interval $[s_{min}, s_{max}]$, then, 
$\mathcal{C}$ also encloses the spectrum of the matrices $(\mathcal{R}_{m,L})_{m=1, \ldots, M}$.
This in particular implies that matrix $\mathcal{R}_{m,L}^{-1/2}$ can be written as 
\begin{equation}
    \label{eq:expre-sqrtmathcalRmL}
 \mathcal{R}_{m,L}^{-1/2} = \frac{1}{2 i \pi} \int_{\mathcal{C}_-} \frac{1}{\sqrt{\lambda}} \, \left(  \mathcal{R}_{m,L} - \lambda {\bf I}_L \right)^{-1} \, d \lambda
\end{equation}
where $\mathcal{C}_-$ means that the contour is negatively oriented. 
In the following, we denote by $(\lambda_{k,m})_{k=1, \ldots, K_m}$ the distinct eigenvalues of $\mathcal{R}_{m,L}$, and by $(\Pi_{k,m})_{k=1, \ldots, K_m}$
the orthogonal projection matrices over the corresponding eigenspaces. Therefore, 
$\mathcal{R}_{m,L}$ can be written as 
\begin{equation}
    \label{eq:spectral-decomposition-mathcalRmL}
 \mathcal{R}_{m,L} = \sum_{k=1}^{K_m} \lambda_{k,m} \, \Pi_{k,m}.
\end{equation}
We denote by $\bs{\Delta}_{m,L}$ the matrix defined by 
\begin{equation}
\label{eq:def-deltamL}
{\bs \Delta}_{m,L} = \widehat{\mathcal{R}}_{m,L} - \mathcal{R}_{m,L}.
\end{equation}
In order to investigate ${\bs \Delta}_{m,L}$, it will be convenient to introduce a collection of operators $\mathcal{D}_{m,L} \left( \bf{X} \right)$, $m = 1,\ldots,M$, which transform $L \times L$ matrices into $L \times L$ matrices and are defined as
\begin{equation}
    \label{eq:def-differential-1}
    \mathcal{D}_{m,L} \left( {\bf X} \right)  =  \frac{1}{2 \pi i} \int_{\mathcal{C}_-} \frac{1}{\sqrt{\lambda}} \left(  \mathcal{R}_{m,L} - \lambda {\bf I}_L \right)^{-1}   {\bf X}  \left(  \mathcal{R}_{m,L} - \lambda {\bf I}_L \right)^{-1} \, d \lambda
\end{equation}
where, as before, $\mathcal{C}_-$ is a negatively oriented simple closed contour on the half plane ${\mathrm{Re}\lambda>0}$ enclosing $[s_{min},s_{max}]$. As seen below,  $\mathcal{D}_{m,L}$ can be interpreted as the differential operator of the matrix valued-function 
${\bf A} \rightarrow {\bf A}^{-1/2}$ evaluated at $\mathcal{R}_{m,L}$.
Note that, using the definitions in (\ref{eq:spectral-decomposition-mathcalRmL}), we can express 
$$
\left(  \mathcal{R}_{m,L} - \lambda {\bf I}_L \right)^{-1} = \sum_{k=1}^{K_m} \frac{\Pi_{k,m}}{\lambda_{k,m} - \lambda}.
$$
Plugging this expression into  (\ref{eq:def-differential-1}) and 
using the residue theorem, we can trivially check that this operator can also be expressed as 
\begin{equation}
    \mathcal{D}_{m,L} \left( {\bf X} \right)  = \sum_{k=1}^{K_m} \sum_{l=1}^{K_m} \frac{1}{\sqrt{\lambda_{k,m}} \sqrt{\lambda_{l,m}}(\sqrt{\lambda_{k,m}} + \sqrt{\lambda_{l,m}})} \, \Pi_{k,m} {\bf X} \Pi_{l,m}.
    \label{eq:def-differential-2}
\end{equation}
We summarize next a number of properties that will be useful about these operators throughout the paper.
\begin{lemma}
    \label{lem:operator-Dm-props}
    Consider the operator $   \mathcal{D}_{m,L}$ as defined in (\ref{eq:def-differential-1})-(\ref{eq:def-differential-2}). Then, for every $L \times L$ matrix ${\bf A}$:
    \begin{enumerate}[label=(\roman*),leftmargin=20pt]
        \item If ${\bf B}$ denotes another $L \times L$ matrix,
        \begin{equation}
            \label{eq:swapDmL}
            \mathrm{Tr}\left(\mathcal{D}_{m,L}({\bf A}) {\bf B} \right) = \mathrm{Tr}\left({\bf A} \mathcal{D}_{m,L}({\bf B})  \right)
        \end{equation}
        \item There exists a nice constant $\kappa > 0$ such that 
        \begin{equation}
        \label{eq:continuity-mathcalD}
        \| \mathcal{D}_{m,L}({\bf A}) \| \leq \kappa \| {\bf A} \|.
        \end{equation}
         \item There exists a nice constant $\kappa >0 $ such that
         \begin{equation}
        \label{eq:upperbound-mathcalDmathcalDH}
        \frac{1}{L} \mathrm{Tr} \left[\mathcal{D}_{m,L}({\bf A}) \mathcal{D}^{H}_{m,L}({\bf A}) \right] \leq \kappa \, \frac{1}{L} \mathrm{Tr}({\bf A} {\bf A}^{H}).
        \end{equation}
    \end{enumerate}
\end{lemma}
\begin{proof}
    The identity in (\ref{eq:swapDmL}) follows directly from the definition of $\mathcal{D}_{m,L}$. To see (\ref{eq:continuity-mathcalD}), simply consider the definition of $\mathcal{D}_{m,L}$ in (\ref{eq:def-differential-1}) and note that $\sup_{\lambda \in \mathcal{C}} \| \left( \mathcal{R}_{m,L} - \lambda {\bf I}_L \right)^{-1} \| \leq \kappa$ for some nice constant $\kappa$. 
    In order to justify (\ref{eq:upperbound-mathcalDmathcalDH}), we express $\mathcal{D}_{m,L}({\bf A})$ using (\ref{eq:def-differential-2}) so that, noting that $\Pi_{l,m} \Pi_{l',m} = \Pi_{l,m} \delta_{l-l'}$, we can write
    $$
    \mathcal{D}_{m,L}({\bf A})  \mathcal{D}_{m,L}({\bf A})^{H}  = 
    \sum_{k,k',l} \frac{\Pi_{k,m}}{\lambda_{k,m}^{1/2} (\lambda_{k,m}^{1/2} + \lambda_{l,m}^{1/2})} {\bf A} \, \frac{\Pi_{l,m}}{\lambda_{l,m}}  {\bf A}^{H} \frac{\Pi_{k',m}}{\lambda_{k',m}^{1/2} \lambda_{l,m}^{1/2}(\lambda_{k',m}^{1/2} + \lambda_{l,m}^{1/2})} 
    $$
    Taking the normalized trace, changing the order of the matrices, and using again the fact that $\Pi_{k,m} \Pi_{k',m} = \Pi_{k,m} \delta_{k-k'}$, we obtain 
    $$
    \frac{1}{L} \mathrm{Tr} \left[\mathcal{D}_{m,L}({\bf A}) \mathcal{D}_{m,L}({\bf A})^{H} \right] =
    \sum_{k,l} \frac{1}{L} \mathrm{Tr} \left[ \frac{\Pi_{k,m} {\bf A} \,\Pi_{l,m} {\bf A}^{H} \Pi_{k,m}}{\lambda_{k,m}\lambda_{l,m}(\lambda_{k,m}^{1/2} + \lambda_{l,m}^{1/2})^{2}} \right].
    $$
    Using that $\lambda_{k,m} \geq s_{min}$ for each $k$ and $m$, we obtain immediately that 
    $$
    \frac{\Pi_{k,m} {\bf A} \,\Pi_{l,m} {\bf A}^{H} \Pi_{k,m}}{\lambda_{k,m}\lambda_{l,m}(\lambda_{k,m}^{1/2} + \lambda_{l,m}^{1/2})^{2}}  \leq \kappa \, \Pi_{k,m} {\bf A} \,\Pi_{l,m} {\bf A}^{H} \Pi_{k,m}
    $$
    from where the inequality 
    $$
    \frac{1}{L} \mathrm{Tr} \left[\mathcal{D}_{m,L}({\bf A}) \mathcal{D}_{m,L}({\bf A})^{H} \right]  \leq \kappa \, \frac{1}{L} \mathrm{Tr} \sum_{k,l} \Pi_{k,m} {\bf A} \,\Pi_{l,m} {\bf A}^{H}
    $$
    follows directly. Noting that $\sum_{k} \Pi_{k,m} = {\bf I}_L$, we obtain (\ref{eq:upperbound-mathcalDmathcalDH}).
\end{proof}
Having introduced these operators, we now formulate a result that will be useful here and in the following sections. 
\begin{lemma}
\label{le:perturbation}
Under Assumptions \ref{assum:statistics}-\ref{ass:bounds-spectral-densities} and \ref{as:norm-r-omega}, it holds that 
\begin{equation}
\label{eq:formula-perturbation}
\widehat{\mathcal{R}}_{m,L}^{-1/2} - \mathcal{R}_{m,L}^{-1/2} = - \mathcal{D}_{m,L} \left( \bs{\Delta}_{m,L} \right) +
\bs{\Upsilon}_{m,L}
\end{equation}
% where the matrix 
% $\mathcal{D}_{m,L} \left( \bs{\Delta}_{m,L} \right)$ is defined by 
% \begin{eqnarray}
% \label{eq:def-differential-1}
% \mathcal{D}_{m,L} \left( \bs{\Delta}_{m,L} \right) & = & \frac{1}{2 i \pi} \int_{\mathcal{C}_-} \frac{1}{\sqrt{\lambda}} \left(  \mathcal{R}_{m,L} - \lambda {\bf I}_L \right)^{-1}   \bs{\Delta}_{m,L}  \left(  \mathcal{R}_{m,L} - \lambda {\bf I}_L \right)^{-1} \, d \lambda \\
%  & = &  \sum_{(k,l)} \frac{1}{\sqrt{\lambda_{k,m}} \sqrt{\lambda_{l,m}}(\sqrt{\lambda_{k,m}} + \sqrt{\lambda_{l,m}})} \, \Pi_{k,m} \bs{\Delta}_{m,L} \Pi_{l,m}
% \label{eq:def-differential-2}
% \end{eqnarray}
%and 
where the matrix $\bs{\Upsilon}_{m,L}$, implicitely defined by (\ref{eq:formula-perturbation}), verifies 
\begin{equation}
    \label{eq:domination-Upsilon}
    \| \bs{\Upsilon}_{m,L} \| \prec \max\left(\frac{1}{M},\frac{1}{L^{2 \gamma_0}}\right).
\end{equation}
\end{lemma}
\begin{proof}
    See Appendix \ref{sec:app_proof_sqrt}.
\end{proof}
% \begin{equation}
%     \label{eq:perturbation-theory}
% % \widehat{\mathcal{R}}_{m,L}^{-1/2} = \mathcal{R}_{m,L}^{-1/2} - \frac{1}{2} \mathcal{R}_{m,L}^{-3/4} {\bs \Delta}_{m,L} \mathcal{R}_{m,L}^{-3/4} + {\bs \Upsilon}_{m,L}
%  \widehat{\mathcal{R}}_{m,L}^{-1/2} = \mathcal{R}_{m,L}^{-1/2} - \frac{1}{2} \mathcal{R}_{m,L}^{-3/4} \left( \widehat{\mathcal{R}}_{m,L} - \mathcal{R}_{m,L} \right) \mathcal{R}_{m,L}^{-3/4} + {\bs \Upsilon}_{m,L}
% \end{equation}
%where ${\bs \Delta}_{m,L} = \widehat{\mathcal{R}}_{m,L} - \mathcal{R}_{m,L}$ and 
%where $\|{\bs \Upsilon}_{m,L} \| \prec M^{-1}$. {\color{red} [No dependence on $\gamma_0$ ?]}
%The proof of (\ref{eq:perturbation-theory}) is provided in the Appendix {\color{red} [To be completed]}. 
Since $\|\bs{\Delta}_{m,L} \| = \|\widehat{\mathcal{R}}_{m,L} - \mathcal{R}_{m,L} \| \prec \max{(M^{-1/2},L^{-\gamma_0})}$, we directly observe from the above proposition and (\ref{eq:continuity-mathcalD}) that $\| \widehat{\mathcal{B}}_L^{-1/2} -  \mathcal{B}_L^{-1/2} \| \prec \max{(M^{-1/2},L^{-\gamma_0})}$. This of course implies that  $\|\widehat{\mathcal{B}}_L^{-1/2}\| \prec 1$. Moreover, using the fact that $\|\widehat{\mathcal{R}}_L\| \prec 1$ (see (\ref{eq:hatcalR-domination})), (\ref{eq:expre-Theta}) leads to  
\begin{equation}
    \label{eq:thetaprec}
    \| {\bs \Theta}_N \| \prec \max{\left(\frac{1}{\sqrt{M}},\frac{1}{L^{\gamma_0}}\right)}.     
\end{equation}

    % Noting that  
    % $\mathrm{Bdiag}(\mathcal{R}_L) > s_{min} \I_{ML}$ for each $N$, matrix 
    % $ \mathrm{Bdiag}(\widehat{\mathcal{R}}_L)$ verifies $\mathrm{Bdiag}(\widehat{\mathcal{R}}_L) > \frac{s_{min}}{2} \I_{ML}$ almost surely for each $N$ large enough. We thus also have 
    % $$
    %  \left\| \left(\mathrm{Bdiag}(\widehat{\mathcal{R}}_L)\right)^{-1/2} - \left(\mathrm{Bdiag}(\mathcal{R}_L) \right)^{-1/2}\right\| \rightarrow 0
    % $$
    % almost surely as $N \rightarrow +\infty$. As $\| \widehat{\mathcal{R}}_L \|$ is almost surely bounded by a constant for each $N$ large 
    % enough, we eventually conclude that (\ref{eq:simplification}) holds. 

%Before we close this section, it is worth pointing out that, given the fact that $\sup_N \|\mathcal{B}_L^{-1/2}\| \leq s_{min}^{-1/2}$, we can formulate the following result. 
To close this section, we remark that the identity $\widehat{\mathcal{R}}_{m,L}^{-1} - \mathcal{R}_{m,L}^{-1} = - \widehat{\mathcal{R}}_{m,L}^{-1} \boldsymbol{\Delta}_{m,L} \mathcal{R}_{m,L}^{-1}$ leads immediately to $\| \widehat{\mathcal{R}}_{m,L}^{-1} - \mathcal{R}_{m,L}^{-1} \| \prec \max{\left(\frac{1}{\sqrt{M}},\frac{1}{L^{\gamma_0}}\right)}$. Using this and 
(\ref{eq:control-norm-calRbis}), we obtain the following Proposition.
\begin{proposition}
\label{prop:concentration-hatcorr-overlinecorr}
Under Assumptions \ref{assum:statistics}-\ref{ass:bounds-spectral-densities} and \ref{as:norm-r-omega}, there exists $\alpha_0 > 0$ such that for each $\alpha \geq \alpha_0 $, one may find $\epsilon > 0$ and $N_0$ (both depending on $\alpha$) such that 
\begin{eqnarray}
\label{eq:concentration-hatcorr}
\mathbb{P}\left( \| \overline{\mathcal{R}}_{\mathrm{corr},L} \| > \alpha \right) & \leq &\exp{-N^{\epsilon}} \\
\label{eq:concentration-hatcorr}
\mathbb{P}\left( \| \widehat{\mathcal{R}}_{\mathrm{corr},L} \| > \alpha \right) & \leq & \exp{-N^{\epsilon}} 
\end{eqnarray}
for each $N \geq N_0$. 
\end{proposition}

From all the above, we can therefore conclude that the spectral behavior of the sample block correlation matrix $\widehat{\mathcal{R}}_{\mathrm{corr},L}$ is equivalent to the spectral behavior of the matrix $\overline{\mathcal{R}}_{\mathrm{corr},L} = \mathcal{B}^{-1/2}_L \widehat{\mathcal{R}}_L \mathcal{B}^{-1/2}_L$. 
%In the following section, we will analyze the spectral convergence of the eigenvalues of this matrix by analyzing the behavior of the corresponding resolvent and co-resolvent. 
    
%We should certainly add to our current section 2 the following Proposition

%%%%%%%%%%%%%%%%%%%%%%%%%%%%%%%%%%%%%%%%%%%%%%%%%%%%%%%%%%%%%%%%%

\section{Study of the influence of the estimation of matrices $(\mathcal{R}_{m,L})_{m=1, \ldots, M}$.}
\label{sec:influenceblockmat}
In this section, we study the impact of the estimation of matrices $(\mathcal{R}_{m,L})_{m=1, \ldots, M}$ 
on the asymptotic behaviour of the linear statistics $\widehat{\phi}_N$, defined as % - \int_{\mathbb{R}^{+}} \phi(\lambda) d\mu_{mp,c_N}(\lambda)$ where we recall that $\widehat{\phi}_N$ is defined by 
$$
\widehat{\phi}_N = \frac{1}{ML} \sum_{k=1}^{ML} \phi(\hat{\lambda}_{k,N}) = \int_{\mathbb{R}^{+}} \phi(\lambda) d\hat{\mu}_N(\lambda)
$$
More specifically, we evaluate the behaviour of $\widehat{\phi}_N - \overline{\phi}_N$ 
where $\overline{\phi}_N$ is defined in (\ref{eq:defbarphi}) by establishing the following result. 
%In order to avoid to address a number of different cases, we assume in this section 
%that the parameter $\gamma_0$ defined in Assumption \ref{as:norm-r-omega} is large enough so that
%\begin{assum}
%\label{as:simplification} 
%$\frac{1}{L^{\gamma_0}} \ll \frac{1}{M}$
%\end{assum}, 
%or equivalently that $\gamma_0 > \frac{1}{\beta} - 1$. 
%The purpose of this section is to 
%In order to characterize the effect of using the estimated, establish the following result. 
\begin{theorem}
\label{th:hatphi-overlinephi}
Let Assumptions \ref{assum:statistics}-\ref{ass:bounds-spectral-densities} and \ref{as:norm-r-omega} hold true. Assume that the function $\phi$ is defined on $(-\delta,+\infty)$ for some $\delta > 0$ and smooth in a neighbourhood of the interval $[0,\alpha_0]$ where $\alpha_0$ is defined in Proposition \ref{prop:concentration-hatcorr-overlinecorr}. Then, it holds that 
\begin{equation}
\label{eq:magnitude-hatphi-overlinephi}
|\widehat{\phi}_N - \overline{\phi}_N| \prec \max\left(\frac{1}{M}, \frac{1}{L^{\gamma_0}}\right).
\end{equation}
\end{theorem}
%We recall that the parameter $\gamma_0$ is defined in Assumption \ref{as:norm-r-omega}. 
In order to establish Theorem \ref{th:hatphi-overlinephi}, we first mention that Proposition 
\ref{prop:concentration-hatcorr-overlinecorr} implies that it is possible to assume without restriction that 
$\phi$ is compactly supported by the interval $[-\delta, \alpha]$ for some $\alpha > \alpha_0$. 
To justify this claim, we consider $\xi \in (\alpha_0, \alpha)$ and introduce the event $\mathcal{A}_N$ defined by 
$$
\mathcal{A}_N = \left\{   \| \overline{\mathcal{R}}_{\mathrm{corr},L} \| \leq \xi \right\} 
\cap \left\{   \| \widehat{\mathcal{R}}_{\mathrm{corr},L} \| \leq \xi \right\}. 
$$
Proposition \ref{prop:concentration-hatcorr-overlinecorr} implies that there exists a $\eta > 0$
for which $P(\mathcal{A}_N^{c}) \leq \exp-N^{\eta}$ for each $N$ 
large enough. We denote by $\phi_c$ a smooth function, supported by $[-\delta, \alpha]$, 
and which coincides with $\phi$ on the interval $[-\delta/2, \xi]$. Then, it is clear 
that $\widehat{\phi}_N$ and $\overline{\phi}_N$ coincide with $\widehat{\phi}_{c,N}$ and $\overline{\phi}_{c,N}$ respectively on $\mathcal{A}_N$. 
For each $\epsilon > 0$, by conditioning on the event $\mathcal{A}_N$ and its complementary $\mathcal{A}_N^c$ we can express
\begin{multline*}
    \mathbb{P}\left(|\widehat{\phi}_N - \overline{\phi}_N| > N^{\epsilon} \max\left(\frac{1}{M}, \frac{1}{L^{\gamma_0}}\right)\right)  =  \\
%\mathbb{P}\left(|\widehat{\phi}_N - \overline{\phi}_N| > N^{\epsilon} \max\left(\frac{1}{M}, \frac{1}{L^{\gamma_0}}\right), \, \mathcal{A}_N \right) + 
%\mathbb{P}\left(|\widehat{\phi}_N - \overline{\phi}_N| > N^{\epsilon} \max\left(\frac{1}{M}, \frac{1}{L^{\gamma_0}}\right), \, \mathcal{A}_N^{c} \right)
%or as 
%$$
= \mathbb{P}\left(|\widehat{\phi}_{c,N} - \overline{\phi}_{c,N}| > N^{\epsilon} \max\left(\frac{1}{M}, \frac{1}{L^{\gamma_0}}\right), \, \mathcal{A}_N \right) + \\
+ \mathbb{P}\left(|\widehat{\phi}_N - \overline{\phi}_N| > N^{\epsilon} \max\left(\frac{1}{M}, \frac{1}{L^{\gamma_0}}\right), \, \mathcal{A}_N^{c} \right)
%$$
\end{multline*}
where we have used the fact that $\widehat{\phi}_N$ and $\overline{\phi}_N$ respectively coincide with $\widehat{\phi}_{c,N}$ and $\overline{\phi}_{c,N}$ on $\mathcal{A}_N$. Now, for $N$ large enough we can bound the first term of the above equation by ${\mathbb{P}(|\widehat{\phi}_{c,N} - \overline{\phi}_{c,N}| > N^{\epsilon} \max(M^{-1}, {L^{-\gamma_0}}))}$  and the second term by $\mathbb{P}( \mathcal{A}_N^{c}) \leq \exp(-N^{\eta})$. Therefore, it is sufficient to establish that $|\widehat{\phi}_{c,N} - \overline{\phi}_{c,N}| \prec \max(M^{-1}, L^{-\gamma_0})$ to prove (\ref{eq:magnitude-hatphi-overlinephi}). For this reason, from now on we assume 
without loss of generality that $\phi$ is supported by $[-\delta, \alpha]$.
% \begin{multline*}
%     \mathbb{P}\left(|\widehat{\phi}_{c,N} - \overline{\phi}_{c,N}| > N^{\epsilon} \max\left(\frac{1}{M}, \frac{1}{L^{\gamma_0}}\right), \, \mathcal{A}_N \right) \leq \mathbb{P}\left(|\widehat{\phi}_{c,N} - \overline{\phi}_{c,N}| > N^{\epsilon} \max\left(\frac{1}{M}, \frac{1}{L^{\gamma_0}}\right) \right)
% \end{multline*}
% As the inequalities 
% $$
% \mathbb{P}\left(|\widehat{\phi}_N - \overline{\phi}_N| > N^{\epsilon} \max\left(\frac{1}{M}, \frac{1}{L^{\gamma_0}}\right), \, \mathcal{A}_N^{c} \right) \leq \mathbb{P}\left(  \mathcal{A}_N^{c} \right)
% \leq \exp(-N^{\eta})
% $$
% and 
% $$
% \mathbb{P}\left(|\widehat{\phi}_{c,N} - \overline{\phi}_{c,N}| > N^{\epsilon} \max\left(\frac{1}{M}, \frac{1}{L^{\gamma_0}}\right), \, \mathcal{A}_N \right) \leq \mathbb{P}\left(|\widehat{\phi}_{c,N} - \overline{\phi}_{c,N}| > N^{\epsilon} \max\left(\frac{1}{M}, \frac{1}{L^{\gamma_0}}\right) \right)
% $$
% hold for $N$ large enough, 
% it is clear that it is sufficient to establish that $|\widehat{\phi}_{c,N} - \overline{\phi}_{c,N}| \prec 
% \max\left(\frac{1}{M}, \frac{1}{L^{\gamma_0}} \right)$ to prove (\ref{eq:magnitude-hatphi-overlinephi}). Therefore, we assume from now on 
% without any restriction $\phi$ is supported by $[-\delta, \alpha]$. \\

The main tool that we will use in order to analyze the asymptotic behavior of the linear spectral statistics is the Helffer-Sjöstrand formula for sufficiently regular, compactly supported functions.This formula was already used in the large random matrices literature, see e.g. \cite{and-gui-zei-2010},  \cite{anderson-2013}, \cite{najim-yao-2016}. In order to introduce this tool, assume that $\phi(\lambda)$ is compactly supported and of class $\mathcal{C}^{k+1}$ for a certain integer $k$, and denote by $\Phi_k(\phi): \mathbb{C} \rightarrow \mathbb{C}$ the function of complex variable
\begin{equation}
\label{eq:cleverfunctHS}
    \Phi_k(\phi)(x+iy) = \sum_{l=0}^{k} \frac{(iy)^l}{l!} \phi^{(l)}(x) \rho (y)
\end{equation}
where $\rho: \mathbb{R} \rightarrow \mathbb{R}^+$  is a smooth, compactly supported function (to fix the ideas, we assume that the support of $\rho$ is $[-2,2]$) that takes the value $1$ in a neighbourhood of zero. Now, taking $z=x + iy$, we see that the function $\Phi_k(\phi)(z)$ is compactly supported on the complex plane, and therefore by \cite[Lemma 20.3]{rudin-book} we have 
\begin{equation*}
    \int \phi(\lambda) d\mu(\lambda) = \frac{1}{\pi} \mathrm{Re} \int_{\mathbb{C}^+} dx\,dy\,\overline{\partial} \Phi_k(\phi) (z) s_\mu (z) 
\end{equation*}
where $\mu$ is a probability measure, $s_\mu(z)$ its Stieltjes transform and where we define 
$$
\overline{\partial}  \Phi_k(\phi)(z) = \frac{\partial \Phi_k(\phi)(x + iy)}{\partial x} + i\frac{\partial \Phi_k(\phi)(x + iy)}{\partial y}.
$$
In particular, according to the definition of $\Phi_k(\phi)(z)$ in (\ref{eq:cleverfunctHS}), we can see that
$$
\overline{\partial}  \Phi_k(\phi)(z) = \frac{(iy)^k}{k!}\phi^{(k+1)}(x)
$$
when $y$ belongs to a neighbourhood of zero where $\rho(y)=1$. The regularity of $\phi$ will allow us to bound quantities of the form $|\overline{\partial}\Phi_k(\phi)(z) y^{-k}|$ when $y$ is in a neighbourhood of zero. 

Consider now the two resolvents $\widehat{\mathbf{Q}}_N(z)$ and $\mathbf{Q}_N(z)$ defined in (\ref{eq:def-resolvent-tildeRcorr}) and (\ref{eq:def_resolvent_barRcorr}) respectively. Recall that their normalized trace is equal to the Stieltjes transform of the empirical eigenvalue distribution of $\widehat{\mathcal{R}}_{\mathrm{corr},L}$ and $\overline{\mathcal{R}}_{\mathrm{corr},L}$ respectively. Hence, a direct application of the Helffer-Sjöstrand formula to our problem leads to the identity
\begin{equation}
\label{eq:expre-hatphi-overlinephi-hs}
\widehat{\phi}_N - \overline{\phi}_N = \frac{1}{\pi} \mathrm{Re} \int_{\mathcal{D}} dx \, dy \,\bar{\partial}\Phi_k(\phi)(z)\left(\frac{1}{ML}\mathrm{Tr} (\widehat{\mathbf{Q}}_N(z)) -  \frac{1}{ML}\mathrm{Tr} (\mathbf{Q}_N(z)) \right)
\end{equation}
where $\mathcal{D}$ is defined by $\mathcal{D} = [-\delta, \alpha] \times [0,2]$ and where $k$ is large enough. Before going into 
the details of the proof of Theorem \ref{th:hatphi-overlinephi}, we first present the main steps of the proof. In what follows, we will omit the dependence on $N$ and $z$ in all the matrices in order to simplify the notation. 

We recall that $\bs{\Theta}$ is the matrix defined in (\ref{eq:def-Theta}) and remark that, by the definition of resolvents, we can write 
$$
\widehat{{\bf Q}} - {\bf Q} = - {\bf Q} \bs{\Theta} \widehat{{\bf Q}} = - {\bf Q} \bs{\Theta} {\bf Q} + 
{\bf Q}  \bs{\Theta} {\bf Q} \bs{\Theta} \widehat{{\bf Q}}
$$
Therefore, (\ref{eq:expre-hatphi-overlinephi-hs}) can also be written as 
\begin{eqnarray}
    \widehat{\phi}_N - \overline{\phi}_N = &-& \frac{1}{\pi}\mathrm{Re}\int_{\mathcal{D}} dx \, dy \,\bar{\partial}\Phi_k(\phi)(z) \frac{1}{ML}(\mathrm{Tr}  {\bf Q}^{2} \bs{\Theta}) \nonumber \\
    &+&  \frac{1}{\pi}\mathrm{Re}\int_{\mathcal{D}} dx \, dy \,\bar{\partial}\Phi_k(\phi)(z) \frac{1}{ML}\mathrm{Tr}( {\bf Q}  \bs{\Theta} {\bf Q} \bs{\Theta} \widehat{{\bf Q}}). \label{eq:expre-hatphi-overlinephi-hs-improved}
\end{eqnarray}
Having established these basic facts, the proof of Theorem \ref{th:hatphi-overlinephi} proceeds as follows:
\begin{enumerate}[leftmargin=3pt]
\item The first step of the proof consists in showing that, by virtue of (\ref{eq:thetaprec}), the second term of (\ref{eq:expre-hatphi-overlinephi-hs-improved}) can be disregarded from the evaluation, in the sense that
\begin{equation}
    \label{eq:term-Theta2-negligible}
\left| \frac{1}{\pi} \mathrm{Re} \int_{\mathcal{D}} dx \, dy \,\bar{\partial}\Phi_k(\phi)(z) 
\frac{1}{ML} \mathrm{Tr} ({\bf Q}  \bs{\Theta} {\bf Q} \bs{\Theta} \widehat{{\bf Q}}) \right| \prec \max\left(\frac{1}{M},\frac{1}{L^{2\gamma_0}}\right)
\end{equation}
%{\color{blue} should we mentioned that a tighter evaluation is $\max\left(\frac{1}{M},\frac{1}{L^{2 \gamma_0}}\right)$ ?} {\color{red} Certainly!}
We therefore just need to evaluate the first term of the right hand side of (\ref{eq:expre-hatphi-overlinephi-hs-improved}).
\item In the second step, it is proved that ${\bs \Theta}$ can be written as 
\begin{equation}
\label{eq:simplification-Theta}
{\bs \Theta} = \left( \widehat{\mathcal{B}}^{-1/2}  - \mathcal{B}^{-1/2} \right) \, \mathcal{B}^{1/2} \overline{\mathcal{R}}_{\mathrm{corr}} + 
\overline{\mathcal{R}}_{\mathrm{corr}} \, \mathcal{B}^{1/2} \left( \widehat{\mathcal{B}}^{-1/2}  - \mathcal{B}^{-1/2} \right) + {\bs \Theta}_2
\end{equation}
where $\|{\bs \Theta}_2\| \prec \max(M^{-1},L^{-2\gamma_0})$. This will imply that the contribution of $\bs{\Theta}_2$ to 
the first term of the right hand side of (\ref{eq:expre-hatphi-overlinephi-hs-improved}) can be omitted. 
\item If we take ${\bs \Theta}_1 = {\bs \Theta} - {\bs \Theta}_2$, the purpose of the third step is to establish that 
\begin{equation}
\label{eq:difficult-to-prove-1}
\left|  \int_{\mathcal{D}} dx \, dy \,\bar{\partial}\Phi_k(\phi)(z) 
\frac{1}{ML}\mathrm{Tr} ({\bf Q}^{2} {\bs \Theta}_1) \right| \prec \max\left(\frac{1}{M}, \frac{1}{L^{\gamma_0}}\right).
\end{equation}
For this, we will just verify that 
\begin{equation}
\label{eq:difficult-to-prove-2}
\left| \int_{\mathcal{D}} dx \, dy \,\bar{\partial}\Phi_k(\phi)(z) 
\frac{1}{ML} \mathrm{Tr} \left[ {\bf Q}^{2} (\widehat{\mathcal{B}}^{-1/2} -\mathcal{B}^{-1/2}) \, \mathcal{B}^{1/2} \overline{\mathcal{R}}_{\mathrm{corr}} \right] \right| \prec \max\left(\frac{1}{M}, \frac{1}{L^{\gamma_0}}\right)
\end{equation}
(note that the second term in (\ref{eq:simplification-Theta}) can be handled similarly). The proof of (\ref{eq:difficult-to-prove-2}) is demanding. Using Lemma \ref{le:perturbation}, 
%a first order expansion of matrix 
%$\widehat{\mathcal{R}}_{m,L}^{-1/2}$ based on perturbation %theory of Hermitian matrices 
we only need to show that 
\begin{multline}
 \label{eq:difficult-to-prove-3}
\left| \int_{\mathcal{D}} dx \, dy \,\bar{\partial}\Phi_k(\phi)(z) 
\frac{1}{M} \sum_{m=1}^{M}  \frac{1}{L} \mathrm{Tr} \left[ \mathcal{D}_{m,L}({\bs \Delta}_{m,L}) \mathcal{R}_{m,L}^{1/2} ({\bf Q} + z {\bf Q}^{2})_{m,m}  \right] \right| \\ \prec \max\left(\frac{1}{M}, \frac{1}{L^{\gamma_0}}\right)
\end{multline}
where ${\bf Q}_{m,m}$ denotes the $m$th $L \times L$ diagonal block of ${\bf Q}$, where the operator $\mathcal{D}_{m,L}$ is defined in (\ref{eq:def-differential-1})-(\ref{eq:def-differential-2}) and where ${\bs \Delta}_{m,L}$ is defined in (\ref{eq:def-deltamL}).  
We will only establish that 
\begin{equation}
\label{eq:difficult-to-prove-4}
\left|  \int_{\mathcal{D}} dx \, dy \,\bar{\partial}\Phi_k(\phi)(z) 
\frac{1}{M} \sum_{m=1}^{M} \frac{1}{L}  \mathrm{Tr} \left[   \mathcal{D}_{m,L}({\bs \Delta}_{m,L}) \mathcal{R}_{m,L}^{1/2} {\bf Q}_{m,m} \right] \right|  \prec \max\left(\frac{1}{M}, \frac{1}{L^{\gamma_0}}\right)
\end{equation}
because the term due to $z ({\bf Q}^{2})_{m,m}$ can be handled similarly. In order to show this, we rely on the fact that, up to a term 
stochastically dominated by $\frac{1}{M}$, it is possible to replace matrices $(\widehat{\mathcal{R}}_{m,L})_{m=1, \ldots, M}$ in (\ref{eq:difficult-to-prove-4}) by their Toeplitz approximations $\widehat{\mathcal{R}}_{m,L}^{t} = \widetilde{\mathbf{W}}^m_N(\widetilde{\mathbf{W}}^m_N)^H $, $m=1,\ldots,M$ introduced in (\ref{eq:def_RLmToep}) of Section \ref{subsec:convergence-estimators-RmL}. The upper bound in (\ref{eq:biais-hatsm}) will
%and Assumption \ref{as:simplification} 
imply that the contribution of the bias of the Toeplitz estimates of $(\mathcal{R}_{m,L})_{m=1, \ldots, M}$ to (\ref{eq:difficult-to-prove-4}) is a term of order $\mathcal{O}(\max{(M^{-1},L^{-\gamma_0}}))$. At this point, it will remain to 
study the term $\zeta$ defined by 
\begin{equation}
    \label{eq:def-zeta}
    \zeta = \int_{\mathcal{D}} dx \, dy \,\bar{\partial}\Phi_k(\phi)(z) 
\frac{1}{M} \sum_{m=1}^{M} \frac{1}{L}  \mathrm{Tr} \left[  \mathcal{D}_{m,L} \left(\widehat{\mathcal{R}}^{t}_{m,L} - \mathbb{E}(\widehat{\mathcal{R}}^{t}_{m,L})\right) \mathcal{R}_{m,L}^{1/2} {\bf Q}_{m,m} \right].
\end{equation}
Recall that $\hat{r}_m(l) = \frac{1}{N} \sum_{n=1}^{N-l} y_m(n+l) y_m^{*}(n)$ 
and $\hat{r}_m(-l) = \hat{r}_m^{*}(l)$ for $l \geq 0$ represent the empirical estimate of the autocovariance sequence of $y_m$ at lag $l$, and consider $\hat{r}_m^{\circ}(l) = \hat{r}_m - \mathbb{E}\hat{r}_m(l)$. With these definitions and using (\ref{eq:swapDmL}) and (\ref{eq:trace-toeplitz}, the term $\zeta$ can be re-written as 
\begin{equation}
\label{eq:expre-bis-zeta}
\zeta =   \int_{\mathcal{D}} dx \, dy \,\bar{\partial}\Phi_k(\phi)(z) 
\frac{1}{M} \sum_{m=1}^{M}  \sum_{u=-(L-1)}^{L-1} \hat{r}_m^{\circ}(u)\;
\tau\left( \mathcal{D}_{m,L}\left(\mathcal{R}_{m,L}^{1/2} {\bf Q}_{m,m}\right) \right)(-u)
\end{equation}
where we recall that if ${\bf A}$ is a $L \times L$ matrix, $\tau({\bf A})(u)$ is defined by $\tau({\bf A})(u) = 
\frac{1}{L} \mathrm{Tr}( {\bf A} {\bf J}_L^{u})$, see (\ref{eq:definition_tau}). This way of expressing $\zeta$ will be the key to showing that 
\begin{equation}
    \label{eq:zeta-prec-1/M}
    |\zeta| \prec \frac{1}{M}
\end{equation}
which will complete the proof of Theorem \ref{th:hatphi-overlinephi}. This will be shown in two final steps. \begin{itemize}[leftmargin=12pt]
\item We check that $\mathbb{E}(\zeta) =  \mathcal{O}(M^{-1})$. To verify this, we use that the Nash-Poincaré inequality and obtain that 
$\mathrm{Var}\left[ \tau\left( \mathcal{D}_{m,L}(\mathcal{R}_{m,L}^{1/2} {\bf Q}_{m,m}) \right)(-u) \right] = \mathcal{O}(N^{-1})$. Since $\mathrm{Var}(\hat{r}_m(u))$ is also a term of order $\mathcal{O}(N^{-1})$, we obtain immediately 
from the Schwartz inequality that $\mathbb{E}(\zeta) = \mathcal{O}(M^{-1})$.
\item The most difficult part of the proof consists in establishing that
\begin{equation}
\label{eq:zeta-Ezeta-prec-1/M}
\left| \zeta - \mathbb{E}(\zeta) \right| \prec \frac{1}{M}.
\end{equation}
For this, for each $m$, we introduce the $(N+L-1)$--dimensional row vector 
${\bf y}_m = (y_{m,1}, \ldots, y_{m,N+L-1})$, which can be re-written as 
\begin{equation}
\label{eq:def-xm}
{\bf y}_m = {\bf x}_m \mathcal{R}_{m,N+L-1}^{1/2}
\end{equation}
for some $\mathcal{N}_\mathbb{C}(0, \I_{N+L-1})$-distributed row vector ${\bf x}_m$. By using the above definition, we can re-interpret $\zeta$ as a function of  of the $M(N+L-1)$ i.i.d. $\mathcal{N}_\mathbb{C}(0,1)$ entries of vector 
\begin{equation}
    \label{eq:def-x}
    {\bf x} = ({\bf x}_1, \ldots, {\bf x}_M).
\end{equation}
%$\zeta$ depending on vectors ${\bf y}_1, \ldots, {\bf y}_M$, we can interpret $\zeta$ as a function of the $M(N+L-1)$ i.i.d. $\mathcal{N}_c(0,1)$ entries of vector ${\bf x} = ({\bf x}_1, \ldots, {\bf x}_M)$. 
If $\zeta$, considered as a function of $({\bf x},{\bf x}^{*})$
%the $2(N+L-1)$--dimensional real-valued vector\footnote{to simplify the notations, we will identify
%in the following $(\mathrm{Re}({\bf x}), \Imm({\bf x}))$ with ${\bf x}.$} $(\mathrm{Re}({\bf x}), \Imm({\bf x}))$,
were a Lipschitz function with constant of order $\mathcal{O}(M^{-1})$, the result in (\ref{eq:zeta-Ezeta-prec-1/M}) would follow from conventional concentration inequalities of Lipschitz functions of Gaussian random vectors (see (\ref{eq:concentrationGaussian}) above). 
%a Gaussian concentration inequality. the Gaussian concentration inequality (to be introduced) would immediately provide (\ref{eq:zeta-Ezeta-prec-1/M}). 
Unfortunately, the terms $\hat{r}_m(u)$ are not Lipschitz functions 
of ${\bf x}$ due to the quadratic dependence on this vector. In any case, it is still true that for each $\epsilon > 0$, the inequality $| \hat{r}_m(u) - \mathbb{E}(\hat{r}_m(u)) | \leq \frac{N^{\epsilon}}{N}$ holds for fixed $u$ and $m$, except for an event that has exponentially small probability. Therefore, we show that it is possible to replace (for each $u$ and $m$)  $ \hat{r}_m(u) - \mathbb{E}(\hat{r}_m(u))$ by a well chosen function, and that the corresponding modification $\tilde{\zeta}$ of $\zeta$
is Lipschitz with constant $\frac{N^{\epsilon}}{M}$. We deduce from this that $|\zeta - \mathbb{E}(\zeta) | \prec 
\frac{N^{\epsilon}}{M}$ for each $\epsilon > 0$, a property which will directly imply (\ref{eq:zeta-Ezeta-prec-1/M}). 
\end{itemize}
\end{enumerate}
We now proceed with the three steps of the proof. 

{\bf Step 1}. 
In order to establish (\ref{eq:term-Theta2-negligible}), we simply notice that 
$$
\left|  \int_{\mathcal{D}} dx \, dy \,\bar{\partial}\Phi_k(\phi)(z) 
\frac{1}{ML} \mathrm{Tr} ({\bf Q}  \bs{\Theta} {\bf Q} \bs{\Theta} \hat{{\bf Q}}) \right| \leq 
 \int_{\mathcal{D}} dx \, dy \,|\bar{\partial}\Phi_k(\phi)(z)| \left|\frac{1}{ML} \mathrm{Tr} ( {\bf Q}  \bs{\Theta} {\bf Q} \bs{\Theta} \hat{{\bf Q}})\right|.
 $$
 It is clear that if $z \in \mathbb{C}^{+}$, we can use the item (iv) in Proposition \ref{prop:class-S} to establish that
 $$
 \left|\frac{1}{ML}\mathrm{Tr} ({\bf Q}  \bs{\Theta} {\bf Q} \bs{\Theta} \hat{{\bf Q}}) \right| \leq \|{\bf Q}\|^{2}  \| \hat{{\bf Q}} \| \|\bs{\Theta}\|^{2} \leq \frac{1}{(\Imm z)^{3}} \|\bs{\Theta}\|^{2}.
 $$
Since $\phi$ is smooth by assumption, we can choose $k \geq 3$ to guarantee that the integral $ \int_{\mathcal{D}} dx \, dy \,|\bar{\partial}\Phi_k(\phi)(z)| \frac{1}{{(\Imm z)}^{3}}$ is finite. This, together with  (\ref{eq:thetaprec}), shows that 
$$
\left|  \int_{\mathcal{D}} dx \, dy \,\bar{\partial}\Phi_k(\phi)(z) 
\frac{1}{ML}(\mathrm{Tr} {\bf Q}  \bs{\Theta} {\bf Q} \bs{\Theta} \hat{{\bf Q}}) \right| \leq \kappa  \|\bs{\Theta}\|^{2} \prec \max\left(\frac{1}{M},\frac{1}{L^{2\gamma_0}}\right)
$$
which completes the proof of (\ref{eq:term-Theta2-negligible}). \\

{\bf Step 2}. In order to establish (\ref{eq:simplification-Theta}), we take (\ref{eq:expre-Theta}) as a 
starting point and express $\widehat{\mathcal{R}}$ as $\widehat{\mathcal{R}} = \mathcal{B}^{1/2} \overline{\mathcal{R}}_{\mathrm{corr}} \mathcal{B}^{1/2}$, that is
\begin{align*}
\bs{\Theta} = \left( \widehat{\mathcal{B}}^{-1/2}  - \mathcal{B}^{-1/2} \right) \, \mathcal{B}^{1/2} \overline{\mathcal{R}}_{\mathrm{corr}} \mathcal{B}^{1/2}  \widehat{\mathcal{B}}^{-1/2} + 
\overline{\mathcal{R}}_{\mathrm{corr}} \, \mathcal{B}^{1/2} \left( \widehat{\mathcal{B}}^{-1/2}  - \mathcal{B}^{-1/2} \right) \\ = \left( \widehat{\mathcal{B}}^{-1/2}  - \mathcal{B}^{-1/2} \right) \, \mathcal{B}^{1/2} \overline{\mathcal{R}}_{\mathrm{corr}}  + 
\overline{\mathcal{R}}_{\mathrm{corr}} \, \mathcal{B}^{1/2} \left( \widehat{\mathcal{B}}^{-1/2}  - \mathcal{B}^{-1/2} \right)
+ \bs{\Theta}_2
\end{align*}
where $\bs{\Theta}_2$ is given by 
\begin{equation*}
\bs{\Theta}_2 
%\left( \widehat{\mathcal{B}}^{-1/2}  - \mathcal{B}^{-1/2} \right)  \mathcal{B}^{1/2} \overline{\mathcal{R}}_{\mathrm{corr}} \left( \mathcal{B}^{1/2}  \widehat{\mathcal{B}}^{-1/2} - \I \right) \\
= \left( \widehat{\mathcal{B}}^{-1/2}  - \mathcal{B}^{-1/2} \right)  \mathcal{B}^{1/2} \overline{\mathcal{R}}_{\mathrm{corr}}
 \mathcal{B}^{1/2} \left( \widehat{\mathcal{B}}^{-1/2}  - \mathcal{B}^{-1/2} \right).
\end{equation*}
As we showed that $\|\widehat{\mathcal{B}}^{-1/2}  - \mathcal{B}^{-1/2}\| \prec \max (M^{-1/2},L^{-\gamma_0})$, 
$\| \bs{\Theta}_2 \|$ clearly verifies $\| \bs{\Theta}_2 \| \prec \max (M^{-1},L^{-2\gamma_0})$ as expected. Hence, using 
the inequality $|\frac{1}{ML} \mathrm{Tr} ({\bf Q}^{2} \bs{\Theta}_2) | \leq \frac{\| \bs{\Theta}_2 \|}{(\Imm z)^{2}}$ together with
$$
 \int_{\mathcal{D}} dx \, dy \,\left|\bar{\partial}\Phi_k(\phi)(z)\right| \frac{1}{(\Imm z)^{2}} < +\infty
 $$
we obtain that 
$$
\left|  \int_{\mathcal{D}} dx \, dy \,\bar{\partial}\Phi_k(\phi)(z) 
\frac{1}{ML} \mathrm{Tr} ({\bf Q}^{2}  \bs{\Theta}_2) \right| \leq \kappa \, \| \bs{\Theta}_2 \| \prec \max \left(\frac{1}{M},\frac{1}{L^{2\gamma_0}}\right).
$$

{\bf Step 3}. We finally establish (\ref{eq:difficult-to-prove-1}), and just verify (\ref{eq:difficult-to-prove-2}) because the 
contribution of the second term of $\bs{\Theta}_1$ 
can be handled similarly. Using (\ref{eq:formula-perturbation}) and the resolvent identity 
$\overline{\mathcal{R}}_{\mathrm{corr}} {\bf Q} = \I+ z {\bf Q}$, we obtain immediately that
the term on the left hand side of (\ref{eq:difficult-to-prove-2}) can be written as the sum of the term on the left hand side of 
(\ref{eq:difficult-to-prove-3}) plus a term depending on the matrices $(\bs{\Upsilon}_m)_{m=1, \ldots, M}$. As this last term is easily seen to be stochastically dominated by $\max({L^{- 2 \gamma_0}}, {M}^{-1})$, (\ref{eq:difficult-to-prove-2}) becomes  
equivalent to (\ref{eq:difficult-to-prove-3}). 

We now prove (\ref{eq:difficult-to-prove-4}). We first reason that we can replace the matrices $\widehat{\mathcal{R}}_{m,L}$ with their Toeplitz approximations $\widehat{\mathcal{R}}^t_{m,L} $. Indeed, it was shown in (\ref{eq:toeplitz-approximation-hatRm}) of Section \ref{subsec:convergence-estimators-RmL}
that $\| \widehat{\mathcal{R}}_{m,L} -  \widehat{\mathcal{R}}_{m,L}^{t} \| \prec \frac{1}{M}$. We claim that this 
implies that 
$$
\left| \int_{\mathcal{D}} dx \, dy \,\bar{\partial}\Phi_k(\phi)(z) 
\frac{1}{M} \sum_{m=1}^{M} \frac{1}{L}  \mathrm{Tr} \left[ \mathcal{D}_{m,L} \left(\widehat{\mathcal{R}}_{m,L} - \widehat{\mathcal{R}}_{m,L}^{t}\right) \mathcal{R}_{m,L}^{1/2} {\bf Q}_{m,m} \right] \right| \prec \frac{1}{M} 
$$
Indeed, a direct use of (\ref{eq:continuity-mathcalD}) together with the fact that $\| {\bf Q}_{m,m} \| \leq (\Imm z)^{-1}$ for $z \in \mathbb{C}^+$ shows that
$$
\left| \frac{1}{M} \sum_{m=1}^{M} \frac{1}{L}  \mathrm{Tr} \left[  \mathcal{D}_{m,L} \left(\widehat{\mathcal{R}}_{m,L} - \widehat{\mathcal{R}}_{m,L}^{t}\right) \mathcal{R}_{m,L}^{1/2} {\bf Q}_{m,m} \right] \right| \leq 
\frac{\kappa}{\Imm z} \, \sup_{m=1, \ldots, M} \| \widehat{\mathcal{R}}_{m,L} -  \widehat{\mathcal{R}}_{m,L}^{t} \|
$$
for some nice constant $\kappa >0$. However, since the integral $\int_{\mathcal{D}} dx \, dy \,|\bar{\partial}\Phi_k(\phi)(z)| \frac{1}{\Imm z} $ is finite as long as $k \geq 1$, 
we readily see that 
\begin{multline*}
\left| \int_{\mathcal{D}} dx \, dy \,\bar{\partial}\Phi_k(\phi)(z) 
\frac{1}{M} \sum_{m=1}^{M} \frac{1}{L}  \mathrm{Tr} \left[  \mathcal{D}_{m,L} \left(\widehat{\mathcal{R}}_{m,L} - \widehat{\mathcal{R}}_{m,L}^{t}\right) \mathcal{R}_{m,L}^{1/2} {\bf Q}_{m,m} \right] \right| \leq \\ \leq \kappa \, \sup_{m=1, \ldots, M} \| \widehat{\mathcal{R}}_{m,L} -  \widehat{\mathcal{R}}_{m,L}^{t} \| \prec \frac{1}{M}.
\end{multline*}
Consequently, in order to establish  (\ref{eq:difficult-to-prove-4}), it only remains to prove that 
\begin{multline}
\label{eq:difficult-to-prove-5}
\left|  \int_{\mathcal{D}} dx \, dy \,\bar{\partial}\Phi_k(\phi)(z) 
\frac{1}{M} \sum_{m=1}^{M} \frac{1}{L}  \mathrm{Tr} \left[  \mathcal{D}_{m,L} \left(\widehat{\mathcal{R}}_{m,L}^{t} - \mathcal{R}_{m,L}\right) \mathcal{R}_{m,L}^{1/2} {\bf Q}_{m,m} \right] \right| \prec \\ \prec \max\left(\frac{1}{M}, \frac{1}{L^{\gamma_0}}\right).
\end{multline}
Given the Toeplitz structure of $\widehat{\mathcal{R}}_{m,L}^{t} - \mathcal{R}_{m,L}$, the bound that was established in (\ref{eq:biais-hatsm}) directly implies that $\sup_{m=1, \ldots, M} \| \mathbb{E}\widehat{\mathcal{R}}_{m,L}^{t} - 
\mathcal{R}_{m,L} \|  \leq \kappa (L^{-\gamma_0} + N^{-1}L^{(1-\gamma_0)_{+}})$. % \frac{L^{(1-\gamma_0)_{+}}}{N})$. 
This leads immediately to 
\begin{multline*}
\left| \int_{\mathcal{D}} dx \, dy \,\bar{\partial}\Phi_k(\phi)(z) 
\frac{1}{M} \sum_{m=1}^{M} \frac{1}{L}  \mathrm{Tr} \left[  \mathcal{D}_{m,L} \left(\mathbb{E}\left(\widehat{\mathcal{R}}_{m,L}^{t}\right) - 
\mathcal{R}_{m,L} \right) \mathcal{R}_{m,L}^{1/2} {\bf Q}_{m,m} \right] \right|  \leq \\  \leq  \kappa 
 \left(\frac{1}{L^{\gamma_0}} + \frac{L^{(1-\gamma_0)_{+}}}{N}\right) 
      \leq   \kappa \left(\frac{1}{L^{\gamma_0}} + \frac{1}{M}\right).
\end{multline*}
It thus remains to study $\zeta$ defined in (\ref{eq:def-zeta}). Noting that $\widehat{\mathcal{R}}_{m,L}^{t} - \mathbb{E}(\widehat{\mathcal{R}}_{m,L}^{t})$ is the $L \times L$ Toeplitz matrix with entries $\hat{r}^\circ_m(i-j) = \hat{r}_m(i-j) - \mathbb{E}\hat{r}_m(i-j)$, 
$1 \leq i,j \leq L$, we can establish, by virtue of (\ref{eq:swapDmL}) and (\ref{eq:trace-toeplitz}), 
\begin{multline*}
\frac{1}{L}  \mathrm{Tr} \left[  \mathcal{D}_{m,L} (\widehat{\mathcal{R}}^{t}_{m,L} - \mathbb{E}(\widehat{\mathcal{R}}^{t}_{m,L})) \mathcal{R}_{m,L}^{1/2} {\bf Q}_{m,m} \right] = \frac{1}{L}  \mathrm{Tr} \left[ (\widehat{\mathcal{R}}^{t}_{m,L} - \mathbb{E}(\widehat{\mathcal{R}}^{t}_{m,L}))  \mathcal{D}_{m,L}( \mathcal{R}_{m,L}^{1/2} {\bf Q}_{m,m}) \right] \\  = \sum_{u=-(L-1)}^{L-1} \hat{r}^\circ_m(u) \; \tau( \mathcal{D}_{m,L}(\mathcal{R}_{m,L}^{1/2} {\bf Q}_{m,m}) )(-u).
\end{multline*}
We first study the expectation of $\zeta$ and prove that
\begin{equation}
    \label{eq:eval-Ezeta}
    \mathbb{E}(\zeta) = \mathcal{O}\left(\frac{1}{M}\right).
\end{equation}
For this, 
we first note that, using again (\ref{eq:swapDmL}), we can write
$$
\mathbb{E}(\zeta)  = \int_{\mathcal{D}} dx \, dy \,\bar{\partial}\Phi_k(\phi)(z) \frac{1}{M} \sum_{m=1}^{M}  \sum_{u=-(L-1)}^{L-1} \mathbb{E} \left( \hat{r}^\circ_m(u) \; \frac{1}{L} \mathrm{Tr} \left( {\bf Q}^{\circ}_{m,m} 
\mathcal{D}_{m,L}( {\bf J}_L^{-u}) \mathcal{R}_{m,L}^{1/2} \right)  \right)
$$
where we recall that ${\bf Q}_{m,m}^{\circ} = {\bf Q}_{m,m} - \mathbb{E}({\bf Q}_{m,m})$. By the Cauchy-Schwarz inequality, we have 
\begin{multline}
    \left|  \mathbb{E} \left( \hat{r}^\circ_m(u) \; \frac{1}{L} \mathrm{Tr} \left( {\bf Q}^{\circ}_{m,m} \mathcal{D}_{m,L}({\bf J}_L^{-u}) \mathcal{R}_{m,L}^{1/2} \right)  \right) \right| \leq  \\ \leq \mathrm{Var}^{1/2}(\hat{r}_m(u)) \,  \mathrm{Var}^{1/2}\left(  \frac{1}{L} \mathrm{Tr} \left( {\bf Q}_{m,m} \mathcal{D}_{m,L}( {\bf J}_L^{-u}) \mathcal{R}_{m,L}^{1/2} \right)  \right)
    \label{eq:bound_multivar}
\end{multline}
and it is therefore enough to bound these two variances. Regarding $\mathrm{Var}(\hat{r}_m(u))$, we observe that we can write 
\begin{equation}
    \label{eq:expre-rmu}
    \hat{r}_m(u) = \frac{1}{N} {\bf x}_m \mathcal{R}_{m,N+L-1}^{1/2} \left( \begin{array}{c} \I_N \\ 0 \end{array} \right) {\bf J}_N^{-u} \left(\I_N, 0\right) \mathcal{R}_{m,N+L-1}^{1/2} {\bf x}_m^{H}
\end{equation}
from which we deduce immediately that $\mathrm{Var}(\hat{r}_m(u)) \leq \frac{\kappa}{N}$. Regarding the term corresponding to the second variance in (\ref{eq:bound_multivar}), we first introduce the following lemma, proven in Appendix \ref{app:prove-lemma-trace}.
\begin{lemma}
\label{le:sumsquare-derivatives-trace}
Let ${\bf x}_{m,i}$ the $i$th entry of vector ${\bf x}_m$ defined in (\ref{eq:def-xm}) and consider an $ML \times ML$ deterministic matrix ${\bf A}$.  For $z \in \mathbb{C}^+$, we have
\begin{equation}
    \label{eq:sumsquare-derivatives-trace}
    %\sum_{m,i} 
    \sum_{m=1}^{M} \sum_{i=1}^{N+L-1}
    \left| \frac{1}{ML} \mathrm{Tr}  \left( \frac{\partial {\bf Q}}{\partial {\bf x}_{m,i}} {\bf A} \right) \right|^{2} \leq \frac{\kappa}{MN} \, \frac{1+|z|}{(\Imm z)^{3}} \left( 1 + \frac{1}{\Imm z} \right) \, \frac{1}{ML} \mathrm{Tr}({\bf A} {\bf A}^{H})
\end{equation}
for some nice constant $\kappa$. 
\end{lemma}
As a consequence of the above lemma, we have the following result, which follows from a direct application of the Poincar\'e-Nash inequality (see further \cite[Lemma 3.1]{loubaton-mestre-2017}).
\begin{corollary}
    Let $({\bf A}_N)_{N \geq 1}$ denote a sequence of deterministic $ML \times ML$ matrices. Then,
    \begin{equation}
        \label{eq:var_norm_trace}
        \mathrm{Var} \frac{1}{ML} \mathrm{Tr}  \left( \mathbf{A}_N \mathbf{Q} \right) \leq \frac{\kappa}{MN} \, \frac{1+|z|}{(\Imm z)^{3}} \left( 1 + \frac{1}{\Imm z} \right) \, \frac{1}{ML} \mathrm{Tr}({\bf A}_N {\bf A}_N^{H})
    \end{equation}
    for some nice constant $\kappa$.
\end{corollary}
We can apply the above corollary to study the second variance term in (\ref{eq:bound_multivar}) by defining the $ML \times L$ matrix $\mathbf{E}_m$, which is composed of $M$ blocks of dimension $L \times L$, all of which are zero except for the $m$th one, which is equal to $\I_L$. This means that we can express $\mathbf{Q}_{m,m} = \mathbf{E}_m^H \mathbf{Q} \mathbf{E}_m$. 
%we denote by ${\bf E}_m$ the matrix 
% \begin{equation}
%     \label{eq:def-Em}
%     {\bf E}_m = {\bf f}_m \otimes \I_L = ({\bf f}_1^{m}, \ldots, {\bf f}_L^{m})
% \end{equation}
Hence, the use of (\ref{eq:var_norm_trace}) with ${\bf A}_N = {\bf E}_m \mathcal{D}_{m,L}( {\bf J}_L^{-u}) \mathcal{R}_{m,L}^{1/2}  {\bf E}_m^{H}$ leads immediately to 
$$
\mathrm{Var}\left( \frac{1}{L} \mathrm{Tr} \left( {\bf Q}^{\circ}_{m,m} 
\mathcal{D}_{m,L}({\bf J}_L^{-u}) \mathcal{R}_{m,L}^{1/2}  \right) \right) 
\leq \frac{\kappa}{N}  \frac{1+|z|}{(\Imm z)^{3}} \left(1+\frac{1}{\Imm z}\right).
$$
% Moreover, $\hat{r}_m(u)$ is equal to 
% \begin{equation}
%     \label{eq:expre-rmu}
% \hat{r}_m(u) = \frac{1}{N} {\bf x}_m \mathcal{R}_{m,N+L-1}^{1/2} \left( \begin{array}{c} \I_N \\ 0 \end{array} \right) {\bf J}_N^{uH} \left(\I_N, 0\right) \mathcal{R}_{m,N+L-1}^{1/2} {\bf x}_m^{H}
% \end{equation}
% from which we deduce immediately that $E|\hat{r}_m(u) - \mathbb{E}(\hat{r}_m(u))|^{2} \leq \frac{\kappa}{N}$. 
Using these two bounds in (\ref{eq:bound_multivar}) we can conclude that
$$
|\mathbb{E}(\zeta)| \leq \frac{\kappa}{M}   \int_{\mathcal{D}} dx \, dy \,|\bar{\partial}\Phi_k(\phi)(z)| \; \frac{1}{(\Imm z)^{3/2}} \left(1+\frac{1}{\Imm z}\right)^{1/2}.
$$
Noting that $\phi$ is smooth, we see that the above integral is finite by choosing $k \geq 2$, and consequently  (\ref{eq:eval-Ezeta}) is proved.

We finally establish that $|\zeta - \mathbb{E}(\zeta)| \prec \frac{1}{M}$ following the approach in \cite{loubaton-rosuel}. More specifically, we interpret 
$\zeta$ as a function of the  $M(N+L-1)$--dimensional vector ${\bf x} = ({\bf x}_1, \ldots, {\bf x}_M)$ (where we recall that the vectors $({\bf x}_m)_{m=1, \ldots, M}$ are defined by (\ref{eq:def-xm})) and exchange $\zeta$ by a term that is Lipschitz with a relevant Lipschitz constant. For each $\epsilon > 0$, 
we denote by $\mathcal{A}_{N,\epsilon}$ the composite event 
\begin{eqnarray}
    \label{eq:def-event-A}
  \mathcal{A}_{N,\epsilon}  = & \bigcap\limits_{\substack{m=1, \ldots, M\\ u=-(L-1), \ldots, L-1}}  \left\{|\hat{r}_m(u) - \mathbb{E}(\hat{r}_m(u))| < \frac{N^{\epsilon}}{\sqrt{N}} \right\}  \\ \nonumber 
     & \bigcap \left( \bigcap_{m=1, \ldots, M} 
  \left\{ \frac{\|{\bf x}_m\|^{2}}{N+L-1} \leq 2 \right\} \right).
\end{eqnarray}
It is clear that $\left| \frac{\|{\bf x}_m\|^{2}}{N+L-1} - 1 \right| \prec \frac{1}{\sqrt{N}}$ and that 
$|\hat{r}_m(u) - \mathbb{E}(\hat{r}_m(u))| \prec \frac{1}{\sqrt{N}}$. Therefore, 
the family of events $\left(\mathcal{A}_{N,\epsilon}\right)_{N \geq 1}$ holds with exponentially
high probability, i.e. there exist $N_0$ and $\eta > 0$ such that 
$\mathbb{P}\left(\mathcal{A}_{N,\epsilon}^{c}\right) \leq \exp-N^{\eta}$ for each $N \geq N_0$. Our strategy
is to replace $\zeta$ with a certain random variable $\tilde{\zeta}_{\epsilon}$ such that
$\tilde{\zeta}_{\epsilon} = \zeta$ on $\mathcal{A}_{N,\epsilon}$ and which, considered 
as a function of ${\bf x}$, is Lipschitz with constant $\frac{N^{\epsilon}}{M}$. In order to build $\tilde{\zeta}_{\epsilon}$, we consider a smooth function $g(t)$ satisfying $g(t) = t$ if $t \in [-1,1]$, $g(t) = 0$
if $|t| \geq 2$, $g(t) \geq 0$ if $t \geq 0$,  $g(t) \leq 0$ if $t \leq 0$ and $|g(t)| \leq 2 |t|$ for each $t$. We then define $g_{N,\epsilon}(t)$ by 
\begin{equation}
    \label{eq:def-gNepsilon}
    g_{N,\epsilon}(t) = \frac{N^{\epsilon}}{\sqrt{N}} \, g\left( \frac{\sqrt{N}}{N^{\epsilon}} \, t \right)
\end{equation}
It is easy to check that $g_{N,\epsilon}$ verifies: 
\begin{equation}
\label{eq:propgepsilon}
g_{N,\epsilon}(t)  =  t \; \mbox{if} \; |t| \leq \frac{N^{\epsilon}}{\sqrt{N}}, \;  g_{N,\epsilon}(t)  =  0 \; \mbox{if} \; |t| \geq \frac{2 N^{\epsilon}}{\sqrt{N}}, \; |g_{N,\epsilon}(t)| \leq  \frac{2 N^{\epsilon}}{\sqrt{N}} \; \mbox{for each $t$} 
\end{equation}
and
\begin{equation}
\label{eq:propgepsilonprime}    
g_{N,\epsilon}'(t)  = 0 \; \mbox{if} \; |t| \geq \frac{2 N^{\epsilon}}{\sqrt{N}}, \;  |g_{N,\epsilon}'(t)|  \leq \kappa \; \mbox{if} \; |t| \leq \frac{2 N^{\epsilon}}{\sqrt{N}}.
\end{equation}
We also introduce a $\mathcal{C}_1$ function $\tilde{g}(t)$ verifying $\tilde{g}(t) = 1 $
if $t \in [0,2]$, and $\tilde{g}(t) = 0$ if $t$ does not belong to $[-1,3]$. 
We then define $\tilde{\zeta}_{\epsilon}$ by 
\begin{equation}
    \label{eq:def-zetatilde}
 \tilde{\zeta}_{\epsilon}  = \frac{1}{M} \, \int_{\mathcal{D}} dx \, dy \,\bar{\partial}\Phi_k(\phi)(z) f({\bf x},z)
 \end{equation}
 where $ f({\bf x},z)$ is defined by 
 \begin{equation}
 \label{eq:def-f}
%f({\bf x},z) =  \sum_{m=1}^{M}  \tilde{g}\left(\frac{\|{\bf x}_m\|^{2}}{N+L-1}\right) \left( \sum_{u=-(L-1)}^{L-1} g_{N,\epsilon}(\hat{r}_m(u) - \mathbb{E}(\hat{r}_m(u)) \;  \tau(\mathcal{D}_{m,L}(\mathcal{R}_{m,L}^{1/2} {\bf Q}_{m,m}(z))(-u)  \right) 
f({\bf x},z) =  \sum_{m=1}^{M}  \tilde{g}\left(\frac{\|{\bf x}_m\|^{2}}{N+L-1}\right) \left( \sum_{u=-(L-1)}^{L-1} g_{N,\epsilon}(\hat{r}^\circ_m(u)) \;  \tau(\mathcal{D}_{m,L}(\mathcal{R}_{m,L}^{1/2} {\bf Q}_{m,m}(z))(-u)  \right).
\end{equation}
It is clear that $\tilde{\zeta}_{\epsilon} = \zeta$ on the set $\mathcal{A}_{N,\epsilon}$. In the following, we first establish that, considered as a function of ${\bf x}$,  $\tilde{\zeta}_{\epsilon}$ is a Lipschitz function with constant $\frac{N^{\epsilon}}{M}$.
This property will imply that $| \tilde{\zeta}_{\epsilon} - \mathbb{E}( \tilde{\zeta}_{\epsilon})| \prec \frac{N^{\epsilon}}{M}$. Next, we justify that if
$| \tilde{\zeta}_{\epsilon} - \mathbb{E}( \tilde{\zeta}_{\epsilon})| \prec \frac{N^{\epsilon}}{M}$, then $|\zeta - \mathbb{E}(\zeta)| \prec \frac{N^{\epsilon}}{M}$. 
As this property will be true for each $\epsilon > 0$, we will deduce from this 
that  $|\zeta - \mathbb{E}(\zeta)| \prec \frac{1}{M}$ as expected. \\

In order to show that $\tilde{\zeta}_{\epsilon}$ is a Lipschitz function with constant $\frac{N^{\epsilon}}{M}$, we establish that the norm square of the gradient of 
$\tilde{\zeta}_{\epsilon}$ is a $\mathcal{O}(\frac{N^{2\epsilon}}{M^{2}})$ term. 
\begin{lemma}
\label{le:tildezeta-lipschitz}
Under the assumptions of Theorem \ref{th:hatphi-overlinephi}, the inequality 
\begin{equation}
\| \nabla \tilde{\zeta}_{\epsilon} \|^{2} = \sum_{m_0=1}^{M} \sum_{i=1}^{N+L-1} 
\left( \left| \frac{\partial \tilde{\zeta}_{\epsilon}}{\partial {\bf x}_{m_0,i}} \right|^{2} +
\left| \frac{\partial \tilde{\zeta}_{\epsilon}}{\partial {\bf x}_{m_0,i}^{*}} \right|^{2} \right)
\leq \kappa \frac{N^{2\epsilon}}{M^{2}}
\end{equation}
holds true for some nice constant $\kappa >0$.
\end{lemma}
We just evaluate the contribution of the derivatives w.r.t. the variables ${\bf x}_{m_0,i}$
because the derivatives w.r.t.  ${\bf x}_{m_0,i}^{*}$ can be addressed in a similar way. 
It is clear that 
$$
\left| \frac{\partial \tilde{\zeta}_{\epsilon}}{\partial {\bf x}_{m_0,i}} \right|^{2} \leq 
\frac{\kappa}{M^{2}} \int_{\mathcal{D}} dx \, dy \,\left|\bar{\partial}\Phi_k(\phi)(z)\right|^{2} \left| \frac{\partial f}{\partial {\bf x}_{m_0,i}} \right|^{2}
$$
so that Lemma \ref{le:tildezeta-lipschitz} will be established if we choose $k \geq 4$ and prove that 
\begin{equation}
    \label{eq:property-gradient-f}
 \sum_{m_0=1}^{M} \sum_{i=1}^{N}  \left| \frac{\partial f}{\partial {\bf x}_{m_0,i}} \right|^{2}
 \leq N^{2 \epsilon} \, P_1(|z|) P_2\left(\frac{1}{\Imm z}\right)
\end{equation}
where $P_1$ and $P_2$ are nice polynomials, with $\mathrm{deg}(P_2) \leq 4$. 

 We now observe that the derivative ${\partial f}/{\partial {\bf x}_{m_0,i}}$ is the sum of the followings three terms: 
 $$
 T^{1}_{m_0,i} = \tilde{g}'\left(\frac{\|{\bf x}_{m_0}\|^{2}}{N+L-1}\right) \frac{ {\bf x}_{m_0,i}^{\ast}}{N+L-1}  \left( \sum_{u=-(L-1)}^{L-1} g_{N,\epsilon}(\hat{r}^\circ_{m_0}(u)) \;  \tau(\mathcal{D}_{m_0,L}(  \mathcal{R}_{m_0,L}^{1/2} {\bf Q}_{m_0,m_0}))(-u)  \right)   
 $$
 $$
  T^{2}_{m_0,i} = \tilde{g}\left(\frac{\|{\bf x}_{m_0}\|^{2}}{N+L-1}\right)  \left( \sum_{u=-(L-1)}^{L-1} g_{N,\epsilon}'(\hat{r}^\circ_{m_0}(u)) \; \frac{\partial \hat{r}_{m_0}(u)}{\partial {\bf x}_{m_0,i}}  \;  \tau(\mathcal{D}_{m_0,L}( \mathcal{R}_{m_0,L}^{1/2} {\bf Q}_{m_0,m_0} ))(-u)  \right)   
  $$
 $$
  T^{3}_{m_0,i} =  \sum_{m=1}^{M}  \tilde{g}\left(\frac{\|{\bf x}_m\|^{2}}{N+L-1}\right) \left( \sum_{u=-(L-1)}^{L-1} g_{N,\epsilon}(\hat{r}^\circ_m(u)) \;  \tau(\mathcal{D}_{m,L}(  \mathcal{R}_{m_0,L}^{1/2} \frac{\partial {\bf Q}_{m,m}}{\partial {\bf x}_{m_0,i}} ))(-u)  \right)    
$$
We first address the behaviour of $\sum_{m_0,i}|T^{1}_{m_0,i}|^{2}$. For this, we first 
remark that
$$
\left|\tilde{g}'\left(\frac{\|{\bf x}_{m_0}\|^{2}}{N+L-1}\right)\right|^{2} \leq 
\kappa \, \mathbf{1}_{\frac{\|{\bf x}_{m_0}\|^{2}}{N+L-1} \leq 3}.
$$ 
Now, using the inequality $|g_{N,\epsilon}(t)|\leq 2 \frac{N^{\epsilon}}{\sqrt{N}}$ together with the Cauchy-Schwartz inequality, we obtain
\begin{align*}
\left| \sum_{u=-(L-1)}^{L-1} g_{N,\epsilon}(\hat{r}_{m_0}(u) - \mathbb{E}(\hat{r}_{m_0}(u)) \;  \tau( \mathcal{D}_{m_0,L}(\mathcal{R}_{m_0,L}^{1/2} {\bf Q}_{m_0,m_0}))(-u)  \right|^{2} 
\leq & \\ \leq
\kappa \, L \, \frac{N^{2\epsilon}}{N} \sum_{u=-(L-1)}^{L-1} \left| \tau(\mathcal{D}_{m_0,L}(\mathcal{R}_{m_0,L}^{1/2} {\bf Q}_{m_0,m_0}))(u)  \right|^{2}.
\end{align*}
Now, the inequality in (\ref{eq:parsevalintro}) together with (\ref{eq:upperbound-mathcalDmathcalDH}) imply that, for each $L \times L$ matrix 
${\bf A}$, we have
\begin{equation}
    \label{eq:parseval}
    \sum_{u=-(L-1)}^{L-1} \left| \tau(\mathcal{D}_{m_0,L}({\bf A}))(u) \right|^{2} \leq \kappa \, \frac{1}{L} \mathrm{Tr}({\bf A} {\bf A}^{H})
\end{equation}
for some nice constant $\kappa >0$.
Using this for ${\bf A} = \mathcal{R}_{m_0,L}^{1/2} {\bf Q}_{m_0,m_0}$ leads to the conclusion that 
\begin{equation}
    \label{eq:use-of-parseval-T1}
\sum_{u=-(L-1)}^{L-1} \left| \tau(\mathcal{D}_{m_0,L}(\mathcal{R}_{m_0,L}^{1/2} {\bf Q}_{m_0,m_0}))(u)  \right|^{2} \leq \kappa \frac{1}{L} \mathrm{Tr}\left({\bf Q}_{m_0,m_0} {\bf Q}_{m_0,m_0}^{H}\right).
\end{equation}
Therefore, we obtain that 
$$
|T^{1}_{m_0,i}|^{2} \leq \frac{\kappa}{(N+L-1)} \, \frac{L N^{2\epsilon}}{N} \;  \frac{|{\bf x}_{m_0,i}|^{2}}{N+L-1} \, \mathbf{1}_{\frac{\|{\bf x}_{m_0}\|^{2}}{N+L-1} \leq 3} \; \frac{1}{L} \mathrm{Tr}\left(  {\bf Q}_{m_0,m_0} {\bf Q}_{m_0,m_0}^{H} \right)
$$
from where we deduce that 
$$
\sum_{i=1}^{N+L-1} |T^{1}_{m_0,i}|^{2} \leq  \frac{\kappa}{(N+L-1)} \, \frac{L N^{2\epsilon}}{N} 
\frac{1}{L} \mathrm{Tr}\left(  {\bf Q}_{m_0,m_0} {\bf Q}_{m_0,m_0}^{H} \right)
$$
and 
$$
\sum_{m_0=1}^{M} \sum_{i=1}^{N+L-1} |T^{1}_{m_0,i}|^{2} \leq \kappa \,  \frac{N^{2\epsilon}}{N}
\frac{1}{ML} \mathrm{Tr}\left( {\bf Q}{\bf Q}^{H} \right) \leq \kappa \, \frac{N^{2\epsilon}}{N} \, 
\frac{1}{(\Imm z)^{2}}.
$$
We now consider $\sum_{m_0,i}|T^{2}_{m_0,i}|^{2}$. We first remark that 
$$
\left|\tilde{g}\left(\frac{\|{\bf x}_{m_0}\|^{2}}{N+L-1}\right)\right|^{2} \leq 
\kappa \, \mathbf{1}_{\frac{\|{\bf x}_{m_0}\|^{2}}{N+L-1} \leq 3}
$$
and that 
$ \left|g_{N,\epsilon}'(\hat{r}_{m_0}(u) - \mathbb{E}(\hat{r}_{m_0}(u))\right|^{2} \leq \kappa$. Therefore, a direct use of the Cauchy-Schwartz inequality and (\ref{eq:use-of-parseval-T1}) leads us to 
$$
\left|T^{2}_{m_0,i}\right|^{2} \leq \kappa \,  \mathbf{1}_{\frac{\|{\bf x}_{m_0}\|^{2}}{N+L-1} \leq 3} \, \left(\sum_{u} \left| \frac{\partial \hat{r}_{m_0}(u)}{\partial {\bf x}_{m_0,i}} \right|^{2} \right) \; 
\frac{1}{L} \mathrm{Tr}\left(  {\bf Q}_{m_0,m_0} {\bf Q}_{m_0,m_0}^{H} \right).
$$
Using (\ref{eq:expre-rmu}), we obtain that 
$$
\sum_{i=1}^{N+L-1} \left| \frac{\partial \hat{r}_{m_0}(u)}{\partial {\bf x}_{m_0,i}} \right|^{2}  \leq \frac{\kappa}{N+L-1} \frac{\|{\bf x}_{m_0}\|^{2}}{N+L-1}
$$
and eventually that 
$$
\sum_{m_0,i} \left|T^{2}_{m_0,i}\right|^{2} \leq \kappa \, \frac{1}{ML} \mathrm{Tr}\left( 
{\bf Q} {\bf Q}^{H} \right) \leq \frac{\kappa}{(\Imm z)^{2}}.
$$
We finally study  $\sum_{m_0,i}|T^{3}_{m_0,i}|^{2}$. Using once more the fact that
$\left| g_{N,\epsilon}(t) \right|^{2} \leq 4 \frac{N^{2\epsilon}}{N}$ and that
$\left|\tilde{g}\left(\frac{\|{\bf x}_{m_0}\|^{2}}{N+L-1}\right)\right|^{2} \leq 
\kappa$, Jensen inequality leads immediately that 
\begin{equation}
    \label{eq:eval-intermediate-T3}
\sum_{m_0,i} |T^{3}_{m_0,i}|^{2} \leq \frac{\kappa N^{2\epsilon}}{N} \, ML \, 
\sum_{m=1}^{M} \sum_{u=-(L-1)}^{L-1} \sum_{m_0,i} \left|\tau\left(\mathcal{D}_{m,L}\left(\mathcal{R}_{m,L}^{1/2} \frac{\partial {\bf Q}_{m,m}}{\partial {\bf x}_{m_0,i}}\right)\right)(u)  \right|^{2} 
\end{equation}
In order to evaluate  $\sum_{m_0,i}|T^{3}_{m_0,i}|^{2}$, we use (\ref{eq:swapDmL}) and observe that we can write
$$
\tau\left(\mathcal{D}_{m,L}\left( \mathcal{R}_{m,L}^{1/2} \frac{\partial {\bf Q}_{m,m}}{\partial {\bf x}_{m_0,i}}\right)\right)(u) = \frac{1}{L} \mathrm{Tr} \left( \frac{\partial {\bf Q}_{m,m}}{\partial {\bf x}_{m_0,i}} \mathcal{D}_{m,L}({\bf J}_L^{u}) \mathcal{R}_{m,L}^{1/2} \right) = \frac{1}{L}  \mathrm{Tr} \left( \frac{\partial {\bf Q}}{\partial {\bf x}_{m_0,i}} {\bf A} \right)
$$
where ${\bf A}$ is the $ML \times ML$ matrix defined by 
${\bf A} = {\bf E}_m \mathcal{D}_{m,L}({\bf J}_L^{u}) \mathcal{R}_{m,L}^{1/2} {\bf E}_m^{H}$. Lemma \ref{le:sumsquare-derivatives-trace} leads immediately to 
$$
\sum_{m_0,i} \left|\tau\left(\mathcal{D}_{m,L}\left( \mathcal{R}_{m,L}^{1/2} \frac{\partial {\bf Q}_{m,m}}{\partial {\bf x}_{m_0,i}}\right) \right)(u)\right|^{2} \leq \frac{\kappa}{N} \, \frac{(1+|z|)}{(\Imm z)^{3}} \left( 1 + \frac{1}{\Imm z} \right)
$$
Plugging this into the evaluation (\ref{eq:eval-intermediate-T3}) eventually 
leads to 
$$
\sum_{m_0,i} |T^{3}_{m_0,i}|^{2} \leq \kappa \, N^{2 \epsilon} \, \frac{(1+|z|)}{(\Imm z)^{3}} \left( 1 + \frac{1}{\Imm z} \right) 
$$
This establishes (\ref{eq:property-gradient-f}) and Lemma \ref{le:tildezeta-lipschitz}. 
Therefore, we have shown that, considered as a function of ${\bf x}$, $\tilde{\zeta}_{\epsilon}$ is Lipschitz with constant $\kappa \frac{N^{\epsilon}}{M}$. 
The Gaussian concentration inequality thus implies that $\left| \tilde{\zeta}_{\epsilon} - \mathbb{E}(\tilde{\zeta}_{\epsilon}) \right| \prec \frac{N^{\epsilon}}{M}$. 

It remains to justify that $\left| \tilde{\zeta}_{\epsilon} - \mathbb{E}(\tilde{\zeta}_{\epsilon}) \right| \prec \frac{N^{\epsilon}}{M}$ implies that 
$\left| \zeta - \mathbb{E}(\zeta) \right| \prec \frac{N^{\epsilon}}{M}$. For this, 
it is sufficient to follow the proof of \cite[Lemma 4.1, p.41]{loubaton-rosuel}.

\section{Evaluation of the modified statistic $\overline{\phi}_N$}
\label{sec:evalbarphi}
\subsection{Reduction to the study of the expectation of $\overline{\phi}_N$}
\label{sec:reduction2expect}
In this short section, we show that we can reduce the study of the statistic $\overline{\phi}_N$ to the study of the expectation $\mathbb{E}(\overline{\phi}_N)$ up to an error that is dominated by $\frac{1}{M\sqrt{L}}$. We express the result in terms of the following proposition.

\begin{proposition}
    \label{prop:barphi-Ebarphi}
    Let Assumptions \ref{assum:statistics}-\ref{ass:bounds-spectral-densities} and \ref{as:norm-r-omega} hold true and let $\phi$ have the same properties as in the statement of Theorem \ref{th:hatphi-overlinephi}. Then, 
    \begin{equation}
        \label{eq:barphi-Ebarphi}
        \left| \overline{\phi}_N - \mathbb{E}\left( \overline{\phi}_N \right) \right| \prec \frac{1}{M\sqrt{L}}.
    \end{equation}
\end{proposition}
We devote the rest of the section to the proof of this result. We first reason that, without loss of generality, we can replace the function $\phi$ by a smooth function that is supported by $[-\delta,\alpha]$ for some $\alpha > \alpha_0$ (see the statement of Theorem \ref{th:hatphi-overlinephi} for a definition of $\alpha_0$, $\delta$). The justification is the same that we used at the initial steps of the proof of Theorem \ref{th:hatphi-overlinephi} and is therefore omitted. We therefore focus on this class functions for the rest of the proof. 

In order to show Proposition \ref{prop:barphi-Ebarphi}, consider again the Helffer-Sj\"ostrand representation of $\overline{\phi}_N$ in (\ref{eq:expre-hatphi-overlinephi-hs}), which allows us to write
\begin{equation}
    \overline{\phi}^\circ_N = \overline{\phi}_N - \mathbb{E}\overline{\phi}_N   = \frac{1}{\pi} \mathrm{Re} \int_{\mathcal{D}} dx \, dy \,\bar{\partial}\Phi_k(\phi)(z)\frac{1}{ML}\mathrm{Tr} \left(\mathbf{Q}_N(z) - \mathbb{E}\mathbf{Q}_N(z)\right).
\end{equation}
Here again, the idea is to consider $\overline{\phi}^\circ_N$ as a function of the $\mathcal{N}_\mathbb{C}(0,\I_{M(N+L-1)})$-distributed random vector ${\bf x}$ defined in (\ref{eq:def-x}). We will show that this function is Lipschitz with constant of order $\mathcal{O}((M\sqrt{L})^{-1})$, so that the result follows from conventional concentration results of Gaussian functionals in (\ref{eq:concentrationGaussian}) (see also \cite[Theorem 2.1.12]{tao-book}).

Indeed, let $\nabla \overline{\phi}^\circ_N$ denote the gradient of $\overline{\phi}^\circ_N$ with respect to ${\bf x}$. Then, we can obviously write $$
    \left\| \nabla \overline{\phi}^\circ_N \right\|^2 = \sum_{m_0=1}^{M} \sum_{i=1}^{N+L-1} \left| \frac{\partial \overline{\phi}_N}{\partial \mathbf{x}_{m_0,i}} \right|^2 + \left| \frac{\partial \overline{\phi}_N}{\partial \mathbf{x}^\ast_{m_0,i}} \right|^2
$$
where we recall that $\mathbf{x}_{m_0,i}$ denotes the $i$th entry of $\mathbf{x}_m$, the $m$th block of ${\bf x}$, with dimension $N+L-1$. A direct use of Lemma \ref{le:sumsquare-derivatives-trace} shows that 
$$
    \sum_{m_0=1}^{M} \sum_{i=1}^{N+L-1} \left| \frac{\partial \overline{\phi}_N}{\partial \mathbf{x}_{m_0,i}} \right|^2 \leq
    \frac{\kappa}{MN} \int_\mathcal{D} dx \, dy \, \left|\overline{\partial}\Phi_k(\phi)(z) \right|^2  
    \frac{1+|z|}{(\Imm z)^3} \left( 1+ \frac{1}{\Imm z} \right)
$$
for some nice constant $\kappa$, where the integral on the right hand side is finite if we select $k \geq 4$, which is always possible because $\phi$ is smooth. This concludes the proof of Proposition \ref{prop:barphi-Ebarphi}.

\subsection{Weak convergence of $\overline{\mu}_N(\lambda)$ and evaluation of $\mathbb{E}(\overline{\phi}_N)$}
\label{sec:deterministic-equivalent}

The aim of this section is twofold. On the one hand, we will show that $ \overline{\mu}_N(\lambda) - \mu_N(\lambda) $ converges weakly almost surely to zero. On the other hand, we will evaluate the convergence of the 
$\mathbb{E}(\overline{\phi}_N)$ by establishing that, when $\frac{L^{3/2}}{MN} \rightarrow 0$ (equivalently $\beta < 4/5$), we have
\begin{equation}
    \label{eq:convergEphi-intphidmu}
    \left| \mathbb{E}\left(\overline{\phi}_N\right) - \int \phi(\lambda) d\mu_N(\lambda) \right|  \leq \kappa \frac{L}{MN}
\end{equation}
for some nice constant $\kappa >0$.

% where we recall that $\mu_N(\lambda)$ is a probability measure with Stieltjes transform $\frac{1}{ML}\mathrm{Tr} \mathbf{T}_N(z)$ introduced in Theorem \ref{thm:main_result}. 
% In a second step, by  we establish that the error term is of order $\mathcal{O}\left(\frac{L}{MN}\right)$.

%  Contrary to what happened in previous sections, we are not able to use here Helffer-Sj\"ostrand integral representation of (\ref{eq:convergEphi-intphidmu}) because we are not able to guarantee a proper inequality of $\frac{1}{ML}\mathrm{Tr} (\mathbb{E}\mathbf{Q}_N(z)- \mathbf{T}_N(z))$ on the whole upper complex plane. 
 
We will address the problem by studying $\frac{1}{ML}\mathrm{Tr}\left( \mathbb{E}\mathbf{Q}_N(z) - \mathbf{T}_N(z) \right)$. % invoking \cite[Lemma 5.5.5]{and-gui-zei-2010}. 
%$$
%     \mathbb{E}\left(\overline{\phi}_N\right) - \int \phi(\lambda) d\mu_N(\lambda) 
%     = \frac{1}{\pi}\mathrm{Re}\int_\mathcal{D} dx \, dy\, \overline{\partial} \Phi(\phi)(z) \frac{1}{ML}\mathrm{Tr}\left( \mathbb{E}\mathbf{Q}_N(z) - \mathbf{T}_N(z) \right).
% $$
%Thanks to this result, we can attack the problem by directly studying the quantity $\frac{1}{ML}\mathrm{Tr}\left( \mathbb{E}\mathbf{Q}_N(z) - \mathbf{T}_N(z) \right)$. 
We will study this term by conveniently adapting the tools in \cite{loubaton-mestre-2017} to the present context. First, we study the master equations that define the matrix
function $\mathbf{T}_{N}(z)$ in the statement of Theorem \ref{thm:main_result}
and establish existence and unicity of the solution using again the tools
developed in \cite[Section 5]{loubaton-mestre-2017}. We then establish that, considering a sequence of $ML \times ML$ deterministic matrices $\mathbf{A}_N$ of uniformly bounded spectral norm, we have
\[
\left|\frac{1}{ML}\mathrm{Tr}\left[  \left(  \mathbb{E}\mathbf{Q}_{N} (z)-\mathbf{T}_{N}(z)\right)  \mathbf{A}_{N}\right] \right| \leq  \kappa \frac{L}{MN}
\]
% \[
% \frac{1}{ML}\mathrm{Tr}\left[  \left(  \mathbb{E}\mathbf{Q}_{N}
% (z)-\mathbf{T}_{N}(z)\right)  \mathbf{A}_{N}\right]  \rightarrow0
% \]
for $z$ in a certain subset of $\mathbb{C}^+$, assuming that $\beta < 4/5$. Even if the subset where the above inequality 
holds is not the whole semiplane $\mathbb{C}^+$, it will be sufficient to deduce (\ref{eq:convergEphi-intphidmu}) by conveniently adapting
the arguments in \cite[Lemma 5.5.5]{and-gui-zei-2010}.
%Finally, we characterize the rate of convergence to zero of the above quantity and use it to establish the convergence rate of (\ref{eq:convergEphi-intphidmu}). 

% showing that indeed $\mathbf{T}_{N}(z)$ is the asymptotic deterministic
% equivalent of $\mathbf{Q}_{N}(z)$. Taking $\mathbf{A}_{N}=\mathbf{I}_{ML}$ and
% using that $\frac{1}{ML}\mathrm{Tr}\left[  \mathbf{Q}_{N}(z)-\mathbb{E}
% (\mathbf{Q}_{N}(z))\right]  \rightarrow0$ almost surely, we obtain that
% $q_{N}(z)-\frac{1}{ML}\mathrm{Tr}\left(  \mathbf{T}_{N}(z)\right)
% \rightarrow0$ almost surely. This, in turn, justifies
% (\ref{eq:convergence-q-t}) as well as (\ref{eq:convergence-rate-biais-lss}), completing the proof of Theorem \ref{thm:main_result}. 

First of all, we consider here the two asymptotic equivalents $\mathbf{T}(z)$,$\widetilde
{\mathbf{T}}(z)$, as the solutions to the equations (\ref{eq:canonical-T})-(\ref{eq:canonical-tildeT}).

\begin{proposition}
\label{proposition:existence_unicity} There exists a unique pair of functions
$(\mathbf{T}(z),\widetilde{\mathbf{T}}(z))\in\mathcal{S}_{ML}(\mathbb{R}
^{+})\times\mathcal{S}_{N}(\mathbb{R}^{+})$ that satisfy (\ref{eq:canonical-T})--(\ref{eq:canonical-tildeT}) for each $z\in\mathbb{C}^{+}$.
Moreover, one can find two nice constants $\eta$ and $\tilde{\eta}$ such that
\begin{align}
\mathbf{T}(z)\mathbf{T}^{H}(z)  &  \geq\frac{(\Imm z)^2}{16(\eta
^{2}+|z|^{2})^{2}}\I_{ML}\label{eq:lower-bound-TT*}\\
\widetilde{\mathbf{T}}(z)\widetilde{\mathbf{T}}^{H}(z)  &  \geq\frac
{(\Imm{z})^{2}}{16(\tilde{\eta}^{2}+|z|^{2})^{2}}\I_{N}.
\label{eq:lower-bound-tildeTtildeT*}
\end{align}

\end{proposition}
The proof follows the steps as the proof of Proposition 5.1 in
\cite{loubaton-mestre-2017}. To prove existence, we consider
the composition of (\ref{eq:canonical-T})--(\ref{eq:canonical-tildeT}) as a
mapping in the set of $ML\times ML$ block diagonal matrices. Using Proposition
\ref{prop:Upsilon-tildeUpsilon} one can establish that iterating these two
equations one can create a sequence of $ML\times ML$ diagonal block matrices
with blocks belonging to the class $\mathcal{S}_{L}\left(  \mathbb{R}
^{\mathbb{+}}\right)  $ that has a limit in this set. Then, in a second step,
it can be shown that this limit is a solution to the canonical equation. For
more details, the reader may refer to the proof of Proposition 5.1 in
\cite{loubaton-mestre-2017}.

% In order to prove unicity, we can focus on the case $z\in\mathbb{C}^{+}$
% (elements of $\mathcal{S}_{ML}(\mathbb{R}^{+})$ and $\mathcal{S}
% _{N}(\mathbb{R}^{+})$ are uniquely defined by their values on $\mathbb{C}^{+}
% $). The idea, again,
The proof of unicity follows the same path that was established in
\cite{loubaton-mestre-2017}. More specifically, assume that $\mathbf{T}
(z),\widetilde{\mathbf{T}}(z)$ and $\mathbf{S}(z),\widetilde{\mathbf{S}}(z)$
are matrices solutions of the system (\ref{eq:canonical-T},
\ref{eq:canonical-tildeT}) of equations at point $z$, and assume that
$\mathbf{T}(z)$ and $\mathbf{S}(z)$ have positive imaginary parts. Let
$\mathbf{T}_{\mathcal{B}}(z)=\mathcal{B}_{L}^{-1/2}\mathbf{T}(z)\mathcal{B}
_{L}^{-1/2}$ and $\mathbf{S}_{\mathcal{B}}(z)=\mathcal{B}_{L}^{-1/2}
\mathbf{S}(z)\mathcal{B}_{L}^{-1/2}$. It is easily seen that
\begin{equation}
\mathbf{T}_{\mathcal{B}}(z)-\mathbf{S}_{\mathcal{B}}(z)=\Phi_{\mathcal{B}
,0}\left(  \mathbf{T}_{\mathcal{B}}(z)-\mathbf{S}_{\mathcal{B}}(z)\right)
\label{eq:T-S}
\end{equation}
where we have defined the operator $\Phi_{\mathcal{B},0}\left(  \mathbf{X}
\right)  $ as
\begin{equation}
\Phi_{\mathcal{B},0}\left(  \mathbf{X}\right)  =z^{2}c_{N}\mathbf{S}
_{\mathcal{B}}(z)\Psi\left(  \widetilde{\mathbf{S}}^{T}(z)\overline{\Psi
}\left(  \mathbf{X}\right)  \widetilde{\mathbf{T}}^{T}(z)\right)
\mathbf{T}_{\mathcal{B}}(z) \label{eq:def-Phi0}
\end{equation}
where $\mathbf{X}$ is an $ML\times ML$ matrix. This operator is the analog of
$\Phi_{0}\left(  \mathbf{X}\right)  $ in \cite{loubaton-mestre-2017}
translated to our current matrix model. Operating like in
\cite{loubaton-mestre-2017} we write $\Phi_{\mathcal{B},0}^{(1)}\left(
\mathbf{X}\right)  =\Phi_{\mathcal{B},0}\left(  \mathbf{X}\right)  $ and
recursively define $\Phi_{\mathcal{B},0}^{(n+1)}\left(  \mathbf{X}\right)
=\Phi_{\mathcal{B},0}(\Phi_{\mathcal{B},0}^{(n)}(\mathbf{X}))$ for $n\geq1$.
By (\ref{eq:T-S}), unicity is proven if we are able to show that
$\lim_{n\rightarrow\infty}\Phi_{\mathcal{B},0}^{(n)}(\mathbf{X})=\mathbf{0}$
for every $ML\times ML$ matrix $\mathbf{X}$.\ Now, using Proposition
\ref{prop:Phi-Phi_S-Phi_TH} it is easily established that, for any two
$L$-dimensional column vectors $\mathbf{a}$, $\mathbf{b}$, we can write
\begin{equation}
\left\vert \mathbf{a}^{H}\left(  \Phi_{\mathcal{B},0}^{(n)}\left(
\mathbf{X}\right)  \right)  _{m,m}\mathbf{b}\right\vert \leq\left[
\mathbf{a}^{H}\left(  \Phi_{\mathbf{S}_{\mathcal{B}}}^{(n)}\left(
\mathbf{X}\mathbf{X}^{H}\right)  \right)  _{m,m}\mathbf{a}\right]
^{1/2}\left[  \mathbf{b}^{H}\left(  \Phi_{\mathbf{T}_{\mathcal{B}}^{H}}
^{(n)}\left(  \I_{ML}\right)  \right)  _{m,m}\mathbf{b}\right]  ^{1/2}
\label{eq:inequality-fundamental-phi0B}
\end{equation}
where $\Phi_{\mathbf{T}_{\mathcal{B}}^{H}}$ and $\Phi_{\mathbf{S}
_{\mathcal{B}}}$ are the positive operators defined by
\begin{align}
\Phi_{\mathbf{T}_{\mathcal{B}}^{H}}\left(  \mathbf{X}\right)   &  =\left\vert
z\right\vert ^{2}c_{N}\mathbf{T}_{\mathcal{B}}^{H}(z)\Psi\left(
\widetilde{\mathbf{T}}^{\ast}(z)\overline{\Psi}\left(  \mathbf{X}\right)
\widetilde{\mathbf{T}}^{T}(z)\right)  \mathbf{T}_{\mathcal{B}}
(z)\label{eq:defPhi_Tb}\\
\Phi_{\mathbf{S}_{\mathcal{B}}}\left(  \mathbf{X}\right)   &  =\left\vert
z\right\vert ^{2}c_{N}\mathbf{S}_{\mathcal{B}}(z)\Psi\left(  \widetilde
{\mathbf{S}}^{T}(z)\overline{\Psi}\left(  \mathbf{X}\right)  \widetilde
{\mathbf{S}}^{\ast}(z)\right)  \mathbf{S}_{\mathcal{B}}^{H}(z).
\label{eq:defPhi_Sb}
\end{align}
Thus, by Proposition \ref{prop:Phi-Phi_S-Phi_TH}, $\lim_{n\rightarrow\infty
}\Phi_{\mathcal{B},0}^{(n)}(\mathbf{X})=\mathbf{0}$ will follow directly if we
are able prove that there exist two positive definite matrices $\mathbf{Y}
_{1}$ and $\mathbf{Y}_{2}$ such that $\Phi_{\mathbf{T}_{\mathcal{B}}^{H}
}^{(n)}\left(  \mathbf{Y}_{1}\right)  $ and $\Phi_{\mathbf{S}_{\mathcal{B}}
}^{(n)}\left(  \mathbf{Y}_{2}\right)  $ converge towards $\mathbf{0}$.

\begin{lemma}
\label{le:convergence-series-PhinTB} Let $\mathbf{T}(z),\widetilde{\mathbf{T}
}(z)$ be a solution to the canonical equation (\ref{eq:canonical-T},
\ref{eq:canonical-tildeT}) at point $z \in\mathbb{C}^{+}$ satisfying
$\mathrm{Im}(\mathbf{T}(z)) \geq0$, and define $\mathbf{T}_{\mathcal{B}
}(z)=\mathcal{B}_{L}^{-1/2}\mathbf{T} (z)\mathcal{B}_{L}^{-1/2}$. Let
$\mathbf{X}$ be a positive semi definite matrix. Then, it holds that
\begin{equation}
\Phi_{\mathbf{T}_{\mathcal{B}}}^{(n)}\left(  \mathbf{X}\right)  \rightarrow
\mathbf{0} \label{eq:convergence-Phi_Tn-zero}
\end{equation}
and
\begin{equation}
\Phi_{\mathbf{T}_{\mathcal{B}}^{H}}^{(n)}\left(  \mathbf{X}\right)
\rightarrow\mathbf{0} \label{eq:convergence-Phi_THn-zero}
\end{equation}
as $n\rightarrow\infty$. Moreover, the series $\sum_{n=0}^{+\infty}
\Phi_{\mathbf{T}_{\mathcal{B}}}^{(n)}\left(  \mathbf{X}\right)  $ and
$\sum_{n=0}^{+\infty}\Phi_{\mathbf{T}_{\mathcal{B}}^{H}}^{(n)}\left(
\mathbf{X}\right)  $ converge. 
% Finally, consider $\alpha(z) > 0$ such that
% $\mathbf{T}(z) \mathbf{T}^{H}(z) \geq\alpha(z) \, \mathbf{I}_{ML}$ (as
% $\mathbf{T}(z)$ is full rank, $\mathbf{T}(z) \mathbf{T}^{H}(z)$ is positive
% definite). Then, we have
% \begin{align}
% \label{eq:bound-sum-PhiTBn-I}\sum_{n=0}^{+\infty} \Phi_{\mathbf{T}
% _{\mathcal{B}}}^{(n)}\left(  \I_{ML} \right)   &  \leq\kappa\frac{1}
% {\alpha(z)} \frac{\mathrm{Im}\mathbf{T}_{\mathcal{B}}}{\mathrm{Im}z}\\
% \sum_{n=0}^{+\infty} \Phi_{\mathbf{T}_{\mathcal{B}}^{H}}^{(n)}\left(  \I_{ML}
% \right)   &  \leq\kappa\frac{1}{\alpha(z)} \frac{\mathrm{Im}\mathbf{T}
% _{\mathcal{B}}}{\mathrm{Im}z}\label{eq:bound-sum-PhiTBHn-I}
% \end{align}
% for some nice constant $\kappa$.
\end{lemma}

\begin{proof}
The proof of the lemma follows the same steps as the proof of Lemma 5.4 in
\cite{loubaton-mestre-2017} and is therefore omitted. 
% The first part of the lemma follows the same steps as the proof of Lemma 5.4 in
% \cite{loubaton-mestre-2017}. Indeed, using the same reasoning we can prove the convergence of (\ref{eq:convergence-Phi_Tn-zero})--(\ref{eq:convergence-Phi_THn-zero}) together with the corresponding series.
% Furthermore, following the steps in the proof of \cite[Lemma 5.4]{loubaton-mestre-2017} we can establish the identity 
% \begin{equation}
% \label{eq:expre-iterated-ImTB-infini}
% \frac{\mathrm{Im}\mathbf{T}_{\mathcal{B}}}{\mathrm{Im}z}
% =\sum_{n=0}^{+\infty}\Phi_{\mathbf{T}_{\mathcal{B}}}^{(n)}\left(  \mathcal{B}
% _{L}^{-1/2}\mathbf{T}(z)\mathbf{T}^{H}(z)\mathcal{B}_{L}^{-1/2}\right).
% \end{equation}
% Noting that $ \mathcal{B}_{L}^{-1/2} \mathbf{T}(z) \mathbf{T}^{H}(z) \mathcal{B}_{L}^{-1/2} \geq \kappa \, \alpha(z) \,  \I_{ML}$ (because the spectral densities are uniformly bounded away from zero), 
% (\ref{eq:expre-iterated-ImTB-infini}) implies that 
% $$
% \sum_{n=0}^{+\infty}\Phi_{\mathbf{T}_{\mathcal{B}}}^{(n)}\left(  \I_{ML}\right)  \leq \kappa \, \frac{1}{\alpha(z)}  \frac{\mathrm{Im}\mathbf{T}_{\mathcal{B}}}{\mathrm{Im}z}
% $$
% which directly shows (\ref{eq:bound-sum-PhiTBn-I}). The inequality in (\ref{eq:bound-sum-PhiTBHn-I}) is established in a similar way. 
\end{proof}

As a consequence of all the above, Theorem \ref{thm:main_result} will be a
direct implication of the following result.

\begin{proposition}
\label{prop:convergence-Tr(Q-T)} We consider a sequence $(\mathbf{A}
_{N})_{N\geq1}$ of $ML\times ML$ deterministic matrices such that $\sup_{N}\Vert\mathbf{A}_{N}\Vert\leq a$ for some nice constant $a$. Then, for each
$z\in\mathbb{C}^{+}$, we have
\begin{equation}
\frac{1}{ML}\mathrm{Tr}\left(  \mathbf{A}_{N}(\mathbf{Q}_{N}(z)-
\mathbf{T}_{N}(z))\right)  \rightarrow0\label{eq:Tr-Q-R}
\end{equation}
almost surely. For any bounded continuous function $\phi$ we have 
\begin{equation}
    \label{eq:weakconv}
    \left\vert \frac{1}{ML}\mathrm{Tr}\left(  \phi(\overline{\mathcal{R}
}_{\mathrm{corr},L})\right)  -\int\phi(\lambda)d\mu_{N}(\lambda)\right\vert
\rightarrow 0
\end{equation}
almost surely. 

Assume, in addition, that $\frac{L^{3/2}}{MN}\rightarrow0$, i.e. that $\beta < \frac{4}{5}$. In this case, we have
% \begin{equation}
% \frac{1}{ML}\mathrm{Tr}\left(  \mathbf{A}_{N}(\mathbb{E}\mathbf{Q}
% _{N}(z)-\mathbf{T}_{N}(z))\right)  \rightarrow0.\label{eq:Tr-E(Q)-T}
% \end{equation}
% If , it holds that
\begin{equation}
\left|\frac{1}{ML}\mathrm{Tr}\left(  \mathbf{A}_{N}(\mathbb{E}\mathbf{Q}
_{N}(z)-\mathbf{T}_{N}(z))\right)\right|  \leq C(z)\frac{L}{MN}
\label{eq:-speed-Tr-E(Q)-T}
\end{equation}
when $z$ belongs to a set $E_{N}$ defined as
\[
E_{N}=\left\{z\in\mathbb{C}^{+},\frac{L^{3/2}}{MN}
P_{1}(|z|)P_{2}(1/\Imm {z})<1\right\}
\]
and where $P_{1}$ and $P_{2}$ are two nice polynomials. Finally, for each
compactly supported smooth function $\phi$, we have
\begin{equation}
\left\vert \frac{1}{ML}\mathbb{E}\mathrm{Tr}\left(  \phi(\overline{\mathcal{R}
}_{\mathrm{corr},L})\right)  -\int\phi(\lambda)d\mu_{N}(\lambda)\right\vert
\leq \kappa \frac{L}{MN}
\label{eq:convergence-rate-biais-lss-bis}
\end{equation}
for some nice constant $\kappa >0$.
\end{proposition}

\begin{proof}
The proof of (\ref{eq:Tr-Q-R}) can be established by essentially following the approach in \cite{loubaton-mestre-2017}. The main idea is to consider the resolvent in (\ref{eq:def_resolvent_barRcorr}) together with the co-resolvent, defined as 
\begin{equation*}
    \label{eq:def_coresolvent_barRcorr}
    \widetilde{\mathbf{Q}}_N(z) = \left( \mathbf{W}_N^H \mathcal{B}_L^{-1} \mathbf{W}_N - z \I_N \right)^{-1}.
\end{equation*}
Using a trivial modification of \cite[Lemma 3.1]{loubaton-mestre-2017} one can reduce the problem to the study of the expectations $\mathbb{E}{\mathbf{Q}}_N(z)$ and $\mathbb{E}\widetilde{\mathbf{Q}}_N(z)$. We can then introduce two matrix-valued functions $\mathbf{R}_N(z)$ and $\widetilde{\mathbf{R}}_N(z)$ defined as 
\begin{eqnarray}
\widetilde{\mathbf{R}}_N(z) & = & -\frac{1}{z}\left(  \mathbf{I}_{N}+c_{N}
\overline{\Psi}^{T}\left(  \mathcal{B}_{L}^{-1/2}\mathbb{E}\mathbf{Q}_N
(z)\mathcal{B}_{L}^{-1/2}\right)  \right)  ^{-1}\label{eq:def_Rztilde} \\
\mathbf{R}_N(z) & = & -\frac{1}{z}\left(  \mathbf{I}_{ML}+\mathcal{B}_{L}^{-1/2}
\Psi\left(  \widetilde{\mathbf{R}}_N^{T}(z)\right)  \mathcal{B}_{L}
^{-1/2}\right)  ^{-1}\label{eq:def_Rz}
\end{eqnarray}
which are the analogous of the same quantities in \cite[Section 4]{loubaton-mestre-2017}. In particular, one can establish that Lemma 4.1 and Proposition 4.3 in \cite{loubaton-mestre-2017} also hold true with these new definitions, so that
\[
\left|\frac{1}{ML} \mathrm{Tr} \left[ \mathbf{A}_N (\mathbb{E}\mathbf{Q}_N(z) - \mathbf{R}_N(z)) \right] \right| \leq C(z) \frac{L}{MN}
\]
% together with 
% \[
% \| \overline{\Psi} (\mathbb{E}\mathbf{Q}_N(z) - \mathbf{R}_N(z)) \| \leq C(z) \frac{L^{3/2}}{MN}
% \]
for all $z \in \mathbb{C}^+$. In order to see this, we need to make explicit use of Assumption \ref{as:norm-vector-r}. 
At this point, in order to show (\ref{eq:Tr-Q-R}) and (\ref{eq:-speed-Tr-E(Q)-T}) one only needs to evaluate the quantity $\mathbf{R}_N(z) - \mathbf{T}_N(z)$ using the approach in Section 6 of \cite{loubaton-mestre-2017}, which essentially holds verbatim after replacing the operators $\Phi_1(\mathbf{X})$ and $\Phi_1^t(\mathbf{X})$ with
\begin{align}
\Phi_{\mathcal{B},1}\left(  \mathbf{X}\right) & = z^{2}c_{N}\mathbf{R}_{\mathcal{B}}
(z)\Psi\left(  \widetilde{\mathbf{R}}_{N}^{T}(z)\overline{\Psi}\left(
\mathbf{X}\right)  \widetilde{\mathbf{T}}_{N}^{T}(z)\right)  \mathbf{T}
_{\mathcal{B}}(z) \label{eq:def-Phi1} \\
    \label{eq:def-Phi1Bt}
  {\Phi}_{\mathcal{B},1}^{t}(\mathbf{X}) & = z^{2}c_{N}\,{\Psi}\left(  \widetilde{\mathbf{T}
}_{N}^{T}(z)\overline{{\Psi}}(\mathbf{T}_{\mathcal{B}}(z)\mathbf{X}\mathbf{R}_{\mathcal{B}}(z))\widetilde{\mathbf{R}}_{N}^{T}(z)\right)
\end{align}
where now $\mathbf{R}_{\mathcal{B}}(z) = \mathcal{B}_L^{-1/2} \mathbf{R}_{N}(z) \mathcal{B}_L^{-1/2}$. In particular, (\ref{eq:-speed-Tr-E(Q)-T}) will follow the arguments in \cite[Section 6.1]{loubaton-mestre-2017}, which basically requires the application of Montel's theorem. 
To see that (\ref{eq:Tr-Q-R}) implies (\ref{eq:weakconv}) we need to check that $(\bar{\mu}_N)_{N \geq 1}$ is almost surely tight and $({\mu}_N)_{N \geq 1}$ is tight (see \cite[Corollary 2.7]{hachem-loubaton-najim-aap-2007}). The fact that $(\bar{\mu}_N)_{N \geq 1}$ is almost surely tight follows from the fact that
\[
\int_{\mathbb{R}^{+}}\lambda d\bar{\mu}_{N}(\lambda)=\frac{1}{ML}
\mathrm{Tr}\mathcal{B}_{L}^{-1/2}\mathbf{W}_{N}\mathbf{W}_{N}^{H}
\mathcal{B}_{L}^{-1/2}=\frac{1}{M}\sum_{m=1}^{M}\frac{1}{L}\mathrm{Tr}\left[
\mathcal{R}_{m,L}^{-1/2}\hat{\mathcal{R}}_{m,L}\mathcal{R}_{m,L}
^{-1/2}\right]  .
\]
The identity in (\ref{eq:convergence-blocks}) implies that
\[
\sup_{m=1,\ldots,M}\left\vert \frac{1}{L}\mathrm{Tr}\left[  \mathcal{R}
_{m,L}^{-1/2}\hat{\mathcal{R}}_{m,L}\mathcal{R}_{m,L}^{-1/2}\right]
-1\right\vert \rightarrow0,\,a.s.
\]
Therefore, $\int_{\mathbb{R}^{+}}\lambda d\bar{\mu}_{N}(\lambda)\rightarrow1$
almost surely, and tightness holds with probability one. To verify that $({\mu}_N)_{N \geq 1}$ is tight, we evaluate $\int_{\mathbb{R}^{+}}\lambda
d\bs{\mu}_{N}(\lambda)$ using item (v) of Proposition \ref{prop:class-S} and
immediately obtain that $\int_{\mathbb{R}^{+}}\lambda d\bs{\mu}_{N}
(\lambda)=\I_{ML}$ and $\int_{\mathbb{R}^{+}}\lambda d\mu_{N}(\lambda)=1$, so tightness established.

To establish (\ref{eq:-speed-Tr-E(Q)-T}) when $\beta < 4/5$, we follow the corresponding arguments in \cite[Section 6.2]{loubaton-mestre-2017}. Regarding (\ref{eq:convergence-rate-biais-lss-bis}), it will be a direct consequence of \cite[Lemma 5.5.5]{capitaine2007freeness}\footnote{The statement of \cite[Lemma 5.5.5]{capitaine2007freeness} requires that the function $\phi$ vanishes on the support of $\mu_N$. However, the reader may check that this assumption is in fact not needed.} provided that we are able to show that, given two nice constants $C_0, C_0'$, there exist three nice constants $C_1, C_2, C_3$ and an integer $N_0$ such that 
$$
\left| \frac{1}{ML} \mathrm{Tr}\left(\mathbb{E}\mathbf{Q}_N(z) - 
 {\bf T}_N (z) \right)  \right| \leq C_2 \frac{L}{MN}\frac{1}{(\Imm z)^{C_3}}
$$
for all $z$ inside the domain $|\mathrm{Re}z|\leq C_0$, $N^{-C_1}\leq \Imm z \leq C_0'$ and $N > N_0$. For this, it is sufficient to to follow the arguments used to establish Theorem 10.1 in \cite{loubaton2016}. 
\end{proof}
\begin{remark}
\label{re:compactly-supported-phi-1}
We notice that (\ref{eq:convergence-rate-biais-lss-bis}) is just established for compactly 
supported functions $\phi$. In order to extend (\ref{eq:convergence-rate-biais-lss-bis}) 
to non compactly supported $\phi$, it would be necessary to establish that the support 
of $\mu_N$ is included for each $N$ large enough in a compact subset independent from $N$. While we feel that this property holds, its proof does not seem obvious. In Section \ref{sec:sims} we provide an example of non-compactly supported $\phi$ for which (\ref{eq:convergence-rate-biais-lss-bis}) still holds.

% We finally mention that if 
% $\phi(\lambda) = (\lambda - 1)^{2}$, it is possible to calculate in closed form 
% $\frac{1}{ML}\mathbb{E}\mathrm{Tr}\left(  \phi(\overline{\mathcal{R}
% }_{\mathrm{corr},L})\right)$ as well as  $\int\phi(\lambda)d\mu_{N}(\lambda)$, and to 
% establish that 
% \begin{equation}
%     \label{eq:particular-phi-1}
% \frac{1}{ML}\mathbb{E}\mathrm{Tr}\left(  \phi(\overline{\mathcal{R}
% }_{\mathrm{corr},L})\right)  = \int\phi(\lambda)d\mu_{N}(\lambda)
% \end{equation}
% (\ref{eq:convergence-rate-biais-lss-bis}) is thus verified if $\phi(\lambda) = (\lambda - 1)^{2}$. 
% The expression of $\int\phi(\lambda)d\mu_{N}(\lambda)$ will be given in section \ref{sec:approximation-MP} (see (\ref{eq:expre-particular-phi})).
\end{remark}

% an easy generalization of \cite[Theorem 6.2]{haagerupnew2005} (see e.g. \cite{capitaine2007freeness} {\color{red} [Perhaps say something more]} implies that, if $\phi$ is a compactly supported smooth function, we have 
% $$
%  \lim_{\epsilon \rightarrow 0} \left| \int_{\mathbb{R}^{+}} \phi(\lambda) \, \left(\frac{1}{ML} \mathrm{Tr}\left(\mathbb{E}\mathbf{Q}_N(\lambda + i \epsilon) - 
%  {\bf T}_N (\lambda + i \epsilon) \right)\right) d \lambda \right| \leq \kappa\frac{L}{MN}.
% $$
% The inequality in (\ref{eq:convergence-rate-biais-lss-bis}) thus follows from the Stieltjes transform inversion formula
% \begin{multline*}
% \frac{1}{ML}\mathbb{E}\mathrm{Tr}\left(  \phi(\bar{\mathcal{R}
% }_{\mathrm{corr},L})\right)  -\int\phi(\lambda)d\mu_{N}(\lambda)
% = \\ =
% \frac{1}{\pi} \lim_{\epsilon \rightarrow 0} \mathrm{Im} \left(\int_{\mathbb{R}^{+}} \phi(\lambda) \, \frac{1}{ML} \mathrm{Tr}\left(\mathbb{E}\mathbf{Q}_N(\lambda + i \epsilon)  - 
% {\bf T}_N(\lambda + i \epsilon) \right) d \lambda \right).
% \end{multline*}

%\input{expectationresolvent}
%\input{deterministiceq}
%\input{convergenceRT}
\section{Approximation by a Marchenko-Pastur distribution}
\label{sec:approximation-MP}

Let us denote by $t_N(z)$ the Stieltjes transform of the Marcenko-Pastur law $\mu_{mp,c_N}$ associated to 
the parameter $c_N = \frac{ML}{N}$. In other words, for each $z \in \mathbb{C}^+$, 
$t_N(z)$ is the unique solution of the equation 
\begin{equation}
    \label{eq:equation-marcenko-pastur}
    t_N(z) = \frac{1}{-z + \frac{1}{1 + c_N t_N(z)}}
\end{equation}
for which ${\mathrm{Im}(t_N(z))} \geq 0$. % if $z \in \mathbb{C} \setminus \mathbb{R}$ and $t_N(z) \geq 0$ if $z < 0$. 
If ${\bf T}_N(z)$ represents the deterministic equivalent of ${\bf Q}_N(z)$, solution of the equations 
(\ref{eq:canonical-T}, \ref{eq:canonical-tildeT}), the following theorem establishes that, for each $\gamma < \gamma_0$, $\gamma \neq 1$, the Stieltjes transform $\frac{1}{ML}\mathrm{Tr}\mathbf{T}_N(z)$ is well approximated by $t_N(z)$, up to an error of order $\mathcal{O}(L^{-2\min(1,\gamma)}$. 

The strategy of the proof follows two steps. In a first step, we will establish that the spectral norm of the error between the two Stieltjes transforms $ \| {\bf T}_N(z) - t_N(z) \I_{ML} \|$ is upper bounded by a term that decays as $L^{-\min(1,\gamma)}$ for each $\gamma < \gamma_0, \gamma \neq 1$. In a second stage, this result is used to obtain a refined convergence rate for the normalized trace of the result, so that, in fact
\begin{equation}
\label{eq:bound_ntrace_stieltjeswrtmp}
    \frac{1}{ML} \mathrm{Tr} \left( \mathbf{T}_N(z) - t_N(z) \I_{ML} \right) \leq \frac{1}{L^{2\min(\gamma,1)}}P_1(z)P_2\left( \frac{1}{\Imm z} \right)
\end{equation}
for each $z \in \mathbb{C}^+$ and for two nice polynomials $P_1(z)$, $P_2(z)$.

We observe here that a direct application of the above result to the Helffer-Sj\"ostrand formula implies (\ref{eq:convergencemu-MP}) in Theorem \ref{thm:main_result}. Indeed, observe that in this case we can write 
\begin{multline*}
\int_{\mathbb{R}^+} \phi(\lambda)  d\mu_N(\lambda) - \int_{\mathbb{R}^+} \phi(\lambda) d\mu_{mp,N}(\lambda) = \\ 
    = \frac{1}{\pi}\mathrm{Re}\int_{\mathcal{D}}dx \, dy \, \overline{\partial} \Phi_k (\phi) (z) \frac{1}{ML} \mathrm{Tr}\left( \mathbf{T}_N(z) - t_N(z)\I_{ML} \right).
\end{multline*}
% Now, note that
% $$
% \left| \frac{1}{ML} \mathrm{Tr} \left(  {\bf T}(z) - t(z) \I_{ML}\right) \right| \leq \| {\bf T}(z) - t(z) \I_{ML}  \| \leq \frac{1}{L^{\min(\gamma,1)}} P_1(|z|) P_{2}\left(\frac{1}{\Imm z}\right)
% $$
% where $\gamma < \gamma_0$, $\gamma \neq 1$. 
If $k$ is taken to be larger than or equal to the degree of $P_2$ in (\ref{eq:bound_ntrace_stieltjeswrtmp}), this directly shows (\ref{eq:convergencemu-MP}). On the other hand, from the convergence of $\frac{1}{ML} \mathrm{Tr} \left(  {\bf T}_N(z) - t_N(z) \I_{ML}\right)$ for all $z \in \mathbb{C}^+$ to zero together with the fact that both $(\mu_N)_{N \geq 1}$ and  $(\mu_{mp,c_N})_{N \geq 1}$ are tight\footnote{Tightness of $(\mu_N)_{N \geq 1}$ has been established before, whereas tightness of $(\mu_{mp,c_N})_{N \geq 1}$ follows from the fact that $c_N \rightarrow c_\star$.}, we see that $\mu_N - \mu_{mp,c_N}$ converges weakly to zero. But since $\mu_{mp,c_N}$ in turn converges weakly to $\mu_{mp,c_\star}$, the proof of Theorem \ref{thm:main_result} is completed. 

  \begin{remark}
 \label{re:compactly-supported-phi-2}
 We again notice that (\ref{eq:convergencemu-MP}) is established for compactly supported smooth 
 functions $\phi$. As in the context of Remark \ref{re:compactly-supported-phi-1}, the 
 generalization of (\ref{eq:convergencemu-MP}) to non compactly supported functions would need 
 to prove that the support of $\mu_N$ is included in a compact independent of $N$. 
%  We finally notice that (\ref{eq:convergencemu-MP}) is valid for $\phi(\lambda) = (\lambda - 1)^{2}$. This is because a direct calculation leads to 
%  \begin{equation}
%      \label{eq:expre-particular-phi}
%      \int\phi(\lambda)d\mu_{N}(\lambda) = c_N + c_N \frac{1}{ML} \mathrm{Tr} \left( \mathcal{B}_L^{-1} 
%      {\boldsymbol \Psi}({\bf E}_N) \right)
%  \end{equation}
%  As $\int\phi(\lambda)d\mu_{mp,c_N}(\lambda) = c_N$, it holds that 
%  \begin{equation}
%  \label{eq:expre-particular-bias}
%  \int\phi(\lambda)d\mu_{N}(\lambda) - \int\phi(\lambda)d\mu_{mp,c_N}(\lambda) = c_N \frac{1}{ML} \mathrm{Tr} \left( \mathcal{B}_L^{-1} 
%      {\boldsymbol \Psi}({\bf E}_N) \right)
%  \end{equation}
%  Corollary \ref{cor:control-norm-E-Delta} thus leads to the conclusion that (\ref{eq:convergencemu-MP})
%  holds  $\phi(\lambda) = (\lambda - 1)^{2}$. 
 \end{remark}

We will present the two stages of the proof in two separate subsections that follow.  In order to simplify the notation, we will drop from now on the subindex $N$ in all relevant quantities, i.e. $t_N(z), \tilde{t}_N(z), c_N, \mathbf{T}_N(z), \widetilde{\mathbf{T}}_N(z)$, etc. 

\subsection{Bounding the spectral norm  $\left\|  {\bf T}_N(z) - t_N(z) \I_{ML}\right\|$}
The objective of this section is to prove the following result.

\begin{theorem}
\label{th:T-t}
Under Assumptions \ref{as:asymptotic-regime}, \ref{ass:bounds-spectral-densities} and \ref{as:norm-r-omega}, there exist two nice polynomials $P_1$ and $P_2$ as given in Definition \ref{def:nice}, such that for each $\gamma < \gamma_0$, $\gamma \neq 1$, the inequality
\begin{equation}
    \label{eq:bound-norm-T-t}
    \| {\bf T}_N(z) - t_N(z) \I_{ML} \| \leq \frac{1}{L^{\min(\gamma,1)}} P_1(|z|) P_{2}\left(\frac{1}{\Imm z}\right)
\end{equation}
holds for each $z \in \mathbb{C}^+$. 
% Moreover, for each compactly supported smooth {\color{red} Drop?} function $\phi$, we have
% \begin{equation}
%     \label{eq:rate-convergence-linear-statistics-mp}
%     \int_{\mathbb{R}^{+}} \phi(\lambda) \, d \mu_N(\lambda) - \int_{\mathbb{R}^{+}} \phi(\lambda) \, d \mu_{mp,c_N}(\lambda)  \leq  \frac{\kappa}{L^{\min(\gamma,1)}} 
% \end{equation}
% for some nice constant $\kappa$.% {\color{red} PL: The case $\gamma = 1$ was actually not considered in section 6 (Corollary \ref{cor:control-norm-E-Delta})} 
%If $\gamma = 1$, the term $\frac{1}{L^{\min(\gamma,1)}}$ has to be 
%replaced by $\frac{\log L}{L}$.
\end{theorem}

% We observe here that a direct application of the above result to the Helffer-Sj\"ostrand formula implies (\ref{eq:convergencemu-MP}) in Theorem \ref{thm:main_result}. Indeed, observe that in this case we can write 
% \begin{multline*}
% \int_{\mathbb{R}^+} \phi(\lambda)  d\mu_N(\lambda) - \int_{\mathbb{R}^+} \phi(\lambda) d\mu_{mp,N}(\lambda) = \\ 
%     = \frac{1}{\pi}\mathrm{Re}\int_{\mathcal{D}}dx \, dy \, \overline{\partial} \Phi_k (\phi) (z) \frac{1}{ML} \mathrm{Tr}\left( \mathbf{T}_N(z) - t_N(z)\I_{ML} \right).
% \end{multline*}
% Now, note that
% $$
% \left| \frac{1}{ML} \mathrm{Tr} \left(  {\bf T}(z) - t(z) \I_{ML}\right) \right| \leq \| {\bf T}(z) - t(z) \I_{ML}  \| \leq \frac{1}{L^{\min(\gamma,1)}} P_1(|z|) P_{2}\left(\frac{1}{\Imm z}\right)
% $$
% where $\gamma < \gamma_0$, $\gamma \neq 1$. If $k$ is taken to be larger than or equal to the degree of $P_2$ in (\ref{eq:bound-norm-T-t}), this directly shows (\ref{eq:convergencemu-MP}). On the other hand, from the convergence of $\frac{1}{ML} \mathrm{Tr} \left(  {\bf T}(z) - t(z) \I_{ML}\right)$ for all $z \in \mathbb{C}^+$ to zero together with the fact that both $(\mu_N)_{N \geq 1}$ and  $(\mu_{mp,c_N})_{N \geq 1}$ are tight\footnote{Tightness of $(\mu_N)_{N \geq 1}$ has been established before, whereas tightness of $(\mu_{mp,c_N})_{N \geq 1}$ follows from the fact that $c_N \rightarrow c_\star$.}, we see that $\mu_N - \mu_{mp,c_N}$ converges weakly to zero. But since $\mu_{mp,c_N}$ in turn converges weakly to $\mu_{mp,c_\star}$, the proof of Theorem \ref{thm:main_result} is completed. 

We devote the rest of this section to the proof of Theorem \ref{th:T-t}. First of all, it is well known that the 
function $\tilde{t}(z) = c t(z) - \frac{1-c}{z}$ coincides with the
Stieltjes transform of the probability measure $c \mu_{mp,c} + (1 - c) \delta_{0}$ and is equal to 
\begin{equation}
    \label{eq:equation-marcenko-pastur-tilde}
    \tilde{t}(z) = - \frac{1}{z(1 + c t(z))}
\end{equation}
so that $t(z)$ can also be written as 
$$
t(z) = - \frac{1}{z(1+\tilde{t}(z))}.
$$

Consider here the two matrix-valued functions 
$\widetilde{{\bf T}}_{mp}(z)$ and ${\bf T}_{mp}(z)$ defined by 
\begin{eqnarray}
\label{eq:def-tildeTmp}
\widetilde{{\bf T}}_{mp}(z) & = & -\frac{1}{z}\left(
\mathbf{I}_{N}+c_{N}\overline{\Psi}^{T}\left(\mathcal{B}_{L}^{-1/2} \,
t(z) \I_{ML} \, \mathcal{B}_{L}^{-1/2}\right)\right)^{-1} \\
{\bf T}_{mp}(z) & = & -\frac{1}{z}
\left(\mathbf{I}_{ML}+\mathcal{B}_{L}^{-1/2}\Psi\left(\widetilde{\mathbf{T}}_{mp}^{T}(z)\right)  \mathcal{B}_{L}^{-1/2}\right)^{-1}.
\end{eqnarray}
According to Proposition \ref{prop:Upsilon-tildeUpsilon}, these functions belong to $\mathcal{S}_N(\mathbb{R}^{+})$ and  $\mathcal{S}_{ML}(\mathbb{R}^{+})$
respectively, and verify the various properties of functions $\widetilde{\bs{\Upsilon}}(z)$
and $\bs{\Upsilon}(z)$ defined in the statement of that proposition. In order 
to establish Theorem \ref{th:T-t}, we define $\Delta_{mp}(z)$ by 
\begin{equation}
    \label{eq:def-Deltamp}
 \Delta_{mp}(z) = t(z) \I_{ML} - {\bf T}_{mp}(z)   
\end{equation}
and express $t(z) \I_{ML} - {\bf T}(z)$ as 
\begin{equation}
    \label{eq:diff_t_T}
t(z) \I_{ML} - {\bf T}(z) = \left({\bf T}_{mp}(z) - {\bf T}(z)\right) + \Delta_{mp}(z).
\end{equation}
We also define ${\bf t}_{\mathcal{B}}(z)$, ${\bf T}_{\mathcal{B},mp}(z)$
and $\Delta_{\mathcal{B},mp}(z)$ by 
${\bf t}_{\mathcal{B}}(z) =  \mathcal{B}_L^{-1/2}  \, t(z) \I_{ML}  \, \mathcal{B}_L^{-1/2}$, 
${\bf T}_{\mathcal{B},mp}(z) =  \mathcal{B}_L^{-1/2}   {\bf T}_{mp}(z)  \mathcal{B}_L^{-1/2}$
and $\Delta_{\mathcal{B},mp}(z) =  \mathcal{B}_L^{-1/2}   \Delta_{mp}(z)  \mathcal{B}_L^{-1/2}$ respectively. 
Using the definition of ${\bf T}_{mp}$ and $\widetilde{{\bf T}}_{mp}$ as well as 
the canonical equations (\ref{eq:canonical-T}, \ref{eq:canonical-tildeT}), we obtain easily that  \begin{equation}
    \label{eq:expre-TBmp-TB}
    {\bf T}_{\mathcal{B},mp}(z) - {\bf T}_{\mathcal{B}}(z) = \Phi_{\mathcal{B},2}\left( {\bf t}_{\mathcal{B}}(z)  - {\bf T}_{\mathcal{B}}(z) \right) 
\end{equation}
where $\Phi_{\mathcal{B},2}$ is the linear operator acting on $ML \times ML$ matrices defined as 
\begin{equation}
    \label{eq:def-Phi2}
    \Phi_{\mathcal{B},2}({\bf X}) = c z^{2}  {\bf T}_{\mathcal{B},mp}(z) 
\Psi\left(  \widetilde{\mathbf{T}}_{mp}^{T}(z)\overline{\Psi}\left(
\mathbf{X}\right)  \widetilde{\mathbf{T}}^{T}(z)\right)  \mathbf{T}
_{\mathcal{B}}(z).
\end{equation}
Using this definition, we can re-write (\ref{eq:diff_t_T}) as  
\begin{equation}
 \label{eq:expre-tB-TB}
{\bf t}_{\mathcal{B}}(z) - {\bf T}_{\mathcal{B}}(z) = \Phi_{\mathcal{B},2}\left( {\bf t}_{\mathcal{B}}(z)  - {\bf T}_{\mathcal{B}}(z) \right) + \Delta_{\mathcal{B},mp}(z).
\end{equation}
Our approach is to use Proposition \ref{prop:Phi-Phi_S-Phi_TH} in order to establish that 
\begin{equation}
    \label{eq:expre-somme-tB-TB}
    {\bf t}_{\mathcal{B}}(z) - {\bf T}_{\mathcal{B}}(z)  =  \sum_{n=0}^{+\infty} 
\Phi_{\mathcal{B},2}^{(n)}\left(\Delta_{\mathcal{B},mp}(z)\right)    
\end{equation}
and that $\|   {\bf t}_{\mathcal{B}}(z) - {\bf T}_{\mathcal{B}}(z) \| \leq C(z) \, 
\| \Delta_{\mathcal{B},mp}(z) \|$. The identity in (\ref{eq:bound-norm-T-t}) will then be established if we are able to show that $\| \Delta_{\mathcal{B},mp}(z) \| \leq C(z) \frac{1}{L^{\min(1,\gamma)}}$ if $\gamma < \gamma_0$, $\gamma \neq 1$.

We begin by evaluating the spectral norm of $ \Delta_{mp}(z)$ and  $ \Delta_{\mathcal{B},mp}(z)$. 
For this, we observe that $\widetilde{{\bf T}}_{mp}^{T}(z)$ is given by
$$
\widetilde{{\bf T}}_{mp}^{T}(z) = -\frac
{1}{z}\left(  \mathbf{I}_{N} + c t(z)\overline{\Psi}\left(  \mathcal{B}
_{L}^{-1}\right)  \right)  ^{-1}
$$
where we can express $\overline{\Psi}\left(  \mathcal{B}
_{L}^{-1}\right) $ as
$$
\overline{\Psi}\left(  \mathcal{B}
_{L}^{-1}\right) = \int_{0}^{1} \frac{1}{M} \sum_{m=1}^{M} \mathcal{S}_m(\nu) {\bf a}_L^{H}(\nu)
\mathcal{R}_{m,L}^{-1}  {\bf a}_L(\nu)  {\bf d}_N(\nu)  {\bf d}^{H}_N(\nu) \, d\nu.
$$
Let us denote by $\mathbf{E}_N$ the $N \times N$ matrix defined by 
\begin{equation}
\label{eq:def_errormatrix_E_N}
\mathbf{E}_N  =\int_{0}^{1}\left(  \frac{1}{M}\sum_{m=1}^{M}\epsilon
_{m,L}\left(  \nu\right)  \right)  \mathbf{d}_{N}\left(  \nu\right)
\mathbf{d}^{H}_{N} \left(  \nu\right)  d\nu
\end{equation}
where $\epsilon_{m,L}(\nu)$ is defined by 
\[
\epsilon_{m,L}\left(\nu\right)  =\mathcal{S}_{m}(\nu)\mathbf{a}_{L}^{H}
\left(\nu\right)  \mathcal{R}_{m,L}^{-1}
\mathbf{a}_{L}\left(\nu\right) - 1.
\]
It is clear that $\overline{\Psi}\left(  \mathcal{B}
_{L}^{-1}\right) = \I_{N} + \mathbf{E}_N$, so that 
$\widetilde{{\bf T}}_{mp}^{T}(z)$ can be written as 
$$
\widetilde{{\bf T}}_{mp}^{T}(z) = \left[-z(1+ct(z)) \left( \I_{N} + \frac{c t(z)}{1+ct(z)} \mathbf{E}_N \right) \right]^{-1}
$$
or equivalently as 
\begin{align*}
\widetilde{{\bf T}}_{mp}^{T}(z) & = \tilde{t}(z) \, \I_{N}  \, \left( \I_{N} - c \, z \, t(z) \, \tilde{t}(z) {\bf E}_N \right)^{-1} \\ & = \tilde{t}(z) \I_{N} + c z t(z) \tilde{t}^{2}(z) {\bf E}_N \left( \I_{N} - c \, z \, t(z) \, \tilde{t}(z) {\bf E}_N \right)^{-1}.
\end{align*}
In order to express ${\bf T}_{mp}(z)$ in a convenient way, we define ${\bs \Gamma(z)}$ as the 
$ML \times ML$ block diagonal matrix given by
\begin{equation}
    \label{eq:def-Gamma}
    \bs{\Gamma}(z) = \Psi\left(  {\bf E}_N \left( \I_{ML} - c \, z \, t(z) \, \tilde{t}(z) {\bf E}_N \right)^{-1} \right).
\end{equation}
Using that $\Psi(\I_{N}) = \mathcal{B}_L$, we obtain 
$$
{\bf T}_{mp}(z) = \left[ -z \left( (1+\tilde{t}(z)) \I_{ML} + c z t(z) \tilde{t}^{2}(z) \mathcal{B}_L^{-1/2} 
\bs{\Gamma} (z) \mathcal{B}_L^{-1/2} \right) \right]^{-1}
$$
or, equivalently, 
\begin{align*}
{\bf T}_{mp}(z) &= t(z) \left( \I_{ML} - c (z t(z) \tilde{t}(z))^{2} \bs{\Gamma}_{\mathcal{B}}(z) \right)^{-1}\\
&= t(z) \I_{ML} +  t(z) c (z t(z) \tilde{t}(z))^{2} \bs{\Gamma}_{\mathcal{B}}(z)  \left( \I_{ML} - c (z t(z) \tilde{t}(z))^{2} \bs{\Gamma}_{\mathcal{B}}(z) \right)^{-1}
\end{align*}
where $\bs{\Gamma}_{\mathcal{B}}(z) = \mathcal{B}_L^{-1/2} \bs{\Gamma} (z) \mathcal{B}_L^{-1/2}$. 
We eventually obtain that 
\begin{equation}
    \label{eq:expre-Deltamp}
    \Delta_{mp}(z) = -  t(z) c (z t(z) \tilde{t}(z))^{2} \bs{\Gamma}_{\mathcal{B}}(z)  \left( \I - c (z t(z) \tilde{t}(z))^{2} \bs{\Gamma}_{\mathcal{B}}(z) \right)^{-1}.
\end{equation}
The asymptotic behaviour of $\Delta_{mp}(z)$ depends on the behaviour of matrix 
${\bf E}_N$, which itself depends on the  properties of 
the terms $(\epsilon_m(\nu))_{m=1, \ldots, M}$. The following Lemma, 
established in the Appendix \ref{sec:orthogonal-polynomials}, is the key point 
of the proof of Theorem \ref{th:T-t}. 
\begin{lemma}
\label{lem:orthogonal-polynomials}
For each $\gamma < \gamma_0$, it holds that 
\begin{equation}
    \label{eq:control-epsilon_m}
    \sup_{m \geq 1} \sup_{\nu \in [0,1]} |\epsilon_{m,L}(\nu)| \leq \frac{\kappa}{L^{\min(\gamma,1)}}
\end{equation}
for some nice constant $\kappa$ (depending on $\gamma$) if $\gamma \neq 1$ while if $\gamma = 1$, 
\begin{equation}
    \label{eq:control-epsilon_m-gamma=1}
    \sup_{m \geq 1} \sup_{\nu \in [0,1]} |\epsilon_{m,L}(\nu)| \leq \kappa \, \frac{\log L}{L}.
\end{equation}
\end{lemma}
In the following, we use Lemma \ref{lem:orthogonal-polynomials} for a value of $\gamma$ as close as 
possible to $\gamma_0$ in order to obtain the fastest speed of convergence for 
$\sup_{m \geq 1} \sup_{\nu \in [0,1]} |\epsilon_{m,L}(\nu)|$. If $\gamma_0 \leq 1$, 
$\gamma < \gamma_0 \leq 1$ cannot be equal to 1. If $\gamma_0 > 1$, we will of course consider 
a value of $\gamma$ for which $1 < \gamma < \gamma_0$. Therefore, in the following, we assume that $\gamma \neq 1$. If $\gamma_0 \leq 1$, we thus obtain 
that for each $\gamma < \gamma_0$
\begin{equation}
    \label{eq:bound-epsilonm-gamma0-less-1}
  \sup_{m \geq 1} \sup_{\nu \in [0,1]} |\epsilon_{m,L}(\nu)| \leq \frac{\kappa}{L^{\gamma}}   
\end{equation}
holds, while if $\gamma_0 > 1$, 
\begin{equation}
    \label{eq:bound-epsilonm-gamma0-larger-1}
  \sup_{m \geq 1} \sup_{\nu \in [0,1]} |\epsilon_{m,L}(\nu)| \leq \frac{\kappa}{L}.
\end{equation}
Noting that ${\bf E}_N$ is the $N \times N$ Toeplitz matrix with symbol
$\frac{1}{M} \sum_{m=1}^{M} \epsilon_{m,L}(\nu)$, we immediately infer from this discussion 
the following corollary.
\begin{corollary}
\label{cor:control-norm-E-Delta}
If $\gamma_0 \leq 1$, then, for each $\gamma < \gamma_0$, there exists a nice constant 
$\kappa$ depending on $\gamma$ for which $\| {\bf E}_N \| \leq \frac{\kappa}{L^{\gamma}}$. 
If $\gamma_0 > 1$, there exists a nice constant $\kappa$ 
such that $\| {\bf E}_N \| \leq \frac{\kappa}{L}$. 
\end{corollary}
In order to control the norm of $\bs{\Gamma}(z)$, we mention that for each 
$z \in \mathbb{C}^{+}$, then $c |z t(z) \tilde{t}(z)|^{2} < 1$ 
(see e.g. Lemma 1.1 in \cite{loubaton2016}).
Therefore, the inequalities  $|z t(z) \tilde{t}(z)| \leq \frac{1}{\sqrt{c}}$ and 
$c |z t(z) \tilde{t}(z)| \leq \sqrt{c}$ hold on $ \mathbb{C}^{+}$. Corollary \ref{cor:control-norm-E-Delta} 
thus implies that for $L$ large enough, 
$ \| \I_{N} - c \, z \, t(z) \, \tilde{t}(z) {\bf E}_N \| > 1 - \sqrt{c} \, \| {\bf E}_N \| > \frac{1}{2}$ and 
$ \| \left(\I_{N} - c \, z \, t(z) \, \tilde{t}(z) {\bf E}_N \right)^{-1} \| < 2$
hold for each $z \in  \mathbb{C}^{+}$. For $L$
large enough, we thus have $\| {\bf E}_N \left(\I_{N} - c \, z \, t(z) \, \tilde{t}(z) {\bf E}_N \right)^{-1} \| \leq \frac{\kappa}{L^{\min(\gamma,1)}}$ for some nice constant $\kappa$,  a property which 
also implies that $\| \bs{\Gamma}(z) \| \leq \frac{\kappa}{L^{\min(\gamma,1)}}$ because 
if $\widetilde{{\bf X}}$ is any $N \times N$ matrix, then 
$\| \Psi(\widetilde{{\bf X}})\| \leq s_{max} \, \| \widetilde{{\bf X}} \|$, where we recall that $s_{max}$ is an upper bound on the spectral densities (cf. Assumption \ref{ass:bounds-spectral-densities}). We also notice that for $L$ large enough, $ \| \I_{ML} - c (z t(z) \tilde{t}(z))^{2} \bs{\Gamma}_{\mathcal{B}}(z) \| > 1 - \|\bs{\Gamma}_{\mathcal{B}}(z) \| > \frac{1}{2}$ for each $z \in \mathbb{C}^{+}$, and therefore, that $\| \left(  \I_{ML} - c (z t(z) \tilde{t}(z))^{2} \bs{\Gamma}_{\mathcal{B}}(z) \right)^{-1} \| < 2$ on $\mathbb{C}^{+}$.  This, in turn, implies that 
\begin{equation}
\label{eq:bound_Delta_mp}
    \| \Delta_{mp}(z) \| \leq \frac{C(z)}{L^{\min(\gamma,1)}}
\end{equation}
 and also $\| \Delta_{\mathcal{B},mp}(z) \| \leq \frac{C(z)}{L^{\min(\gamma,1)}}$ for $L$ large enough, as we wanted to show.

We now establish that (\ref{eq:expre-somme-tB-TB}) holds. We first prove 
that for any $ML \times ML$ block matrix matrix ${\bf X}$ the series
$\sum_{n=0}^{+\infty} \Phi_{\mathcal{B},2}^{(n)}({\bf X})$ is convergent. 
For this, we use Proposition \ref{prop:Phi-Phi_S-Phi_TH}. According to Lemma \ref{le:convergence-series-PhinTB}, $\sum_{n=0}^{+\infty} \Phi_{{\bf T}_{\mathcal{B}}^{H}}^{(n)}({\bf Y}) < +\infty$ 
for each positive matrix ${\bf Y}$. In order to establish a similar property 
for operator $\Phi_{{\bf T}_{\mathcal{B},mp}}$, we notice that a simple calculation leads to the identity
$$
\frac{\mathrm{Im} {\bf T}_{\mathcal{B},mp}(z)}{\mathrm{Im} z} = 
\mathcal{B}_L^{-1/2} {\bf T}_{mp}(z) {\bf T}^{H}_{mp}(z)  \mathcal{B}_L^{-1/2} +  \Phi_{{\bf T}_{\mathcal{B},mp}} \left( \frac{\mathrm{Im} {\bf t}_{\mathcal{B}}}{\mathrm{Im} z} \right) 
$$
if $z \in \mathbb{C}^{+}$. %, where we recall that, when  $z \in \mathbb{R}^{-*}$, we follow the convention of denoting as $\frac{\mathrm{Im} {\bf T}_{\mathcal{B},mp}(z)}{\mathrm{Im} z}$ and $\frac{\mathrm{Im} {\bf t}_{\mathcal{B}}}{\mathrm{Im} z}$ the derivatives of $ {\bf T}_{\mathcal{B},mp}$ and $ {\bf t}_{\mathcal{B}}$ at $z$ respectively. 
This implies that 
$$
\frac{\mathrm{Im} {\bf t}_{\mathcal{B}}(z)}{\mathrm{Im} z} = 
\mathcal{B}_L^{-1/2} {\bf T}_{mp}(z) {\bf T}^{H}_{mp}(z) \mathcal{B}_L^{-1/2} + 
\frac{\mathrm{Im} \Delta_{\mathcal{B},mp}(z)}{\mathrm{Im} z} + \Phi_{{\bf T}_{\mathcal{B},mp}} \left( \frac{\mathrm{Im} {\bf t}_{\mathcal{B}}}{\mathrm{Im} z} \right). 
$$
Noting that $\| \Delta_{\mathcal{B},mp}(z) \| \leq \frac{C(z)}{L^{\min(\gamma,1)}}$, Lemma B.1 in \cite{hachem-loubaton-najim-vallet-jmva-2013} implies that  
$$
\left \| \frac{\mathrm{Im} \Delta_{\mathcal{B},mp}(z)}{\mathrm{Im} z} \right \| 
\leq \frac{C(z)}{L^{\min(\gamma,1)}}.
$$
Proposition \ref{prop:Upsilon-tildeUpsilon} implies  that ${\bf T}_{mp}(z) {\bf T}^{H}_{mp}(z) \geq \frac{1}{C(z)} \I_{ML}$ for each 
$z \in \mathbb{C}^{+}$. Therefore, if we denote by ${\bf Y}_1(z)$ 
the matrix ${\bf Y}_1(z) = \mathcal{B}_L^{-1/2} {\bf T}_{mp}(z) {\bf T}^{H}_{mp}(z)  \mathcal{B}_L^{-1/2} + 
\frac{\mathrm{Im} \Delta_{\mathcal{B},mp}(z)}{\mathrm{Im} z}$, then, 
${\bf Y}_1(z) > \frac{1}{C(z)} \, \I_{ML} > 0$ if $z \in F_N$ where $F_N$ is a subset of $\mathbb{C}^{+}$ defined by
\begin{equation}
    \label{eq:def-FN}
    F_N = \left \{ z \in \mathbb{C}^{+}, \frac{1}{L^{\min(1,\gamma)}} P_1(|z|) P_2\left(\frac{1}{\Imm z}\right) \leq \kappa \right\}
\end{equation}
for some nice constant $\kappa$. Using the same arguments as in \cite{loubaton-mestre-2017}, 
we obtain that for each $z \in F_N$, the series 
$\sum_{n=0}^{+\infty} \Phi_{{\bf T}_{\mathcal{B},mp}}^{(n)}\left( {\bf Y}_1(z) \right)$ 
is convergent. Proposition \ref{prop:Phi-Phi_S-Phi_TH} implies that for each positive matrix ${\bf Y}$, 
$\sum_{n=0}^{+\infty} \Phi_{{\bf T}_{\mathcal{B},mp}}^{(n)}\left( {\bf Y} \right) < +\infty$ and 
that for each matrix ${\bf X}$, the series $\sum_{n=0}^{+\infty} \Phi_{\mathcal{B},2}^{(n)}\left( {\bf X} \right)$ is convergent if $z \in F_N$. 
Therefore, (\ref{eq:expre-somme-tB-TB}) holds true for $z \in F_N$, and 
$$
\left\| \sum_{n=0}^{+\infty} 
\Phi_{\mathcal{B},2}^{(n)}\left(\Delta_{\mathcal{B},mp}(z)\right)  \right\| 
\leq \| \Delta_{\mathcal{B},mp}(z) \| \left\| \sum_{n=0}^{+\infty} \Phi_{{\bf T}_{\mathcal{B},mp}}^{(n)}\left(\I_{ML}\right) \right\|^{1/2}
 \left\| \sum_{n=0}^{+\infty} \Phi_{{\bf T}_{\mathcal{B}}}^{(n)}\left( \I_{ML}\right) \right\|^{1/2}.
$$
It is easy to check that $\sum_{n=0}^{+\infty} \Phi_{{\bf T}_{\mathcal{B},mp}}^{(n)}\left( \I_{ML} \right) < C(z)\I_{ML}$ for $z \in F_N$. Therefore, 
we obtain that $\| {\bf t}_{\mathcal{B}}(z) - {\bf T}_{\mathcal{B}}(z) \| \leq \frac{C(z)}{L^{\min(\gamma,1)}}$
for each $z \in F_N$. It remains to evaluate $\| {\bf t}_{\mathcal{B}}(z) - {\bf T}_{\mathcal{B}}(z) \|$ if $z$ does not belong to $F_N$. For this, we remark that  $\| {\bf t}_{\mathcal{B}}(z) - {\bf T}_{\mathcal{B}}(z) \| \leq 
\| {\bf t}_{\mathcal{B}} (z)\|  +\| {\bf T}_{\mathcal{B}} (z) \| \leq C(z)$. As $z$ does not belong to $F_N$, the inequality
$1 \leq \frac{C(z)}{L^{\min(1,\gamma)}} $ holds for a certain $C(z)$, from which we deduce that 
 $\| {\bf t}_{\mathcal{B}}(z) - {\bf T}_{\mathcal{B}}(z) \| \leq \frac{C(z)}{L^{\min(\gamma,1)}}$ as expected. 
 Since the matrix $\mathcal{B}_L^{-1/2}$ verifies $\mathcal{B}_L^{-1/2} > \frac{1}{\sqrt{s_{min}}} \I_{ML}$, we obtain (\ref{eq:bound-norm-T-t}) for each $z \in \mathbb{C}^{+}$.

%{\color{red} I will finally not discuss the issue of estimating $ c_N \frac{1}{ML} \mathrm{Tr} \left( \mathcal{B}_L^{-1} {\boldsymbol \Psi}({\bf E}_N) \right)$.}

% It remains to justify (\ref{eq:rate-convergence-linear-statistics-mp}). For this, we remark that $$
% \left| \frac{1}{ML} \mathrm{Tr} \left( t(z) \I_{ML} - {\bf T}(z) \right) \right| \leq \| t(z) \I_{ML} - {\bf T}(z) \| \leq \frac{C(z)}{L^{\min(\gamma,1)}}
% $$
% where $\gamma < \gamma_0$, $\gamma \neq 1$. 
% An easy generalization of Theorem 6.2 in \cite{haagerupnew2005} (see e.g. \cite{capitaine2007freeness}) thus implies 
% that if $\phi$ is a compactly supported $\mathcal{C}^{\infty}$ function, we have 
% $$
% \lim_{\epsilon \rightarrow 0} \left| \int_{\mathbb{R}^{+}} \phi(\lambda) \, \left(\frac{1}{ML} \mathrm{Tr}\left(t(\lambda + i \epsilon) \I_{ML} - 
% {\bf T} (\lambda + i \epsilon) \right)\right) d \lambda \right| \leq \frac{\kappa}{L^{\min(\gamma,1)}}.
% $$
% The inequality in (\ref{eq:rate-convergence-linear-statistics-mp}) thus follows from the Stieltjes transform inversion formula
% $$
% \int_{\mathbb{R}^{+}} \phi(\lambda) \, \left(d\mu_N(\lambda) - d\mu_{mp, c}(\lambda) \right)= 
% \frac{1}{\pi} \lim_{\epsilon \rightarrow 0} \mathrm{Im} \left(\int_{\mathbb{R}^{+}} \phi(\lambda) \, \left(\frac{1}{ML} \mathrm{Tr}\left(t(\lambda + i \epsilon) \I_{ML} - 
% {\bf T}(\lambda + i \epsilon) \right) \right) d \lambda \right).
% $$

\subsection{Bounding the term $\frac{1}{ML} \mathrm{Tr} \left( t_N(z) \I_{ML} - \mathbf{T}_N(z) \right) $} 

We begin by considering the identity in (\ref{eq:expre-tB-TB}), and obtain that 
\begin{equation}
\label{eq:expre-tB-TBbis}
t(z) {\bf I}_{ML}  - {\bf T}(z) = \mathcal{B}_L^{1/2} \Phi_{\mathcal{B},2}\left( \mathcal{B}_L^{-1/2} ( t(z) {\bf I}_{ML}  - {\bf T}(z)) \mathcal{B}^{-1/2}_L \right)  \mathcal{B}_L^{1/2} + \Delta_{mp}(z)
\end{equation}
which directly implies that 
\begin{multline}
\label{eq:expre-t-normalized-trace-T}
t (z) - \frac{1}{ML} \mathrm{Tr} ({\bf T}(z)) = \frac{1}{ML} \mathrm{Tr} \left( \Phi_{\mathcal{B},2} \left( \mathcal{B}_L^{-1/2} (t(z) {\bf I}_{ML}  - {\bf T}(z)) \mathcal{B}_L^{-1/2} \right) \mathcal{B}_L \right) + \\ + \frac{1}{ML} \mathrm{Tr}( \Delta_{mp}(z)).
\end{multline}
We introduce the operator $\Phi_{\mathcal{B},2}^{t}$ defined by the property that, for any two $ML \times ML$ matrices $\mathbf{X}$, $\mathbf{Y}$, we have
\begin{equation}
\label{eq:def-transpose-Phi}
\frac{1}{ML} \mathrm{Tr} \left( {\bf X} \Phi_{\mathcal{B},2}({\bf Y}) \right) = \frac{1}{ML} \mathrm{Tr} \left( {\bf Y} \Phi_{\mathcal{B},2}^{t}({\bf X}) \right).    
\end{equation}
This can be seen as a transpose operator of $\Phi_{\mathcal{B},2}$. Using (\ref{eq:transpose_ops_gen}) it can be expressed in closed form as
$$
\Phi_{B,2}^{t}({\bf X}) = c z^{2} \Psi\left( \widetilde{{\bf T}}^{T} \overline{\Psi}( {\bf T}_{\mathcal{B}} {\bf X} {\bf T}_{\mathcal{B},mp}) \widetilde{{\bf T}}_{mp}^{T}  \right).
$$
Using (\ref{eq:def-transpose-Phi}), the expression in (\ref{eq:expre-t-normalized-trace-T}) can be rewritten as 
\begin{multline}
\label{eq:expre-t-normalized-trace-T-bis}
t (z) - \frac{1}{ML} \mathrm{Tr} ({\bf T}(z)) = \frac{1}{ML} \mathrm{Tr} \left( (t(z) {\bf I}_{ML}- {\bf T}(z)) \mathcal{B}_L^{-1/2} \Phi_{\mathcal{B},2}^{t}(\mathcal{B}_L) \mathcal{B}_L^{-1/2} \right) \\ +  \frac{1}{ML} \mathrm{Tr}( \Delta_{mp}(z)).    
\end{multline}
In order to simplify (\ref{eq:expre-t-normalized-trace-T-bis}), we observe that there exists $C(z) = P_1(|z|) P_2(\frac{1}{\Imm z})$ for some nice polynomials $P_1$ and $P_2$ such that
\begin{eqnarray*}
 \| {\bf T}_{\mathcal{B}}(z)  - t(z) \mathcal{B}_L^{-1} \| & \leq & \frac{C(z)}{L^{\min(\gamma,1)}} \\
 \| {\bf T}_{\mathcal{B},mp}(z) - t(z) \mathcal{B}_L^{-1} \| & \leq & \frac{C(z)}{L^{\min(\gamma,1)}} 
\end{eqnarray*}
and
\begin{eqnarray*}
 \| z \widetilde{{\bf T}}(z) -  z \tilde{t}(z) {\bf I}_{N} \| & \leq & \frac{C(z)}{L^{\min(\gamma,1)}} \\
 \| z \widetilde{{\bf T}}_{mp}(z) - z \tilde{t}(z) {\bf I}_{N} \| & \leq & \frac{C(z)}{L^{\min(\gamma,1)}} 
\end{eqnarray*}
which follow directly from Theorem \ref{th:T-t} and (\ref{eq:bound_Delta_mp}). From this, 
it is easily checked that for each matrix ${\bf X}$, $\Phi_{B,2}^{t}({\bf X})$ can be written as
\begin{equation}
    \label{eq:approx-transpose-Phi-bis}
\Phi_{B,2}^{t}({\bf X}) = c (z t(z) \tilde{t}(z))^{2} \, \Psi \left( \overline{\Psi}(\mathcal{B}_L^{-1} {\bf X} \mathcal{B}_L^{-1}) \right)  + \boldsymbol{\Upsilon}({\bf X})
\end{equation} 
where $\boldsymbol{\Upsilon}$ is a linear operator verifying 
\begin{equation}
    \label{eq:norm-upsilon}
    \| \boldsymbol{\Upsilon}({\bf X}) \| \leq \kappa \; C(z) \frac{\| {\bf X} \|}{L^{\min(\gamma,1)}}
\end{equation}
for each $z \in \mathbb{C}^{+}$. By (\ref{eq:approx-transpose-Phi-bis}), and using the fact that $\overline{\Psi}(\mathcal{B}_L^{-1}) = {\bf I}_N + {\bf E}_N$ and that 
$\Psi({\bf I}_N) = \mathcal{B}_L$, we obtain
\begin{equation}
 \label{eq:approx-transpose-Phi-ter}
\Phi_{B,2}^{t}(\mathcal{B}) = u(z) \, \Psi \left( \overline{\Psi}(\mathcal{B}_L^{-1}) \right) + \boldsymbol{\Upsilon}(\mathcal{B}_L) 
= u(z) \, ( \mathcal{B}_L + \Psi(\mathbf{E}_N)) +  \boldsymbol{\Upsilon}(\mathcal{B}_L)
\end{equation}
where we have introduced the definition $u(z) =  c (z t(z) \tilde{t}(z))^{2} $. We can express the above equation as  
\begin{equation}
 \label{eq:approx-transpose-Phi-quatro}
\Phi_{B,2}^{t}(\mathcal{B}_L) = u(z) \mathcal{B}_L +  \boldsymbol{\Upsilon}( \mathcal{B}_L) + u(z) \Psi(\mathbf{E}_N).
\end{equation}
Plugging (\ref{eq:approx-transpose-Phi-quatro}) into (\ref{eq:expre-t-normalized-trace-T-bis}), 
we obtain 
$$
t (z) - \frac{1}{ML} \mathrm{Tr} ({\bf T}(z)) = u(z) \, \left( t (z) - \frac{1}{ML} \mathrm{Tr} ({\bf T}(z)) \right) +  \\ \frac{1}{ML} \mathrm{Tr}( \Delta_{mp}(z)) + \delta_1(z) 
$$
where $\delta_1(z)$ is the error term defined as 
$$
\delta_1(z) = \frac{1}{ML} \mathrm{Tr} \left(   (t(z) {\bf I}_{ML}- {\bf T}(z)) \mathcal{B}_L^{-1/2} ( \boldsymbol{\Upsilon}( \mathcal{B}_L) + u(z) \Psi(\mathbf{E}_N)) \mathcal{B}_L^{-1/2} \right).
$$
We recall (see e.g. \cite{loubaton2016}, Lemma 1.1) that $u(z)$ verifies 
$1 - |u(z)| > \frac{1}{C(z)}$ on $\mathbb{C}^{+}$, where $C(z) = P_1(|z|) P_2(\frac{1}{\Imm z})$ for some nice polynomials $P_1$ and $P_2$.  Therefore, we have the inequality 
\begin{equation}
    \label{eq:inequality-t-normalized-trace-T}
    \left| t (z) - \frac{1}{ML} \mathrm{Tr} ({\bf T}(z)) \right| \leq C(z) \, \left| \frac{1}{ML}  \mathrm{Tr}( \Delta_{mp}(z)) \right| + |\delta_1(z)|.
\end{equation}
The bound in (\ref{eq:bound-norm-T-t}) together with Corollary \ref{cor:control-norm-E-Delta}, and the properties of 
operator $\boldsymbol{\Upsilon}$ imply that $|\delta_1(z)| \leq \frac{C(z)}{L^{2\min(\gamma,1)}}$. 
As a consequence, in order to complete the proof of (\ref{eq:bound_ntrace_stieltjeswrtmp}), we only need to establish the following fundamental Lemma.
\begin{lemma}
Under the above assumptions and for any $z \in \mathbb{C}^+$, we have
\label{le:trace-Deltamp}
\begin{equation}
    \label{eq:evaluation-trace-Deltamp}
 \left| \frac{1}{ML}  \mathrm{Tr}( \Delta_{mp}(z)) \right| \leq  \frac{C(z)}{L^{2\min(\gamma,1)}}
\end{equation}
where $C(z) = P_1(|z|) P_2(\frac{1}{\Imm z})$ for two nice polynomials $P_1$ and $P_2$.
\end{lemma}
To justify (\ref{eq:evaluation-trace-Deltamp}), we consider Eq. (\ref{eq:expre-Deltamp})
and express $({\bf I}_{ML} - u(z) \boldsymbol{\Gamma}_{\mathcal{B}}(z))^{-1}$ as 
$$
({\bf I}_{ML} - u(z) \boldsymbol{\Gamma}_{\mathcal{B}}(z))^{-1} = {\bf I}_{ML} + 
u(z)  \boldsymbol{\Gamma}_{\mathcal{B}}(z) ({\bf I}_{ML} - u(z) \boldsymbol{\Gamma}_{\mathcal{B}}(z))^{-1}.
$$
Hence, $ \Delta_{mp}(z)$ can thus be rewritten as
$$
\Delta_{mp}(z) = -t(z) u(z) \boldsymbol{\Gamma}_{\mathcal{B}}(z) + \Delta_{mp,1}(z)
$$
where $ \Delta_{mp,1}(z)$ is now defined by 
$$
 \Delta_{mp,1}(z) = - t(z) u^{2}(z) \left(\boldsymbol{\Gamma}_{\mathcal{B}}(z) \right)^{2} ({\bf I}_{ML} - u(z) \boldsymbol{\Gamma}_{\mathcal{B}}(z))^{-1}.
$$
Using the fact that $\| \boldsymbol{\Gamma}_{\mathcal{B}}(z) \| \leq \frac{\kappa}{L^{\min(\gamma,1)}}$
on $\mathbb{C}^{+}$, we obtain immediately that $\| \Delta_{mp,1}(z) \| \leq \frac{C(z)}{L^{2\min(\gamma,1)}}$ for each $z \in \mathbb{C}^{+}$. 
We finally remark that 
$ \boldsymbol{\Gamma}_{\mathcal{B}}(z)$ can be written as
\begin{equation}
    \label{eq:expre-Gamma-bis}
 \boldsymbol{\Gamma}_{\mathcal{B}}(z) = \mathcal{B}_L^{-1/2} \Psi({\bf E}_N) \mathcal{B}_L^{-1/2} + c z t(z) \tilde{t}(z) \Psi\left( 
 {\bf E}_N^{2}({\bf I}_{N} - c z t(z) \tilde{t}(z) {\bf E}_N)^{-1} \right).
 \end{equation}
 The spectral norm of the right hand side of (\ref{eq:expre-Gamma-bis}) is clearly upper bounded 
 by a term such as $\frac{\kappa}{L^{2\min(\gamma,1)}}$ for each $z \in \mathbb{C}^{+}$. 
 Therefore, $\frac{1}{ML}  \mathrm{Tr}( \Delta_{mp}(z))$ can be written as 
 \begin{equation}
     \label{eq:expre-trace-Deltamp}
 \frac{1}{ML}  \mathrm{Tr}( \Delta_{mp}(z)) = -t(z) u(z) \frac{1}{ML} \mathrm{Tr}( \Psi({\bf E}_N) \mathcal{B}_L^{-1})  + \delta_2(z)
 \end{equation}
where $\delta_2(z)$ verifies $|\delta_2(z)| \leq \frac{C(z)}{L^{2\min(\gamma,1)}}$ for each $z \in \mathbb{C}^{+}$. Using (\ref{eq:transpose_ops_gen}), we notice that $\frac{1}{ML} \mathrm{Tr}( \Psi({\bf E}_N) \mathcal{B}_L^{-1})$
is equal to 
$$
\frac{1}{ML} \mathrm{Tr}( \Psi({\bf E}_N) \mathcal{B}_L^{-1}) = \frac{1}{N} \mathrm{Tr}( {\bf E}_N \overline{\Psi}(\mathcal{B}_L^{-1}))=  \frac{1}{N} \mathrm{Tr}\left( {\bf E}_N({\bf I}_N + {{\bf E}_N)} \right)
$$
so that $\frac{1}{ML}  \mathrm{Tr}( \Delta_{mp}(z))$ can in turn be rewritten as 
\begin{equation}
     \label{eq:expre-trace-Deltamp-bis}
 \frac{1}{ML}  \mathrm{Tr}( \Delta_{mp}(z)) = -t(z) u(z) \frac{1}{N} \mathrm{Tr}( {\bf E}_N) +
 \delta_3(z)
 \end{equation}
where  $|\delta_3(z)| \leq \frac{C(z)}{L^{2\min(\gamma,1)}}$ for each $z \in \mathbb{C}^{+}$. We complete the proof of (\ref{eq:evaluation-trace-Deltamp}) by simply noting that 
\begin{equation}
    \label{eq:trace-E}
    \mathrm{Tr}({\bf E}_N) = 0.
\end{equation}
This can be shown by noting that we can express $\mathcal{R}_{m,L}$ as
$$
\mathcal{R}_{m,L} = \int_{0}^{1} \mathcal{S}_m(\nu) \mathbf{d}_L(\nu) \mathbf{d}_L^H(\nu) d\nu.
$$
As a consequence of this, 
$$
\int_{0}^{1} \mathcal{S}_m(\nu) \mathbf{a}^H_L(\nu)\mathcal{R}^{-1}_{m,L} \mathbf{a}_L(\nu) d\nu = \frac{1}{L}  \mathrm{Tr} \left[  \mathcal{R}^{-1}_{m,L}  \mathcal{R}_{m,L}\right]  = 1
$$
which directly implies that 
$$
 \int_0^1 \epsilon_{m,L} (\nu) d\nu = \int_{0}^{1} \mathcal{S}_m(\nu) \mathbf{a}^H_L(\nu)\mathcal{R}^{-1}_{m,L} \mathbf{a}_L(\nu) d\nu  -1 = 0.
$$
However, from the definition of $\mathbf{E}_N$ we see that 
$$
    \mathrm{Tr}\left( \mathbf{E}_N \right) = \frac{1}{M} \sum_{m=1}^M \int_0^1 \epsilon_{m,L} (\nu) d\nu = 0 
$$
which completes the proof.

\section{Numerical Validation}
\label{sec:sims}

The aim of this section is to validate the asymptotic study carried out above
via simulations. To that effect, we consider a simple example in which the
$M$\ independent time series are all autoregressive processes of order one
with parameter $\rho$ and unit power. By this, we mean that we generate each
time series independently by the recursion $y_{m,n+1}=\rho y_{m,n}+e_{m,n}$
where $e_{m,n}\sim\mathcal{N}_{\mathbb{C}}( 0,1-|\rho| ^{2})  $.

Let us first compare the empirical eigenvalue distribution of the sample cross
correlation matrix $\widehat{\mathcal{R}}_{\mathrm{corr},L}$ with the measure
$\mu_{N}$ and the Marchenko-Pastur distribution with parameter $c_{N}$.
Figure \ref{fig:histvsMP} represents the histogram of the eigenvalues of $\widehat
{\mathcal{R}}_{\mathrm{corr},L}$ together with the Marchenko-Pastur
distribution $\mu_{mp,c_{N}}$ for different values of $M,N,L$. In general terms, the 
Marchenko-Pastur approximation provides a relatively good approximation of the actual 
eigenvalue density. In general terms, we observe that the Marchenko-Pastur law
is a very good approximation of the actual empirical eigenvalue distribution, even for
relatively low values of $M,L$.

\begin{figure}
    \begin{minipage}{0.5\textwidth}
    \centering
    \includegraphics[width=\textwidth]{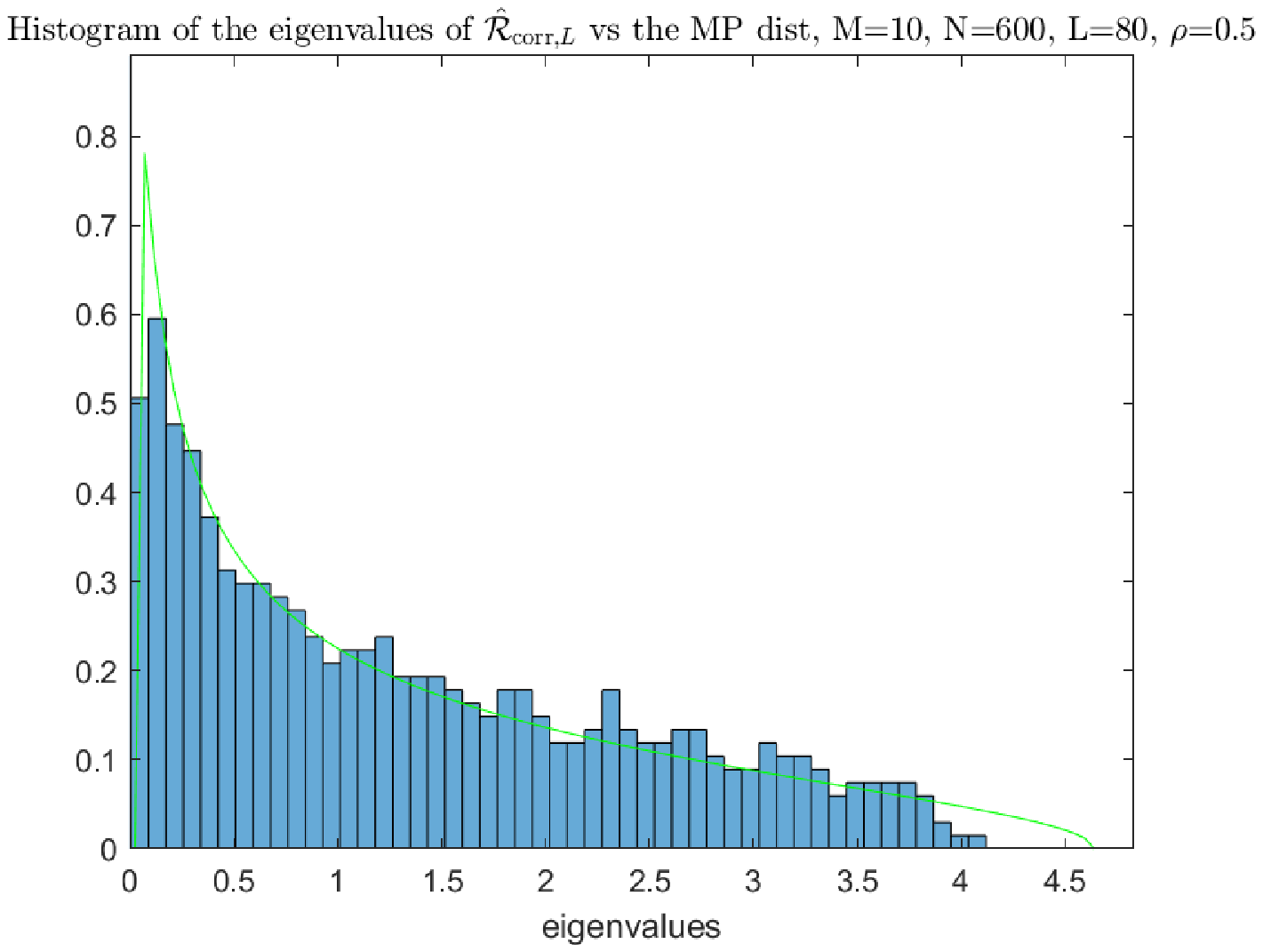}
    
   \tiny{ (a) $M=10, N=600, L=80$.}
    \end{minipage}%
    \begin{minipage}{0.5\textwidth}
    \centering
    \includegraphics[width=\textwidth]{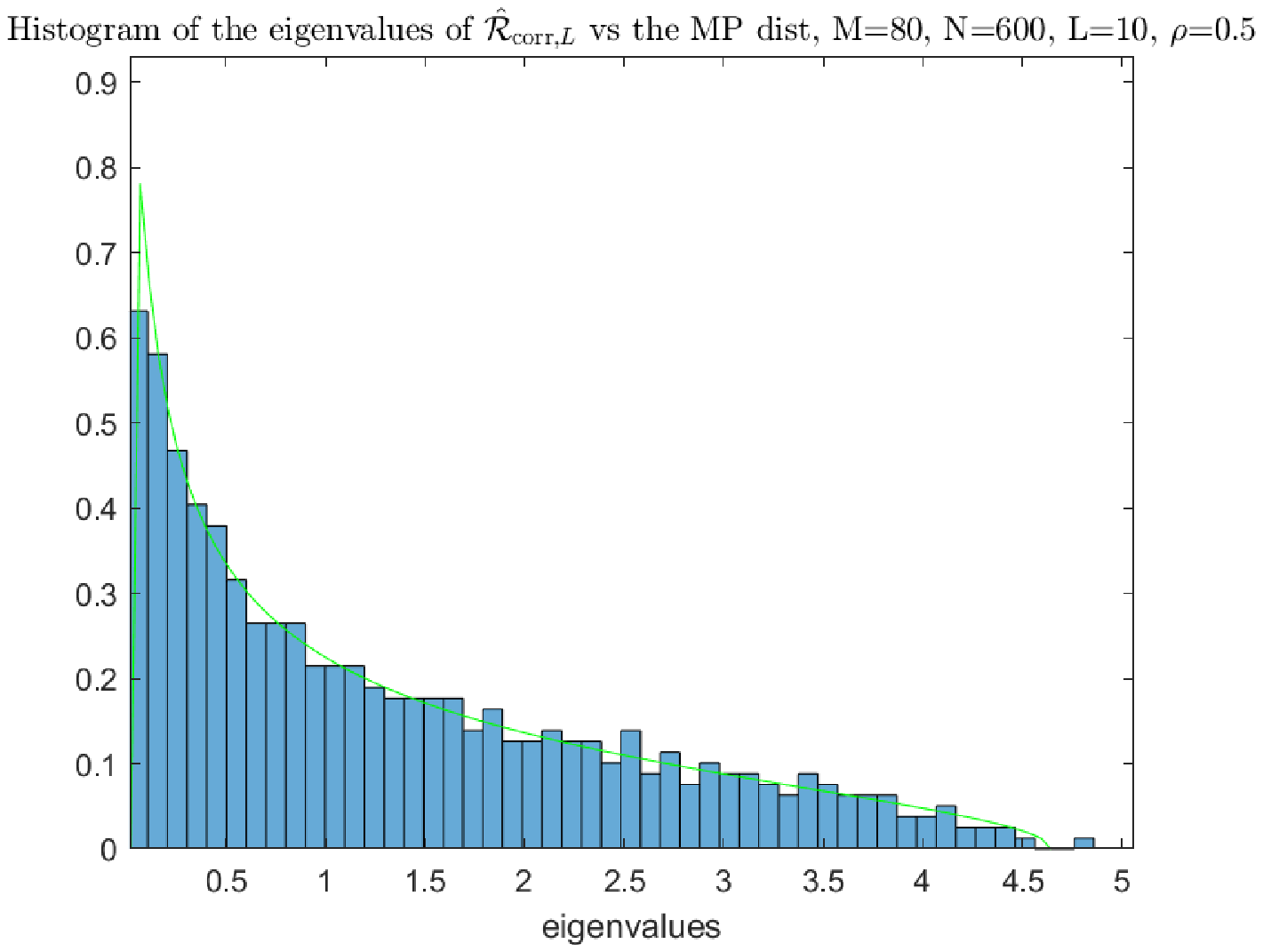}
    
    \tiny{(b) $M=80, N=600, L=10$.}
    \end{minipage}\\
    
    \begin{minipage}{0.5\textwidth}
    \centering
    \includegraphics[width=\textwidth]{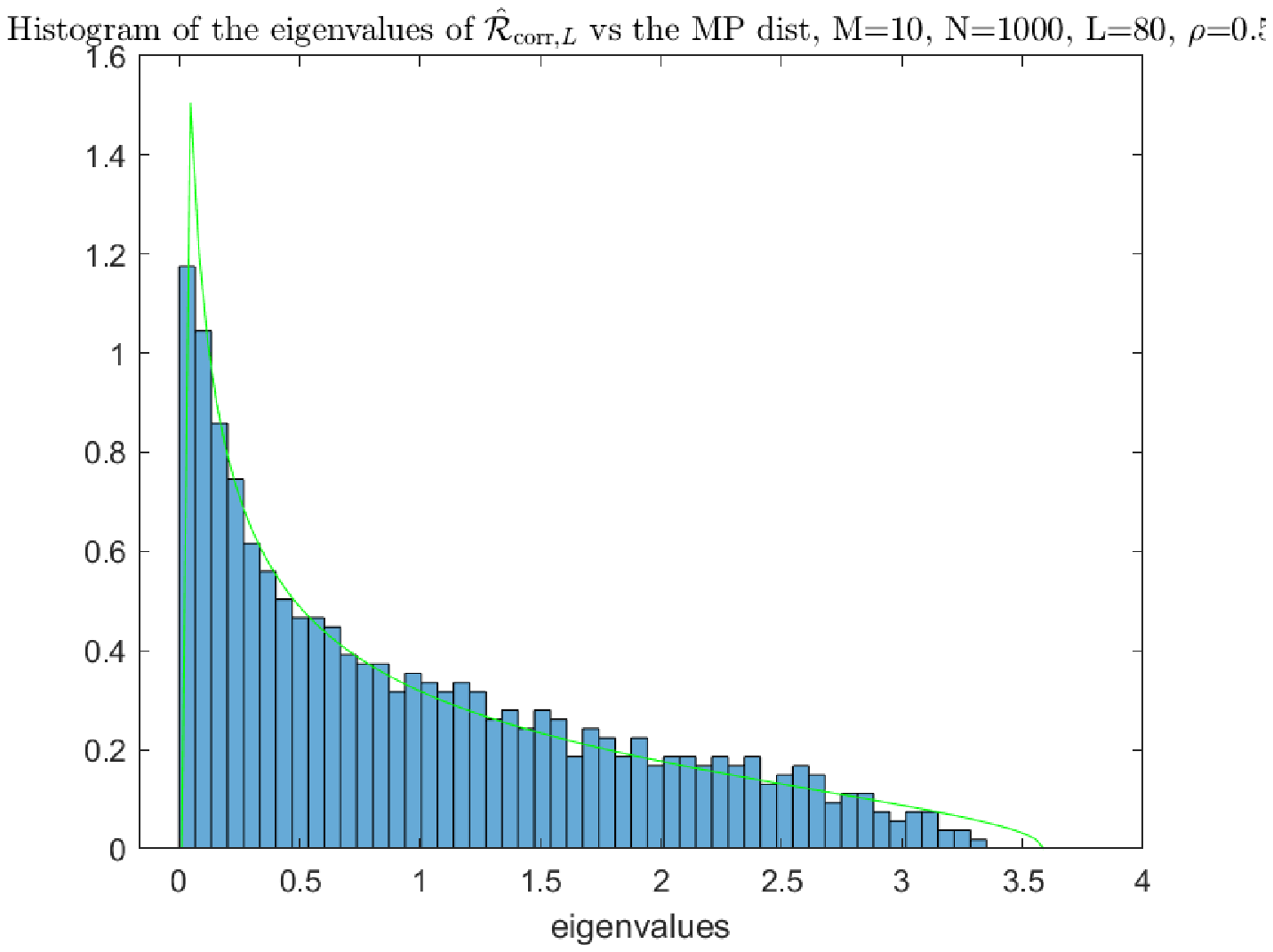}
    
    \tiny{(c) $M=10, N=1000, L=80$.}
    \end{minipage}%
    \begin{minipage}{0.5\textwidth}
    \centering
    \includegraphics[width=\textwidth]{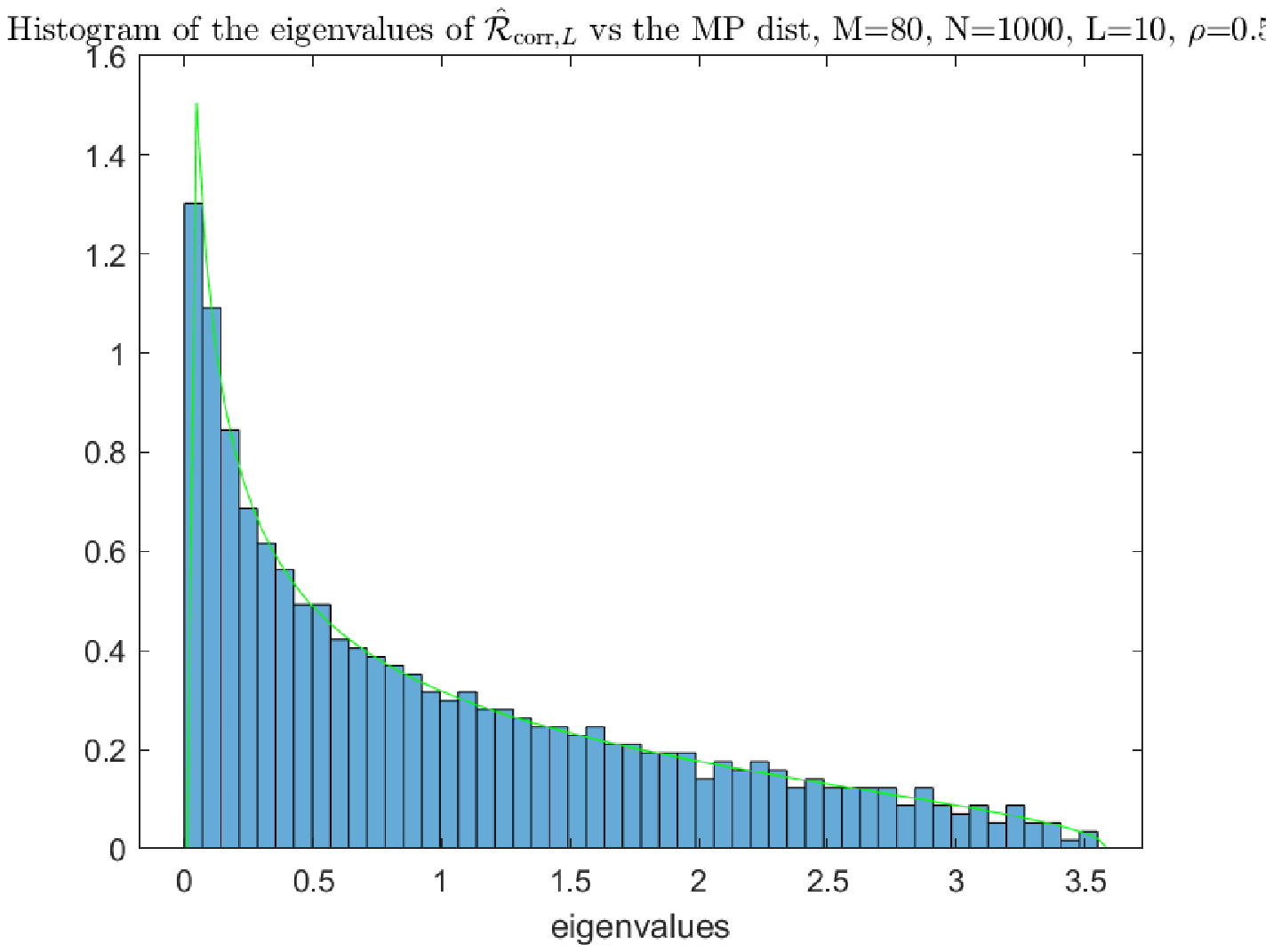}
    
    \tiny{(d) $M=80, N=1000, L=10$.}
    \end{minipage}
    \caption{Histogram of the eigenvalues of $\widehat{\mathcal{R}}_{\mathrm{corr},L}$ and Marchenko-Pastur law for different values of $M,N,L$ with $\rho=0.5$. Upper plots correspond to a situation where $c_N>1$ whereas lower plots deploy the case $c_N <1$.}
    \label{fig:histvsMP}
\end{figure}

Next, consider a correlation detection test statistic consisting of the sum of
the squared value of all the off-diagonal entries of $\widehat{\mathcal{R}
}_{\mathrm{corr},L}$. As mentioned in the introduction, this is reasonable
test since under $\mathrm{H}_{0}$ the true cross-correlation matrix
$\mathcal{R}_{\mathrm{corr},L}$ is equal to an identity. This corresponds to a
linear spectral statistic of $\widehat{\mathcal{R}}_{\mathrm{corr},L}$ built
with the function $\phi(\lambda)=(\lambda-1)^{2}$.

\begin{remark}
We observe that this function is not compactly supported so that in principle
the asymptotic rates predicted in items (ii) and (iii)\ of Theorem
\ref{thm:main_result} are not guaranteed to hold. However, we claim here that these two
items still hold for the choice $\phi(\lambda)=(\lambda-1)^{2}$. Indeed,
consider first item (ii) in the statement of this theorem. The only point in
the proof of this item where the hypothesis of compactly supported
$\phi(\lambda)$ is used is in order to establish
(\ref{eq:convergence-rate-biais-lss-bis}). However,  for this choice of
$\phi(\lambda)$ it is possible to compute $\frac{1}{ML}\mathbb{E}
\mathrm{Tr}\left(  \phi(\overline{\mathcal{R}}_{\mathrm{corr},L})\right)  $ in
closed form as well as $\int\phi(\lambda)d\mu_{N}(\lambda)$, and to establish
that
\begin{equation}
\frac{1}{ML}\mathbb{E}\mathrm{Tr}\left(  \phi(\overline{\mathcal{R}
}_{\mathrm{corr},L})\right)  =\int\phi(\lambda)d\mu_{N}(\lambda
)\label{eq:particular-phi-1}
\end{equation}
so that (\ref{eq:convergence-rate-biais-lss-bis}) is, in fact, trivial.
Indeed, the quantity on the left hand side can be computed by using
conventional formulas on the expectation of four Gaussian random vectors,
whereas the quantity on the right hand side can be evaluated by relating the
second order moment of the measure $\mu_{N}(\lambda)$ with its Stieltjes
transform. In both cases, we can establish that both quantities are equal to
\begin{equation}
\int\phi(\lambda)d\mu_{N}(\lambda)=c_{N}+c_{N}\frac{1}{ML}\mathrm{Tr}\left(
\mathcal{B}_{L}^{-1}{\boldsymbol\Psi}(\mathbf{E}_{N})\right)
\label{eq:expre-particular-phi}
\end{equation}
where we recall that $\mathbf{E}_{N}$ is defined in
(\ref{eq:def_errormatrix_E_N}). Regarding item (iii) in Theorem
\ref{thm:main_result}, we simply need to observe that $\int\phi(\lambda
)d\mu_{mp,c_{N}}(\lambda)=c_{N}$, so that
\begin{equation}
\int\phi(\lambda)d\mu_{N}(\lambda)-\int\phi(\lambda)d\mu_{mp,c_{N}}
(\lambda)=c_{N}\frac{1}{ML}\mathrm{Tr}\left(  \mathcal{B}_{L}^{-1}
{\boldsymbol\Psi}(\mathbf{E}_{N})\right)  .\label{eq:expre-particular-bias}
\end{equation}
Consequently, a direct application of Corollary \ref{cor:control-norm-E-Delta}
thus leads to the conclusion that (\ref{eq:convergencemu-MP}) also holds for
this particular choice of $\phi(\lambda).$ We may therefore consider this
statistic to validate the results of the paper. 
\end{remark}

In order to assess the error between $\hat{\phi}_{N}$ and the corresponding
integral of $\phi(\lambda)$ with respect to the Marchenko-Pastur distribution,
we considered here a set of $10^{4}$ realizations of the multivariate
autoregressive process described above. In each experiment, we fixed the three parameters $c_{\ast}$, $N$ and $\beta$ and considered a set of $M=[(c_{\ast}N)^{1-\beta}]$
independent time series, where $[x]$ here denotes the integer that is closest
to $x$. The number of time lags was therefore fixed to $L=[(c_{\ast
}N)^{\beta}]$. Figure \ref{fig:errors} represents the error between $\hat{\phi}_N$ and its
corresponding asymptotic limit as a function of $\beta$ for different values of $N$. 
The errors are represented as the square root of the empirical mean of the corresponding normalized difference,
averaged over the $10^4$ realizations. The plots on the left hand side represent the total error
$\hat{\phi}_N - \int \phi(\lambda)d\mu_{mp,N}$ whereas plot on the right hand side represent 
the two main constituent errors, namely: ``Error 1'' (solid lines) represents the square root of the empirical mean of the square of $\hat{\phi}_N - \int \phi(\lambda)d\mu_{N}$, and ``Error 2'' (dotted lines) represents $ \int \phi(\lambda)d\mu_{N} - \int \phi(\lambda)d\mu_{mp,N}$ as given in (\ref{eq:expre-particular-bias}).

These numerical results tend to confirm the fact that the error between the considered
statistic and its asymptotic deterministic approximation tends to be dominated by two different
phenomena depending on whether $M \ll L$ (large $\beta$) or $M \gg L$ (small $\beta$). 
In the fist case, the main contribution to the error corresponds to the term $\hat{\phi}_N - \int \phi(\lambda)d\mu_{N}$ (Error 1). We recall that, since the correlation sequence considered here decays exponentially to zero, this error term is dominated by $N^{-(1-\beta)}$, which in particular increases with $\beta$. Conversely, when $M \gg L$ (small $\beta$), the error is dominated by the difference between the two measures $\mu_{N}$ and $\mu_{mp,N}$.
We have seen that this error term is dominated by a term of order $N^{-2\beta}$, which in particular decreases with $\beta$. Observe also that the optimum choice of $\beta$ appears to be close to $1/3$, which corresponds to the case where the two error rates coincide.

\begin{figure}
    \begin{minipage}{0.5\textwidth}
    \centering
    \includegraphics[width=\textwidth]{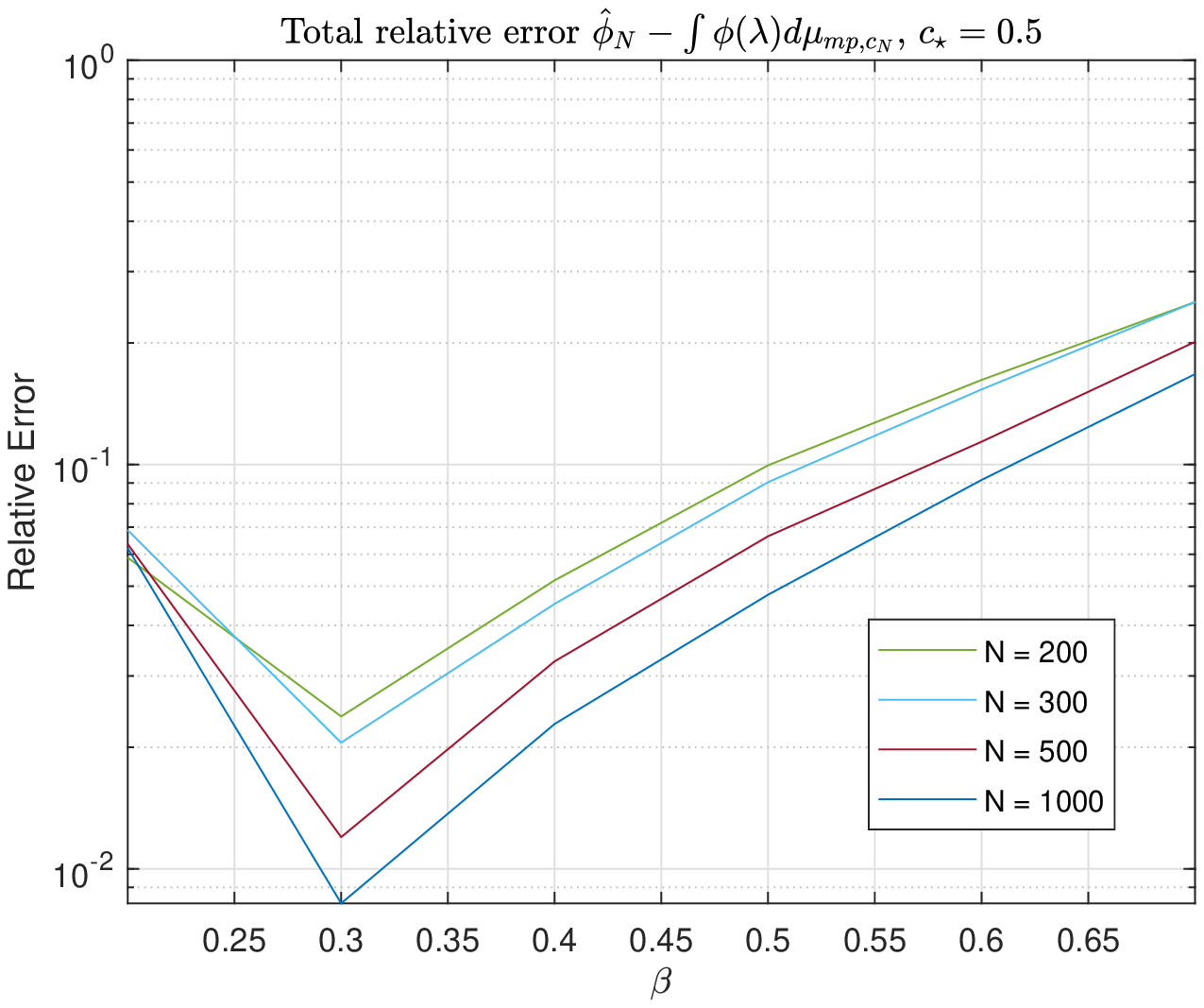}
    
    \tiny{ (a) Total error $c_\star = 0.5$.}
    \end{minipage}%
    \begin{minipage}{0.5\textwidth}
    \centering
    \includegraphics[width=\textwidth]{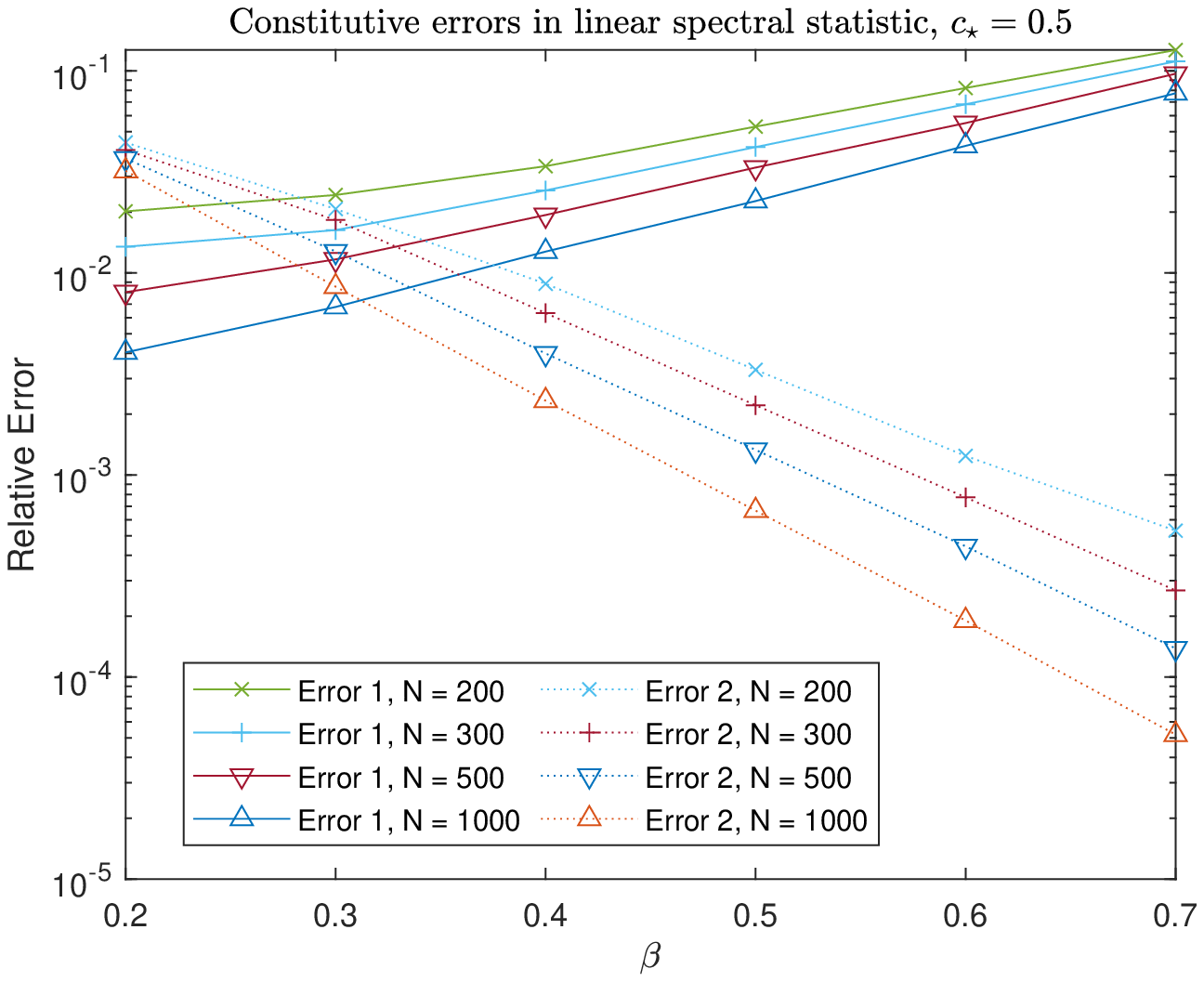}
    
   \tiny{ (b) Constituent errors $c_\star = 0.5$.}
    \end{minipage}\\
    
    \begin{minipage}{0.5\textwidth}
    \centering
    \includegraphics[width=\textwidth]{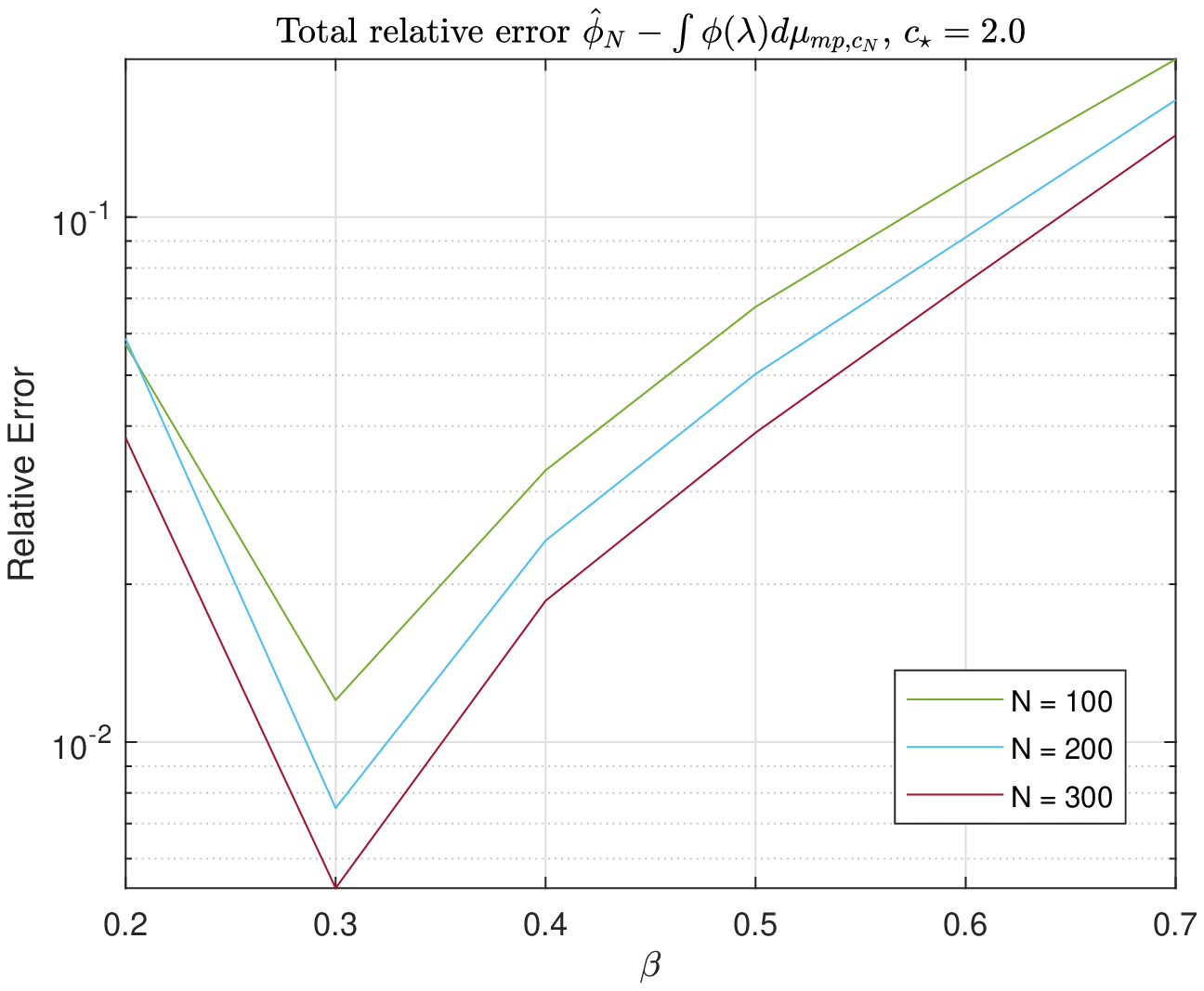}
    
    \tiny{(c) Total error $c_\star = 2$.}
    \end{minipage}%
    \begin{minipage}{0.5\textwidth}
    \centering
    \includegraphics[width=\textwidth]{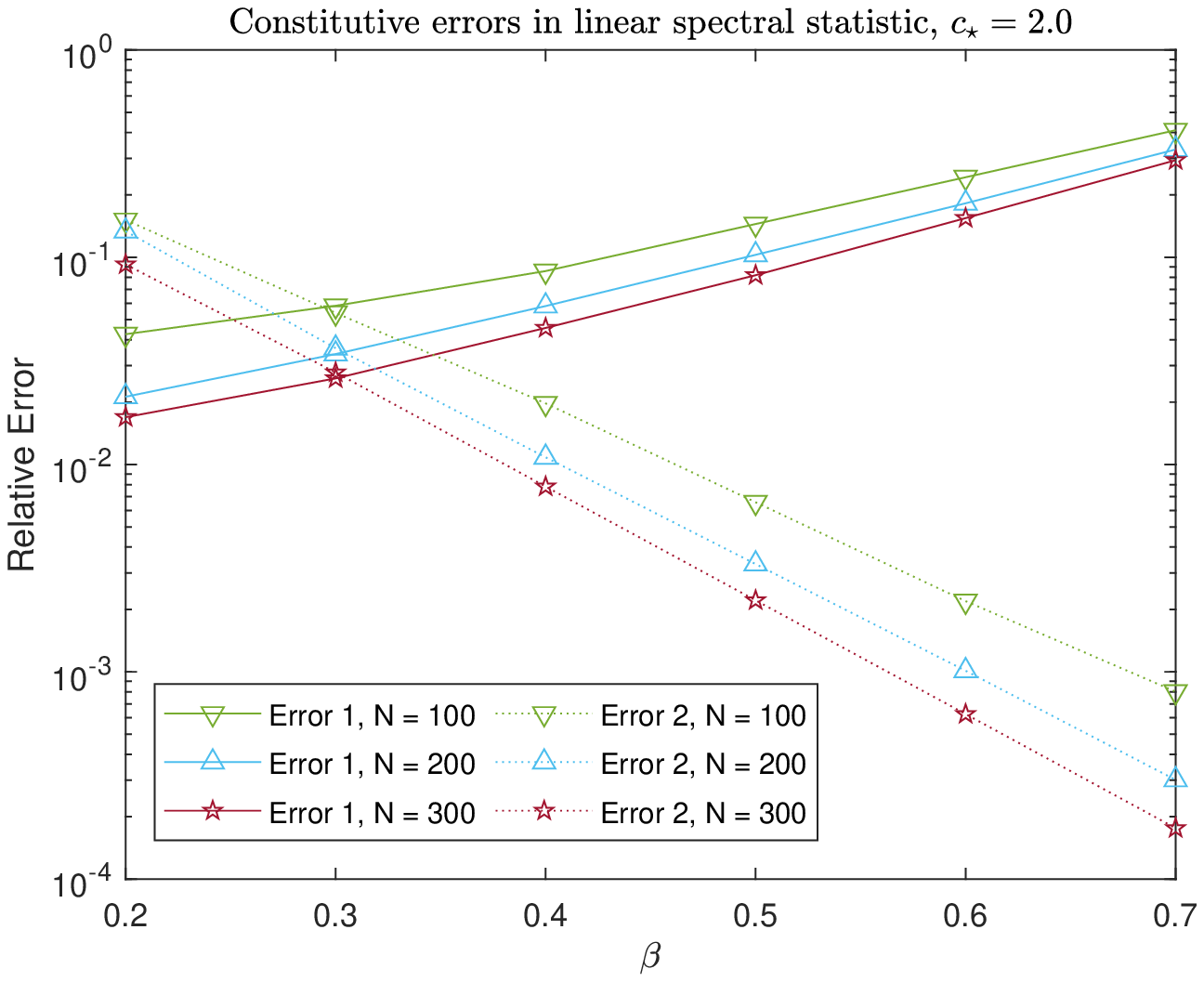}
    
    \tiny{(d) Constituent errors $c_\star = 2$.}
    \end{minipage}
    \caption{Evolution of the error of $\hat{\phi}_N$ with respect to the Marchenko-Pastur limit as a function of $\beta$. On the left hand side we represent the square root of the empirical mean of the square of $\hat{\phi}_N - \int \phi(\lambda)d\mu_{mp,N}$ over $10^4$ realizations of the statistic. On the right hand side, we represent the two main constituent errors. Error 1 (solid lines) represents the square root of the empirical mean of the square of $\hat{\phi}_N - \int \phi(\lambda)d\mu_{N}$. Error 2 (dotted lines) represents $ \int \phi(\lambda)d\mu_{N} - \int \phi(\lambda)d\mu_{mp,N}$. Upper plots correspond to a situation where $c_\star=0.5$ whereas lower plots deploy the case $c_\star =2$.}
    \label{fig:errors}
\end{figure}

\section*{Acknowledgments}
This work is partially funded by the B\'ezout Labex, funded by ANR, reference ANR-10-LABX-58, and by the ANR Project HIDITSA, reference ANR-17-CE40-0003.

\vspace{2em}
\appendix
\centerline{\textbf{APPENDICES}}
%appendix

\section{Proof of Lemma \ref{le:expectation-periodogram}}
\label{sec:app_proof_lemma_period}
 A classical calculation (see e.g. Theorem 4.3.2 in \cite{brillinger} in the non Gaussian case) leads to 
     $$
     \mathbb{E}| \xi_{L,y_m}(\nu) |^{2} = \sum_{-(L-2)}^{L-2} (1 - |l|/L) r_m(l) e^{- 2 i \pi l \nu }.
     $$
     Taking into account that $\mathcal{S}_m(\nu) = \sum_{l} r_m(l) e^{- 2 i \pi l \nu }$, we obtain immediately that 
     $$
     \mathbb{E}| \xi_{L,y_m}(\nu) |^{2}  = \mathcal{S}_m(\nu) + \epsilon_{m,L}(\nu)
     $$
     where $\epsilon_{m,L}(\nu)$ is defined by 
     $$
     \epsilon_{m,L}(\nu)= - \sum_{|l| \geq L-1} r_m(l) e^{- 2 i \pi l \nu } - \frac{1}{L} \sum_{-(L-2)}^{L-2} |l| r_m(l) e^{- 2 i \pi l \nu }.
     $$
     It is clear that 
     $$
     | \epsilon_{m,L}(\nu) | \leq \sum_{|l| \geq L-1} |r_m(l)| + \frac{1}{L} \sum_{-(L-2)}^{L-2} |l| |r_m(l)|.
     $$
     Using the bound in (\ref{eq:uniform-control-reminder-rm}) we directly obtain an upper bound of the first term, namely
     $$
     \sum_{|l| \geq L-1} |r_m(l)| \leq \frac{\kappa}{(L-1)^{\gamma_0}}.
     $$
     If $\gamma_0 \geq 1$, $\sum_{-(L-2)}^{L-2} |l| |r_m(l)| \leq \|r_m\|_{\omega_0}$
     and it holds that $\frac{1}{L} \sum_{-(L-2)}^{L-2} |l| |r_m(l)| \leq \frac{\kappa}{L}$. Therefore, if $\gamma_0 \geq 1$, we obtain that 
     $$
     | \epsilon_{m,L}(\nu) | \leq \frac{\kappa}{L}.
     $$
     If $\gamma_0 < 1$, we equivalently have
     $$
     \sum_{-(L-2)}^{L-2} |l| |r_m(l)| \leq L^{1-\gamma_0} \|r_m\|_{\omega_0}.
     $$
    Therefore, the inequality
     $$
     \frac{1}{L} \sum_{-(L-2)}^{L-2} |l| |r_m(l)| \leq \frac{\kappa}{(L-1)^{\gamma_0}}
     $$
     holds, as well as 
     $$
     | \epsilon_{m,L}(\nu) | \leq \frac{\kappa}{L^{\gamma_0}}.
     $$
     This completes the proof of Lemma \ref{le:expectation-periodogram}.

\section{Proof of  Lemma \ref{le:perturbation}}
\label{sec:app_proof_sqrt}

In order to establish (\ref{eq:formula-perturbation}), we first recall that 
$\| \widehat{\mathcal{R}}_{m,L} - \mathcal{R}_{m,L} \| \prec \max{(M^{-1/2},L^{-\gamma_0})}$.
We consider some $\delta > 0$ for which $N^{\delta}  \max{(M^{-1/2},L^{-\gamma_0})} \rightarrow 0$ and introduce the event $\mathcal{E}_N$ defined by 
\begin{equation}
    \label{eq:def-mathcalEN}
 \mathcal{E}_N = \left\{  \max_{m=1, \ldots, M} \|  \widehat{\mathcal{R}}_{m,L} - \mathcal{R}_{m,L} \| <  N^{\delta}  \max{(M^{-1/2},L^{-\gamma_0})}  \right\}
\end{equation}
Then, the event $\mathcal{E}_N$ holds with exponentially high probability. 
 In order to establish (\ref{eq:domination-Upsilon}), we have to evaluate $\mathbb{P}( \| \bs{\Upsilon}_{m,L} \| > N^{\epsilon}  \max{(M^{-1},L^{-2\gamma_0})})$ for each $\epsilon > 0$. For this,  we express $\mathbb{P}( \| \bs{\Upsilon}_{m,L} \| > N^{\epsilon}  \max{(M^{-1},L^{-2\gamma_0})})$ as 
 $$
 \mathbb{P}\left( \| \bs{\Upsilon}_{m,L} \| > N^{\epsilon}  \max{(M^{-1},L^{-2\gamma_0})}, \mathcal{E}_N \right) + \mathbb{P}\left( \| \bs{\Upsilon}_{m,L} \| > N^{\epsilon}  \max{(M^{-1},L^{-2\gamma_0})}, \mathcal{E}_N^{c} \right).
 $$
 Therefore, it holds that 
\begin{multline*}
\mathbb{P}(\| \bs{\Upsilon}_{m,L} \| > N^{\epsilon}  \max{(M^{-1},L^{-2\gamma_0})}) \leq \\ \leq \mathbb{P}( \mathcal{E}_N^{c}) + \mathbb{P}\left( \| \bs{\Upsilon}_{m,L} \| > N^{\epsilon}  \max{(M^{-1},L^{-2\gamma_0})}, \mathcal{E}_N \right).
\end{multline*}
In order to establish (\ref{eq:domination-Upsilon}), we thus just need to prove that 
there exists a $\gamma > 0$ such that 
$\mathbb{P}\left( \| \bs{\Upsilon}_{m,L} \| > N^{\epsilon}  \max{(M^{-1/2},L^{-\gamma_0})}, \mathcal{E}_N \right) \leq \exp(-N^{\gamma})$ for each $N$ large enough. For this, we 
remark that for each $N$ large enough, on $\mathcal{E}_N$, all the eigenvalues of matrices $\widehat{\mathcal{R}}_{m,L}$ are enclosed by the contour $\mathcal{C}$. Therefore, on 
$\mathcal{E}_N$, the equality 
\begin{equation}
    \label{eq:expre-sqrthatmathcalRmL}
    \widehat{\mathcal{R}}_{m,L}^{-1/2} = \frac{1}{2 i \pi} \int_{\mathcal{C}_-} \frac{1}{\sqrt{\lambda}} \, \left(  \widehat{\mathcal{R}}_{m,L} - \lambda {\bf I}_L \right)^{-1} \, d \lambda  
\end{equation}
holds. We note here that $\left( \widehat{\mathcal{R}}_{m,L} - \lambda {\bf I}_L \right)^{-1}$ can be
written as 
\begin{multline*}
\left(\widehat{\mathcal{R}}_{m,L} - \lambda {\bf I}_L \right)^{-1} = 
\left(  \mathcal{R}_{m,L}  - \lambda {\bf I}_L \right)^{-1} + \\ - 
\left(\widehat{\mathcal{R}}_{m,L} - \lambda {\bf I}_L \right)^{-1} \left( \widehat{\mathcal{R}}_{m,L} -  \mathcal{R}_{m,L} \right) \left(  \mathcal{R}_{m,L} - \lambda {\bf I}_L \right)^{-1} 
\end{multline*}
so that, by iterating this formula, we obtain
\begin{align*}
\left(\widehat{\mathcal{R}}_{m,L} - \lambda {\bf I}_L \right)^{-1} = & \left(  \mathcal{R}_{m,L} - \lambda {\bf I}_L \right)^{-1} 
-\left(  \mathcal{R}_{m,L} - \lambda {\bf I}_L \right)^{-1} \bs{\Delta}_{m,L} \left(  \mathcal{R}_{m,L} - \lambda {\bf I}_L \right)^{-1} + \\  
  + & \left(\widehat{\mathcal{R}}_{m,L} - \lambda {\bf I}_L \right)^{-1}  \bs{\Delta}_{m,L} \left(  \mathcal{R}_{m,L} - \lambda {\bf I}_L \right)^{-1} \bs{\Delta}_{m,L} \left(  \mathcal{R}_{m,L} - \lambda {\bf I}_L \right)^{-1}.
\end{align*}
We deduce from this expression together with (\ref{eq:expre-sqrtmathcalRmL}) and (\ref{eq:expre-sqrthatmathcalRmL}) that on $\mathcal{E}_N$ we can write
\begin{align}
 \widehat{\mathcal{R}}_{m,L}^{-1/2} & - \mathcal{R}_{m,L}^{-1/2} =  -\frac{1}{2 i \pi} \int_{\mathcal{C}_-} \frac{1}{\sqrt{\lambda}} \, \left(  \mathcal{R}_{m,L} - \lambda {\bf I}_L \right)^{-1}  \,  \bs{\Delta}_{m,L} \, \left(  \mathcal{R}_{m,L} - \lambda {\bf I}_L \right)^{-1} \, d \lambda +  \nonumber \\
 \label{eq:intermediate-perturbation-formula}
+ & \frac{1}{2 i \pi} \int_{\mathcal{C}_-} \frac{1}{\sqrt{\lambda}} \left(\widehat{\mathcal{R}}_{m,L} - \lambda {\bf I}_L \right)^{-1}  \bs{\Delta}_{m,L} \left(  \mathcal{R}_{m,L} - \lambda {\bf I}_L \right)^{-1} \bs{\Delta}_{m,L} \left(  \mathcal{R}_{m,L} - \lambda {\bf I}_L \right)^{-1} \, d \lambda.
\end{align}
Now, it is clear that on the contour $\mathcal{C}$, $| \frac{1}{\sqrt{\lambda}} |$ and the spectral norm of $\left(  \mathcal{R}_{m,L} - \lambda {\bf I}_L \right)^{-1}$ are upper bounded by a nice constant. This property also 
holds for $( \widehat{\mathcal{R}}_{m,L} - \lambda {\bf I}_L )^{-1}$ on the event $\mathcal{E}_N$. Therefore, on $\mathcal{E}_N$, the spectral norm 
of the second term on the right hand side of (\ref{eq:intermediate-perturbation-formula}) is upper bounded by $\kappa \| \bs{\Delta}_{m,L} \|^{2}$, which is stochastically dominated by  $\max(M^{-1},{L^{-2 \gamma_0}})$. This, in turn, establishes that there exists a $\gamma > 0$ such that 
$\mathbb{P}\left( \| \bs{\Upsilon}_{m,L} \| > N^{\epsilon}  \max{(M^{-1},L^{-2\gamma_0})}, \mathcal{E}_N \right) \leq \exp(-N^{\gamma})$ for each $N$ large enough. This completes the proof of Lemma \ref{le:perturbation}. 
\section{Proof of Lemma \ref{le:sumsquare-derivatives-trace}}
\label{app:prove-lemma-trace}

We first express 
matrix ${\bf W}^{m}_N$ in terms of vector ${\bf x}_m$. For this, we observe that 
for each $l=1, \ldots, L$, the $N$--dimensional vector $({\bf y}_{m,l}, \ldots, {\bf y}_{m,N+l-1})$ 
can be written as
\begin{align*}
({\bf y}_{m,l}, \ldots, {\bf y}_{m,N+l-1}) & =  {\bf y}_m {\bf J}_{N+L-1}^{-(l-1)} \left( \begin{array}{c} \I_N \\ 0 \end{array} \right) \\
                                           & =   {\bf x}_m \mathcal{R}_{m,N+L-1}^{1/2}  {\bf J}_{N+L-1}^{-(l-1)}\left( \begin{array}{c} \I_N \\ 0 \end{array} \right)
\end{align*}
Therefore, matrix ${\bf W}_N^{m}$ can be written as 
\begin{equation}
    \label{eq:expre-Wm}
 {\bf W}_N^{m} = \frac{1}{\sqrt{N}} \; \left( \begin{array}{c} {\bf x}_m \mathcal{R}_{m,N+L-1}^{1/2}   \\ \vdots \\  {\bf x}_m \mathcal{R}_{m,N+L-1}^{1/2}  {\bf J}_{N+L-1}^{-(L-1)} \end{array} \right) \; 
 \left( \begin{array}{c} \I_N \\ 0 \end{array} \right)
\end{equation}
We recall that ${\bf W}_N$ is the matrix ${\bf W}_N = \left( ({\bf W}^{1}_N)^{T}, \ldots, ({\bf W}^{L}_N)^{T}\right)^{T}$, and that  $\widehat{\mathcal{R}}_L = {\bf W}_N {\bf W}_N^{H}$. Using this notation, we can write
%Matrix $\frac{\partial {\bf Q}}{\partial {\bf x}_{m_0,i}}$ can thus be written as
\begin{multline*}
\frac{\partial {\bf Q}_N(z)}{\partial {\bf x}_{m_0,i}}  =  - {\bf Q}_N(z) \, \mathcal{B}_L^{-1/2} \, \frac{\partial \widehat{\mathcal{R}}_L }{\partial {\bf x}_{m_0,i}} \, \mathcal{B}_L^{-1/2} \, {\bf Q}_N(z)  \\
                                = - \frac{1}{\sqrt{N}} {\bf Q}_N(z) \, \mathcal{B}_L^{-1/2} \, {\bf E}_{m_0}  \; \left( \begin{array}{c} {\bf e}_i^{H} \mathcal{R}_{m_0,N+L-1}^{1/2}  \\ \vdots \\  {\bf e}_i^{H} \mathcal{R}_{m_0,N+L-1}^{1/2}  {\bf J}_{N+L-1}^{-(L-1)} \end{array} \right) \; \left( \begin{array}{c} \I_N \\ 0 \end{array} \right) {\bf W}_N^{H} \, 
                                 \mathcal{B}_L^{-1/2} \, {\bf Q}_N(z)
\end{multline*}
where we recall that $\mathbf{E}_{m_0}$ is an $ML \times L$ selection matrix with entries $(\mathbf{E}_{m_0})_{i,j} = \delta_{i = (m_0-1)M +j} $ and where $\mathbf{e}_i$ denotes the $i$th column of $\I_{N+L-1}$.
We introduce the matrix ${\bf H}_N(z)$ defined by 
$$
{\bf H}_N(z) = \frac{1}{\sqrt{N}} \left( \begin{array}{c} \I_N \\ 0 \end{array} \right) 
{\bf W}_N^{H} \mathcal{B}_L^{-1/2}  {\bf Q}_N(z) {\bf A}_N  {\bf Q}_N(z) \mathcal{B}_L^{-1/2}.
$$
It is easily seen that 
$$
 \frac{1}{ML} \mathrm{Tr}  \left( \frac{\partial {\bf Q}_N(z)}{\partial {\bf x}_{m_0,i}}  \, {\bf A}_N \right) = - \frac{1}{ML} \mathrm{Tr} \left( \begin{array}{c} {\bf e}_i^{H} \mathcal{R}_{m_0,N+L-1}^{1/2} \\ \vdots \\  {\bf e}_i^{H} \mathcal{R}_{m_0,N+L-1}^{1/2}  {\bf J}_{N+L-1}^{-(L-1)} \end{array} \right) \, {\bf H}_N(z) \, {\bf E}_{m_0}.
 $$
If we denote by $ {\bf f}_{l}^{m_0}$ the $l$-th column of $\mathbf{E}_{m_0}$, we can re-write the above expression as
$$
\left|  \frac{1}{ML} \mathrm{Tr}  \left( \frac{\partial {\bf Q}_N(z)}{\partial {\bf x}_{m_0,i}}  \, {\bf A}_N \right) \right|^{2} = \frac{1}{M^{2}} \left| \frac{1}{L} \sum_{l=1}^{L}  {\bf e}_i^{H} \mathcal{R}_{m_0,N+L-1}^{1/2}  {\bf J}_{N+L-1}^{-(l-1)} \; {\bf H}_N(z) \; {\bf f}_{l}^{m_0}  \right|^{2}.
$$
Consequently, a direct application of Jensen's inequality leads to 
\begin{multline*}
    \left|  \frac{1}{ML} \mathrm{Tr}  \left( \frac{\partial {\bf Q}_N(z)}{\partial {\bf x}_{m_0,i}}  \, {\bf A}_N \right) \right|^{2} \leq  \frac{1}{M^{2}} \frac{1}{L} \sum_{l=1}^{L} ({\bf f}_{l}^{m_0})^{H} \, {\bf H}_N^{H}(z) {\bf J}_{N+L-1}^{(l-1)} \mathcal{R}_{m_0,N+L-1}^{1/2} 
    {\bf e}_i \times \\ \times {\bf e}_i^{H}  \mathcal{R}_{m_0,N+L-1}^{1/2} {\bf J}_{N+L-1}^{-(l-1)} \; {\bf H}_N(z) \; {\bf f}_{l}^{m_0}.
\end{multline*}
Hence, using $\sum_{i} {\bf e}_i {\bf e}_i^{H} = \I_{N+L-1}$ and ${\bf J}_{N+L-1}^{(l-1)} \mathcal{R}_{m_0,N+L-1}
 {\bf J}_{N+L-1}^{-(l-1)} \leq \kappa \,  \I_{N+L-1}$, we obtain 
 $$
\sum_{m_0,i} \left|  \frac{1}{ML} \mathrm{Tr}  \left( \frac{\partial {\bf Q}_N(z)}{\partial {\bf x}_{m_0,i}}  \, {\bf A}_N \right) \right|^{2} \leq \kappa \,  \frac{1}{M} \frac{1}{ML} \mathrm{Tr} \left( {\bf H}_N^{H}(z)  {\bf H}_N(z)\right)
$$
so that, inserting the expression of ${\bf H}_N(z)$ above,
\begin{multline*}
    \sum_{m_0,i} \left|  \frac{1}{ML} \mathrm{Tr}  \left( \frac{\partial {\bf Q}_N(z)}{\partial {\bf x}_{m_0,i}}  \, {\bf A}_N \right) \right|^{2} \leq \\ \leq \frac{\kappa}{MN} \,  \frac{1}{ML} \mathrm{Tr} \left( 
    \mathcal{B}_L^{-1/2}  {\bf Q}_N^{H}(z) {\bf A}_N^{H}  {\bf Q}_N^{H}(z) \mathcal{B}_L^{-1/2} {\bf W}_N {\bf W}_N^{H} 
    \mathcal{B}_L^{-1/2}  {\bf Q}_N(z) {\bf A}_N  {\bf Q}_N(z) \mathcal{B}_L^{-1/2} \right).
\end{multline*}
Finally, using the resolvent identity $\mathcal{B}_L^{-1/2}  {\bf W}_N {\bf W}_N^{H}  \mathcal{B}_L^{-1/2}  {\bf Q}_N(z) = 
\I_{ML} + z \, {\bf Q}_N(z)$, we obtain
$$
 \| {\bf Q}_N^{H}(z) \mathcal{B}_L^{-1/2} {\bf W}_N {\bf W}_N^{H} 
\mathcal{B}_L^{-1/2}  {\bf Q}_N(z) \| \leq \frac{1}{\Imm z} \left( 1 + \frac{|z|}{\Imm z} \right) 
\leq \frac{1+|z|}{\Imm z}  \left( 1 + \frac{1}{\Imm z} \right) 
$$
so that (\ref{eq:sumsquare-derivatives-trace}) follows directly from
\begin{align*}
\sum_{m_0,i} \left|  \frac{1}{ML} \mathrm{Tr}  \left( \frac{\partial {\bf Q}_N(z)}{\partial {\bf x}_{m_0,i}}  \, {\bf A}_N \right) \right|^{2} & \leq  \frac{\kappa}{MN} \, \frac{1+|z|}{\Imm z}  \left( 1 + \frac{1}{\Imm z} \right) \,  \frac{1}{ML} \mathrm{Tr} \left( {\bf A}_N {\bf Q}_N(z) \mathcal{B}_L^{-1} {\bf Q}_N^{H}(z) {\bf A}_N^{H} \right) \\  &  \leq  \frac{\kappa}{MN} \, (1+|z|) \frac{1}{(\Imm z)^{3}}  \left( 1 + \frac{1}{\Imm z} \right) \,  \frac{1}{ML} \mathrm{Tr} \left( {\bf A}_N  {\bf A}_N^{H} \right).
\end{align*}

\section{Proof of Lemmas \ref{lem:orthogonal-polynomials}}

\label{sec:orthogonal-polynomials}
The proof of Lemma \ref{lem:orthogonal-polynomials} follows from the observation that the term 
${\bf a}_L^{H}(\nu) \mathcal{R}_{m,L}^{-1} {\bf a}_L(\nu)$ can be expressed in terms of 
the Szegö orthogonal polynomials associated to the scalar product 
\begin{equation}
\label{eq:def-scalarproduct-polynomial}
 \langle z^{k}, z^{l} \rangle = \int_{0}^{1} \mathcal{S}_m(\nu) e^{2 i \pi (k-l) \nu} \, d\nu .
\end{equation}
For each integer $l$, we introduce the monic orthogonal polynomial $\Phi_l(z)$ defined by 
\begin{equation}
    \label{eq:def-Phi}
    \Phi^{(m)}_{l}(z) = z^{l} - z^{l} | \mathrm{sp}(1,z, \ldots, z^{l-1})
\end{equation}
where the symbol $| A$ stands for the orthogonal projection over the space $A$ in the sense of the scalar product (\ref{eq:def-scalarproduct-polynomial}). We denote by $\sigma^{2,m}_{l}$ the norm square of $\Phi^{(m)}_{l}$, and define for each $l$ the normalized orthogonal polynomial 
$\phi^{(m)}_{l}(z)$ by 
\begin{equation}
\label{eq:def-phi}
\phi^{(m)}_{l}(z) = \frac{\Phi^{(m)}_{l}(z)}{\sigma^{m}_{l}}.
\end{equation}
It is well known that the sequence $(\sigma^{2,m}_{l})_{l \geq 0}$ is decreasing, 
that $\sigma^{2,m}_0 = r_m(0)$, and that $\lim_{l \rightarrow +\infty} \sigma^{2,m}_{l} = \sigma^{2,m}$
coincides with $\exp \int_{0}^{1} \log \mathcal{S}_m(\nu) d \nu$. It is clear that the normalized orthogonal polynomials satisfy 
$$
\langle \phi^{(m)}_{l}, \phi^{(m)}_{l^{'}} \rangle = \int_{0}^{1} \phi^{(m)}_{l}(e^{2 i \pi \nu}) \left(\phi^{(m)}_{l^{'}}(e^{2 i \pi \nu})\right)^{*} \mathcal{S}_m(\nu) d \nu = \delta_{l-l^{'}}.
$$
In the following, we also denote by $\Phi^{(m)*}_l(z)$ and $\phi^{(m)*}_l(z)$ the degree $l$ polynomials defined by 
$$
\Phi^{(m)*}_l(z) = z^{l} \left(\Phi^{(m)}_l(z^{-*})\right)^{*}, \; \phi^{(m)*}_l(z) = z^{l} \left(\phi^{(m)}_l(z^{-*})\right)^{*}.
$$
Noting that $\Phi_l$ is for each $l$ a monic polynomial, it is clear that $\Phi^{(m)*}_l(z)$ can be written as 
\begin{equation}
    \label{eq:expre-coeff-Phi*}
 \Phi^{(m)*}_l(z) = 1 + \sum_{k=1}^{l} a_{k,l}^{(m)} z^{k}   
\end{equation}
for some coefficients $( a_{k,l}^{(m)})_{k=1, \ldots, l}$. Moroever, $\Phi^{(m)*}_l(z)$
coincides with 
$$
\Phi^{(m)*}_l(z) = 1 - 1 | \mathrm{sp}(z, z^{2}, \ldots, z^{l})
$$
and the $l$--dimensional vector ${\bf a}_l^{(m)} = (a_{1,l}^{(m)}, \ldots, a_{l,l}^{(m)})^{T}$
is given by 
\begin{equation}
    \label{eq:yule-walker}
    \left( \begin{array}{c} 1 \\ {\bf a}_l^{(m)} \end{array} \right) = \sigma^{2,m}_l \,  \mathcal{R}_{m,l+1}^{-T} \, {\bf e}_1
\end{equation}
where ${\bf e}_{1}$ is the $l+1$--dimensional vector ${\bf e}_1 = (1, 0, \ldots, 0)^{T}$. It is moreover easily checked that 
\begin{equation}
    \label{eq:connection-linear-prediction}
    y_{m,n} - y_{m,n} | \mathrm{sp}(y_{m,n-1}, \ldots, y_{m,n-l}) = y_{m,n} + \sum_{k=1}^{l} 
    a_{k,l}^{(m)*} y_{m,n-k}
\end{equation}
where the orthogonal projection operator is this time defined on the space of all finite second moment complex valued random variables. For more details on these polynomials, we refer the reader to \cite{barry-simon-book} and
\cite{geronimus}. 

The matrix $\mathcal{R}_{m,L}^{-1}$ can be written as 
\begin{equation}
    \label{eq:cholevsky}
    \mathcal{R}_{m,L}^{-1} = {\bf A}_{m,L} \, \mathrm{Diag}\left(\frac{1}{\sigma^{2,m}_{0}}, \ldots, \frac{1}{\sigma^{2,m}_{L-1}} \right) \, {\bf A}_{m,L}^{H}
\end{equation}
where ${\bf A}_{m,L}$ is the upper-triangular matrix defined by 
\begin{equation}
    \label{eq:def-AmL}
    {\bf A}_{m,L} = \left( \begin{array}{ccccc} 1 & a_{1,1}^{(m)} & a_{2,2}^{(m)} & \ldots &  a_{L-1,L-1}^{(m)} \\
    0 & 1 & a_{1,2}^{(m)} & \ldots & a_{L-2,L-1}^{(m)} \\
    \vdots & \ddots & 1 & \ddots & \vdots \\
    \vdots & \ddots & \ddots & \ddots & \vdots \\
    0 & \ldots & \ldots & 0 & 1 \end{array} \right).
\end{equation}
In order to  see this, simply observe that $\mathcal{R}_{m,L}\mathbf{A}_{m,L}$ is lower triangular because of (\ref{eq:yule-walker}) and the fact that $\mathcal{R}_{m,l+1}^T = {\bf J}_{l+1} \mathcal{R}_{m,l+1} {\bf J}_{l+1}$. Since $\mathbf{A}^H_{m,L}$ is also lower triangular, so is the product $\mathbf{A}^H_{m,L}\mathcal{R}_{m,L}\mathbf{A}_{m,L}$. However, matrix $\mathbf{A}^H_{m,L}\mathcal{R}_{m,L}\mathbf{A}_{m,L}$ is also hermitian, which implies that it must be diagonal. Close examination of (\ref{eq:yule-walker}) reveals that its diagonal entries are equal to $\sigma^{2,m}_l$ for $l=0,\ldots,L-1$. Inverting the corresponding equation we obtain (\ref{eq:cholevsky}). 

Using the above decomposition of the matrix $\mathcal{R}_{m,L}^{-1}$ we immediately obtain that 
$$
{\bf a}_L(\nu)^{H}  {\bf A}_{m,L} = \frac{1}{\sqrt{L}} \left( 1, e^{-2 i \pi \nu} \Phi_1^{(m)*}(e^{2i\pi\nu}), \ldots,  e^{-2 i \pi (L-1)  \nu} \Phi_{L-1}^{(m)*}(e^{2i\pi\nu})\right)
$$
and consequently
\begin{equation}
    \label{eq:fundamental-equality}
    {\bf a}_L(\nu)^{H} \, \mathcal{R}_{m,L}^{-1}  {\bf a}_L(\nu) = 
    \frac{1}{L} \sum_{l=0}^{L-1} |\phi_{l}^{(m)*}(e^{2 i \pi \nu})|^{2}.
\end{equation}
%{\color{red} The proof of Lemma  \ref{le::trace-E} is a direct consequence of (\ref{eq:fundamental-equality}) and of the well known fact that 
%\begin{equation}
%    \label{eq:pade-approximation}
%    r_m(l) = \int_{0}^{1} \frac{1}{ |\phi_{l}^{(m)*}(e^{2 i \pi \nu})|^{2}} e^{2 i \pi k \nu} \, d\nu
%\end{equation}
%for each $m$ and for each $k \in \{-l, \ldots, l}$. In other words, we have
%$$
%\int_{0}^{1} \left( \mathcal{S}_m(\nu) -  \frac{1}{ |\phi_{l}^{(m)*}(e^{2 i \pi \nu})|^{2}} \right) e^{2 i \pi k \nu} \, d\nu = 0
%$$
%if $|k| \leq l$. This of course implies that 
%$$
%\int_{0}^{1} \left( \mathcal{S}_m(\nu) -  \frac{1}{ |\phi_{l}^{(m)*}(e^{2 i \pi \nu})|^{2}} \right)
% |\phi_{l}^{(m)*}(e^{2 i \pi \nu})|^{2} \, d\nu = 0
%$$
%and that 
%$$
%\frac{1}{L} \sum_{l=0}^{L-1} \int_{0}^{1} \left( \mathcal{S}_m(\nu)  |\phi_{l}^{(m)*}(e^{2 i \pi %\nu})|^{2} -1 \right) \, d\nu = 0
%$$
%Lemma \ref{le::trace-E} thus follows from (\ref{eq:fundamental-equality}) and from the observation %that 
%$$
%\mathrm{Tr}({\bf E}_N) = \frac{1}{M} \sum_{m=1}^{M} \int_{0}^{1} \left( \mathcal{S}_m(\nu) 
%{\bf a}_L(\nu)^{H} \, \mathcal{R}_{m,L}^{-1}  {\bf a}_L(\nu) - 1 \right)
%$$
%We now complete the proof of Lemma  \ref{lem:orthogonal-polynomials}.}
We first explain informally why, for each $m$, $\mathcal{S}_m(\nu) {\bf a}_L(\nu)^{H} \, \mathcal{R}_{m,L}^{-1}  {\bf a}_L(\nu) - 1$ converges uniformly towards $0$. For this, 
we need to recall certain results that are summarized next. 

Since the spectral densities $\mathcal{S}_{m}\left(\nu\right) $ are
uniformly bounded from below, we can define the cepstrum coefficients $
(c_{m}(k))_{k \in \mathbb{Z}}$, namely
\begin{equation*}
c_{m}\left( k\right) =\int_{0}^{1}\log\mathcal{S}_{m}\left( \nu\right)
\mathrm{e}^{2\pi\mathrm{i}\nu k}d\nu.
\end{equation*}
We notice that $\lim_{l \rightarrow +\infty} \sigma^{2,m}_{l} = \sigma^{2,m}$ coincides 
with $\exp c_m(0)$.
Assumption  \ref{as:norm-r-omega} and a generalization of the Wiener-Lévy theorem (see e.g. \cite{barry-simon-book}) implies that for each $m$, $c_{m}\in\ell_{\omega}$ for each $\gamma \leq \gamma_0$. 
We define the function
 $\pi^{(m)}(z)$ given by 
$$
\pi^{(m)}(z) = \exp -\left(c_m(0)/2 + \sum_{n=1}^{+\infty} c_m(-n) z^{n} \right).
$$
Then, $\pi^{(m)}(z)$ and $\psi^{(m)}(z) = \frac{1}{\pi^{(m)}(z)}$ are analytic in the open unit disk $\mathbb{D}$ and continuous on the closed unit disk. In the following, 
we denote by $\pi^{(m)}(z) = \sum_{n=0}^{+\infty} \pi^{(m)}(n) z^{n}$ and $\psi^{(m)}(z) = \sum_{n=0}^{+\infty} \psi^{(m)}(n) z^{n}$ their expansion in $\mathbb{D}$. Moreover, functions $\nu \rightarrow \pi^{(m)}(e^{2 i \pi \nu})$ and 
$\nu \rightarrow \psi^{(m)}(e^{2 i \pi \nu})$ also belong to $\ell_{\omega_0}$. To check this, we denote by $(\tilde{c}_m(n))_{n \geq 0}$
the one-sided sequence defined by $\tilde{c}_m(0) = c_m(0)/2$ and $\tilde{c}_m(n) = c_m(-n)$ for $n \geq 1$. 
Then, the sequences $\pi^{(m)}$ and $\psi^{(m)}$ can be written as
$$
\pi^{(m)} = \sum_{k=0}^{+\infty} \frac{(-1)^{k}}{k!} (\tilde{c}_m)^{*(k)},\  \psi^{(m)} = \sum_{k=0}^{+\infty} \frac{1}{k!} (\tilde{c}_m)^{*(k)}
$$
where for a sequence $a$, $a^{*(k)}$ represents $\underbrace{a * a* \ldots * a}_{k}$. Observe, in particular, that both sequences are one-sided. Now, for each $\gamma \leq \gamma_0$, it holds that 
\begin{eqnarray}
\label{eq:upperbound-norm-pi}
\|\pi^{(m)}\|_{\omega} & \leq & \sum_{k=0}^{+\infty} \frac{1}{k!} \| \tilde{c}_m \|_{\omega}^{k} = \exp(\| \tilde{c}_m \|_{\omega}) \leq \exp(\| c_m \|_{\omega})\\
\label{eq:upperbound-norm-psi}
\|\psi^{(m)}\|_{\omega} & \leq & \sum_{k=0}^{+\infty} \frac{1}{k!}  \| \tilde{c}_m \|_{\omega}^{k} = \exp(\| \tilde{c}_m \|_{\omega}) \leq \exp(\| c_m \|_{\omega}).
\end{eqnarray}
In the following, we also need a version of (\ref{eq:upperbound-norm-pi}, \ref{eq:upperbound-norm-psi}) holding uniformly w.r.t. $m$. For this, we establish the following lemma,
which can be seen as a uniform version of the generalized Wiener-Lévy theorem. 
\begin{lemma}
\label{le:wiener-levy-uniform}
Consider a function $F(z)$ holomorphic in a neighbourhood of the interval $[s_{min}, s_{max}]$
where $s_{min}$ and $s_{max}$ are defined in Assumption \ref{ass:bounds-spectral-densities}. Then, for each $\gamma < \gamma_0$ and for each $m$, the function $F\, \mathrm{o} \,\mathcal{S}_m$ belongs to 
$\ell_{\omega}$ and \footnote{We make the slight abuse of notation by identifying the $\omega$-norm of a function on the unit circle as the corresponding norm of its Fourier coefficient sequence.}
\begin{equation}
    \label{eq:uniform-wiener-levy}
    \sup_{m \geq 1} \| F \, \mathrm{o} \, \mathcal{S}_m \|_{\omega} < +\infty.
\end{equation}
\end{lemma}
\begin{proof}
We adapt the proof of the Wiener-Levy theorem in \cite{zygmund} (Theorem 5.2, p. 245).
We first claim that if $p$ is an integer such that $p > 1 + \gamma_0$ and if $G(\nu) = \sum_{n \in \mathbb{Z}} g(n) e^{2 i \pi n \nu}$ belongs to $\mathcal{C}_p$, then, $g \in \ell_{\omega_0}$, and 
\begin{equation}
    \label{eq:bound-norm-g}
    \| g \|_{\omega_0} \leq \kappa \, \left( \sup_{\nu} |G(\nu)| + \sup_{\nu} |G^{(p)}(\nu)| \right)
    \end{equation}
    for some constant $\kappa$ depending only on $\gamma_0$.
To verify (\ref{eq:bound-norm-g}), we remark that $|G(0)| \leq \sup_{\nu} |G(\nu)|$. 
Moreover, for each $n \neq 0$, the integration by parts formula leads to 
$$
g(n) = \frac{1}{(2 i \pi n)^{p}} \, \int_{0}^{1} G^{(p)}(\nu) e^{-2 i \pi n \nu} \, d \nu
$$
and to $|g(n)| \leq \frac{1}{(2 \pi)^{p}} \, \frac{1}{|n|^{p}} \, \sup_{\nu} | G^{(p)}(\nu)|$. 
As $p > 1 + \gamma_0$, we obtain immediately that (\ref{eq:bound-norm-g}) holds. \\

Since $F$ is holomorphic in a neighbourhood of $[s_{min}, s_{max}]$, there exists a $\rho > 0$ 
for which $F$ is holomorphic in the open disk $\mathbb{D}(s, 2 \rho)$ for each $s \in [s_{min}, s_{max}]$. 
In particular,  for each $m$ and each $\nu$, $F$ is holomorphic in 
$\mathbb{D}(\mathcal{S}_m(\nu), 2 \rho)$. We consider a partial sum $\mathcal{S}_{m,n_0}(\nu) = \sum_{k=-n_0}^{n_0} r_m(k) e^{- 2 i \pi k \nu}$, and claim that for each $\gamma < \gamma_0$, we have
\begin{equation}
\label{eq:bound-remainder}
    \| \mathcal{S}_{m}(\nu)-\mathcal{S}_{m,n_0}(\nu) \|_{\omega} = \sum_{|k| \geq (n_0+1)} (1+|k|)^{\gamma} |r_m(k)| \leq \frac{\kappa}{n_0^{\gamma_0 - \gamma}}
\end{equation}
for some nice constant $\kappa$. To justify (\ref{eq:bound-remainder}), we remark that 
$$
\| r_m \|_{\omega_0} \geq \sum_{|k| \geq (n_0+1)} (1 + |k|)^{\gamma_0} |r_m(k)| \geq 
n_0^{\gamma_0 - \gamma} \, \sum_{|k| \geq (n_0+1)}  (1 + |k|)^{\gamma} |r_m(k)| = n_0^{\gamma_0 - \gamma} \, \| \mathcal{S}_{m}(\nu)-\mathcal{S}_{m,n_0}(\nu) \|_{\omega}.
$$
Assumption \ref{as:norm-r-omega} implies that $\sup_{m} \| r_m \|_{\omega_0} < +\infty$. This leads immediately to (\ref{eq:bound-remainder}). We choose $n_0$ in such a way that  
$\frac{\kappa}{n_0^{\gamma_0 - \gamma}} \leq \frac{\rho}{2}$, and notice that (\ref{eq:bound-remainder}) leads to $\sup_{\nu} |\mathcal{S}_m(\nu) - \mathcal{S}_{m,n_0}(\nu)| \leq \frac{\rho}{2} $ 
for each $m$. Therefore, the circle $\mathbb{C}(\mathcal{S}_{m,n_0}(\nu), \rho)$ with center $\mathcal{S}_{m,n_0}(\nu)$ and radius $\rho$ is included into $\mathbb{D}(\mathcal{S}_m(\nu), 2 \rho)$, 
and $\mathcal{S}_m(\nu)$ belongs to the disk $\mathbb{D}(\mathcal{S}_{m,n_0}(\nu), \rho)$. The Cauchy formula 
implies that 
\begin{equation}
    \label{eq:cauchy-formula}
\left(F \, \mathrm{o} \, \mathcal{S}_m \right)(\nu) = 
\frac{1}{2 \pi} \int_{0}^{2 \pi} \frac{F(\mathcal{S}_{m,n_0}(\nu) + \rho e^{i \theta})}{\mathcal{S}_m(\nu) - 
\mathcal{S}_{m,n_0}(\nu) - \rho e^{i \theta}} \, \rho \, e^{i \theta} \, d\theta. 
\end{equation}
Since $ |\mathcal{S}_m(\nu) - \mathcal{S}_{m,n_0}(\nu)| \leq \frac{\rho}{2} $, it holds that 
$$
\frac{\rho e^{i \theta}}{\mathcal{S}_m(\nu) - 
\mathcal{S}_{m,n_0}(\nu) - \rho e^{i \theta}} = - \sum_{k=0}^{+\infty} \rho^{-k} e^{- i k \theta} \left( \mathcal{S}_m(\nu) - 
\mathcal{S}_{m,n_0}(\nu) \right)^{k}
$$
and that 
$$
\left \| \frac{\rho e^{i \theta}}{\mathcal{S}_m(\nu) - 
\mathcal{S}_{m,n_0}(\nu) - \rho e^{i \theta}}  \right \|_{\omega} \leq \sum_{k=0}^{+\infty} \rho^{-k} \| \mathcal{S}_m - 
\mathcal{S}_{m,n_0} \|_{\omega}^{k} \leq 2.
$$
Using (\ref{eq:bound-norm-g}), it is easy to check that $G_m(\nu,\theta)$ defined 
by $G_m(\nu,\theta) = F(\mathcal{S}_{m,n_0}(\nu) + \rho e^{i \theta})$ verifies
$$
\sup_{m,\theta,\nu} \| G_m(\nu,\theta)\|_{\omega} \leq \kappa
$$
for each $\gamma \leq \gamma_0$ for some nice constant $\kappa$. We thus obtain 
that for some nice constant $\kappa$, it holds that 
$$
\left \|  \frac{F(\mathcal{S}_{m,n_0}(\nu) + \rho e^{i \theta})}{\mathcal{S}_m(\nu) - 
\mathcal{S}_{m,n_0}(\nu) - \rho e^{i \theta}} \, \rho \, e^{i \theta} \right \|_{\omega}
\leq \kappa
$$
for each $\gamma < \gamma_0$, each $m$ and each  $\theta$. (\ref{eq:cauchy-formula}) 
thus implies (\ref{eq:uniform-wiener-levy}). The proof of Lemma \ref{le:wiener-levy-uniform}
is thus complete. %$\blacksquare$ \\
\end{proof}

The use of Lemma \ref{le:wiener-levy-uniform} for $f(x) = \log x$ shows that 
\begin{equation}
    \label{eq:uniform-control-norm-c}
    \sup_{m} \| c_m \|_{\omega} < +\infty
\end{equation}
for each $\gamma < \gamma_0$. Therefore, (\ref{eq:upperbound-norm-pi}, \ref{eq:upperbound-norm-psi}) imply that 
\begin{equation}
    \label{eq:control-norm-psi-pi}
\sup_{m}  \|\pi^{(m)}\|_{\omega} \leq \kappa, \,\,\,  \sup_{m}  \|\psi^{(m)}\|_{\omega} \leq \kappa.
\end{equation}

It also holds that $\mathcal{S}_m(\nu) = \left| \psi^{(m)}(e^{2 i \pi \nu}) \right|^{2}$
and therefore $\psi^{(m)}(z)$ coincides with the outer spectral factor of $\mathcal{S}_m$
in the sense that both $\psi^{(m)}(z)$ and $\frac{1}{\psi^{(m)}(z)} = \pi^{(m)}(z)$  are analytic in the unit disc. Theorem 5.1.8 in \cite{barry-simon-book} leads to the conclusion that $\| \phi_l^{(m)*} - \pi^{(m)} \|_{\omega} \rightarrow 0$ when $l \rightarrow +\infty$, a result which implies that 
\begin{equation}
    \label{eq:uniform-convergence-m-fixed}
    \sup_{\nu} \left| \phi_l^{(m)*}(e^{2 i \pi \nu}) - \pi^{(m)}(e^{2 i \pi \nu}) \right| \rightarrow 0.
\end{equation}
Given the fact that $\mathcal{S}_m(\nu) = \left| \frac{1}{\pi^{(m)}(e^{2 i \pi \nu})} \right|^{2}$, 
(\ref{eq:upperbound-S}) and (\ref{eq:lowerbound-S}) allow us to conclude that 
\begin{equation}
    \label{eq:lower-upper-bounds-pi}
0 < \inf_{m} \inf_{\nu} |\pi^{(m)}(e^{2 i \pi \nu})| \leq \sup_{m} \sup_{\nu} |\pi^{(m)}(e^{2 i \pi \nu})| < +\infty. 
\end{equation}
Therefore, (\ref{eq:uniform-convergence-m-fixed}) leads to 
$\sup_{\nu} | \frac{1}{\pi^{(m)}(e^{2 i \pi \nu})} \phi_l^{(m)*}(e^{2 i \pi \nu}) - 1 | \rightarrow 0$, and to $\sup_{\nu} \left| |\frac{1}{\pi^{(m)}(e^{2 i \pi \nu})}|^{2} |\phi_l^{(m)*}(e^{2 i \pi \nu})|^{2} - 1 \right| \rightarrow 0$, or equivalently, to 
\begin{equation}
    \label{eq:uniform-convergence-m-fixed-2}
 \sup_{\nu} \left| \mathcal{S}_m(\nu) |\phi_l^{(m)*}(e^{2 i \pi \nu})|^{2} - 1 \right| \rightarrow 0.
 \end{equation}
 This, in turn, implies that 
 \begin{equation}
    \label{eq:uniform-convergence-m-fixed-3}
 \sup_{\nu} \left| \mathcal{S}_m(\nu) \frac{1}{L} \sum_{l=1}^{L} |\phi_l^{(m)*}(e^{2 i \pi \nu})|^{2} - 1 \right| \rightarrow 0
 \end{equation}
when $L \rightarrow +\infty$ as expected. In order to complete the proof of Lemma 
\ref{lem:orthogonal-polynomials}, we have thus to prove that (\ref{eq:uniform-convergence-m-fixed-3}) holds uniformly w.r.t. $m$, and to evaluate the rate of 
convergence. For this, we can follow the proof of Theorem 5.1.8 in \cite{barry-simon-book}, adapting the corresponding arguments to our particular context. 

Theorem 5.1.8 in \cite{barry-simon-book} follows from general results concerning Wiener-Hopf operators defined on the Wiener algebra $\ell_{1}$. As explained below, we will show that $\sup_{m} \| \phi_l^{(m)*} - \pi^{(m)} \|_{1} \rightarrow 0$, and will only use that $\sup_{m} \| r_ m\|_{\omega} < +\infty$ and $\sup_{m} \| c_ m\|_{\omega} < +\infty$ for each $\gamma < \gamma_0$ in order to obtain an upper bound of the above term. In the following, 
we denote by $C^{(m)}$ the operator defined on the Wiener algebra $\ell_{1}$ by 
$$
C^{(m)} a = \overline{r}_m * a
$$
where $\overline{r}_m$ is the sequence defined by $\overline{r}_m(n) = r_m(-n)$ for each $n \in \mathbb{Z}$. 
$C^{(m)}$ can alternatively be defined in the Fourier transform domain as the multiplication operator
$$
\sum_{n \in \mathbb{Z}} a(n) e^{2 i \pi n \nu} \rightarrow \mathcal{S}_m(\nu)  \sum_{n \in \mathbb{Z}} a(n) e^{2 i \pi n \nu}.
$$
It is well known that $\|C^{(m)}\|_1 = \|\overline{r}_m \|_1 = \|{r}_m \|_1$. As $\mathcal{S}_m(\nu) = |\psi^{(m)}(e^{2 i \pi \nu})|^{2}$, the operator 
$C^{(m)}$ can be factorized as $C^{(m)} = L^{(m)} U^{(m)} = U^{(m)} L^{(m)}$ where $U^{(m)}$ and $L^{(m)}$ represent the multiplication operators by $\psi^{(m)}(e^{2 i \pi \nu})$ and $\left(\psi^{(m)}(e^{2 i \pi \nu})\right)^{*}$ defined on $\ell_1$ respectively. We denote by $P_{+}$
the projection operator defined on $\ell_{1}$ by 
$$
P_{+}\left( \{a(n), n \in \mathbb{Z} \} \right) = \{a(n), n \geq 0 \}
$$
or equivalently in the Fourier transform domain by 
$$
P_{+}\left(\sum_{n \in \mathbb{Z}} a(n) e^{2 i \pi n \nu} \right) = \sum_{n=0}^{+\infty} a(n) e^{2 i \pi n \nu}.
$$
The operator $P_{-}$ is defined by $P_{-} = I - P_{+}$. The operator $U^{(m)}$ is called upper triangular in 
the sense that $P_{-} U^{(m)} P_{+} = 0$ while $L^{(m)}$ is lower triangular because $P_{+} L^{(m)} P_{-} = 0$. 
Moreover, as $\pi^{(m)} = \frac{1}{\psi^{(m)}}$ belongs to $\ell_{1}$ and $\pi^{(m)}(e^{2 i \pi  \nu}) = 
\sum_{n=0}^{+\infty} \pi^{(m)}(n) \, e^{2 i \pi n \nu}$, the operators $U^{(m)}$ and $L^{(m)}$ are invertible, 
and $(U^{(m)})^{-1}$ and $(L^{(m)})^{-1}$ are upper triangular and lower triangular respectively.
In the Fourier domain,  $(U^{(m)})^{-1}$ and $(L^{(m)})^{-1}$ correspond 
respectively to the multiplication operator by $\pi^{(m)}(e^{2 i \pi \nu})$ 
and  $(\pi^{(m)}(e^{2 i \pi \nu}))^{*}$ These properties 
imply that the factorization $C^{(m} = L^{(m)} U^{(m)} = U^{(m)} L^{(m)}$ is a Wiener-Hopf factorization. In the following, we denote by $T^{(m)}$ the Toeplitz operator defined 
on $\ell_1$ by 
\begin{equation}
    \label{eq:def:T}
    T^{(m)} = P_+ C^{(m)} P_+.
    \end{equation}
    It is clear that if $j \geq 0$ and if $\delta_j$ is the sequence 
    $\delta_j$ defined by $\delta_j(n) = \delta_{n-j}$, then, 
    $<\delta_i, T^{(m)} \delta_j>$, defined as $\left(T^{(m)} \delta_j\right)(i)$, is equal to   $r_m(j-i)$. Therefore, the matrix representation of $T^{(m)}$ in the basis $(\delta_j)_{j \geq 0}$ is the infinite matrix $\mathcal{R}_{m,\infty}^{T}$. 
Theorem 5.1.1 in \cite{barry-simon-book} implies that, considered as an operator defined on 
$\mathrm{Range}(P_+)$,  $T^{(m)}$ is invertible, i.e. that for each $a \in \mathrm{Range}(P_{+})$, there exists a unique $b \in \mathrm{Range}(P_{+})$ such that 
$T^{(m)} b = a$. $\left(T^{(m)}\right)^{-1} b$ is of course defined as $a$. If an element $a$ does not belong to $\mathrm{Range}(P_+)$, $\left(T^{(m)}\right)^{-1} a$ is defined as $\left(T^{(m)}\right)^{-1} P_{+} a$. We also 
notice that $\left(T^{(m)}\right)^{-1} = P_+ \left(U^{(m)}\right)^{-1} P_+ \left(L^{(m)}\right)^{-1} P_+$.  
For each $n \geq 1$, we denote by $Q_n$ the projection operator defined  by 
\begin{equation}
    \label{eq:def-Qn}
    Q_n\left( \{ a(l), l \in \mathbb{Z} \} \right) =  \{ a(l), 0 \leq l \leq n \}
\end{equation}
or equivalently by
$$
Q_n \left( \sum_{l \in \mathbb{Z}} a(l) e^{2 i \pi l \nu} \right) = 
\sum_{l=0}^{n} a(l) e^{2 i \pi l \nu}.
$$
We also introduce the truncated Toeplitz operator $T^{(m)}_n$  defined by
\begin{equation}
    \label{eq:def-Tn}
    T^{(m)}_n = Q_n C^{(m)} Q_n = Q_n T^{(m)} Q_n .
\end{equation}
We note that in the basis $(\delta_{j})_{j=0, \ldots, n}$, the matrix representation of 
$T^{(m)}_n$ is the matrix $\mathcal{R}_{m,n+1}^{T}$. We now introduce the projection operator $R_n$ defined by $R_n = P_+ - Q_n$, and state the following Lemma which appears as an immediate  consequence of Theorem 5.1.2 and Theorem 5.1.3 in \cite{barry-simon-book}. 
\begin{lemma}
\label{le:adaptation-baxter-lemma}
For each $n \geq 0$, it holds that $R_n L^{(m)} Q_n = R_n L^{-(m)} Q_n = Q_n U^{(m)} R_n = 
Q_n U^{-(m)} R_n = 0$. Moreover, there exists an integer $n_0$ independent of $m$  such that for each $n \geq n_0$, $T^{(m)}_n$, considered as an operator defined on $\mathrm{Range}(Q_n)$, is invertible, in the sense that for each $a \in \mathrm{Range}(Q_n)$, it exists a unique $b \in \mathrm{Range}(Q_n)$, defined as $(T^{(m)}_n)^{-1} a$,  such that $T^{(m)}_n b = a$. If $a \in \mathrm{Range}(P_+)$, $(T^{(m)}_n)^{-1} a$ is defined as 
$(T^{(m)}_n)^{-1} a = (T^{(m)}_n)^{-1} Q_n a$. Moreover, there exists a nice constant $\alpha$ such that, for each $n \geq n_0$ and each $a \in \mathrm{Range}(P_+)$, the inequality 
\begin{equation}
    \label{eq:inequality-norm-inverse}
    \left\| \left(T^{(m)}_n\right)^{-1} a \right\|_1 \leq \alpha \,  \| a \|_1
\end{equation}
holds.
\end{lemma}
\begin{proof} We just verify that $R_n L^{(m)} Q_n = 0$, and omit the proof of the three other identities. For this, we have just to check that if $a(e^{2 i \pi \nu}) = \sum_{l=0}^{n} a(n) e^{2 i \pi l \nu}$, then $\left( \psi^{(m)}(e^{2 i \pi \nu}) \right)^{*} a(e^{2 i \pi \nu})$  
can be written as
$$
\left( \psi^{(m)}(e^{2 i \pi \nu}) \right)^{*} a(e^{2 i \pi \nu}) = \sum_{l=-\infty}^{n} 
b(l) e^{2 i \pi l \nu}
$$
for some coefficients $(b(l))_{l=-\infty, \ldots, n}$. This, of course, holds true because 
$\left( \psi^{(m)}(e^{2 i \pi \nu}) \right)^{*} = \sum_{l=0}^{\infty} (\psi^{(m)}(l))^{*} e^{- 2 i \pi l \nu}$. 

In order to be able to use Theorem 5.1.2 in \cite{barry-simon-book}, we establish that 
it exists an integer $n_0$ such that $\| P_{-} (L^{(m)})^{-1} R_n U^{(m)} \|_1\leq \frac{1}{2}$ and
$\| R_n (U^{(m)})^{-1} P_{-} L^{(m)} \|_1 \leq \frac{1}{2}$ for each $n \geq n_0$ and for each $m$. 
If $a \in \ell_{1}$, we evaluate $P_{-} (L^{(m)})^{-1} R_n U^{(m)} a$ in the Fourier transform domain, and denote $x^{(m)}_n(e^{2 i \pi \nu})$ the function defined by 
$x^{(m)}_n(e^{2 i \pi \nu}) = R_n \psi^{(m)}(e^{2 i \pi \nu}) a(e^{2 i \pi \nu})$, which, of course, can be written as 
$x^{(m)}_n(e^{2 i \pi \nu}) = \sum_{l=n+1}^{+\infty} x^{(m)}_n(l) e^{2 i \pi l \nu}$. The operation of $(L^{(m)})^{-1}$ is equivalent to the multiplication by $(\pi^{(m)}(e^{2 i \pi \nu}))^{*}$ in the Fourier transform domain, which is associated to a left-sided series. Therefore, 
\[
P_{-} \left(\pi^{(m)}(e^{2 i \pi \nu})\right)^{*} x^{(m)}_n(e^{2 i \pi \nu}) = P_{-} \left[ \sum_{l=n+1}^{+\infty} \left( \pi^{(m)}(l) \right)^{*} e^{-2 i \pi l \nu} x^{(m)}_n(e^{2 i \pi \nu}) \right].
\]
The norm of the right hand side can be bounded as
$$
\left \| P_{-} \left[\sum_{l=n+1}^{+\infty} \left( \pi^{(m)}(l) \right)^{*} e^{-2 i \pi l \nu} x^{(m)}_n(e^{2 i \pi \nu}) \right] \right \|_{1} \leq \left\| \sum_{l=n+1}^{+\infty} \left( \pi^{(m)}(l) \right)^{*} e^{-2 i \pi l \nu} \right\|_{1} \| \psi^{(m)} \|_1 \| a \|_{1}
$$
or equivalently, 
$$
\left\| P_{-} (L^{(m)})^{-1} R_n U^{(m)} \right\|_{1} \leq \left( \sum_{l=n+1}^{+\infty} |\pi^{(m)}(l)| \right) \, \| \psi^{(m)} \|_{1}.
$$
The bound in (\ref{eq:control-norm-psi-pi}) implies that $\sup_{m} \| \psi^{(m)} \|_{1} \leq \kappa$ and that 
$\sup_{m} \| \pi^{(m)} \|_{\omega} \leq \kappa$ for some nice constant $\kappa$. It is therefore clear that for each $\gamma < \gamma_0$ and for each $m$, we have
$$
\kappa \geq \| \pi^{(m)} \|_{\omega} \geq \sum_{l=n+1}^{+\infty} (1 + l)^{\gamma} | \pi^{(m)}(l) | \geq (1 + n)^{\gamma} 
\sum_{l=n+1}^{+\infty} | \pi^{(m)}(l) |.
$$
We conclude from this that 
\begin{equation}
    \label{eq:control-reminder-pi}
    \sum_{l=n+1}^{+\infty} | \pi^{(m)}(l) | \leq \frac{\kappa}{n^{\gamma}}
\end{equation}
and therefore
\begin{equation}
    \label{eq:control-norm-P_L-1RU}
  \left\| P_{-} (L^{(m)})^{-1} R_n U^{(m)} \right\|_{1} \leq \frac{\kappa}{n^{\gamma}}  
\end{equation}
for some nice constant $\kappa$. It can be shown similarly that 
\begin{equation}
    \label{eq:control-norm-RU-1P_L}
  \left\| R_n (U^{(m)})^{-1} P_{-} L^{(m)} \right\|_1  \leq \frac{\kappa}{n^{\gamma}}.
\end{equation}
This implies that it exists an integer $n_0$ such that $\| P_{-} (L^{(m)})^{-1} R_n U^{(m)} \|_1\leq \frac{1}{2}$ and
$\| R_n (U^{(m)})^{-1} P_{-} L^{(m)} \|_1 \leq \frac{1}{2}$ for each $n \geq n_0$ and for each $m$. Therefore, Theorem 5.1.2 in \cite{barry-simon-book}
implies that for each $n \geq n_0$ and for each $m$, it holds that $T^{(m)}_n$ is invertible and that for each $a \in \mathrm{Range}(Q_n)$, it holds that $\| (T^{(m)}_n)^{-1} a \| \leq \alpha_{m,n} \| a \|_1$ where $\alpha_{m,n}$ is given by 
\begin{align*}
\alpha_{m,n} &= \left\| (L^{(m)})^{-1} (U^{(m)})^{-1} \right\|_1 \\
&+ 2 \max \left( \left\| (U^{(m)})^{-1} \right\|_1,  \left\| (L^{(m)})^{-1} \right\|_1 \right) \left( \left\| P_{-} (L^{(m)})^{-1} \right\|_1 + \left\| R_n (U^{(m)})^{-1} \right\|_1 \right).
\end{align*}
The bounds in (\ref{eq:control-norm-psi-pi}) imply that for each $m$ and $n$, $\alpha_{m,n} \leq \alpha$ for some nice constant $\alpha$. 
Therefore, $\| (T^{(m)}_n)^{-1} a \| \leq \alpha \| a \|_1$ for each $n \geq n_0$, for each $m$, and for each $a \in \mathrm{Range}(Q_n)$. If $a \in \mathrm{Range}(P_+)$,  $(T^{(m)}_n)^{-1} a$ is equal to 
$(T^{(m)}_n)^{-1} Q_n a$. Therefore, $\| (T^{(m)}_n)^{-1} a \|_1 \leq \alpha \|Q_n a\|_1 \leq \alpha \|a\|_1$. This completes
the proof of the lemma. %$\blacksquare$ \\
\end{proof}

Lemma \ref{le:adaptation-baxter-lemma} and Theorem 5.1.3 in \cite{barry-simon-book} imply the 
following corollary. 
\begin{corollary}
\label{cor:convergence-inverse}
For each integer $m$ and for each $a \in \mathrm{Range}(P_+)$, it holds that 
\begin{equation}
    \label{eq:convergence-inverse}
    \lim_{n \rightarrow +\infty} \left\| (T_n^{(m)})^{-1} a - (T^{(m)})^{-1} a \right\|_1 = 0.
\end{equation}
\end{corollary}
\begin{proof}
(\ref{eq:inequality-norm-inverse}) implies that $T_n^{(m)}$ is invertible for each $n \geq n_0$. We use the observation that $ (T_n^{(m)})^{-1}  T_n^{(m)} = Q_n$. Therefore, 
the operator $(T_n^{(m)})^{-1} - (T^{(m)})^{-1}$ can be written as
$$
(T_n^{(m)})^{-1} - (T^{(m)})^{-1} = (T^{(m)}_n)^{-1} \left( T^{(m)} - T_n^{(m)} \right)  (T^{(m)})^{-1} +
(Q_n - I)  (T^{(m)})^{-1}.
$$
We conclude from this and (\ref{eq:inequality-norm-inverse}) that for each $n \geq n_0$, 
it holds that 
\begin{equation}
    \label{eq:control-1}
\| (T_n^{(m)})^{-1} a - (T^{(m)})^{-1} a \|_1 \leq \alpha \, \| ( T^{(m)} - T_n^{(m)}) (T^{(m)})^{-1} a \|_1 
+ \| (T^{(m)})^{-1} a - Q_n  (T^{(m)})^{-1} a \|_1.
\end{equation}
It is clear that $\| (T^{(m)})^{-1} a - Q_n  (T^{(m)})^{-1} a \|_1 \rightarrow 0$ when $n \rightarrow +\infty$. Moreover, for each $b \in \mathrm{Range}(P_+)$, $ ( T^{(m)} - T_n^{(m)}) \, b$ can be expressed as 
\begin{equation}
\label{eq:control-2}
( T^{(m)} - T_n^{(m)}) \, b = - \left( Q_n C^{(m)} \left( Q_n - P_{+} \right) \, b +  \left( Q_n - P_{+} \right)
C^{(m)} P_{+} b  \right).
\end{equation}
From this, we obtain immediately that for each $m$, $\|( T^{(m)} - T_n^{(m)})\, b \|_1 \rightarrow 0$
when $n \rightarrow +\infty$. Taking $b =  (T^{(m)})^{-1} a$ leads to (\ref{eq:convergence-inverse}). %$\blacksquare$ \\
\end{proof}

Corollary \ref{cor:convergence-inverse} implies that for each $m$, $\|(T_n^{(m)})^{-1} \delta_0 - (T^{(m)})^{-1} \delta_0 \|_{1}$ converges towards $0$ when $n \rightarrow +\infty$. Since the matrix representation 
of $T^{(m)}_n$ in the basis $(\delta_j)_{j=0, \ldots, n}$ coincides with matrix $\mathcal{R}_{m,n+1}^{T}$, (\ref{eq:yule-walker}) implies that $(T_n^{(m)})^{-1} \delta_0$
coincides with the sequence $\frac{1}{\sigma^{m}_n} (1, a^{(m)}_{1,n}, \ldots, a^{(m)}_{n,n}, 0, \ldots)$ whose Fourier transform coincides with 
$\frac{1}{\sigma^{m}_n} \phi^{*(m)}_n(e^{2 i \pi \nu})$. Therefore, the Fourier 
transform of the $\ell_1$ sequence  $(T^{(m)})^{-1} \delta_0$ is the limit of 
$\frac{1}{\sigma^{m}_n} \phi^{*(m)}_n(e^{2 i \pi \nu})$ in the $\ell_1$ metric. 
Theorem 5.1.8 in \cite{barry-simon-book} implies that for each $\gamma < \gamma_0$ and for each $m$, $\| \phi_n^{(m)*} - \pi^{(m)} \|_{\omega} \rightarrow 0$, and therefore that  $\| \phi_n^{(m)*} - \pi^{(m)} \|_{1} \rightarrow 0$ as $n \rightarrow +\infty$. As it is well known that $\sigma^{m}_n \rightarrow \sigma^{m} = \exp \frac{c_m(0)}{2}$, 
this discussion leads to the conclusion that for each $m$, 
\begin{equation}
    %\label{eq:expre-T-1-\delta0}
    (T^{(m)})^{-1} \delta_0 = \frac{1}{\sigma^{m}} \, \pi^{(m)}.
\end{equation}
In the following, we establish the following proposition. 
\begin{proposition}
\label{prop:uniform-convergence-T-m_n-delta0}
If $\gamma < \gamma_0$, there exist an integer $n_1$ and a nice constant $\kappa$ such that 
\begin{equation}
    \label{eq:uniform-convergence-T-m_n-delta0}
    \sup_{m \geq 1} \| (T_n^{(m)})^{-1} \delta_0 - (T^{(m)})^{-1} \delta_0 \|_{1} \leq \frac{\kappa}{n^{\gamma}}
\end{equation}
for each $n \geq n_1$.
\end{proposition}
\begin{proof}
In order to establish (\ref{eq:uniform-convergence-T-m_n-delta0}), we use (\ref{eq:control-1}) and (\ref{eq:control-2}) for $a = \delta_0$ and $b = (T^{(m)})^{-1} \delta_0 = \frac{1}{\sigma^{m}} \, \pi^{(m)}$. We first evaluate $\| (T^{(m)})^{-1} \delta_0 - Q_n (T^{(m)})^{-1} \delta_0 \|_1$, or equivalently $\frac{1}{\sigma^{m}} \sum_{k=n+1}^{+\infty} |\pi^{(m)}(n)|$. In order to 
check that $\sup_{m} \frac{1}{\sigma^{m}} < +\infty $, we notice that (\ref{eq:lowerbound-S}) 
implies that $\inf_{m} c_0(m) > -\infty$, and that $\inf_{m} \exp \frac{c_0(m)}{2} > 0$. 
Therefore, it holds that  $\sup_{m} \frac{1}{\sigma^{m}} < +\infty $. The bound in (\ref{eq:control-reminder-pi})
thus implies that for each $n \geq n_0$ and for each $m$, it holds that 
$$
\left\| (T^{(m)})^{-1} \delta_0 - Q_n (T^{(m)})^{-1} \delta_0 \right\|_1 \leq \frac{\kappa}{n^{\gamma}}
$$
for some nice constant $\kappa$. It remains to control $\| (T^{(m)} - T^{(m)}_n)  (T^{(m)})^{-1} \delta_0 \|_1 $. As $\sup_{m} \frac{1}{\sigma^{m}} < +\infty$, it is sufficient to study 
$\| (T^{(m)} - T^{(m)}_n) \pi^{(m)} \|_1$. For this, we use (\ref{eq:control-2}) for 
$b = \pi^{(m)}$, and obtain that 
\begin{equation}
    \label{eq:control-3}
\left\| (T^{(m)} - T^{(m)}_n) \pi^{(m)} \right\|_1 \leq \left\| C^{(m)} \right\|_1 \left\| \pi^{(m)} - Q_n \pi^{(m)} \right\|_1 + \left\| (P_+ - Q_n) C^{(m)} \pi^{(m)} \right\|_1.
\end{equation}
The bound in (\ref{eq:control-reminder-pi}) implies that the first term of the right hand side of (\ref{eq:control-3}) is upper bounded by $\frac{\kappa}{n^{\gamma}}$ for some nice 
constant $\kappa$ for each $n$ and each $m$. The second term of the right hand side 
of (\ref{eq:control-3}) is given by
$$
\left\| (P_+ - Q_n) C^{(m)} \pi^{(m)} \right\|_1 = \sum_{k=n+1}^{\infty} \left| \left(C^{(m)} \pi^{(m)}\right)(k) \right|
$$
where it holds that 
$$
\left(C^{(m)} \pi^{(m)}\right)(k) = \sum_{l=0}^{+\infty} \overline{r}_m(k-l) \, \pi^{(m)}(l).
$$
Therefore, 
$$
\sum_{k=n+1}^{\infty} \left| \left(C^{(m)} \pi^{(m)}\right)(k) \right| \leq \sum_{k=n+1}^{+\infty} 
\sum_{l=0}^{+\infty} |r_m(k-l)| |\pi^{(m)}(l)|.
$$
We express the right hand side of the above inequality as 
$$
\sum_{k=n+1}^{+\infty} 
\sum_{l=0}^{+k} |r_m(k-l)| |\pi^{(m)}(l)| + \sum_{k=n+1}^{+\infty}  \sum_{l=k+1}^{+\infty} |r_m(k-l)| |\pi^{(m)}(l)|
$$
or equivalently as 
$$
\sum_{k=n+1}^{+\infty}  \sum_{u+v=k, u \geq 0, v \geq 0} |r_m(u)| |\pi^{(m)}(v)| + 
\sum_{k=n+1}^{+\infty}  \sum_{u+v=k, u \leq -1, v \geq 0} |r_m(u)| |\pi^{(m)}(v)|.
$$
It is clear that 
\begin{multline*}
\sum_{k=n+1}^{+\infty}  \sum_{u+v=k, u \geq 0, v \geq 0} |r_m(u)| |\pi^{(m)}(v)| \leq 
\left( \sum_{l=0}^{+\infty} |\pi^{(m)}(l)| \right) \left( \sum_{k=[(n+1)/2]}^{+\infty} |r_m(k)| \right) + \\
\left( \sum_{k=0}^{+\infty} |r_m(k)|  \right) \left( \sum_{l=[(n+1)/2]}^{+\infty}  |\pi^{(m)}(l)| \right)
\end{multline*}
and that
$$
\sum_{k=n+1}^{+\infty}  \sum_{u+v=k, u \leq -1, v \geq 0} |r_m(u)| |\pi^{(m)}(v)| \leq 
\left( \sum_{k \leq -1} |r_m(k)|  \right) \left( \sum_{l=n+1}^{+\infty}  |\pi^{(m)}(l)| \right).
$$
Using the fact that that $\sup_{m} \| r_m\|_{\omega} < +\infty$, we obtain, using the same arguments as in (\ref{eq:control-reminder-pi}), that 
$$
\sup_{m} \sum_{l=n+1}^{+\infty}  |\pi^{(m)}(l)| < \frac{\kappa}{n^{\gamma}}
$$
for some nice constant $\kappa$. We have thus shown that
$$
\sup_{m} \| (P_+ - Q_n) C^{(m)} \pi^{(m)} \|_1  \leq  \frac{\kappa}{n^{\gamma}}
$$
and this completes the proof of Proposition \ref{prop:uniform-convergence-T-m_n-delta0}. %$\blacksquare$ \\
\end{proof}

Proposition \ref{prop:uniform-convergence-T-m_n-delta0} immediately allows 
to study the behaviour of $\| \phi^{*(m)}_n - \pi^{(m)}\|_1$ when $n \rightarrow +\infty$.
\begin{corollary}
If $\gamma < \gamma_0$, it exists an integer $n_2$ and a nice constant $\kappa$ for which 
\begin{equation}
    \label{eq:control-phi-pi}
 \| \phi^{(m)*}_n - \pi^{(m)}\|_1 \leq \frac{\kappa}{n^{\gamma}}   
\end{equation}
for each $n \geq n_2$ and each $m$.
\end{corollary}
\begin{proof}
$\phi^{(m)*}_n - \pi^{(m)} $ coincides with $ \sigma^{m}_n (T^{(m)}_n)^{-1} \delta_0 - \sigma^{m} (T^{(m)})^{-1} \delta_0 $, which can also be written as 
$$
\phi^{(m)*}_n - \pi^{(m)}  = \sigma^{m}_n \left( (T^{(m)}_n)^{-1} \delta_0 - (T^{(m)})^{-1} \delta_0 \right) + ( \sigma^{m}_n - \sigma^{m})  (T^{(m)})^{-1} \delta_0
$$
or equivalently as 
\begin{equation}
\label{eq:expre-phi*- pi}
\phi^{(m)*}_n - \pi^{(m)}  = \sigma^{m}_n \left( (T^{(m)}_n)^{-1} \delta_0 - (T^{(m)})^{-1} \delta_0 \right) +  (\sigma^{m}_n - \sigma^{m}) \frac{\pi^{(m)}}{\sigma^{m}}.
\end{equation}
We notice that $\sigma^{m}_n = \langle (T^{(m)}_n)^{-1} \delta_0, \delta_0 \rangle^{-1}$ and that 
$\sigma^{m} = \langle (T^{(m)})^{-1} \delta_0, \delta_0 \rangle^{-1}$. We express $\sigma^{m}_n - \sigma^{m}$
as 
$$
\sigma^{m}_n - \sigma^{m} = \sigma^{m}_n \sigma^{m} \left( \frac{1}{\sigma^{m}} -  \frac{1}{\sigma^{m}_n} \right) = \sigma^{m}_n \sigma^{m} \left\langle (T^{(m)}_n)^{-1} \delta_0 -(T^{(m)})^{-1} \delta_0), \delta_0 \right\rangle.
$$
Noting that $\sup_{m,n} \sigma^{m}_n \leq \sup_{m} r_0(m) < +\infty$, we obtain that 
for each $n$ large enough and for each $m$, the inequality 
$$
\sigma^{m}_n - \sigma^{m} \leq \kappa \| (T^{(m)}_n)^{-1} \delta_0 -(T^{(m)})^{-1} \delta_0\|_1 \leq \frac{\kappa}{n^{\gamma}}
$$
holds for some nice constant $\kappa$. (\ref{eq:control-phi-pi}) thus follows immediately 
from Proposition \ref{prop:uniform-convergence-T-m_n-delta0}. 
%$\blacksquare$ \\
 \end{proof}

We finally complete the proof of Lemma  \ref{lem:orthogonal-polynomials}. 
(\ref{eq:control-phi-pi}) implies that 
$$
\sup_{m} \sup_{\nu} |\phi^{(m)*}_n(e^{2 i \pi \nu}) - \pi^{(m)}(e^{2 i \pi \nu})| \leq \frac{\kappa}{n^{\gamma}}
$$
for each $n \geq n_2$. Using (\ref{eq:lower-upper-bounds-pi}) and 
$\mathcal{S}_m(\nu) = \frac{1}{|\pi^{(m)}(e^{2 i \pi \nu})|^{2}}$, we obtain that 
\begin{equation}
\label{eq:uniform-cv-S-phi2}
\sup_{m} \sup_{\nu} \left|\mathcal{S}_m(\nu) |\phi^{(m)*}_n(e^{2 i \pi \nu})|^{2} - 1 \right| \leq \frac{\kappa}{n^{\gamma}}
\end{equation}
for  each $n \geq n_2$. We recall that $\epsilon_m(\nu)$ is equal to 
$$
\epsilon_{m,L}(\nu) = \frac{1}{L} \sum_{n=0}^{L-1} \mathcal{S}_m(\nu) |\phi_{n}^{(m)*}(e^{2 i \pi \nu})|^{2} -1.
$$
Therefore, 
$$
|\epsilon_{m,L}(\nu)| \leq \frac{1}{L} \sum_{n=0}^{L-1} \left| \mathcal{S}_m(\nu) |\phi^{(m)*}_n(e^{2 i \pi \nu})|^{2} - 1 \right|.
$$
We express the right hand side as 
$$
 \frac{1}{L} \sum_{n=0}^{n_2-1} \left| \mathcal{S}_m(\nu) |\phi^{(m)*}_n(e^{2 i \pi \nu})|^{2} - 1 \right| + \frac{1}{L} \sum_{n=n_2}^{L} \left| \mathcal{S}_m(\nu) |\phi^{(m)*}_n(e^{2 i \pi \nu})|^{2} - 1 \right|
 $$
 and handle the two terms separately. On the one hand, (\ref{eq:uniform-cv-S-phi2}) implies that 
 $$
\frac{1}{L} \sum_{n=n_2}^{L} \left| \mathcal{S}_m(\nu) |\phi^{(m)*}_n(e^{2 i \pi \nu})|^{2} - 1 \right| \leq \kappa \, \frac{1}{L} \sum_{n=n_2}^{L} \frac{1}{n^{\gamma}}.
$$
If $\gamma > 1$, $\sum_{n=n_2}^{L} \frac{1}{n^{\gamma}}$ is a bounded term, and we obtain 
that 
 $$
\sup_{m} \sup_{\nu} \frac{1}{L} \sum_{n=n_2}^{L} \left| \mathcal{S}_m(\nu) |\phi^{(m)*}_n(e^{2 i \pi \nu})|^{2} - 1 \right| \leq \frac{\kappa}{L}.
 $$
If $\gamma = 1$, the above term is bounded by $\kappa \, \frac{\log L}{L}$, and if 
$0 < \gamma < 1$, it holds that 
$$
\sum_{n=n_2}^{L} \frac{1}{n^{\gamma}} \leq \kappa \, L^{1 - \gamma}
$$
and that 
$$
\sup_{m} \sup_{\nu} \frac{1}{L} \sum_{n=n_2}^{L} \left| \mathcal{S}_m(\nu) |\phi^{(m)*}_n(e^{2 i \pi \nu})|^{2} - 1 \right| \leq \frac{\kappa}{L^{\gamma}}.
$$
We finally justify that there exists a nice constant $\kappa$ such that 
$$
\sup_{m} \sup_{\nu} \, \sum_{n=0}^{n_2-1} \left| \mathcal{S}_m(\nu) |\phi^{(m)*}_n(e^{2 i \pi \nu})|^{2} - 1 \right| \leq \kappa.
$$
Indeed, since $n_2$ is a fixed integer, we have just to verify that for each $n \leq n_2$, 
$\sup_{m} \sup_{\nu} |\phi^{(m)*}_n(e^{2 i \pi \nu})| < +\infty$. For this, we recall that 
the non normalized polynomials $\Phi^{(m)}_n$ and $\Phi^{(m)*}_n$ verify the relation 
 the well known recursion formula
\begin{eqnarray}
\label{eq:levinson-1}
\Phi_{n+1}^{(m)}(z) & = & z \, \Phi_n^{(m)}(z) \, - \, \alpha^{(m)}_n \, \Phi_n^{(m)*}(z) \\
\label{eq:levinson-2}
\Phi_{n+1}^{(m)*}(z) & = & \Phi_{n}^{(m)*}(z) \, - \, \alpha^{(m)*}_n \, z \, \Phi_n^{(m)}(z).
\end{eqnarray}
Here, $(\alpha_m(n))_{n \geq 0}$ are the reflection coefficients sequence associated to 
autocovariance $(r_m(n))_{n \in \mathbb{Z}}$, also  called in  \cite{barry-simon-book} the Verblunsky coefficients. For each $n$, it holds that $|\alpha_m(n)| < 1$. 
It is obvious that $\| \Phi_n^{(m)} \|_1 = \| \Phi_n^{(m)*} \|_1 $. Therefore, 
(\ref{eq:levinson-1}) implies that 
$$
 \| \Phi_{n+1}^{(m)*} \|_1 \leq (1+ |\alpha_m(n)|)  \| \Phi_n^{(m)*} \|_1 \leq 2  \| \Phi_n^{(m)*} \|_1.
$$
Noting that $\| \Phi_0^{(m)*} \|_1 = 1$, we obtain that
$ \| \Phi_n^{(m)*} \|_1 \leq 2^{n}$, and that $\sup_{m} \sup_{\nu}  |\Phi_n^{(m)*}(e^{2 i \pi \nu})| \leq 2^{n}$. As $\inf_{m,n} \sigma^{m}_n > 0$, the normalized polynomials 
verify $\sup_{m} \sup_{\nu} |\phi^{(m)*}_n(e^{2 i \pi \nu})| < +\infty$. This completes the proof of 
Lemma \ref{lem:orthogonal-polynomials}.


\begin{thebibliography}{99}


\bibitem{alpay-tsekanovskii} D. Alpay, E. Tsekanovskii, ``Subclasses of Herglotz-Nevanlinna matrix-valued functions
and linear systems``, Proc. Int. Conf. on Dynamical Systems and Differential Equations, May 18-21, 2000, Atlanta, USA. 

\bibitem{and-gui-zei-2010} G.W. Anderson, A. Guionnet, O. Zeitouni, ``An Introduction to Random Matrices", {\sl Cambridge 
Studies in Advanced Mathematics}, vol. 118, Cambridge University Press, 2010. 

\bibitem{anderson-2013} G.W. Anderson, "Convergence of the largest singular value of a polynomial in independent Wigner matrices", Annals of Proba., 2013, vol. 41, No. 3B, 2103-2181. 

\bibitem{batta-bose-2016} M. Bhattachargee, A. Bose, ``Large sample behaviour of high-dimensional autocovariance matrices'',
  Ann. of Stat., vol. 44, no. 2, pp. 598-628, 2016a.

\bibitem{brillinger} D.R. Brillinger, ``Time Series, Data Analysis and Theory", Classics in 
Applied Mathematics 36, SIAM, 2001.

\bibitem{capitaine2007freeness} M. Capitaine, C. Donati-Martin, ``Strong asymptotic freeness of Wigner and Wishart matrices", {\sl Indiana Univ. Math. Journal}, 
vol. 25, pp. 295-309, 2007

\bibitem{dette-dornemann-2020} H. Dette, N. Dörnemann, "Likelihood ratio tests for 
many groups in high dimensions", J. Mult. Anal., 178 (2020) 104605. 

\bibitem{duchesne2003robust}  P. Duchesne, R. Roy, 
``Robust tests for independence of two time series",
Statistica Sinica, pp. 827--852, 2003.


\bibitem{eichler2007frequency} M. Eichler, ``A frequency-domain based test for non-correlation between stationary time series",
Metrika, vol. 65, no. 2, pp. 133--157, 2007.
 
 \bibitem{eichler2008testing}  M. Eichler, ``Testing nonparametric and semiparametric hypotheses in vector stationary processes",
J. of Multi.  Anal., vol. 99, no. 5, pp. 968--1009, 2008. 
 
\bibitem{elhimdiduchesneroy2003} K. El Himdi, R. Roy, P. Duchesne, 
``Tests for non-norrelation of two multivariate time series: A nonparametric approach", Lecture Notes-Monograph Series, vol. 42,
pp. 397--416, 2003, Mathematical Statistics and Applications: Festschrift for Constance van Eeden.

\bibitem{erdos13} L. Erd\"os and A. Knowled and {H.T.} Yao, 
``Averaging Fluctuations in Resolvents of Random Band Matrices'', Ann. Henri Poincaré, vol. 14, pp. 1837–1926,  2013.

\bibitem{geronimus} Ya. L. Geronimus, 
``Polynomials Orthogonal on a Circle and Interval", 
Pergamon Press, 1960. 

\bibitem{haagerupnew2005} U. Haagerup, S. Thorbjornsen, ``A new application of random matrices: $\mathrm{Ext}(C^*_{red}(F_2))$ is not a group", 
{\sl Annals of Mathematics}, vol. 162, no. 2, 2005. 

\bibitem{hachem-loubaton-najim-aap-2007} W. Hachem, P. Loubaton, J. Najim, ``Deterministic equivalents for certain functionals of large random matrices", Annals of Applied Probability, 17(3):875–930, 2007.

\bibitem{hachem-loubaton-najim-vallet-jmva-2013} W. Hachem, P. Loubaton, J. Najim, P. Vallet, ``On bilinear forms based on the resolvent of large random matrices", Annales de l'Institut Henri Poincaré, Prob. Stats., vol. 45, no. 1, 2013, pp. 36-63. 

\bibitem{haugh1976checking} L.D. Haugh, 
``Checking the independence of two covariance-stationary time series: a univariate residual cross-correlation approach", 
J.  of the Am. Stat.  Asso., vol. 71, no. 354,
pp. 378--385, 1976
 
\bibitem{hong1996testing} Y. Hong, 
``Testing for independence between two covariance stationary time series",
 Biometrika, vol. 83, no. 3, 1996
 

\bibitem{jiang2004} T. Jiang, ``The limiting distributions of eigenvalues of sample correlation matrices", Sankhya: The Indian Journal of Statistics, vol. 66, pp. 35-48, 2004.

\bibitem{jiangyang2013} T. Jiang, F. Yang, ``Central limit theorems for classical likelihood ratio tests for high-dimensional distributions", Ann. of Stats., vol. 41, 
no. 4, pp. 2029-2074, 2013.

\bibitem{jin-bai-el-al-2014} B. Jin, C. Wang, Z.D. Bai, K.K. Nair, M. Harding, ``Limited spectral distribution of a symmetrised auto-cross
  covariance matrices'', Ann. of Proba., vol. 24, no. 3, pp. 1199-1225, 2014.
  
\bibitem{kato} T. Kato, "Perturbation Theory for Linear Operators", Classics in Mathematics, 
Springer, Reprint of the second edition, 
1995.


\bibitem{kim2005test} E. Kim, S. Lee, 
``A test for independence of two stationary infinite order autoregressive processes",
Ann.  Inst. Stat. Math., vol. 57, no. 1, pp. 105--127, 2005.

\bibitem{li-pan-yao-jmva-2015} Z. Li, G. Pan, J. Yao, "On singular value distribution of large-dimensional autocovariance matrices",  J. Multivariate Analysis, vol. 137, pp. 119-140, May 2015. 

\bibitem{liu-aue-paul-2015} H. Liu, A. Aue, D. Paul, ``On the Marcenko-Pastur law for linear time series'', Ann. of Stat., vol. 43, no. 2, pp. 675-712, 2015.  

\bibitem{loubaton2016} P. Loubaton,  ``On the almost sure location of the singular values of certain Gaussian block-Hankel large random matrices", J. of Theoretical Probability, vol. 29, no. 4, pp. 1339-1443, December 2016

\bibitem{loubaton-mestre-2017} P. Loubaton, X. Mestre, ``Spectral convergence of large block-Hankel Gaussian random matrices", Colombo F., Sabadini I., Struppa D., Vajiac M. (eds) Advances in Complex Analysis and Operator Theory. Trends in Mathematics. Birkha{\"u}ser, Cham, 2017, available on Arxiv (arXiv:1704.06651).

\bibitem{loubaton-tieplova-2020} P. Loubaton, D. Tieplova, "On the behaviour of large autocovariance matrices
between the past and the future", to appear in Random Matrix: Theory and Applications, https://doi.org/10.1142/S2010326321500210. 

\bibitem{loubaton-rosuel} P. Loubaton, A. Rosuel, "Large random matrix approach for testing independence of a large number of Gaussian time series", preprint  arXiv:2007.08806.

\bibitem{mestre2017correlation} X. Mestre, P. Vallet, ``Correlation tests and linear spectral statistics of the sample correlation matrix", IEEE Trans. Info. Theory, vol. 63, no. 7, 2017.

\bibitem{najim-yao-2016} J. Najim,  J. Yao.
"Gaussian fluctuations for linear spectral statistics of large random covariance matrices" .
Annals of Applied Probability, vol. 26(3), 2016.

\bibitem{nowak17} M. A Nowak, W. Tarnowski,
``Spectra of large time-lagged correlation matrices from random matrix theory",
Journal of Statistical Mechanics: Theory and Experiment, vol. 2017, June 2017.

\bibitem{pastur-shcherbina-book} L. Pastur, M. Shcherbina,
``Eigenvalue Distribution of Large Random Matrices", Mathematical Surveys and Monographs, AMS, 2011.

\bibitem{rudelson2013hanson} M. Rudelson, R. Vershynin, "Hanson-Wright inequality and sub-gaussian concentration", Electronic Communications in Probability, vol. 18, 2013. 

\bibitem{rudin-book} W. Rudin, ``Real and Complex Analysis", Mc Graw Hill, 1966.

\bibitem{rozanov} Yu.A. Rozanov, "Stationary Random Processes", Holden Day, 1967. 


\bibitem{barry-simon-book} B. Simon, ``Orthogonal Polynomials on the Unit Circle, Part 1: Classical 
Theory", American Mathematical Society Colloquium 
Publications, Vol. 54, Part 1, 2005. 


\bibitem{tao-book} T. Tao, ``Topics in Random Matrix Theory", Graduate Studies in Mathematics 132, 
2011. 

\bibitem{taniguchi1996nonparametric} M. Taniguchi, M.L. Puri, M. Kondo, 
``Nonparametric approach for non-Gaussian vector stationary processes",
J. of Mult. Anal., vol. 56, no. 2, pp. 259--283, 1996. 

\bibitem{wahba1971some} G. Wahba, 
``Some tests of independence for stationary multivariate time series",
J. Royal Stat. Soc.: Series B (Methodological), vol. 33, no. 1, 1971

\bibitem{xiao-wu-2012} H. Xiao, W.B. Wu, ``Covariance matrix estimation for stationary time series", Ann. of Statistics, vol. 40, pp. 466-493, 2012.


\bibitem{zygmund} A. Zygmund, ``Trigonometric Series", Third Edition, Cambridge University Press, 
2002.

\end{thebibliography}
\end{document}